\documentclass[11pt,a4paper]{amsart}

\usepackage[margin=25mm]{geometry}
\usepackage[T1]{fontenc}
\usepackage{lmodern,microtype,xcolor}
\usepackage{amsmath,amssymb,amsthm,mathtools}
\usepackage{array,booktabs,longtable,needspace}
\usepackage{enumitem}
\usepackage[
  colorlinks=true,
  linkcolor=blue!45!black,
  urlcolor=blue!45!black,
  citecolor=blue!45!black,
  bookmarksnumbered=true,
  bookmarksopen=false,
  bookmarksdepth=3
]{hyperref}
\usepackage{bookmark}

\setlist{itemsep=3pt,topsep=5pt}
\numberwithin{equation}{section}

\newtheorem{theorem}{Theorem}[section]
\newtheorem{proposition}[theorem]{Proposition}
\newtheorem{lemma}[theorem]{Lemma}
\newtheorem{corollary}[theorem]{Corollary}

\theoremstyle{definition}

\theoremstyle{remark}

\newcommand{\R}{\mathbb R}

\newcommand{\Z}{\mathbb Z}
\newcommand{\one}{\mathbf1}
\newcommand{\ii}{\mathrm i}

\newcommand{\supp}{\operatorname{supp}}
\newcommand{\diam}{\operatorname{diam}}
\newcommand{\dist}{\operatorname{dist}}
\newcommand{\dd}{\mathsf d}

\newcommand{\TV}{\mathrm{TV}}
\newcommand{\np}{\mathfrak n_p}
\newcommand{\Cr}{C_{\rm rad}}
\newcommand{\qs}{q_*}
\newcommand{\qr}{q_{\rm ret}}
\newcommand{\CP}{C_P}
\newcommand{\ct}{c_\vartheta}

\newcommand{\D}{\mathcal D}
\newcommand{\C}{\mathcal C}
\newcommand{\M}{\mathcal M}
\newcommand{\G}{\mathcal G}

\newcommand{\Par}[2]{\mathrm{Par}[#1,#2]}
\newcommand{\Proj}[2]{\mathrm{Proj}[#1,#2]}

\newcommand{\code}[1]{{\small\nolinkurl{#1}}}
\newcolumntype{P}[1]{>{\raggedright\arraybackslash}p{#1}}

\title[A quantitative planar distance theorem]
{A quantitative form of the planar Falconer distance theorem}

\author[G. Cheng]{Guozheng Cheng}
\address{School of Mathematical Sciences\\
Dalian University of Technology\\
Dalian 116024\\
P. R. China}
\email{gzhcheng@dlut.edu.cn}

\author[X. Fang]{Xiang Fang}
\address{Department of Applied Mathematics\\
National Yang Ming Chiao Tung University\\
Hsinchu 30010\\
Taiwan}
\email{xfang@nycu.edu.tw}

\author[M. Li]{Menghan Li}
\address{School of Mathematical Sciences\\
Dalian University of Technology\\
Dalian 116024\\
P. R. China}
\email{lmh2025@mail.dlut.edu.cn}

\author[H. Qu]{Hongdou Qu}
\address{School of Mathematical Sciences\\
Dalian University of Technology\\
Dalian 116024\\
P. R. China}
\email{quhongdou0926@mail.dlut.edu.cn}

\author[C. Xiao]{Chengbo Xiao}
\address{School of Mathematical Sciences\\
Dalian University of Technology\\
Dalian 116024\\
P. R. China}
\email{xiaochengbo@mail.dlut.edu.cn}

\keywords{Falconer distance problem, Frostman measure, Fourier transform,
radial projection, dyadic decomposition}

\subjclass[2020]{28A75, 28A80, 42B10}

\date{}

\hypersetup{
 pdftitle={A quantitative form of the planar Falconer distance theorem},
 pdfauthor={Guozheng Cheng, Xiang Fang, Menghan Li, Hongdou Qu, Chengbo Xiao},
 pdfsubject={Uniform quantitative bounds for planar distance sets},
 pdfkeywords={Falconer distance problem, Frostman measure, Fourier transform, radial projection, dyadic decomposition}
}

\begin{document}

\begin{abstract}
In this paper we prove a uniform lower bound for the length of a planar distance set under a bound on the concentration of mass. For each $0<\eta<1$, we obtain $b(\eta)>0$ such that every compact $F\subset\overline B(0,1)$ carrying a Borel probability $\nu$ with $\nu(B(z,r))\le2r^{1+\eta}$ for all $z\in\R^2$ and $0<r\le1$ satisfies $|\Delta F|\ge b(\eta)$. For a compact planar set $E$ of diameter $D>0$ carrying a probability with $\mu(B(z,r))\le A(r/D)^{1+\varepsilon}$ for all centers and $0<r\le D$, where $A\ge1$ and $0<\varepsilon<1$, localization gives $$|\Delta E|\ge c_{\varepsilon,p}DA^{-p}$$ for each fixed $p>1/\varepsilon$, with $c_{\varepsilon,p}>0$ depending only on $\varepsilon$ and $p$. In particular, every compact planar set of Hausdorff dimension greater than one has a distance set of positive length. The proof uses scale-dependent cutoffs to construct approximants to a fixed positive distance measure. The same approximants have summable total-variation errors and summable integrals of the squared modulus of their Fourier transforms over the corresponding dyadic bands. We develop the multiscale method of OpenAI's release on October 6, 2026 in the plane and prove the estimates with constants uniform over the normalized family. This paper is for expository purpose only and is not intended for publication.
\end{abstract}

\maketitle
{\small\tableofcontents}
\clearpage

\section{Introduction and main statements}\label{main:introduction}
In this paper we prove a uniform lower bound for the length of a planar distance set when the concentration of mass is controlled. For a compact set $E\subset\R^2$, we write
\begin{equation}\label{main:distance-set}
 \Delta E=\{|x-y|:x,y\in E\}.
\end{equation}
This is a compact subset of $[0,\diam E]$, and $|\Delta E|$ denotes its one-dimensional Lebesgue measure. The planar positive-measure form of Falconer's distance problem asks whether $\dim_H E>1$ implies $|\Delta E|>0$.

Our quantitative question asks for a lower bound that works for an entire family of sets. We first fix the spatial scale and an upper bound on the mass of small balls. We then show how the resulting distance-set bound changes with the diameter and the concentration coefficient. Both pieces of information are needed. A similarity can make all distances arbitrarily small, and even at diameter one there are sets of Hausdorff dimension two whose distance sets have arbitrarily small length. The example in Subsection~\ref{main:concentration} explains the latter point.

We develop the multiscale method acknowledged in Subsection~\ref{main:credit}. The proof first constructs distance measures that serve both the total-variation and Fourier estimates, then uses their common bounds to choose a reconstruction cutoff uniformly. Localization translates the resulting normalized bound into one that depends on diameter and concentration.

\subsection{The normalized quantitative statement}\label{main:normalized-section}
We begin with a probability measure in a fixed disk. Let $0<\eta<1$, and suppose that a Borel probability $\nu$ carried by a compact set $F\subset\overline B(0,1)$ satisfies
\begin{equation}\label{main:normalized-ball}
 \nu(B(z,r))\le2r^{1+\eta}
 \qquad(z\in\R^2,\ 0<r\le1).
\end{equation}
Here and throughout,
\begin{equation}\label{main:balls}
 B(z,r)=\{x:|x-z|<r\},\qquad
 \overline B(z,r)=\{x:|x-z|\le r\}.
\end{equation}
The centers in~\eqref{main:normalized-ball} range over the whole plane. The condition limits concentration from above; it imposes no lower bound on the mass of a ball and allows measures singular with respect to planar area. We use the term \emph{Frostman bound} for such a ball inequality, with its coefficient, exponent, and radius range specified.

\begin{theorem}\label{main:normalized}
For each $0<\eta<1$, there is $b(\eta)>0$ such that every compact $F\subset\overline B(0,1)$ carrying a Borel probability satisfying~\eqref{main:normalized-ball} obeys
\[
 |\Delta F|\ge b(\eta).
\]
The same number $b(\eta)$ works for every such set and probability.
\end{theorem}

The probability normalization, the coefficient two, and the unit disk fix the quantities over which this conclusion is uniform. The exponent $1+\eta$ gives a fixed positive margin above the one-dimensional mass bound. The proof in Section~\ref{main:normalized-proof} gives a permissible choice of $b(\eta)$ in terms of common analytic constants. We do not evaluate its optimal dependence on $\eta$ as $\eta$ tends to zero.

\subsection{Dependence on diameter and concentration}\label{main:localized-section}
We next allow the spatial scale and the concentration coefficient to vary. Let $E\subset\R^2$ be compact with diameter $D>0$, and suppose that a Borel probability $\mu$ carried by $E$ satisfies
\begin{equation}\label{main:diameter-ball}
 \mu(B(z,r))\le A(r/D)^{1+\varepsilon}
 \qquad(z\in\R^2,\ 0<r\le D),
 \qquad A\ge1,
\end{equation}
where $0<\varepsilon<1$. The ratio $r/D$ makes this hypothesis unchanged by a common spatial similarity. The length of the distance set scales linearly with $D$.

\begin{theorem}\label{main:localized}
Under~\eqref{main:diameter-ball}, for every $0<\eta<\varepsilon$,
\begin{equation}\label{main:localized-bound}
 |\Delta E|\ge
 D(2A)^{-1/(\varepsilon-\eta)}b(\eta),
\end{equation}
where $b(\eta)$ is the same normalized constant as in Theorem~\ref{main:normalized}.
\end{theorem}

The localization argument in Section~\ref{main:localization} selects a closed ball on which the ratio of mass to the $(1+\eta)$th power of the radius is large. We restrict $\mu$ to that ball, divide by the mass retained, and rescale to the unit disk. The radius of the selected ball in the original coordinates is at least $D(2A)^{-1/(\varepsilon-\eta)}$, so Theorem~\ref{main:normalized} gives~\eqref{main:localized-bound}.

For each fixed $p>1/\varepsilon$, we may take $\eta=\varepsilon-1/p$. Corollary~\ref{loc:cor:power} then gives
\begin{equation}\label{main:power-bound}
 |\Delta E|\ge c_{\varepsilon,p}DA^{-p},\qquad
 c_{\varepsilon,p}=2^{-p}b(\varepsilon-1/p)>0.
\end{equation}
The constant depends on the fixed pair $(\varepsilon,p)$. The endpoint $p=1/\varepsilon$ would require the normalized statement at $\eta=0$ and is not included.

\subsection{Why the concentration coefficient is needed}\label{main:concentration}
Fixing dimension and diameter alone cannot give the uniform lower bound in Theorem~\ref{main:localized}. For $0<\delta<1/4$, consider
\[
 E_\delta=\overline B(0,\delta)\cup\{(1-\delta,0)\}.
\]
This set has diameter one and Hausdorff dimension two. Distances between disk points lie in $[0,2\delta]$, and distances from the isolated point to the disk lie in $[1-2\delta,1]$. Thus
\begin{equation}\label{main:concentration-example}
 \Delta E_\delta\subseteq[0,2\delta]\cup[1-2\delta,1],
 \qquad |\Delta E_\delta|\le4\delta\longrightarrow0.
\end{equation}

For any fixed $1<s<2$, normalized area measure on the disk satisfies
\[
 \mu_\delta(B(z,r))\le\min\{1,(r/\delta)^2\}
 \le\delta^{-s}r^s
 \qquad(z\in\R^2,\ r>0).
\]
The permissible coefficient $A=\delta^{-s}$ grows as the distance-set length decreases. This explains why the quantitative conclusion records concentration as well as dimension and diameter. It also illustrates the meaning of a measure \emph{carried by} a set: the probability is carried by $E_\delta$, although it gives no mass to the isolated point.

\subsection{The Hausdorff-dimension consequence}\label{main:dimension}
The quantitative theorem implies that every compact planar set of Hausdorff dimension greater than one has a distance set of positive length:
\[
 \dim_H E>1\quad\Longrightarrow\quad|\Delta E|>0
 \qquad(E\subset\R^2\text{ compact}).
\]
Indeed, for $1<s<\min\{\dim_H E,2\}$, Frostman's theorem~\cite[Theorem~2.7, p.~18]{Mattila2015} supplies a nonzero finite measure on $E$ with an $s$-power ball bound. After normalization, it satisfies the hypothesis of Theorem~\ref{main:localized} with $\varepsilon=s-1$ and a finite coefficient $A$. The normalization and the ball convention are recorded in Subsection~\ref{loc:sec:frostman}, where Corollary~\ref{loc:cor:dimension} gives the full argument. The coefficient can depend on $E$, as the preceding example requires.

\subsection{How the proof works}\label{main:reduction-overview}
The proof studies a distance set through a positive measure carried by it. Integration and the Fourier transform then give a way to estimate its length. We first explain the two estimates needed for this step and then describe how the measures are constructed.

\bigskip

\paragraph{\emph{Reconstruction from two estimates.}}
The argument in Section~\ref{main:reconstruction} starts with one fixed positive measure $\alpha$ on the line. For each sufficiently large integer $N$, we construct a finite measure $\tau_N$ whose Fourier transform can be estimated on the corresponding dyadic band. The two estimates give, for a sufficiently large integer $J$,
\[
 \sum_{N\ge J}\|\alpha-\tau_N\|_{\TV}<\infty,
 \qquad
 \sum_{N\ge J}\int_{2^{N-1}\le|r|\le2^{N+1}}
       |\widehat{\tau_N}(r)|^2\,dr<\infty.
\]
Here $\|\cdot\|_{\TV}$ is the total-variation norm, with the convention stated in Subsection~\ref{main:measure-conventions}. Both estimates concern the same measures $\tau_N$ and the same fixed target $\alpha$.

The reconstruction theorem yields
\[
 \alpha=(g+h)\,dt,\qquad g\in L^2(\R),\quad h\in L^1(\R).
\]
By choosing $J$ sufficiently large, we make $\|h\|_1$ small compared with the positive mass of $\alpha$. Cauchy--Schwarz then gives a lower bound for the length of the set carrying $\alpha$. The proof retains the sum of the Fourier integrals displayed above; it does not require the stronger summability of their square roots.

The approximating measure may depend on the scale. This allows us to remove a small part of the original pair measure when estimating a particular frequency band. The total-variation estimate ensures that these choices still recover the fixed target. To obtain a bound uniform over the normalized family, the prefactors and starting scale in both estimates must also be uniform. Corollary~\ref{rec:cor:uniform} records this dependence, and Section~\ref{main:normalized-proof} chooses one reconstruction cutoff for the whole family.

\bigskip

\paragraph{\emph{Preparation of the measures.}} \label{main:preparation-overview}
We first restrict the starting measure to two separated compact patches. Separation bounds the weight $|x-y|^{-1/2}$, makes the direction map continuous on the product of the supports, and puts the directions from one patch to the other in a common small arc. The preparation in Section~\ref{main:preparation} makes these bounds uniform.

The radial estimate supplies the needed control of directions. A radial projection records the direction of a point as seen from a chosen center. Finite energies of the source and center measures bound an average $L^p$ norm of the resulting direction densities, for one $p>1$. The proof gives a constant independent of the distance between the two disjoint compact supports.

The average estimate gives a finite density norm at almost every center. The later geometric estimates require a common bound at every retained center. We obtain compact sets of such centers from the lower semicontinuity of the extended density norm. The two restrictions retain a uniformly positive amount of mass. Dividing by their retained masses produces the prepared probabilities, and the preparation constants include the effect of this normalization.

\bigskip

\paragraph{\emph{The estimates across scales.}} \label{main:proof-route}
The prepared probabilities may give very different masses to squares of the same size. Section~\ref{main:dyadic} groups the finest dyadic squares into classes. Within each class a mass profile, which is a function of depth, gives model masses comparable to the actual square masses. The actual masses normalize the restrictions, and the original class weights recover the retained part of the prepared measure.

Section~\ref{main:profiles} uses the profiles to choose intervals of depths and assign condition indices to them. The geometric tests in Section~\ref{main:directions} bound how many occupied squares lie near a line and how often two points from a conditional measure have nearly equal projections. For each class pair, a common starting depth determines the indexed families of cutoff factors. Section~\ref{main:cutoffs-mass} uses these smooth cutoffs to construct $\tau_N$ and bounds the discarded mass in total variation.

The smooth factors also permit angular differentiation. The circle expansions convert the distance phase to angular integrals, and the graph estimate bounds the retained terms by counting pairs of square labels. The local reductions then replace an expression on one interval by expressions on shorter intervals, keeping track of the required conditions and derivative orders. Section~\ref{main:iteration} controls the coefficients and errors across the resulting finite expansion. Section~\ref{main:frequency} combines the two profiles and sums with the original class weights to bound the Fourier transform of the same $\tau_N$.

\subsection{Earlier work and acknowledgments}\label{main:credit}
Falconer's work~\cite{Falconer1985} initiated the study of the relation between Hausdorff dimension and distance sets. An earlier published positive-measure result of Guth, Iosevich, Ou, and Wang~\cite{GuthIosevichOuWang2020} shows that a compact planar set $E$ of Hausdorff dimension greater than $5/4$ has a point $x\in E$ for which $\{|x-y|:y\in E\}$ has positive length. This is a pinned-distance conclusion. Our quantitative statements concern the full distance set $\Delta E$.

The finite-energy approach to radial projections in Orponen's work~\cite[Section~3]{Orponen2019} is the basis for the radial argument used in preparation. We prove the unweighted estimate and the uniform constant needed here. Keleti and Shmerkin~\cite{KeletiShmerkin2019} developed a decomposition of measures across scales and organized scale choices through functions of depth. Their work provides context for the dyadic constructions, which are proved here with the weights and constants used later. Mattila's book~\cite{Mattila2015} supplies the Fourier-analytic and Hausdorff-dimension background, including the Frostman theorem used in Subsection~\ref{main:dimension}.

\bigskip

\paragraph{\emph{Methodological acknowledgment.}}
The proof follows the multiscale method of OpenAI~\cite{OpenAIFalconer2026} through mass profiles, directional conditions, angular estimates, recursive reductions, and reconstruction. We prove the specialized estimates in the planar setting and track the common constants needed for the normalized bound.

\Needspace{14\baselineskip}
\subsection{Background and organization}\label{main:reading}
We assume real analysis, Fourier analysis, and Hausdorff dimension theory, including integration against finite measures, the standard convergence theorems, $L^p$ duality, Fourier inversion, Plancherel, and Lebesgue differentiation. The hypotheses that require care for the measures used here are checked in the proofs.
The main stages are arranged as follows.

\bigskip

\begin{center}\small
\begin{tabular}{@{}P{29mm}P{116mm}@{}}
\toprule Sections & Purpose\\\midrule
\ref{main:reconstruction}--\ref{main:preparation} & Prove reconstruction and prepare two compact probabilities with uniform bounds.\\[4pt]
\ref{main:dyadic}--\ref{main:cutoffs-mass} & Build the classes and cutoffs and prove total-variation approximation.\\[4pt]
\ref{main:amplitudes}--\ref{main:graph} & Convert the distance phase to angular integrals and bound their sum by the graph estimate.\\[4pt]
\ref{main:reductions}--\ref{main:iteration} & Prove the local reductions and control their iteration.\\[4pt]
\ref{main:frequency} & Obtain the Fourier bound for the same approximating measures.\\[4pt]
\ref{main:normalized-proof}--\ref{main:localization} & Apply reconstruction uniformly, then derive the diameter-dependent bound and the dimension consequence.\\
\bottomrule
\end{tabular}
\end{center}

\bigskip

Appendix~\ref{main:notation-index} collects the recurring notations and points to the relevant definitions.

\clearpage

\section{Distance measures and reconstruction}\label{main:reconstruction}

We establish sufficient conditions under which a fixed positive measure on the real line admits a density. The approximating measures may depend on the frequency scale, provided that the total-variation errors and the integrals of the squared moduli of their Fourier transforms over the corresponding frequency bands are summable. An $L^2$ bound for the $L^2$ component of the resulting density, together with a sufficiently small $L^1$ remainder, then yields a lower bound for the Lebesgue measure of a set carrying the measure. The final proposition identifies sufficient uniform estimates for the planar application.

\subsection{Measure conventions and the Fourier normalization}\label{main:measure-conventions}
All measures are Borel measures. A positive measure $\mu$ is \emph{carried by} a Borel set $H$ when $\mu(\R^d\setminus H)=0$. For a signed or complex measure, the corresponding condition is $|\mu|(\R^d\setminus H)=0$, where $|\mu|$ is its variation measure. If $H$ is closed, this implies $\supp\mu\subseteq H$. When $H$ is not closed, the support can also contain boundary points. For the classes formed later from half-open squares, we distinguish between a measure being carried by a class and its support being contained in that class.

For a Borel set $H$, we write $\mu|_H$ for the restriction
\[
 (\mu|_H)(A)=\mu(A\cap H).
\]
For a positive measure with $\mu(H)>0$, its normalized restriction is $\mu|_H/\mu(H)$. A Borel map $T$ sends $\mu$ to the pushforward $T_\#\mu$, defined by
\[
 (T_\#\mu)(A)=\mu(T^{-1}(A)).
\]
Thus
\[
 \int f\,d(T_\#\mu)=\int f(T(x))\,d\mu(x)
\]
for every nonnegative Borel function $f$ when $\mu$ is positive, and for every Borel function $f$ such that $f\circ T$ is integrable with respect to $|\mu|$ when $\mu$ is signed or complex. In particular, the pushforward of a positive measure preserves its total mass, even when $T$ is not injective.

For a finite complex measure $\zeta$, we use the total-variation norm
\begin{equation}\label{main:variation}
 \|\zeta\|_{\TV}=|\zeta|(\R^d).
\end{equation}
No factor of $1/2$ is included. Thus the total-variation norm of the difference of two mutually singular probability measures is $2$, and the norm of a positive measure equals its total mass. For bounded Borel $f$,
\[
 \left|\int f\,d\zeta\right|\le\int|f|\,d|\zeta|
 \le\|f\|_\infty\|\zeta\|_{\TV}.
\]

Our Fourier convention is
\begin{equation}\label{main:fourier}
 \widehat\zeta(\xi)=\int e^{-\ii x\cdot\xi}\,d\zeta(x),\qquad
 \widehat f(\xi)=\int e^{-\ii x\cdot\xi}f(x)\,dx,
\end{equation}
and, whenever the following integral is well-defined, we write
\[
 \check v(x)=(2\pi)^{-d}\int e^{\ii x\cdot\xi}v(\xi)\,d\xi.
\]
For $f\in L^2(\R^d)$, Plancherel takes the form
\begin{equation}\label{main:plancherel}
 \|f\|_2^2=(2\pi)^{-d}\|\widehat f\|_2^2.
\end{equation}
We use arclength measure on the unit circle:
\begin{equation}\label{main:arclength}
 S^1=\{(\cos\theta,\sin\theta):\theta\in\R/(2\pi\Z)\},\qquad dw=d\theta.
\end{equation}
This measure has total mass $2\pi$, and this normalization is used in the radial estimate of Section~\ref{main:preparation}. The Fourier transform of a finite complex measure is bounded and continuous, with $|\widehat\zeta|\le\|\zeta\|_{\TV}$, although it need not be square integrable.

\subsection{The elementary support-length inequality}\label{rec:elementary}

If a positive measure $\alpha$ carried by a compact set $K\subset\R$ has an $L^2$ density $f$, then Cauchy--Schwarz gives
\[
 \alpha(\R)=\int_K f(t)\,dt\le |K|^{1/2}\|f\|_2.
\]
To accommodate the reconstruction argument, the following lemma also allows an $L^1$ remainder in the density.

\begin{lemma}\label{rec:elementary-length}
Let $\alpha$ be a positive measure carried by compact $K\subset\R$, and suppose
\[
 \alpha(\R)\ge m>0,\qquad \alpha=(g+h)\,dt,
 \qquad \|g\|_2\le G,\quad \|h\|_1\le H<m.
\]
Then $G>0$ and
\[
 |K|\ge\frac{(m-H)^2}{G^2}.
\]
\end{lemma}
\begin{proof}
Since $K$ has finite length, $g$ is integrable over $K$, and
\[
 m\le\alpha(K)=\int_K(g+h)\,dt
 \le |K|^{1/2}G+H.
\]
The inequality $H<m$ gives $G>0$ and the stated bound after rearrangement.
\end{proof}

Neither $g$ nor $h$ is required to be nonnegative or to vanish almost everywhere outside $K$. The required identity is $\alpha=(g+h)\,dt$. We establish this identity before applying the lemma.

\subsection{One fixed smooth frequency cutoff}\label{rec:sec:cutoff}
We choose a nonnegative cutoff in frequency. Its inverse Fourier transform may change sign.

Let $\vartheta\in C_c^\infty(\R)$ be even, nonincreasing in $|r|$, and satisfy
\[
 0\le\vartheta\le1,\qquad
 \vartheta=1\text{ on }[-1,1],\qquad
 \vartheta=0\text{ outside }[-2,2].
\]
We set
\begin{equation}\label{rec:eq:fixed-constants}
 q=\check\vartheta,\qquad
 \ct=\int_\R\vartheta(r)\,dr,\qquad
 \CP=2\|q\|_1.
\end{equation}
Repeated integration by parts shows that $q$ and its derivatives decrease faster than every inverse power. Thus $q$ is a Schwartz function and belongs to $L^1\cap L^2$. Fourier inversion gives
\begin{equation}\label{rec:eq:q-integral}
 \int_\R q(t)\,dt=\vartheta(0)=1.
\end{equation}

\bigskip

\paragraph{\emph{A choice with fixed numerical bounds.}}\label{rec:subsec:explicit}
The later normalized theorem uses $\ct=3$ and $\CP<16$. To obtain these values, we take
\[
 b(u)=
 \begin{cases}
 e^{-1/(u(1-u))},&0<u<1,\\
 0,&\text{otherwise},
 \end{cases}
 \qquad
 Z=\int_0^1 b(u)\,du,\qquad \omega(u)=b(u)/Z,
\]
and define
\begin{equation}\label{rec:eq:explicit-theta}
 \vartheta(r)=
 \begin{cases}
 1,&|r|\le1,\\
 \displaystyle\int_{|r|-1}^{1}\omega(u)\,du,&1<|r|<2,\\
 0,&|r|\ge2.
 \end{cases}
\end{equation}
The function $\omega$ is smooth, nonnegative, symmetric about $1/2$, and has integral one. Its vanishing derivatives at the endpoints make $\vartheta$ smooth. Symmetry and Tonelli give
\begin{equation}\label{rec:eq:ct-three}
 \ct=2+2\int_0^1\int_s^1\omega(u)\,du\,ds
 =2+2\int_0^1u\omega(u)\,du=3.
\end{equation}
Since $\omega$ increases up to $1/2$ and decreases afterward,
\[
 \|\vartheta''\|_1=4\omega(1/2).
\]
Integrating over $[1/3,2/3]$ gives $Z\ge\tfrac13e^{-9/2}$, and hence $\omega(1/2)\le3e^{1/2}<5$. The inverse Fourier transform formula yields the bound $\ct/(2\pi)$ for $|q(t)|$ when $|t|\le1$, while two integrations by parts yield the bound $\|\vartheta''\|_1/(2\pi t^2)$ when $|t|>1$. Consequently,
\begin{equation}\label{rec:eq:CP-sixteen}
 \CP=2\|q\|_1
 \le\frac2\pi(\ct+\|\vartheta''\|_1)
 <\frac{46}{\pi}<16.
\end{equation}
We keep this cutoff fixed. The reconstruction theorem is stated in terms of~\eqref{rec:eq:fixed-constants}; its application to the normalized planar problem uses $\CP\le16$ and $\ct=3$.

\bigskip 
\paragraph{\emph{Convolution operators.}}
For every integer $n$, let
\begin{equation}\label{rec:eq:operators}
 q_n(t)=2^nq(2^nt),\qquad Q_n\zeta=q_n*\zeta,\qquad
 P_n\zeta=Q_n\zeta-Q_{n-1}\zeta.
\end{equation}
Here $(k*\zeta)(t)=\int k(t-s)\,d\zeta(s)$. Integration against the variation measure gives
\[
 \|k*\zeta\|_1\le\|k\|_1\|\zeta\|_{\TV},
 \qquad
 \|k*\zeta\|_2\le\|k\|_2\|\zeta\|_{\TV}
\]
when the corresponding kernel norm is finite. Every operator in~\eqref{rec:eq:operators} therefore produces an $L^1\cap L^2$ function. Since $\|q_n\|_1=\|q\|_1$,
\begin{equation}\label{rec:eq:uniform-P}
 \|P_n\zeta\|_1\le\CP\|\zeta\|_{\TV}
 \qquad(n\in\Z).
\end{equation}
Fubini identifies the Fourier transform of each convolution, and Plancherel applies to these $L^2$ functions even when $\zeta$ is singular.

\subsection{The reconstruction theorem}\label{rec:sec:reconstruction}
At each scale, we require the Fourier estimate only on the frequency band specified below.

\begin{theorem}[Reconstruction and support length]\label{rec:thm:reconstruction}
Let $\alpha$ be a finite positive Borel measure carried by a compact set $K\subset\R$, and suppose
\[
 0<m\le\alpha(\R)\le M_0<\infty.
\]
Fix an integer $J$. For each integer $N\ge J$, let $\tau_N$ be a finite complex Borel measure. Define
\begin{align}
 E_J&=\sum_{N\ge J}\|\alpha-\tau_N\|_{\TV},\label{rec:eq:E}\\
 S_J&=\sum_{N\ge J}\int_{2^{N-1}\le|r|\le2^{N+1}}
                    |\widehat{\tau_N}(r)|^2\,dr.\label{rec:eq:S}
\end{align}
If $E_J,S_J<\infty$, then $\alpha=(g+h)\,dt$ with
\begin{equation}\label{rec:eq:recon-bounds}
 \|g\|_2^2\le\frac{2^{J-1}\ct M_0^2+S_J}{2\pi},
 \qquad \|h\|_1\le\CP E_J.
\end{equation}
If in addition $\CP E_J<m$, then
\begin{equation}\label{rec:eq:recon-length}
 |K|\ge
 \frac{2\pi(m-\CP E_J)^2}
 {2^{J-1}\ct M_0^2+S_J}.
\end{equation}
The approximants need not be positive, nested, or supported on $K$.
\end{theorem}

Finiteness of the two sums implies that $\alpha$ is absolutely continuous with respect to Lebesgue measure. The additional condition $\CP E_J<m$ yields a quantitative lower bound for $|K|$ with the chosen starting index $J$. For a family of measures, the starting index $J$ and the relevant bounds must be uniform across the family; we return to this point in Corollary~\ref{rec:cor:uniform}.

\subsection{Frequency weights and convergence of the two series}\label{rec:sec:series}

We begin the proof of Theorem~\ref{rec:thm:reconstruction} by retaining $Q_{J-1}\alpha$ as the low-frequency term and using $P_N\tau_N$ for the corresponding higher-frequency increments.

The Fourier multiplier associated with $P_N$ is
\begin{equation}\label{rec:eq:pN}
 p_N(r)=\vartheta(2^{-N}r)-\vartheta(2^{-(N-1)}r).
\end{equation}
Monotonicity and the support of $\vartheta$, together with telescoping, give
\begin{equation}\label{rec:eq:partition}
 \begin{gathered}
 p_N\ge0,\qquad 0\le p_N\le1,\qquad
 \supp p_N\subseteq\{2^{N-1}\le|r|\le2^{N+1}\},\\
 \vartheta(2^{-(J-1)}r)+\sum_{N=J}^{L}p_N(r)
 =\vartheta(2^{-L}r)\le1.
 \end{gathered}
\end{equation}

For nonnegative weights $w_j$ and complex numbers $z_j$, Cauchy--Schwarz gives
\begin{equation}\label{rec:eq:weighted-square}
 \left|\sum_jw_jz_j\right|^2
 \le\left(\sum_jw_j\right)\sum_jw_j|z_j|^2.
\end{equation}
We apply this at each frequency with weights $p_N(r)$. Their sum is at most one, so Plancherel yields
\begin{equation}\label{rec:eq:L2-tail}
 \left\|\sum_{N=A}^{B}P_N\tau_N\right\|_2^2
 \le\frac1{2\pi}\sum_{N=A}^{B}
  \int_{2^{N-1}\le|r|\le2^{N+1}}|\widehat{\tau_N}(r)|^2\,dr.
\end{equation}
Here $J\le A\le B$. Since $S_J<\infty$, the right-hand side tends to zero as $A\to\infty$, uniformly for $B\ge A$. The series of increments therefore converges in $L^2$.

The nonnegative frequency multipliers allow us to sum the band energies directly, without an overlap factor. Applying the triangle inequality directly would instead require summability of the square roots of these band energies. For example, $\sum_{N\ge0}(N+1)^{-2}<\infty$, whereas $\sum_{N\ge0}(N+1)^{-1}=\infty$. Nonnegative frequency weights are essential in~\eqref{rec:eq:weighted-square}; choosing weights $2,-1$ and corresponding values $1,-1$ would yield the false inequality $9\le1$. No such positivity is required of the spatial kernel $q$.

We now include the low-frequency term by setting
\begin{equation}\label{rec:eq:gL}
 g_L=Q_{J-1}\alpha+\sum_{N=J}^{L}P_N\tau_N.
\end{equation}
Using $\vartheta(2^{-(J-1)}r)$ as an additional weight and $|\widehat\alpha|\le M_0$, we obtain
\begin{align*}
 \|g_L\|_2^2
 &\le\frac1{2\pi}\left[
 \int_\R\vartheta(2^{-(J-1)}r)|\widehat\alpha(r)|^2\,dr
 +\sum_{N=J}^{L}\int_\R p_N(r)|\widehat{\tau_N}(r)|^2\,dr
 \right]\\
 &\le\frac{2^{J-1}\ct M_0^2+S_J}{2\pi}.
\end{align*}
Thus $g_L\to g$ in $L^2$, with the first bound in~\eqref{rec:eq:recon-bounds}. The coefficient $\ct$ is $\int\vartheta$: one power of each weight occurs on the right side of~\eqref{rec:eq:weighted-square}.

For the error, we set
\begin{equation}\label{rec:eq:hL}
 h_L=\sum_{N=J}^{L}P_N(\alpha-\tau_N).
\end{equation}
By~\eqref{rec:eq:uniform-P},
\[
 \sum_{N\ge J}\|P_N(\alpha-\tau_N)\|_1\le\CP E_J<\infty.
\]
Hence $h_L\to h$ in $L^1$, with $\|h\|_1\le\CP E_J$.

\subsection{Identification of the actual measure}\label{rec:sec:identification}
The two limits must recover the given measure $\alpha$. Before taking limits, linearity gives
\begin{equation}\label{rec:eq:finite-identity}
 g_L+h_L
 =Q_{J-1}\alpha+\sum_{N=J}^{L}P_N\alpha
 =Q_L\alpha.
\end{equation}
This is why the target measure is fixed throughout the construction.

For $\phi\in C_c(\R)$, Fubini and a change of variables give
\begin{equation}\label{rec:eq:test-QL}
 \int_\R\phi(t)Q_L\alpha(t)\,dt
 =\int_\R\int_\R q(u)\phi(s+2^{-L}u)\,du\,d\alpha(s).
\end{equation}
Since $q\in L^1$ and $\int q=1$, boundedness and uniform continuity of $\phi$ imply, by dominated convergence,
\begin{equation}\label{rec:eq:QL-limit}
 \int\phi Q_L\alpha\,dt\longrightarrow\int\phi\,d\alpha.
\end{equation}
The absolute value of the integrand is bounded by $\|\phi\|_\infty|q(u)|$, which is integrable with respect to $du\,d\alpha(s)$; no sign condition on $q$ is needed.

Since $\phi\in L^2\cap L^\infty$, the convergences $g_L\to g$ in $L^2$ and $h_L\to h$ in $L^1$ allow us to pass to the limit in the integrals against $\phi$. Therefore~\eqref{rec:eq:finite-identity} gives
\begin{equation}\label{rec:eq:measure-test}
 \int\phi\,d\alpha=\int\phi(t)(g(t)+h(t))\,dt
 \qquad(\phi\in C_c(\R)).
\end{equation}
Since $g+h$ is locally integrable, equality of the integrals against all functions in $C_c(\R)$ implies equality of the corresponding locally finite measures:
\begin{equation}\label{rec:eq:measure-identity}
 \alpha=(g+h)\,dt.
\end{equation}
Because $\alpha$ is finite, positive, and carried by $K$, its density $g+h$ is nonnegative almost everywhere, vanishes almost everywhere outside $K$, and belongs to $L^1$. This conclusion does not require global $L^1$ integrability of $g$ in advance.

When $\CP E_J<m$, Lemma~\ref{rec:elementary-length} applied with~\eqref{rec:eq:recon-bounds} gives~\eqref{rec:eq:recon-length}. This completes the proof of Theorem~\ref{rec:thm:reconstruction}.

\subsection{One common cutoff for a family}\label{rec:sec:uniform}
For a uniform length bound, we need one cutoff and common tail bounds throughout the family.

\begin{corollary}\label{rec:cor:uniform}
Suppose all data in a family satisfy Theorem~\ref{rec:thm:reconstruction} with common $m,M_0$ and a common cutoff $J$, and
\[
 E_J\le e<m/\CP,\qquad S_J\le S<\infty.
\]
Then
\[
 |K|\ge\frac{2\pi(m-\CP e)^2}{2^{J-1}\ct M_0^2+S}.
\]
Suppose more specifically that there are common $a,\gamma,C_e,C_s>0$ and an integer $N_*$ such that, for every $N\ge N_*$,
\begin{equation}\label{rec:eq:exponential-input}
 \|\alpha-\tau_N\|_{\TV}\le C_e2^{-aN},\qquad
 \int_{2^{N-1}\le|r|\le2^{N+1}}|\widehat{\tau_N}(r)|^2dr
 \le C_s2^{-\gamma N}.
\end{equation}
Every integer
\begin{equation}\label{rec:eq:J-general}
 J\ge\max\left\{N_*,0,
 \left\lceil\frac1a\log_2
 \frac{2\CP C_e}{m(1-2^{-a})}\right\rceil\right\}
\end{equation}
gives
\begin{equation}\label{rec:eq:uniform-length}
 |K|\ge\frac{2\pi m^2}
 {4\left(2^{J-1}\ct M_0^2+
 C_s2^{-\gamma J}/(1-2^{-\gamma})\right)}.
\end{equation}
\end{corollary}
\begin{proof}
The first assertion follows directly from~\eqref{rec:eq:recon-length}. Under~\eqref{rec:eq:exponential-input}, summing the geometric series gives
\begin{equation}\label{rec:eq:geometric-tails}
 E_J\le\frac{C_e2^{-aJ}}{1-2^{-a}},\qquad
 S_J\le\frac{C_s2^{-\gamma J}}{1-2^{-\gamma}}.
\end{equation}
The choice~\eqref{rec:eq:J-general} ensures $\CP E_J\le m/2<m$. Substitution into~\eqref{rec:eq:recon-length} proves~\eqref{rec:eq:uniform-length}.
\end{proof}

Increasing $J$ decreases the upper bound for the total-variation error tail, while increasing the low-frequency contribution $2^{J-1}\ct M_0^2$. The stated choice yields a positive lower bound for $|K|$ without optimizing this trade-off. For the cutoff~\eqref{rec:eq:explicit-theta}, we may replace $\CP$ by $16$ in~\eqref{rec:eq:J-general}.

\subsection{Examples that clarify the hypotheses}\label{rec:sec:examples}

The first two examples explain why total-variation approximation and uniform tail bounds are needed.

\bigskip

\paragraph{\emph{Weak convergence is insufficient.}}
Let $\alpha=\delta_0$, and let $\tau_N$ be the centered Gaussian probability measure with standard deviation $2^{-N/2}$. Then $\tau_N$ converges weakly to $\alpha$ and
\[
 |\widehat{\tau_N}(r)|^2=e^{-2^{-N}r^2}.
\]
The $N$th frequency band has total length $3\cdot2^N$, and $2^{-N}r^2\ge2^{N-2}$ on that band. Thus
\[
 \int_{2^{N-1}\le|r|\le2^{N+1}}
 |\widehat{\tau_N}(r)|^2dr
 \le3\cdot2^N e^{-2^{N-2}}.
\]
These energies are summable, although $\alpha$ is singular and carried by a set of length zero. The required summability of the total-variation errors fails: $\|\alpha-\tau_N\|_{\TV}=2$ for every $N$, since the two probabilities are mutually singular.

\bigskip

\paragraph{\emph{Individual bounds need not be uniform.}}
For $0<\ell\le1$, let $\alpha_\ell=\ell^{-1}\one_{[0,\ell]}dt$, and take every approximant equal to $\alpha_\ell$. The variation errors vanish. The density has squared $L^2$ norm $1/\ell$, and each frequency belongs to at most three dyadic bands, so Plancherel gives a finite sum of band energies for each fixed $\ell$. Nevertheless, the supporting lengths tend to zero. Corollary~\ref{rec:cor:uniform} requires a common bound on these sums.

\bigskip

\paragraph{\emph{The approximants may be complex.}}
Let $\alpha=\one_{[0,1]}dt$ and
\[
 \tau_N=(1+\ii(-1)^N2^{-2N})\alpha\qquad(N\ge1).
\]
Then $\|\alpha-\tau_N\|_{\TV}=2^{-2N}$. Moreover, the estimate $|\widehat\alpha(r)|\le\min\{1,2/|r|\}$ shows that the $N$th band energy of $\tau_N$ is at most $12(1+2^{-4N})2^{-N}$. Both the total-variation error series and the band-energy series therefore converge. The reconstruction theorem permits complex approximating measures, although the later planar construction uses positive measures.

\subsection{A conditional reduction for prepared planar measures}\label{rec:conditional-section}
We now state the estimates that the planar construction must supply. Section~\ref{main:cutoffs-mass} constructs the approximants and proves their total-variation bound; Section~\ref{main:frequency} proves the Fourier bound for the same measures.

Let $\mu_1,\mu_2$ be probability measures carried by disjoint compact sets $K_1,K_2\subset\R^2$, with
\[
 0<d_0\le|x-y|\le d_1<\infty
 \qquad(x\in K_1,\ y\in K_2).
\]
We define the compact cross-distance set and the associated weighted distance measure by
\begin{equation}\label{rec:weighted-measure}
 \begin{gathered}
 D(K_1,K_2)=\{|x-y|:x\in K_1,\ y\in K_2\},\qquad
 \dd(x,y)=|x-y|,\\
 \alpha=\dd_\#\bigl(|x-y|^{-1/2}\mu_1\otimes\mu_2\bigr).
 \end{gathered}
\end{equation}
In the last line, we weight the product measure $\mu_1\otimes\mu_2$ by $|x-y|^{-1/2}$ and then push it forward under $\dd$. The distance bounds give
\begin{equation}\label{rec:weighted-mass}
 d_1^{-1/2}\le\alpha(\R)\le d_0^{-1/2},
 \qquad\alpha\text{ is carried by }D(K_1,K_2).
\end{equation}
The weight $|x-y|^{-1/2}$ arises in the later planar oscillatory calculation. The reconstruction argument uses this measure, its total-mass bounds, and the compact set carrying it.

\begin{proposition}\label{rec:conditional}
Fix common $d_0,d_1,a,\gamma,C_e,C_s>0$ and an integer $N_*$. Suppose that for every prepared pair in a family, Borel functions $0\le\Theta_N\le1$ on $K_1\times K_2$ give the measures
\begin{equation}\label{rec:conditional-approximants}
 \tau_N=\dd_\#\bigl(|x-y|^{-1/2}\Theta_N(x,y)\mu_1\otimes\mu_2\bigr)
 \qquad(N\ge N_*),
\end{equation}
and that these identical measures satisfy
\begin{equation}\label{rec:conditional-estimates}
 \begin{split}
 \|\alpha-\tau_N\|_{\TV}&\le C_e2^{-aN},\\
 \int_{2^{N-1}\le|r|\le2^{N+1}}
 |\widehat{\tau_N}(r)|^2\,dr&\le C_s2^{-\gamma N}.
 \end{split}
\end{equation}
Then the lengths $|D(K_1,K_2)|$ have a common positive lower bound depending only on those numerical data. More precisely, put $m=d_1^{-1/2}$ and $M_0=d_0^{-1/2}$. Any integer
\begin{equation}\label{rec:conditional-J}
 J\ge\max\left\{N_*,0,
 \left\lceil\frac1a\log_2
 \frac{32C_e}{m(1-2^{-a})}\right\rceil\right\}
\end{equation}
gives
\begin{equation}\label{rec:conditional-length}
 |D(K_1,K_2)|\ge
 \frac{2\pi m^2}{4\bigl(3\cdot2^{J-1}M_0^2+
 C_s2^{-\gamma J}/(1-2^{-\gamma})\bigr)}.
\end{equation}
\end{proposition}
\begin{proof}
Apply Corollary~\ref{rec:cor:uniform} to the measure in~\eqref{rec:weighted-measure}, whose mass bounds are~\eqref{rec:weighted-mass}. The two hypotheses in~\eqref{rec:conditional-estimates} concern the same fixed $\alpha$ and the same approximants. Using $\ct=3$ and $\CP\le16$ gives~\eqref{rec:conditional-J} and~\eqref{rec:conditional-length}.
\end{proof}

The first estimate controls the mass removed, while the second provides an $L^2$ bound for the reconstructed density. Since $0\le\Theta_N\le1$, positivity gives
\begin{equation}\label{rec:positive-defect}
 \|\alpha-\tau_N\|_{\TV}
 =\iint |x-y|^{-1/2}(1-\Theta_N(x,y))\,d\mu_1(x)d\mu_2(y),
\end{equation}
since the pushforward of a positive measure preserves its total mass. The auxiliary restriction used to estimate this integral is not included in $\Theta_N$; including it would change $\tau_N$, whose Fourier transform must be estimated. We retain the definition of $\tau_N$ throughout both arguments.

The next section constructs the prepared probabilities with constants depending only on $\eta$. It gives
\[
 d_0=\frac7{2048},\qquad d_1=\frac14,
 \qquad m=2,\qquad M_0^2=\frac{2048}{7},
\]
and shows that $8D(K_1,K_2)$ is contained in the original distance set. Once the two estimates are proved, the lower bound for the length of the original distance set is eight times the bound in~\eqref{rec:conditional-length}; this factor comes from scaling length. Section~\ref{main:normalized-proof} chooses the remaining parameters in terms of $\eta$ and applies the proposition with the common rates, prefactors, and starting scale.

\clearpage

\section{Radial projection estimates and restrictions to compact sets}
\label{main:preparation}

We will choose two compact sets $K_1^0,K_2^0\subset F$, each
having positive $\nu$-measure bounded below independently of
$F$ and $\nu$. The sets will be separated by a distance bounded
below by a positive constant, and the directions
$(x-y)/|x-y|$, for $(x,y)\in K_1^0\times K_2^0$, will lie
in an arc of $S^1$ whose length is bounded by a fixed small
constant. After restricting and normalizing $\nu$ on these
sets, we will obtain probability measures whose radial
projections have $L^p$ densities with uniformly bounded norms
at every center in the opposite set.

The main analytic ingredient is an estimate for the average
$p$th power of the $L^p$ norm of a radial projection density
in terms of the energies of the source and center measures.
Following the finite-energy approach of
Orponen~\cite[Section~3]{Orponen2019}, we prove this estimate
without weights, with a constant independent of the distance
between the disjoint supports. In the proof of this estimate,
positive separation is needed only for continuity of the
radial measures with respect to the center and for the choice
of mollification radii. We apply the estimate with the roles
of the two measures interchanged. Lower semicontinuity of
the density norm then allows us to select compact subsets
on which the required bounds hold at every center. The final
normalization changes these bounds only by factors that are
uniformly controlled.

The precise result is stated in
Theorem~\ref{prep:thm:preparation}. The resulting probabilities
are the pushforwards of the normalized restrictions under the
fixed similarity $S(z)=z/8+(1/2,1/2)$. Their supports
therefore lie in a fixed dyadic square, so the subsequent
dyadic construction uses the same grid for every input.

\subsection{Radial projections, orthogonal projections, and energy}
\label{rad:sec:definitions}

Let
\[
 S^1=\{(\cos\theta,\sin\theta):
       \theta\in\R/(2\pi\Z)\},
 \qquad dw=d\theta,
\]
where $dw$ denotes the unnormalized arclength measure on $S^1$.
For $x\in\R^2$, define the radial projection
\[
 \pi_x:\R^2\setminus\{x\}\longrightarrow S^1,
 \qquad
 \pi_x(y)=\frac{y-x}{|y-x|}.
\]
If $\lambda$ is a finite positive Borel measure on $\R^2$
with $x\notin\operatorname{supp}\lambda$, its radial projection
is the pushforward measure $(\pi_x)_\#\lambda$, given by
\[
 (\pi_x)_\#\lambda(A)=\lambda(\pi_x^{-1}(A))
\]
for every Borel set $A\subset S^1$.

For a finite positive Borel measure $\mu$ on $\R^d$ and
$0<a<d$, define its $a$-energy by
\begin{equation}\label{rad:eq:energy-def}
 I_a(\mu)
 =\iint |z-z'|^{-a}\,d\mu(z)d\mu(z').
\end{equation}
The kernel is understood to take the value $+\infty$ on
the diagonal. The associated potential is
\[
 U_a\mu(z)=\int |z-y|^{-a}\,d\mu(y),
\]
and Tonelli's theorem gives
\[
 I_a(\mu)=\int U_a\mu(z)\,d\mu(z).
\]
In particular, finite energy implies that $\mu$ has no atoms.
We write $c_{d,a}>0$ for the constant in the Fourier-energy
identity~\eqref{rad:eq:energy-fourier}, proved in
Subsection~\ref{rad:sec:energy}.

For $w=(w_1,w_2)\in S^1$, let
$w^\perp=(-w_2,w_1)$. Define the scalar orthogonal projection
$P_w:\R^2\to\R$ by
\[
 P_w z=z\cdot w^\perp.
\]
For a Borel probability $\nu$ on $\R^2$, write
$\nu_w=(P_w)_\#\nu$.

For $1<p<\infty$ and a finite positive Borel measure
$\zeta$ on $S^1$, define
\begin{equation}\label{rad:eq:intrinsic}
 \np(\zeta)=
 \begin{cases}
 \|f\|_{L^p(S^1,dw)},
   &\zeta=f\,dw
    \text{ with }f\in L^p(S^1,dw),\\
 +\infty,
   &\text{otherwise}.
 \end{cases}
\end{equation}
Since arclength measure is invariant under $w\mapsto-w$,
this norm is unchanged when $\pi_x$ is replaced by
$-\pi_x$.
\subsection{Lower semicontinuity of radial density norms}
\label{rad:sec:lsc}

\begin{lemma}\label{rad:lem:lsc}
Let $1<p<\infty$, and set $p'=p/(p-1)$.
For every finite positive Borel measure $\zeta$ on $S^1$,
\begin{equation}\label{rad:eq:dual-norm}
 \np(\zeta)
 =\sup\left\{
 \left|\int h\,d\zeta\right|:
 h\in C(S^1),\
 \|h\|_{L^{p'}(S^1,dw)}\le1
 \right\}.
\end{equation}
In particular, $\zeta\mapsto\np(\zeta)$ is lower
semicontinuous under weak convergence of finite positive
Borel measures on $S^1$.
\end{lemma}

\begin{proof}
Suppose first that $\zeta=f\,dw$ for some
$f\in L^p(S^1,dw)$. Holder's inequality and the density of
$C(S^1)$ in $L^{p'}(S^1,dw)$ show that the supremum
equals $\|f\|_{L^p(S^1,dw)}$.

Conversely, if the supremum is finite, integration against
$\zeta$ defines a bounded linear functional on $C(S^1)$
with respect to the $L^{p'}(S^1,dw)$ norm. By density,
this functional extends to $L^{p'}(S^1,dw)$.
Duality then gives $f\in L^p(S^1,dw)$ such that
\[
 \int h\,d\zeta=\int hf\,dw
 \qquad(h\in C(S^1)).
\]
Hence $\zeta=f\,dw$, and positivity of $\zeta$ implies
$f\ge0$ almost everywhere. This proves
\eqref{rad:eq:dual-norm}, including the case in which
$\np(\zeta)=+\infty$.

For each $h\in C(S^1)$, the function
$\zeta\mapsto|\int h\,d\zeta|$ is continuous under weak
convergence. The supremum in~\eqref{rad:eq:dual-norm}
is therefore lower semicontinuous.
\end{proof}

Let $E,K\subset\R^2$ be disjoint compact sets, and let
$\lambda$ be a finite positive Borel measure supported
on $E$. For every $h\in C(S^1)$, the function
\[
 x\longmapsto\int h(\pi_x(y))\,d\lambda(y)
\]
is continuous on $K$, since
$(x,y)\mapsto h(\pi_x(y))$ is uniformly continuous
on $K\times E$. Thus the map
$x\mapsto(\pi_x)_\#\lambda$ is weakly continuous.
By Lemma~\ref{rad:lem:lsc},
\[
 x\longmapsto\np((\pi_x)_\#\lambda)
\]
is lower semicontinuous, and hence Borel measurable,
on $K$. Consequently, for every finite $M\ge0$,
the set
\[
 \left\{
 x\in K:\np((\pi_x)_\#\lambda)\le M
 \right\}
\]
is compact.
\subsection{Jointly Borel radial densities}
\label{rad:sec:versions}

We first show that the densities of radial projections
can be chosen to be jointly Borel measurable in the center
and the direction.

\begin{lemma}\label{rad:lem:joint}
Let $E,K\subset\R^2$ be disjoint compact sets, let
$\lambda$ be a finite positive Borel measure supported on
$E$, and let $1<p<\infty$.
There exists a Borel function
\[
 f:K\times S^1\longrightarrow[0,\infty)
\]
such that, for every $x\in K$ with
$\np((\pi_x)_\#\lambda)<\infty$,
\[
 (\pi_x)_\#\lambda=f(x,w)\,dw.
\]
At every other center $x\in K$, the function $f(x,\cdot)$
is identically zero.
\end{lemma}

\begin{proof}
For $x\in K$, write
$\zeta_x=(\pi_x)_\#\lambda$.
Let $d_{S^1}$ denote the shorter angular distance
on $S^1$.
For $n\ge1$, set $r_n=2^{-n}$ and define
\begin{equation}\label{rad:eq:density-limsup}
 f_0(x,w)=\limsup_{n\to\infty}
 \frac{\zeta_x(\{v:d_{S^1}(v,w)<r_n\})}{2r_n}.
\end{equation}

For each $n$, the numerator equals
\[
 \int
 \one_{\{d_{S^1}(\pi_x(y),w)<r_n\}}
 \,d\lambda(y).
\]
The integrand is Borel in $(x,w,y)$, so integration
against the fixed finite measure $\lambda$ gives a Borel
function of $(x,w)$.
Consequently, $f_0$ is a nonnegative, possibly
infinite-valued Borel function on $K\times S^1$.

Let
\[
 K_p=\{x\in K:\np(\zeta_x)<\infty\}.
\]
This set is Borel by Lemma~\ref{rad:lem:lsc}.
For each $x\in K_p$, write
$\zeta_x=f_x\,dw$ with $f_x\in L^p(S^1,dw)$.
Since $S^1$ has finite arclength measure,
$f_x\in L^1(S^1,dw)$.
Lebesgue differentiation on the circle gives
\[
 f_0(x,w)=f_x(w)
\]
for arclength-almost every $w\in S^1$.

Define
\begin{equation}\label{rad:eq:joint-density}
 f(x,w)=
 \begin{cases}
 f_0(x,w),
   &x\in K_p\text{ and }f_0(x,w)<\infty,\\
 0,
   &\text{otherwise}.
 \end{cases}
\end{equation}
Then $f$ is finite-valued, nonnegative, and Borel
on $K\times S^1$.
For every $x\in K_p$, the function $f(x,\cdot)$
agrees almost everywhere with a density of $\zeta_x$.
For $x\notin K_p$, it vanishes identically.
\end{proof}

If $\nu$ is a finite positive Borel measure supported
on $K$ and $\nu(K\setminus K_p)=0$, then
\[
 \int_K\int_{S^1}f(x,w)^p\,dw\,d\nu(x)
 =
 \int_K\np((\pi_x)_\#\lambda)^p\,d\nu(x).
\]

\subsection{An averaged radial projection estimate}
\label{rad:sec:theorem}

Fix $0<\chi\le1/16$, and set
\begin{equation}\label{rad:eq:parameters}
 s=1+\chi,\qquad
 u=1-\chi,\qquad
 t=1-\frac{\chi}{2},\qquad
 p=1+\frac{\chi}{4}.
\end{equation}

\begin{theorem}\label{rad:thm:radial}
Let $\lambda$ and $\nu$ be Borel probabilities
supported on disjoint compact subsets of
$\overline B(0,1)\subset\R^2$, and suppose that
\[
 I_s(\lambda)<\infty,
 \qquad
 I_t(\nu)<\infty.
\]
There exists a constant $\Cr(\chi)\in(0,\infty)$,
depending only on $\chi$, such that
\begin{equation}\label{rad:eq:headline}
 \int\np((\pi_x)_\#\lambda)^p\,d\nu(x)
 \le
 \Cr(\chi)I_s(\lambda)^{p/2}I_t(\nu)^{1/2}.
\end{equation}
The constant is independent of the positive distance
between the supports of $\lambda$ and $\nu$.

In particular, $(\pi_x)_\#\lambda$ has an
$L^p(S^1,dw)$ density for $\nu$-almost every $x$.
The function
$x\mapsto\np((\pi_x)_\#\lambda)$
is lower semicontinuous on $\operatorname{supp}\nu$.
Moreover, there is a nonnegative, finite-valued,
jointly Borel density $f(x,w)$ at every center
where this norm is finite, with $f(x,\cdot)=0$
at the remaining centers.
It satisfies
\[
 \int\!\int f(x,w)^p\,dw\,d\nu(x)
 \le
 \Cr(\chi)I_s(\lambda)^{p/2}I_t(\nu)^{1/2}.
\]

The estimate also holds with $\lambda$ and $\nu$
interchanged, provided
$I_s(\nu)<\infty$ and $I_t(\lambda)<\infty$.
\end{theorem}

\subsection{Relations among the exponents}
\label{rad:sec:exponents}

Set
\begin{equation}\label{rad:eq:angular-exponents}
 r_0=\frac1{2-p},\qquad
 z_0=tr_0,\qquad
 d_\chi=t-up.
\end{equation}
The definitions in~\eqref{rad:eq:parameters} give
\begin{equation}\label{rad:eq:margins}
 \begin{gathered}
 s+u=2,\qquad
 t-up=\frac{\chi}{4}+\frac{\chi^2}{4}
     =d_\chi>0,\\[3pt]
 (2-p)-t=\frac{\chi}{4}>0,\qquad
 1-z_0=\frac{\chi}{4-\chi}>0.
 \end{gathered}
\end{equation}
Consequently,
\[
 0<u<1,\quad
 1<s<2,\quad
 1<p<2,\quad
 up<t<2-p<1,\quad
 r_0>1,\quad z_0<1.
\]

The identity $s+u=2$ gives the Fourier weights
used below, while $up<t$ permits comparison of
projected energies on bounded supports.
The inequality $tr_0<1$ ensures the integrability
required for the angular estimate.

\subsection{Fourier representation of energy}
\label{rad:sec:energy}

\begin{lemma}\label{rad:lem:energy}
Let $d\ge1$ be an integer and $0<a<d$.
There exists a constant $c_{d,a}>0$,
depending only on $d$ and $a$, such that
\begin{equation}\label{rad:eq:energy-fourier}
 I_a(\mu)
 =
 c_{d,a}\int_{\R^d}
 |\widehat\mu(\xi)|^2|\xi|^{a-d}\,d\xi
\end{equation}
for every finite positive Borel measure $\mu$
on $\R^d$.
The identity holds with values in $[0,\infty]$.
\end{lemma}

\begin{proof}
For $a>0$, the integral
\[
 \int_0^\infty t^{a/2-1}e^{-t}\,dt
\]
is positive and finite.
A change of variables therefore gives a constant
$k_a>0$, depending only on $a$, such that
\[
 |z|^{-a}
 =
 k_a\int_0^\infty
 v^{a/2-1}e^{-v|z|^2}\,dv,
\]
where both sides are interpreted as $+\infty$
at $z=0$.

For $v>0$, Gaussian Fourier inversion gives
\[
 e^{-v|z|^2}
 =
 (4\pi v)^{-d/2}
 \int_{\R^d}
 e^{\ii z\cdot\xi}
 e^{-|\xi|^2/(4v)}\,d\xi.
\]
Since $\mu$ is finite, Fubini's theorem yields
\[
 \iint e^{-v|x-y|^2}\,d\mu(x)d\mu(y)
 =
 (4\pi v)^{-d/2}
 \int_{\R^d}
 e^{-|\xi|^2/(4v)}
 |\widehat\mu(\xi)|^2\,d\xi.
\]

Integrating in $v$ and applying Tonelli's theorem,
we obtain
\[
 I_a(\mu)
 =
 k_a(4\pi)^{-d/2}
 \int_{\R^d}|\widehat\mu(\xi)|^2
 \left(
 \int_0^\infty
 v^{(a-d)/2-1}e^{-|\xi|^2/(4v)}\,dv
 \right)d\xi.
\]
For $\xi\ne0$, the substitution
$v=|\xi|^2t$ gives
\[
 \int_0^\infty
 v^{(a-d)/2-1}e^{-|\xi|^2/(4v)}\,dv
 =
 |\xi|^{a-d}
 \int_0^\infty
 t^{(a-d)/2-1}e^{-1/(4t)}\,dt.
\]
The last integral is positive and finite:
the exponential factor ensures convergence
at zero, and $a<d$ ensures convergence
at infinity.

This proves~\eqref{rad:eq:energy-fourier}
with a positive finite constant $c_{d,a}$
depending only on $d$ and $a$.
The value at $\xi=0$ does not affect
the integral.
All integrations involving nonnegative
integrands are justified by Tonelli's theorem,
so the identity also holds when the energy
is infinite.
\end{proof}

Let $\lambda$ be a finite positive Borel measure
on $\R^2$, and let $\varphi\in C_c^\infty(\R^2)$
be nonnegative, supported in the unit disk,
and have integral one.
For $\varepsilon>0$, put
\[
 \varphi_\varepsilon(z)
 =
 \varepsilon^{-2}\varphi(z/\varepsilon).
\]
Since
\[
 \widehat{\lambda*\varphi_\varepsilon}(\xi)
 =
 \widehat\lambda(\xi)
 \widehat\varphi(\varepsilon\xi),
 \qquad
 |\widehat\varphi(\varepsilon\xi)|\le1,
\]
Lemma~\ref{rad:lem:energy} gives
\begin{equation}\label{rad:eq:contraction}
 I_s(\lambda*\varphi_\varepsilon)
 \le I_s(\lambda).
\end{equation}

\begin{lemma}\label{rad:lem:pairing}
Let $0<u<1$, let $\zeta$ be a finite positive
Borel measure on $\R$ with $I_u(\zeta)<\infty$,
and let $f\in C_c^\infty(\R)$.
Then
\begin{equation}\label{rad:eq:pairing}
 \left|\int f\,d\zeta\right|
 \le
 \frac{c_{1,u}^{-1/2}}{2\pi}
 I_u(\zeta)^{1/2}
 \left(
 \int_\R
 |\widehat f(r)|^2|r|^{1-u}\,dr
 \right)^{1/2}.
\end{equation}
\end{lemma}

\begin{proof}
Fourier inversion and Fubini's theorem give
\[
 \int f\,d\zeta
 =
 \frac1{2\pi}
 \int_\R
 \widehat f(r)\widehat\zeta(-r)\,dr.
\]
Applying Cauchy--Schwarz with weights
$|r|^{1-u}$ and $|r|^{u-1}$, followed
by~\eqref{rad:eq:energy-fourier},
proves~\eqref{rad:eq:pairing}.
The weighted integral involving $\widehat f$
is finite because $f\in C_c^\infty(\R)$.
\end{proof}

\subsection{Angular averages of projected energies}
\label{rad:sec:angular}

For $0\le z<1$, define
\begin{equation}\label{rad:eq:Jz}
 \mathcal J(z)
 =
 \int_0^{2\pi}
 |\sin\theta|^{-z}\,d\theta.
\end{equation}
By symmetry,
\[
 \mathcal J(z)
 =
 4\int_0^{\pi/2}
 (\sin\theta)^{-z}\,d\theta.
\]
Since $\sin\theta\ge2\theta/\pi$
for $0\le\theta\le\pi/2$,
\begin{equation}\label{rad:eq:J-upper}
 \mathcal J(z)
 \le\frac{2\pi}{1-z}<\infty.
\end{equation}

For every nonzero $h\in\R^2$,
rotational invariance gives
\begin{equation}\label{rad:eq:angle-moment}
 \int_{S^1}|P_wh|^{-z}\,dw
 =
 \mathcal J(z)|h|^{-z}.
\end{equation}

\begin{lemma}\label{rad:lem:B}
Let $0<t<1$ and $r\ge1$ satisfy $tr<1$.
Let $\nu$ be a finite positive Borel measure
on $\R^2$ with $I_t(\nu)<\infty$, and set
\[
 B(w)=I_t((P_w)_\#\nu).
\]
Then $B:S^1\to[0,\infty]$ is Borel measurable
and
\begin{equation}\label{rad:eq:B-norm}
 \|B\|_{L^r(S^1,dw)}
 \le
 \mathcal J(tr)^{1/r}I_t(\nu).
\end{equation}
In particular, $B(w)<\infty$ for
arclength-almost every $w\in S^1$.
\end{lemma}

\begin{proof}
By the definition of pushforward,
\[
 B(w)
 =
 \iint
 |P_w(z-z')|^{-t}\,d\nu(z)d\nu(z').
\]
The integrand is nonnegative and Borel,
so $B$ is Borel measurable.

For $M>0$, truncate the kernel at height $M$
and apply Minkowski's integral inequality.
Since $I_t(\nu)<\infty$, the measure $\nu$
has no atoms, and hence $z\ne z'$ for
$(\nu\otimes\nu)$-almost every $(z,z')$.
By~\eqref{rad:eq:angle-moment}, the $L^r$
norm of the truncated kernel is bounded by
\[
 \mathcal J(tr)^{1/r}|z-z'|^{-t}.
\]
Integrating this bound against
$d\nu(z)d\nu(z')$ gives the right-hand
side of~\eqref{rad:eq:B-norm}.
Monotone convergence as $M\to\infty$
proves the assertion.
\end{proof}

For the exponents in~\eqref{rad:eq:angular-exponents},
the lemma applies with $r=r_0$.

As an example, let $\nu$ be normalized length
measure on a segment of length $\ell>0$
with unit direction $e$.
For $0<t<1$,
\[
 I_t(\nu)
 =
 \frac{2\ell^{-t}}{(1-t)(2-t)}.
\]
Whenever $P_we\ne0$,
\[
 B(w)=|P_we|^{-t}I_t(\nu).
\]
At the two directions for which $P_we=0$,
the projected measure is a point mass and
$B(w)=+\infty$.
Nevertheless,
\[
 \|B\|_{L^{r_0}(S^1,dw)}
 =
 \mathcal J(tr_0)^{1/r_0}I_t(\nu).
\]
Thus the bound~\eqref{rad:eq:B-norm}
is attained in this example.
\subsection{A projection estimate}
\label{rad:sec:projection}

Suppose that $d\lambda(z)=a(z)\,dz$, where
$a\in C_c^\infty(B(0,2))$ is nonnegative and
$\int a(z)\,dz=1$.
Let $\nu$ be a Borel probability supported on
$\overline B(0,1)$.

For $w\in S^1$, the density of $(P_w)_\#\lambda$ is
\begin{equation}\label{rad:eq:fw}
 f_w(q)=\int_\R a(qw^\perp+rw)\,dr.
\end{equation}
The function $f_w$ is nonnegative, smooth, and supported
in $[-2,2]$. Fubini's theorem and the orthogonal change
of variables give
\begin{equation}\label{rad:eq:slice}
 \widehat f_w(\xi)
 =\widehat\lambda(\xi w^\perp).
\end{equation}

Set
\begin{equation}\label{rad:eq:A-def}
 A(w)=\int_\R
 |\widehat\lambda(\xi w^\perp)|^2|\xi|^\chi\,d\xi,
 \qquad
 B(w)=I_t(\nu_w).
\end{equation}

Fix $w\in S^1$ with $B(w)<\infty$.
By Lemma~\ref{rad:lem:B}, this holds for
arclength-almost every $w$.
Write $\gamma=\nu_w$, and let $g\ge0$ be Borel with
\[
 \|g\|_{L^{p'}(\gamma)}=1,
 \qquad p'=\frac{p}{p-1}.
\]
Since $\gamma$ is a probability, $g\gamma$ is a finite
positive measure.
H\"older's inequality on the product measure
$\gamma\otimes\gamma$ gives
\begin{align}
 I_u(g\gamma)
 &=\iint |q-q'|^{-u}g(q)g(q')
       \,d\gamma(q)d\gamma(q')\notag\\
 &\le I_{up}(\gamma)^{1/p}
 \left(
 \iint g(q)^{p'}g(q')^{p'}
       \,d\gamma(q)d\gamma(q')
 \right)^{1/p'}\notag\\
 &=I_{up}(\gamma)^{1/p}.
 \label{rad:eq:weighted-energy}
\end{align}
The inequality follows by truncating the kernel and
applying monotone convergence. Since $I_t(\gamma)<\infty$,
the diagonal has zero $\gamma\otimes\gamma$-measure.

As $\operatorname{supp}\gamma\subset[-1,1]$ and
$up<t$, we have
\[
 |q-q'|^{-up}
 =|q-q'|^{t-up}|q-q'|^{-t}
 \le2^{d_\chi}|q-q'|^{-t}
\]
for $q\ne q'$ in $\operatorname{supp}\gamma$.
Thus
\begin{equation}\label{rad:eq:weighted-final}
 I_u(g\gamma)
 \le2^{d_\chi/p}B(w)^{1/p}.
\end{equation}

Applying Lemma~\ref{rad:lem:pairing} to $f_w$
and $g\gamma$, and using $1-u=\chi$
and~\eqref{rad:eq:slice}, we obtain
\[
 \int f_wg\,d\gamma
 \le
 \frac{c_{1,u}^{-1/2}}{2\pi}
 2^{d_\chi/(2p)}
 B(w)^{1/(2p)}A(w)^{1/2}.
\]
Since $f_w\ge0$, taking the supremum over
nonnegative $g$ of $L^{p'}(\gamma)$ norm one gives
\begin{equation}\label{rad:eq:direction-bound}
 \|f_w\|_{L^p(\nu_w)}
 \le
 Q_\chi B(w)^{1/(2p)}A(w)^{1/2},
\end{equation}
where $Q_\chi>0$ depends only on $\chi$.
Consequently,
\begin{equation}\label{rad:eq:direction-power}
 \|f_w\|_{L^p(\nu_w)}^p
 \le Q_\chi^p A(w)^{p/2}B(w)^{1/2}.
\end{equation}

\subsection{The radial estimate for a smooth measure}
\label{rad:sec:smooth}

For $x\in\overline B(0,1)$, polar coordinates about
$x$ show that $(\pi_x)_\#\lambda$ has density
\begin{equation}\label{rad:eq:radial-density}
 F_a(x,w)=\int_0^\infty a(x+rw)\,r\,dr.
\end{equation}
If $a(x+rw)\ne0$, then $r\le3$.
Hence
\begin{equation}\label{rad:eq:radial-line}
 0\le F_a(x,w)
 \le3\int_\R a(x+rw)\,dr
 =3f_w(P_wx).
\end{equation}

By polar integration and the Fourier-energy
identity~\eqref{rad:eq:energy-fourier},
\begin{align}
 \int_{S^1}A(w)\,dw
 &=\int_{S^1}\int_\R
 |\widehat\lambda(rw^\perp)|^2
 |r|^\chi\,dr\,dw\notag\\
 &=2\int_{\R^2}
 |\widehat\lambda(\xi)|^2
 |\xi|^{\chi-1}\,d\xi\notag\\
 &=\frac2{c_{2,s}}I_s(\lambda).
 \label{rad:eq:A-average}
\end{align}
The factor two accounts for the two signs of $r$.

Tonelli's theorem,~\eqref{rad:eq:radial-line},
and~\eqref{rad:eq:direction-power} give
\begin{align*}
 \int\!\int F_a(x,w)^p\,dw\,d\nu(x)
 &\le3^p\int_{S^1}\int_\R
 f_w(q)^p\,d\nu_w(q)\,dw\\
 &\le3^pQ_\chi^p
 \int_{S^1}A(w)^{p/2}B(w)^{1/2}\,dw.
\end{align*}
H\"older's inequality with exponents
$2/p$ and $2/(2-p)$, followed by
Lemma~\ref{rad:lem:B}, yields
\begin{align*}
 \int_{S^1}A^{p/2}B^{1/2}\,dw
 &\le
 \left(\int_{S^1}A\,dw\right)^{p/2}
 \|B\|_{L^{r_0}(S^1,dw)}^{1/2}\\
 &\le
 \left(\frac2{c_{2,s}}I_s(\lambda)\right)^{p/2}
 \mathcal J(z_0)^{1/(2r_0)}
 I_t(\nu)^{1/2}.
\end{align*}

Since $z_0<1$, the constants in these estimates
are finite and depend only on $\chi$.
Thus, for a suitable $\Cr(\chi)>0$,
\begin{equation}\label{rad:eq:smooth-result}
 \int\|F_a(x,\cdot)\|_p^p\,d\nu(x)
 \le
 \Cr(\chi)
 I_s(a(z)\,dz)^{p/2}I_t(\nu)^{1/2}.
\end{equation}
No absolute continuity assumption on $\nu$
was used.

\subsection{Extension to general measures}
\label{rad:sec:limit}

Let $\lambda$ and $\nu$ satisfy the hypotheses
of Theorem~\ref{rad:thm:radial}, and set
$E=\operatorname{supp}\lambda$ and
$K=\operatorname{supp}\nu$.
Since these are disjoint compact sets,
\[
 \Delta=\operatorname{dist}(E,K)>0.
\]
Choose $\varepsilon_n\downarrow0$ such that
\[
 0<\varepsilon_n<\min\{1,\Delta/4\},
\]
and let
\[
 \lambda_n=\lambda*\varphi_{\varepsilon_n},
\]
where $\varphi_\varepsilon$ is the smooth
probability kernel defined in
Subsection~\ref{rad:sec:energy}.

Each $\lambda_n$ has a nonnegative smooth density
supported in $B(0,2)$, and its support remains
disjoint from $K$.
By~\eqref{rad:eq:contraction},
\[
 I_s(\lambda_n)\le I_s(\lambda).
\]

For $h\in C(S^1)$ and $x\in K$,
\[
 \int h\,d((\pi_x)_\#\lambda_n)
 =
 \iint h(\pi_x(y+z))
 \varphi_{\varepsilon_n}(z)\,dz\,d\lambda(y).
\]
For $y\in E$ and $|z|\le\varepsilon_n$,
\[
 |y+z-x|\ge3\Delta/4.
\]
Moreover, the unit directions of $y+z-x$
and $y-x$ differ by at most
$4\varepsilon_n/\Delta$.
Uniform continuity of $h$ therefore gives
\[
 \int h\,d((\pi_x)_\#\lambda_n)
 \longrightarrow
 \int h\,d((\pi_x)_\#\lambda)
\]
uniformly for $x\in K$.
In particular, the radial measures converge
weakly at every center in $K$.

By Lemma~\ref{rad:lem:lsc},
\[
 \np((\pi_x)_\#\lambda)^p
 \le
 \liminf_{n\to\infty}
 \np((\pi_x)_\#\lambda_n)^p
 \qquad(x\in K).
\]
Fatou's lemma,~\eqref{rad:eq:smooth-result},
and~\eqref{rad:eq:contraction} now give
\begin{align*}
 \int\np((\pi_x)_\#\lambda)^p\,d\nu(x)
 &\le
 \liminf_{n\to\infty}
 \int\np((\pi_x)_\#\lambda_n)^p\,d\nu(x)\\
 &\le
 \Cr(\chi)
 I_s(\lambda)^{p/2}I_t(\nu)^{1/2}.
\end{align*}

The right-hand side is finite, so
$(\pi_x)_\#\lambda$ has an $L^p(S^1,dw)$
density for $\nu$-almost every $x$.
Lemma~\ref{rad:lem:joint} provides the jointly
Borel density asserted in
Theorem~\ref{rad:thm:radial}.

Although the admissible mollification radii
depend on $\Delta$, the constant in
\eqref{rad:eq:smooth-result} does not.
The same constant therefore applies to the
original measures, completing the proof
of Theorem~\ref{rad:thm:radial}.

\subsection{Selection of compact sets}
\label{prep:sec:theorem}

Fix $0<\eta<1$.
We first choose constants independent of the
initial set and probability:
\begin{equation}\label{prep:eq:fixed-geometry}
 r_*=2^{-8},\qquad
 \ell=r_*/100,\qquad
 L=2^{34},\qquad
 \qs=\frac7{8L},\qquad
 \qr=\frac{3\qs}{4}.
\end{equation}
The square $[-1,1]^2$ admits a fixed cover
by $L$ closed squares of side $2^{-16}$,
each of diameter at most $\ell$.

For $0<a<1+\eta$, set
\begin{equation}\label{prep:eq:Bdef}
 B_a(\eta)=1+\frac{2a}{1+\eta-a}.
\end{equation}
We use the radial parameters
\begin{equation}\label{prep:eq:chi}
 \chi=\min\{\eta/8,1/16\},\qquad
 s=1+\chi,\quad
 t=1-\chi/2,\quad
 p=1+\chi/4,
\end{equation}
and define
\begin{equation}\label{prep:eq:energy-M}
 E_s=\qs^{-1}B_s(\eta),\qquad
 E_t=\qs^{-1}B_t(\eta),\qquad
 M=\max\{1,\Cr(\chi)E_s^{p/2}E_t^{1/2}\}.
\end{equation}
Since $s,t<1+\eta$, these quantities are finite.

For the conclusion below, put
\begin{equation}\label{prep:eq:output-constants}
 \boxed{
 a_*=\frac{7r_*}{8}=\frac7{2048},\qquad
 d_*=\frac14,\qquad
 C_F(\eta)=\frac{8^{2+\eta}}{3\qs},\qquad
 M_*(\eta)=8M(4/3)^p.
 }
\end{equation}

\Needspace{390pt}
\begin{theorem}\label{prep:thm:preparation}
Let $F\subset\overline B(0,1)$ be compact,
and let $\nu$ be a Borel probability supported
on $F$ satisfying~\eqref{main:normalized-ball}.
There exist compact sets
$K_1^0,K_2^0\subset F$ such that
\[
 \nu(K_j^0)\ge\qr>0,\qquad j=1,2.
\]

Define
\begin{equation}\label{prep:eq:S}
 S(z)=z/8+(1/2,1/2),\qquad
 K_j=S(K_j^0),\qquad
 \mu_j=S_\#\left(
 \frac{\nu|_{K_j^0}}{\nu(K_j^0)}
 \right).
\end{equation}
Then $K_j\subset[3/8,5/8]^2
\subset(1/4,3/4)^2$, and
\begin{equation}\label{prep:eq:output-ball}
 \mu_j(B(z,r))
 \le C_F(\eta)r^{1+\eta}
 \qquad(z\in\R^2,\ 0<r\le1).
\end{equation}

There exists $w_0\in S^1$ such that,
for every $(x,y)\in K_1\times K_2$,
\begin{equation}\label{prep:eq:output-geometry}
 a_*\le|x-y|\le d_*,
 \qquad
 d_{S^1}\left(
 \frac{x-y}{|x-y|},w_0
 \right)\le\frac\pi{400}.
\end{equation}
The direction $w_0$ may depend on $F$ and $\nu$,
but the angular bound does not.

For every $x\in K_1$, the measure
$(\pi_x)_\#\mu_2$ has an $L^p(S^1,dw)$
density satisfying
\[
 \np((\pi_x)_\#\mu_2)^p\le M_*(\eta).
\]
Likewise, for every $y\in K_2$,
\[
 \np((\pi_y)_\#\mu_1)^p\le M_*(\eta).
\]
Both families have jointly Borel density
representatives. In particular,
\begin{equation}\label{prep:eq:integrated-output}
 \int_{K_1}
 \np((\pi_x)_\#\mu_2)^p\,d\mu_1(x)
 \le M_*,
 \qquad
 \int_{K_2}
 \np((\pi_y)_\#\mu_1)^p\,d\mu_2(y)
 \le M_*.
\end{equation}

Both probabilities assign zero mass to every
line in $\R^2$, and hence to the boundaries
of the fixed dyadic grid.

All the bounds above are independent of
$F$ and $\nu$.
The similarity $S$ is fixed, and no translation
of the dyadic grid is required.
\end{theorem}

The remainder of this section proves the theorem.

\subsection{Energy bounds for normalized restrictions}
\label{prep:sec:potential}

\begin{lemma}\label{prep:lem:potential}
Let $\nu$ satisfy~\eqref{main:normalized-ball}
and let $0<a<1+\eta$.
Then
\begin{equation}\label{prep:eq:potential-bound}
 \sup_{z\in\R^2}U_a\nu(z)
 \le B_a(\eta).
\end{equation}
If $E\subset\R^2$ is Borel with
$q=\nu(E)>0$ and
$\lambda=\nu|_E/q$, then
\begin{equation}\label{prep:eq:restriction-energy}
 \sup_{z\in\R^2}U_a\lambda(z)
 \le q^{-1}B_a(\eta),
 \qquad
 I_a(\lambda)\le q^{-1}B_a(\eta),
\end{equation}
and
\begin{equation}\label{prep:eq:restriction-ball}
 \lambda(B(z,r))
 \le\frac2q r^{1+\eta}
 \qquad(z\in\R^2,\ 0<r\le1).
\end{equation}
\end{lemma}

\begin{proof}
For $d>0$,
\[
 d^{-a}
 \le
 1+a\int_0^1
 r^{-a-1}\one_{\{d<r\}}\,dr,
\]
with equality when $d\le1$.
The inequality also holds in the extended
sense at $d=0$.
By Tonelli's theorem and~\eqref{main:normalized-ball},
\[
 \begin{aligned}
 U_a\nu(z)
 &\le
 1+a\int_0^1r^{-a-1}\nu(B(z,r))\,dr\\
 &\le
 1+2a\int_0^1r^{\eta-a}\,dr
 =B_a(\eta).
 \end{aligned}
\]
The last integral is finite because
$a<1+\eta$.

Since $\lambda\le q^{-1}\nu$,
the asserted potential and ball estimates
follow.
Integrating the potential bound against
the probability $\lambda$ gives the
energy estimate.
\end{proof}

The potential estimate yields a factor
$q^{-1}$ in the energy bound for a
normalized restriction, rather than
the factor $q^{-2}$ obtained by directly
bounding the double integral.

\subsection{Separated compact subsets}
\label{prep:sec:patches}

Among the $L$ squares of the fixed cover,
one square $Q_1$ satisfies
$\nu(Q_1)\ge1/L$.
Choose $x_0\in F\cap Q_1$.
By~\eqref{main:normalized-ball},
\[
 \nu(B(x_0,8r_*))
 \le2(8r_*)^{1+\eta}
 \le\frac1{16}.
\]
Hence some square $Q_2$ in the same cover
satisfies
\[
 \nu(Q_2\setminus B(x_0,8r_*))
 \ge\frac7{8L}=\qs.
\]

Define
\begin{equation}\label{prep:eq:patches}
 \begin{gathered}
 E_1=F\cap Q_1,\qquad
 E_2=F\cap Q_2
     \cap\{z:|z-x_0|\ge8r_*\},\\
 q_j=\nu(E_j),\qquad
 \lambda_j=\frac{\nu|_{E_j}}{q_j}.
 \end{gathered}
\end{equation}
The sets $E_j$ are compact, and
\[
 q_j\ge\qs,\qquad
 \operatorname{diam}E_j\le\ell.
\]
For $x\in E_1$ and $y\in E_2$,
\begin{equation}\label{prep:eq:patch-separation}
 7r_*
 \le8r_*-\ell
 \le|x-y|
 \le2.
\end{equation}

Choose $y_0\in E_2$ and set
\[
 b=x_0-y_0,\qquad
 w_0=\frac b{|b|}.
\]
For nonzero vectors $a,b\in\R^2$,
\begin{equation}\label{prep:eq:unit-vectors}
 \left|
 \frac a{|a|}-\frac b{|b|}
 \right|
 \le\frac{2|a-b|}{|b|}.
\end{equation}
Taking $a=x-y$ gives
\[
 |a-b|
 \le|x-x_0|+|y-y_0|
 \le2\ell,\qquad |b|\ge8r_*.
\]
Consequently,
\begin{equation}\label{prep:eq:chord}
 \left|
 \frac{x-y}{|x-y|}-w_0
 \right|
 \le\frac{4\ell}{8r_*}
 =\frac1{200}.
\end{equation}
Since the angular distance between
two unit vectors is at most $\pi/2$
times their Euclidean distance,
\[
 d_{S^1}\left(
 \frac{x-y}{|x-y|},w_0
 \right)\le\frac\pi{400}.
\]

Lemma~\ref{prep:lem:potential} gives
\[
 I_s(\lambda_j)\le E_s,\qquad
 I_t(\lambda_j)\le E_t.
\]
Applying Theorem~\ref{rad:thm:radial}
with the two measures in either order,
we obtain
\begin{equation}\label{prep:eq:patch-averages}
 \int_{E_1}
 \np((\pi_x)_\#\lambda_2)^p\,d\lambda_1(x)
 \le M,
 \qquad
 \int_{E_2}
 \np((\pi_y)_\#\lambda_1)^p\,d\lambda_2(y)
 \le M.
\end{equation}

\subsection{Compact sets with pointwise radial bounds}
\label{prep:sec:compact}

By Lemma~\ref{rad:lem:lsc}, the functions
\[
 x\longmapsto\np((\pi_x)_\#\lambda_2),
 \qquad
 y\longmapsto\np((\pi_y)_\#\lambda_1)
\]
are lower semicontinuous on $E_1$ and
$E_2$, respectively.
Therefore the sets
\begin{equation}\label{prep:eq:good-sets}
 \begin{split}
 K_1^0
 &=\{x\in E_1:
   \np((\pi_x)_\#\lambda_2)^p\le8M\},\\
 K_2^0
 &=\{y\in E_2:
   \np((\pi_y)_\#\lambda_1)^p\le8M\}
 \end{split}
\end{equation}
are compact.

Put $z_j=\lambda_j(K_j^0)$.
Markov's inequality and~\eqref{prep:eq:patch-averages}
give
\begin{equation}\label{prep:eq:retained-patch}
 z_j\ge7/8>3/4.
\end{equation}
Define
\begin{equation}\label{prep:eq:final-unscaled}
 \widetilde\mu_j
 =\frac{\lambda_j|_{K_j^0}}{z_j}.
\end{equation}

For every $x\in K_1^0$,
positivity of the measures gives
\begin{equation}\label{prep:eq:radial-domination}
 (\pi_x)_\#\widetilde\mu_2
 \le
 z_2^{-1}(\pi_x)_\#\lambda_2.
\end{equation}
Thus $(\pi_x)_\#\widetilde\mu_2$
has an $L^p$ density, and
\begin{equation}\label{prep:eq:pointwise-radial}
 \np((\pi_x)_\#\widetilde\mu_2)^p
 \le
 z_2^{-p}\np((\pi_x)_\#\lambda_2)^p
 \le8Mz_2^{-p}
 \le8M(4/3)^p=M_*.
\end{equation}
The same argument, with the measures
interchanged, applies at every
$y\in K_2^0$.

The mass retained from the original
measure satisfies
\begin{equation}\label{prep:eq:original-mass}
 \nu(K_j^0)
 =q_jz_j
 \ge\frac78\qs
 \ge\qr.
\end{equation}
Moreover,
\begin{equation}\label{prep:eq:identity-normalization}
 \widetilde\mu_j
 =\frac{\nu|_{K_j^0}}{q_jz_j}
 =\frac{\nu|_{K_j^0}}{\nu(K_j^0)}.
\end{equation}

\bigskip \paragraph{\emph{Joint measurability.}}
\label{prep:subsec:joint}

Apply Lemma~\ref{rad:lem:joint} to
$\widetilde\mu_2$ with center set $K_1^0$,
and then to $\widetilde\mu_1$ with
center set $K_2^0$.
By~\eqref{prep:eq:pointwise-radial}
and its counterpart with the roles reversed,
the radial density norm is finite
at every center in the respective set.
Both families therefore admit finite-valued,
nonnegative, jointly Borel density representatives.

\subsection{Normalization by a fixed similarity}
\label{prep:sec:scaling}

Let $S$ be the similarity in~\eqref{prep:eq:S}.
Since $F\subset\overline B(0,1)$,
\[
 S(F)\subset[3/8,5/8]^2
 \subset(1/4,3/4)^2.
\]
For $x\in K_1^0$ and $y\in K_2^0$,
\[
 |S(x)-S(y)|
 =\frac{|x-y|}{8}
 \in[7r_*/8,1/4].
\]
The distance bounds in
\eqref{prep:eq:output-geometry} follow.

For distinct $x,y\in\R^2$,
\begin{equation}\label{prep:eq:direction-invariance}
 \pi_{S(x)}(S(y))=\pi_x(y),
 \qquad
 \frac{S(x)-S(y)}{|S(x)-S(y)|}
 =\frac{x-y}{|x-y|}.
\end{equation}
Thus the angular bound is preserved.

For $0<u\le1$,
\[
 \widetilde\mu_j(B(z,u))
 \le
 \frac2{\nu(K_j^0)}u^{1+\eta}
 \le\frac2\qr u^{1+\eta}.
\]
If $0<r\le1/8$, the inverse image
under $S$ of a ball of radius $r$
is a ball of radius $8r$.
Consequently,
\[
 \begin{aligned}
 \mu_j(B(z,r))
 &\le\frac2\qr(8r)^{1+\eta}\\
 &=\frac{8^{2+\eta}}{3\qs}r^{1+\eta}
 =C_F(\eta)r^{1+\eta}.
 \end{aligned}
\]
For $1/8<r\le1$, the same inequality
follows from $\mu_j(B(z,r))\le1$ and
$C_F(\eta)\ge8^{1+\eta}$.
This proves~\eqref{prep:eq:output-ball}.

\bigskip \paragraph{\emph{Radial densities.}}

By~\eqref{prep:eq:direction-invariance},
the radial projections before and after
applying $S$ are identical as measures
on $S^1$.
Thus~\eqref{prep:eq:pointwise-radial}
gives the asserted bound $M_*$ for
the scaled probabilities.

For the distance argument, define
\[
 e(x,y)=\frac{x-y}{|x-y|}.
\]
For $x\in K_1$ and $y\in K_2$,
\[
 e(x,y)=-\pi_x(y)=\pi_y(x).
\]
If $f_1(x,w)$ and $f_2(y,w)$ are
jointly Borel densities of
$(\pi_x)_\#\mu_2$ and
$(\pi_y)_\#\mu_1$, respectively, then
\[
 F_1(x,w)=f_1(x,-w),\qquad
 F_2(y,w)=f_2(y,w)
\]
give the corresponding densities
with respect to the direction $e(x,y)$.
The antipodal map preserves arclength
and $L^p$ norms, so integration of
the pointwise bounds proves
\eqref{prep:eq:integrated-output}.

\bigskip \paragraph{\emph{A smooth angular cutoff.}}

Choose a fixed function
$\psi_0\in C^\infty(S^1)$ satisfying
$0\le\psi_0\le1$, equal to one on
the arc of angular radius $\pi/64$
about $(1,0)$, and supported in the
arc of angular radius $\pi/32$.

For each pair of sets, rotate $\psi_0$
so that its center is $w_0$.
By~\eqref{prep:eq:output-geometry},
the rotated function equals one on
every direction $(x-y)/|x-y|$ with
$(x,y)\in K_1\times K_2$.
It vanishes in a neighborhood of
the antipodal directions.
Its derivative bounds are unchanged
by rotation and therefore do not
depend on $F$ or $\nu$.

This cutoff is used in the subsequent
inverse conversion and does not enter
the definition of the distance approximants.

\bigskip \paragraph{\emph{Dyadic boundaries.}}

Let $\mathcal L$ be a line.
The intersection of $\mathcal L$
with the support of $\mu_j$ is contained
in a segment of length at most two.
For $0<r\le1$, that segment can be
covered by at most $2/r+2$ balls
of radius $r$.
Hence
\[
 \mu_j(\mathcal L)
 \le(2/r+2)C_F r^{1+\eta}
 \le4C_F r^\eta.
\]
Letting $r\downarrow0$ gives
$\mu_j(\mathcal L)=0$.
Since the boundaries of all dyadic
squares form a countable union of lines,
both probabilities assign zero mass
to that union.

This completes the proof of
Theorem~\ref{prep:thm:preparation}.\qed

\subsection{Examples concerning compactness and normalization}
\label{prep:sec:examples}

\bigskip \paragraph{\emph{Closure of a set of centers.}}

For $x\in(0,1]$, let $\zeta_x$ be
normalized arclength on the angular
interval $[0,x]$, and set
$\zeta_0=\delta_0$.
The family $\zeta_x$ is weakly continuous,
but
\[
 \np(\zeta_x)^p=x^{1-p}
 \quad(x>0),
 \qquad
 \np(\zeta_0)=+\infty.
\]
Thus every $x>0$ has finite density norm,
whereas its limiting measure at $x=0$
does not have an $L^p$ density.
For $T\ge1$,
\[
 \{x:\np(\zeta_x)^p\le T\}
 =[T^{-1/(p-1)},1],
\]
which is compact.
This illustrates why the sets
in~\eqref{prep:eq:good-sets}
are defined by finite sublevel bounds.

\bigskip \paragraph{\emph{Normalization of a restriction.}}

Let $\lambda$ be normalized arclength
on $S^1$, and restrict it to an arc
of $\lambda$-mass $z\in(0,1)$.
After normalization, its density equals
$(2\pi z)^{-1}$ on that arc, and
\[
 \left\|
 \frac{\lambda|_E}{\lambda(E)}
 \right\|_{L^p(dw)}^p
 =(2\pi)^{1-p}z^{1-p}.
\]
The norm can therefore increase under
restriction followed by normalization.
The general estimate
\eqref{prep:eq:pointwise-radial}
accounts for this through the factor
$z_j^{-p}$.

\subsection{Consequences for the distance measure}
\label{prep:sec:summary}

Theorem~\ref{prep:thm:preparation}
provides probabilities $\mu_1,\mu_2$
whose geometric and radial bounds
depend only on $\eta$.
The compact sets and the direction $w_0$
may depend on the input, but the similarity
and the numerical bounds are fixed.

Let $\alpha$ be the weighted distance measure
defined in~\eqref{rec:weighted-measure}
for $\mu_1,\mu_2$.
Since $a_*\le|x-y|\le d_*$ on
$K_1\times K_2$,
\begin{equation}\label{prep:eq:alpha-mass}
 2\le\alpha(\R)\le a_*^{-1/2},
 \qquad
 M_0^2=a_*^{-1}=2048/7.
\end{equation}
The measure $\alpha$ is supported on
$D(K_1,K_2)$.

For $x\in K_1^0$ and $y\in K_2^0$,
\[
 |x-y|=8|S(x)-S(y)|.
\]
Consequently,
\begin{equation}\label{prep:eq:distance-inclusion}
 8D(K_1,K_2)\subseteq\Delta F.
\end{equation}

The weighted distance measure has
a different scaling factor.
Let
\[
 \alpha^0
 =\dd_\#\bigl(
 |x-y|^{-1/2}
 \widetilde\mu_1\otimes\widetilde\mu_2
 \bigr),
 \qquad
 T_{1/8}(r)=r/8.
\]
Then
\begin{equation}\label{prep:eq:weighted-scaling}
 \alpha=\sqrt8\,(T_{1/8})_\#\alpha^0.
\end{equation}

Suppose the later construction establishes
\eqref{rec:conditional-estimates}
with constants $a,\gamma,C_e,C_s,N_*$,
independent of the initial set and measure.
Choose an integer
\[
 J\ge
 \max\left\{
 N_*,1,
 \left\lceil
 \frac1a\log_2
 \frac{16C_e}{1-2^{-a}}
 \right\rceil
 \right\}.
\]
Proposition~\ref{rec:conditional}
and~\eqref{prep:eq:distance-inclusion}
then give
\begin{equation}\label{prep:eq:conditional-normalized}
 |\Delta F|
 \ge
 \frac{16\pi}{
 3\cdot2^{J-1}(2048/7)
 +
 C_s2^{-\gamma J}/(1-2^{-\gamma})
 }>0.
\end{equation}

Sections~\ref{main:cutoffs-mass}
and~\ref{main:frequency}
establish the required estimates
for the same approximating measures.
Section~\ref{main:normalized-proof}
then applies the reconstruction argument.

The probabilities $\mu_1,\mu_2$
remain fixed as the frequency scale varies.
The dyadic decomposition in
Section~\ref{main:dyadic}
may depend on the scale, while the
original masses of its classes provide
the weights used to combine the
corresponding restricted probabilities.

\clearpage

\section{Dyadic decomposition and class weights}
\label{main:dyadic}

Fix a positive integer $N$ and write $R=2^N$.
We partition the depth-$N$ dyadic squares into
classes for which the masses of occupied squares
are comparable to quantities depending only on
their depths. The normalized restriction to each
class is used in subsequent estimates, while its
original mass is retained as the weight for
recombining the classes.

The construction applies to arbitrary Borel
probabilities on the half-open unit square.
The additional ball estimate is used only to
obtain a lower bound for the functions describing
the square masses. The classes may depend on
$N$ and need not be nested as $N$ varies.

\subsection{Dyadic squares and the decomposition theorem}
\label{dyad:sec:theorem}

Let $Q_0=[0,1)^2$, and let $\D_n$ denote its
partition into half-open dyadic squares of
side length $2^{-n}$.
Every point of $Q_0$ belongs to exactly one
square at each depth.

Let $\mu$ be a Borel probability with
$\mu(Q_0)=1$.
A class $C$ is a union of depth-$N$ squares.
If $\mu(C)>0$, set
\begin{equation}\label{dyad:eq:class-weight}
 a_C=\mu(C),\qquad
 \rho_C=\frac{\mu|_C}{a_C}.
\end{equation}
We say that a dyadic square $P$ is occupied
by $\rho_C$ if $\rho_C(P)>0$.
For every positive-mass depth-$N$ square
$L\subset C$,
\[
 \frac{\rho_C|_L}{\rho_C(L)}
 =
 \frac{\mu|_L}{\mu(L)}.
\]
Thus normalization on a class preserves the
conditional distribution within each of
its constituent depth-$N$ squares.

For a function $m_C:[0,1]\to\R$, define
\begin{equation}\label{dyad:eq:model}
 R=2^N,\qquad
 M_C(n)=R^{-m_C(n/N)},\qquad
 A_{C,n}=m_C(n/N)-n/N.
\end{equation}
The quantity $M_C(n)$ will be compared
with the actual masses of occupied
depth-$n$ squares.
For such a square $P$, its conditional
probability is
\begin{equation}\label{dyad:eq:conditional}
 \rho_{C,P}
 =\frac{\rho_C|_P}{\rho_C(P)}.
\end{equation}
In particular, $M_C(n)$ is not used
as a normalization denominator.

\begin{theorem}\label{dyad:thm:main}
Fix $0<\sigma<1$, and set
\begin{equation}\label{dyad:eq:Nreg}
 N_{\rm reg}(\sigma)
 =\left\lceil20\sigma^{-2}\right\rceil,
 \qquad
 Q_N=N^{20/\sigma}.
\end{equation}
For every integer $N\ge N_{\rm reg}(\sigma)$
and every Borel probability $\mu$ on $Q_0$,
there exists a finite, possibly empty,
family $\C_N$ of pairwise disjoint unions
of depth-$N$ squares such that
\begin{gather}
 \#\C_N\le Q_N,\qquad
 a_C\ge R^{-\sigma}
 \quad(C\in\C_N),
 \label{dyad:eq:number}\\
 e_N:=
 \mu\left(
 Q_0\setminus\bigcup_{C\in\C_N}C
 \right)
 \le R^{-2}+Q_NR^{-\sigma}.
 \label{dyad:eq:discard}
\end{gather}

For each $C\in\C_N$, there is a continuous,
nondecreasing, $2$-Lipschitz function
$m_C:[0,1]\to[0,2]$, with $m_C(0)=0$,
such that
\begin{equation}\label{dyad:eq:profile-main}
 R^{-m_C(n/N)-10\sigma}
 \le\rho_C(P)
 \le R^{-m_C(n/N)+10\sigma}
\end{equation}
for every $0\le n\le N$ and every
occupied $P\in\D_n$.

Suppose in addition that $0<\eta\le1$,
$C_F\ge1$, and
\begin{equation}\label{dyad:eq:Frostman}
 \mu(B(z,r))
 \le C_Fr^{1+\eta}
 \qquad(z\in\R^2,\ 0<r\le1).
\end{equation}
If
\begin{equation}\label{dyad:eq:Nlower}
 N\ge N_{\rm lower}(\sigma,C_F)
 :=
 \max\left\{
 N_{\rm reg}(\sigma),
 \left\lceil
 \frac{\log_2(2C_F)}{\sigma}
 \right\rceil
 \right\},
\end{equation}
then
\begin{equation}\label{dyad:eq:lower-main}
 m_C(u)\ge(1+\eta)u-12\sigma
 \qquad(0\le u\le1).
\end{equation}
Consequently,
\[
 A_{C,n}\ge\eta n/N-12\sigma.
\]

Finally, the threshold
\begin{equation}\label{dyad:eq:Nmass}
 N_{\rm mass}(\sigma)
 =\left\lceil1600\sigma^{-4}\right\rceil
\end{equation}
ensures
\begin{equation}\label{dyad:eq:discard-small}
 e_N
 \le R^{-2}+R^{-\sigma/2}
 \le2R^{-\sigma/2}<\frac12.
\end{equation}
Thus $\C_N$ is nonempty whenever
$N\ge N_{\rm mass}(\sigma)$.
All the conclusions hold simultaneously
when $N\ge\max\{N_{\rm mass},N_{\rm lower}\}$.
\end{theorem}

The numerical thresholds above are sufficient;
their optimal values are not needed.

\subsection{Construction of the classes}
\label{dyad:sec:procedure}

We first work with an arbitrary increasing
sequence of selected depths
\begin{equation}\label{dyad:eq:selected-general}
 0=n_0<n_1<\cdots<n_{J_{\rm d}}=N.
\end{equation}

Let $\mathcal L\subseteq\D_N$ be a collection
of depth-$N$ squares of positive $\mu$-mass.
For each $L\in\mathcal L$, assign the initial label
\begin{equation}\label{dyad:eq:initial-label}
 \ell_{J_{\rm d}}(L)
 =\lfloor\log_2\mu(L)\rfloor.
\end{equation}
Thus a label $k$ corresponds precisely to
\[
 2^k\le\mu(L)<2^{k+1}.
\]

Suppose the labels $\ell_i$ have been defined
for some $1\le i\le J_{\rm d}$.
For $Q\in\D_{n_{i-1}}$ and an existing label $u$,
define
\begin{equation}\label{dyad:eq:child-count}
 c_i(Q,u)
 =
 \#\left\{
 P\in\D_{n_i}:
 P\subseteq Q,\
 \begin{array}{l}
 \text{there exists }L\in\mathcal L
 \text{ with }L\subseteq P\\
 \text{and }\ell_i(L)=u
 \end{array}
 \right\}.
\end{equation}
Only pairs $(Q,u)$ with $c_i(Q,u)>0$
are considered.

For every $L\in\mathcal L$ contained in $Q$
with $\ell_i(L)=u$, set
\begin{equation}\label{dyad:eq:label-update}
 \ell_{i-1}(L)
 =
 \left(
 u,\lfloor\log_2c_i(Q,u)\rfloor
 \right).
\end{equation}
Starting with $i=J_{\rm d}$, repeat this
operation down to $i=1$.
Each update appends one coordinate and
leaves the preceding coordinates unchanged.

The final labels have the form
\begin{equation}\label{dyad:eq:final-label}
 \Lambda=
 (k,r_{J_{\rm d}},r_{J_{\rm d}-1},
 \ldots,r_1).
\end{equation}
For each label $\Lambda$, let $C_\Lambda$
be the union of all $L\in\mathcal L$
with $\ell_0(L)=\Lambda$.
Empty unions are omitted.

Write
\begin{equation}\label{dyad:eq:branching-integers}
 w=2^k,\qquad
 b_i=2^{r_i}
 \quad(1\le i\le J_{\rm d}).
\end{equation}
Every depth-$N$ square in $C_\Lambda$
has original mass in $[w,2w)$.

\subsection{Preservation of child counts}
\label{dyad:sec:preserve}

\begin{lemma}\label{dyad:lem:group-preservation}
Fix a final class $C_\Lambda$,
an index $1\le i\le J_{\rm d}$,
and a square $Q\in\D_{n_{i-1}}$
occupied by that class.
Let
\[
 u_i=(k,r_{J_{\rm d}},\ldots,r_{i+1})
\]
be the initial part of $\Lambda$.
Then the depth-$N$ squares of
$C_\Lambda$ contained in $Q$ are exactly
\begin{equation}\label{dyad:eq:group-equality}
 \left\{
 L\in\mathcal L:
 L\subseteq Q,\ \ell_i(L)=u_i
 \right\}.
\end{equation}
Consequently, the number of occupied
depth-$n_i$ descendants of $Q$
is $c_i(Q,u_i)$, and
\begin{equation}\label{dyad:eq:surviving-count}
 b_i\le c_i(Q,u_i)<2b_i.
\end{equation}
\end{lemma}

\begin{proof}
Every depth-$N$ square in $C_\Lambda$
has label $u_i$ at stage $i$,
because subsequent updates preserve
the previously assigned coordinates.
This proves one inclusion in
\eqref{dyad:eq:group-equality}.

For the reverse inclusion, choose a
depth-$N$ square $L_0\subset Q$
belonging to $C_\Lambda$.
If $L\subset Q$ satisfies
$\ell_i(L)=u_i$, then $L$ and $L_0$
have the same parent $Q$ and the same
label at stage $i$.
By~\eqref{dyad:eq:label-update},
they receive the same label at stage $i-1$.
At each subsequent stage they also
have the same selected ancestor.
Induction therefore shows that their
final labels coincide, so $L\subset C_\Lambda$.

A depth-$n_i$ descendant of $Q$
is occupied by $C_\Lambda$ precisely
when it contains a square in the
set~\eqref{dyad:eq:group-equality}.
Hence the number of such descendants
is $c_i(Q,u_i)$.
The definition of $r_i$ gives
\eqref{dyad:eq:surviving-count}.
\end{proof}

Deleting an entire class does not change
the counts within any class that remains.

A depth-$n_{i-1}$ square has at most
$4^{n_i-n_{i-1}}$ descendants at depth $n_i$.
Thus
\begin{equation}\label{dyad:eq:b-range}
 1\le b_i
 \le c_i(Q,u_i)
 \le2^{2(n_i-n_{i-1})},
\end{equation}
and
\[
 0\le\log_2b_i
 \le2(n_i-n_{i-1}).
\]

\subsection{Mass estimates from child counts}
\label{dyad:sec:profile}

For $0\le i\le J_{\rm d}$, set
\[
 B_{\le i}=\prod_{h=1}^i b_h,\qquad
 B_{>i}=\prod_{h=i+1}^{J_{\rm d}}b_h,
\]
where empty products equal one.
Let $T_i(P)$ denote the number of
depth-$N$ squares of $C_\Lambda$
contained in an occupied
$P\in\D_{n_i}$.

\begin{lemma}\label{dyad:lem:leaf-counts}
For every such square $P$,
\begin{equation}\label{dyad:eq:leaf-counts}
 B_{>i}
 \le T_i(P)
 \le2^{J_{\rm d}-i}B_{>i}.
\end{equation}
\end{lemma}

\begin{proof}
The assertion is immediate at
$i=J_{\rm d}$.
For the induction step, each occupied
square at depth $n_i$ has at least
$b_{i+1}$ and fewer than $2b_{i+1}$
occupied descendants at depth $n_{i+1}$,
by Lemma~\ref{dyad:lem:group-preservation}.
Summing the inductive bounds over
these descendants proves the assertion.
\end{proof}

Every depth-$N$ square in $C_\Lambda$
has mass in $[w,2w)$.
Therefore
\begin{equation}\label{dyad:eq:raw-selected-mass}
 wB_{>i}
 \le\mu(P\cap C_\Lambda)
 \le2^{J_{\rm d}-i+1}wB_{>i}.
\end{equation}
At the root,
\begin{equation}\label{dyad:eq:class-total-comparison}
 wB_{\le J_{\rm d}}
 \le a_\Lambda
 \le2^{J_{\rm d}+1}wB_{\le J_{\rm d}}.
\end{equation}

Dividing by the actual class mass
$a_\Lambda$ gives
\begin{equation}\label{dyad:eq:selected-sharp}
 2^{-(J_{\rm d}+1)}B_{\le i}^{-1}
 \le\rho_\Lambda(P)
 \le2^{J_{\rm d}-i+1}B_{\le i}^{-1}
 \le2^{J_{\rm d}+1}B_{\le i}^{-1}.
\end{equation}
In particular,
\begin{equation}\label{dyad:eq:selected-coarse}
 2^{-(2J_{\rm d}+2)}B_{\le i}^{-1}
 \le\rho_\Lambda(P)
 \le2^{2J_{\rm d}+2}B_{\le i}^{-1}.
\end{equation}

Define $m_\Lambda$ at the selected
depths by
\begin{equation}\label{dyad:eq:profile-nodes}
 m_\Lambda(n_i/N)
 =
 \frac1N\log_2B_{\le i}
 =
 \frac1N\sum_{h=1}^i\log_2b_h,
\end{equation}
and extend it linearly between
consecutive selected depths.

By~\eqref{dyad:eq:b-range}, its slope
on $[n_{i-1}/N,n_i/N]$ is
\[
 \frac{\log_2b_i}{n_i-n_{i-1}}
 \in[0,2].
\]
Thus $m_\Lambda$ is continuous,
nondecreasing, and $2$-Lipschitz,
with $m_\Lambda(0)=0$ and
$0\le m_\Lambda(u)\le2u\le2$.

At every selected depth,
\begin{equation}\label{dyad:eq:model-selected}
 R^{-m_\Lambda(n_i/N)}
 =B_{\le i}^{-1}.
\end{equation}

Put
\[
 d_{\max}=\max_i(n_i-n_{i-1}),
 \qquad
 E=2d_{\max}+2J_{\rm d}+2.
\]

\begin{lemma}\label{dyad:lem:all-depths}
For every $0\le n\le N$ and every
square $P\in\D_n$ occupied by
$\rho_\Lambda$,
\begin{equation}\label{dyad:eq:finite-error}
 2^{-E}R^{-m_\Lambda(n/N)}
 \le\rho_\Lambda(P)
 \le2^ER^{-m_\Lambda(n/N)}.
\end{equation}
\end{lemma}

\begin{proof}
At selected depths, the assertion
follows from~\eqref{dyad:eq:selected-coarse}.

Otherwise choose $i$ with
$n_i<n<n_{i+1}$.
Let $U$ be the depth-$n_i$
ancestor of $P$.
Since $P$ is occupied, it contains
an occupied depth-$n_{i+1}$ square $V$.
Hence
\[
 \rho_\Lambda(V)
 \le\rho_\Lambda(P)
 \le\rho_\Lambda(U).
\]
Apply~\eqref{dyad:eq:selected-coarse}
to $U$ and $V$.
Because $m_\Lambda$ is nondecreasing
and $2$-Lipschitz,
\begin{align*}
 0
 &\le N\bigl(
 m_\Lambda(n/N)-m_\Lambda(n_i/N)
 \bigr)
 \le2d_{\max},\\
 0
 &\le N\bigl(
 m_\Lambda(n_{i+1}/N)-m_\Lambda(n/N)
 \bigr)
 \le2d_{\max}.
\end{align*}
Thus replacing the model mass at
either neighboring selected depth
by the model mass at depth $n$
costs at most a factor $2^{2d_{\max}}$.
Together with~\eqref{dyad:eq:selected-coarse},
this proves~\eqref{dyad:eq:finite-error}.
\end{proof}

\subsection{Choice of depths and discarded mass}
\label{dyad:sec:constants}

Retain initially the depth-$N$ squares
\begin{equation}\label{dyad:eq:live-leaves}
 \mathcal L
 =
 \{L\in\D_N:\mu(L)\ge R^{-4}\}.
\end{equation}
Since $\D_N$ contains $R^2$ squares,
the total mass omitted at this step is
at most
\begin{equation}\label{dyad:eq:first-discard}
 \sum_{\mu(L)<R^{-4}}\mu(L)
 \le R^2R^{-4}=R^{-2}.
\end{equation}
The remaining masses are not normalized.

For $L\in\mathcal L$,
the initial label satisfies
\[
 -4N\le\ell_{J_{\rm d}}(L)\le0.
\]

Choose
\begin{equation}\label{dyad:eq:canonical-levels}
 q_{\rm d}=\lfloor\sigma N\rfloor,\qquad
 J_{\rm d}=\lceil N/q_{\rm d}\rceil,\qquad
 n_i=\min\{iq_{\rm d},N\}
 \quad(0\le i\le J_{\rm d}).
\end{equation}
For $N\ge N_{\rm reg}(\sigma)$,
\begin{equation}\label{dyad:eq:qJ}
 \frac12\sigma N
 \le q_{\rm d}\le\sigma N,\qquad
 J_{\rm d}\le2/\sigma+1,\qquad
 d_{\max}\le q_{\rm d}.
\end{equation}
The selected depths are strictly
increasing, although the final
interval may be shorter than $q_{\rm d}$.

At stage $i$, the new label coordinate
has at most $2(n_i-n_{i-1})+1$ possibilities.
Together with the initial mass label,
this gives at most
\begin{equation}\label{dyad:eq:label-count}
 (4N+1)
 \prod_{i=1}^{J_{\rm d}}
 \bigl(2(n_i-n_{i-1})+1\bigr)
 \le(4N+1)(2N+1)^{J_{\rm d}}
\end{equation}
nonempty final classes.
Since $N\ge N_{\rm reg}(\sigma)$,
\begin{equation}\label{dyad:eq:label-Q}
 (4N+1)(2N+1)^{J_{\rm d}}
 \le N^{2+2J_{\rm d}}
 \le N^{20/\sigma}=Q_N.
\end{equation}

By Lemma~\ref{dyad:lem:all-depths}
and~\eqref{dyad:eq:qJ},
\begin{equation}\label{dyad:eq:error-ten}
 E
 \le2q_{\rm d}+2J_{\rm d}+2
 \le10\sigma N.
\end{equation}
Since $R=2^N$, the estimate
\eqref{dyad:eq:finite-error}
therefore implies~\eqref{dyad:eq:profile-main}.

Retain those final classes for which
\begin{equation}\label{dyad:eq:retain-classes}
 a_C=\mu(C)\ge R^{-\sigma}.
\end{equation}
There are at most $Q_N$ classes,
and each deleted class has mass
less than $R^{-\sigma}$.
Combined with~\eqref{dyad:eq:first-discard},
this proves~\eqref{dyad:eq:number}
and~\eqref{dyad:eq:discard}.

For $N\ge16$,
$\log_2N\le\sqrt N$.
If $N\ge N_{\rm mass}(\sigma)$,
then
\begin{equation}\label{dyad:eq:polynomial-absorb}
 \log_2Q_N
 =\frac{20}{\sigma}\log_2N
 \le\frac{20}{\sigma}\sqrt N
 \le\frac{\sigma N}{2}.
\end{equation}
Thus $Q_N\le R^{\sigma/2}$,
which proves~\eqref{dyad:eq:discard-small}.
This completes the construction
of the classes.

\subsection{A lower bound from the ball condition}
\label{dyad:sec:frostman}

Suppose~\eqref{dyad:eq:Frostman} holds.
Every depth-$n$ square $P$ is contained
in an open ball of radius $2^{-n}$
about its center.
Consequently,
\begin{equation}\label{dyad:eq:original-cell-bound}
 \mu(P)
 \le2C_F2^{-(1+\eta)n}.
\end{equation}

For a retained class $C$,
$a_C\ge R^{-\sigma}$, so
\begin{equation}\label{dyad:eq:normalized-upper}
 \rho_C(P)
 =\frac{\mu(P\cap C)}{a_C}
 \le R^\sigma\mu(P)
 \le2C_FR^\sigma2^{-(1+\eta)n}.
\end{equation}
Combining this with the lower bound
in~\eqref{dyad:eq:profile-main} gives
\[
 R^{-m_C(n/N)-10\sigma}
 \le2C_FR^\sigma R^{-(1+\eta)n/N}
\]
for every occupied depth-$n$ square.
Hence
\begin{equation}\label{dyad:eq:lower-before-threshold}
 m_C(n/N)
 \ge
 (1+\eta)n/N
 -11\sigma
 -\frac{\log_2(2C_F)}{N}.
\end{equation}

Under~\eqref{dyad:eq:Nlower},
the last term is at most $\sigma$.
This proves~\eqref{dyad:eq:lower-main}
at every integer depth.
Since $m_C$ is piecewise affine,
the same inequality holds throughout
$[0,1]$ by interpolation.

In particular, if $0<\chi\le\eta$,
then
\[
 A_{C,n}
 \ge\chi n/N-12\sigma.
\]
This completes the proof of
Theorem~\ref{dyad:thm:main}.

\subsection{Recombination with the original class weights}
\label{dyad:sec:weights}

Let
$E_N=\bigcup_{C\in\C_N}C$
and $\mu_N=\mu|_{E_N}$.
Since the classes are disjoint,
\begin{equation}\label{dyad:eq:one-weight-identity}
 \mu_N=\sum_{C\in\C_N}a_C\rho_C,
 \qquad
 \sum_{C\in\C_N}a_C=1-e_N\le1.
\end{equation}
Indeed, $a_C\rho_C=\mu|_C$.
As $\mu-\mu_N$ is positive,
\begin{equation}\label{dyad:eq:one-tv}
 \|\mu-\mu_N\|_{\TV}=e_N.
\end{equation}
No further normalization of the retained
measure is made.

\bigskip \paragraph{\emph{Products of two measures.}}

Apply the construction separately to
probabilities $\mu_1,\mu_2$.
Let $E_{j,N}$ be their retained unions,
$e_j$ their discarded masses, and
$a_{j,C},\rho_{j,C}$ their original
class weights and normalized restrictions.

The product measure satisfies
\begin{equation}\label{dyad:eq:product-identity}
 \mu_1|_{E_{1,N}}\otimes
 \mu_2|_{E_{2,N}}
 =
 \sum_{C,D}a_{1,C}a_{2,D}
 \rho_{1,C}\otimes\rho_{2,D}.
\end{equation}
Since the measure on the left is a
restriction of $\mu_1\otimes\mu_2$,
\begin{align}\label{dyad:eq:product-loss}
 &\left\|
 \mu_1\otimes\mu_2
 -
 \mu_1|_{E_{1,N}}\otimes
 \mu_2|_{E_{2,N}}
 \right\|_{\TV}\notag\\
 &\hspace{20mm}
 =e_1+e_2-e_1e_2
 \le e_1+e_2.
\end{align}

\bigskip \paragraph{\emph{Weighted pushforwards.}}

Let $T$ be a Borel map on the product
of the supports, and let $W$ be a
bounded nonnegative Borel function there.
By~\eqref{dyad:eq:product-identity},
\begin{equation}\label{dyad:eq:weighted-push}
 T_\#\bigl(
 W\,\mu_1|_{E_{1,N}}\otimes
 \mu_2|_{E_{2,N}}
 \bigr)
 =
 \sum_{C,D}a_{1,C}a_{2,D}
 T_\#\bigl(
 W\,\rho_{1,C}\otimes\rho_{2,D}
 \bigr).
\end{equation}
The total variation of the measure
omitted by deleting classes is at most
\begin{equation}\label{dyad:eq:weighted-discard}
 \|W\|_\infty
 (e_1+e_2-e_1e_2).
\end{equation}

For the distance application,
\[
 T(x,y)=|x-y|,
 \qquad
 W(x,y)=|x-y|^{-1/2}.
\]
The separation in
Theorem~\ref{prep:thm:preparation}
gives $W\le a_*^{-1/2}$ on
$K_1\times K_2$.

The cutoff may depend on the ordered
class pair.
Let $0\le\Theta_{C,D}\le1$ be Borel
on $C\times D$.
Define $\Theta_N$ to equal
$\Theta_{C,D}$ on $C\times D$
and zero outside the retained product.
Then
\begin{equation}\label{dyad:eq:theta-weight-identity}
 T_\#(W\Theta_N\,\mu_1\otimes\mu_2)
 =
 \sum_{C,D}a_{1,C}a_{2,D}
 T_\#\bigl(
 W\Theta_{C,D}\,
 \rho_{1,C}\otimes\rho_{2,D}
 \bigr).
\end{equation}
In particular, the weights remain
the original products $a_{1,C}a_{2,D}$.

\bigskip \paragraph{\emph{Fourier estimates.}}

Suppose $T$ is real-valued, and set
\[
 \nu_{C,D}
 =
 T_\#\bigl(
 W\Theta_{C,D}\,
 \rho_{1,C}\otimes\rho_{2,D}
 \bigr).
\]
Let $I\subset\R$ be an interval.
If
\[
 \|\widehat\nu_{C,D}\|_{L^2(I)}
 \le B
\]
for every ordered class pair, then
the triangle inequality gives
\begin{align}\label{dyad:eq:L2-recombine}
 &\left\|
 \sum_{C,D}a_{1,C}a_{2,D}
 \widehat\nu_{C,D}
 \right\|_{L^2(I)}\notag\\
 &\quad\le
 \sum_{C,D}a_{1,C}a_{2,D}
 \|\widehat\nu_{C,D}\|_{L^2(I)}\notag\\
 &\quad\le
 B\left(\sum_Ca_{1,C}\right)
  \left(\sum_Da_{2,D}\right)
 \le B.
\end{align}
Thus no factor depending on the
number of classes is introduced.
If the individual estimates concern
squared $L^2$ norms, their square roots
are taken before applying the triangle
inequality.

\bigskip \paragraph{\emph{Conditional probabilities and supports.}}

The conditional probability on an
occupied square $P$ remains
\[
 \rho_{C,P}
 =\frac{\rho_C|_P}{\rho_C(P)}.
\]
Additional restrictions used in later
estimates do not change the definition
of $\rho_{C,P}$ or the function $m_C$.

If $\mu$ is carried by a compact set
$K$, then $\rho_C$ is carried by
$K\cap C$, and
$\operatorname{supp}\rho_C\subset K$.
Since $C$ is a union of half-open squares,
its topological boundary may meet
$\operatorname{supp}\rho_C$.

\subsection{A finite example}
\label{dyad:sec:example}
\label{dyad:example:seven}

Take $N=2$ and suppose seven
depth-two squares each have mass $1/7$.
Assume that the four depth-one squares
contain respectively $1,1,2,3$
of these occupied squares.
This example illustrates the construction
without satisfying the large-$N$ threshold
of Theorem~\ref{dyad:thm:main}.

Every initial label equals $-3$.
The first two depth-one squares have
child-count label zero, while the
remaining two have child-count label one.
At the root, both resulting labels
occur in two depth-one squares.

Thus the two final classes have labels
\[
 (-3,0,1)
 \qquad\text{and}\qquad
 (-3,1,1),
\]
and original masses $2/7$ and $5/7$,
respectively.

For the first class, the values of
$m_C$ at $0,1/2,1$ are
$0,1/2,1/2$.
For the second they are $0,1/2,1$.
Their model masses at depth two are
$1/2$ and $1/4$, whereas the actual
normalized masses of their occupied
depth-two squares are $1/2$ and $1/5$.

Writing the two class probabilities
as $\rho_{C_1},\rho_{C_2}$, we have
\[
 \mu
 =
 \frac27\rho_{C_1}
 +
 \frac57\rho_{C_2}.
\]
An equally weighted sum would not
recover $\mu$.

The construction also applies to
atomic measures and to measures charging
dyadic boundaries.
The half-open convention assigns each
point to a unique square, and the
logarithmic intervals are closed on
the left and open on the right.

\subsection{Application to the two measures}
\label{dyad:sec:application}

We apply Theorem~\ref{dyad:thm:main}
to the probabilities $\mu_1,\mu_2$
from Theorem~\ref{prep:thm:preparation}.

\begin{corollary}\label{dyad:cor:prepared}
Fix $0<\eta<1$, and let
$\mu_1,\mu_2$ be as in
Theorem~\ref{prep:thm:preparation}.
For $0<\sigma<1$, set
\begin{equation}\label{dyad:eq:prepared-threshold}
 N_{\rm pair}(\eta,\sigma)
 =
 \max\left\{
 N_{\rm mass}(\sigma),
 N_{\rm lower}(\sigma,C_F(\eta))
 \right\}.
\end{equation}
For every integer $N\ge N_{\rm pair}$,
apply Theorem~\ref{dyad:thm:main}
separately to $\mu_1$ and $\mu_2$.

The resulting families $\C_{j,N}$
have original weights and normalized
class probabilities
\begin{equation}\label{dyad:eq:prepared-classes}
 a_{j,C}=\mu_j(C),\qquad
 \rho_{j,C}
 =\frac{\mu_j|_C}{a_{j,C}},
 \qquad j=1,2.
\end{equation}
Each family contains at most $Q_N$
classes, with
\[
 a_{j,C}\ge R^{-\sigma},
\]
and its discarded mass is at most
$2R^{-\sigma/2}$.

For every occupied depth-$n$ square $P$,
\begin{equation}\label{dyad:eq:prepared-comparison}
 R^{-10\sigma}M_{j,C}(n)
 \le\rho_{j,C}(P)
 \le R^{10\sigma}M_{j,C}(n),
 \qquad
 M_{j,C}(n)=R^{-m_{j,C}(n/N)}.
\end{equation}
For
$\chi=\min\{\eta/8,1/16\}$,
\begin{equation}\label{dyad:eq:prepared-lower}
 A_{j,C,n}
 :=m_{j,C}(n/N)-n/N
 \ge\eta n/N-12\sigma
 \ge\chi n/N-12\sigma.
\end{equation}

Each $\rho_{j,C}$ is carried by
$K_j\cap C$, and
$\operatorname{supp}\rho_{j,C}\subset K_j$.
Consequently, the separation and
angular bounds of
Theorem~\ref{prep:thm:preparation}
hold on the supports of every ordered
pair of class probabilities.

The identities
\eqref{dyad:eq:product-identity}
and~\eqref{dyad:eq:theta-weight-identity}
hold with the weights in
\eqref{dyad:eq:prepared-classes}.
The numerical bounds and the
starting depth $N_{\rm pair}$
depend only on $\eta,\sigma$
and the fixed conventions.
\end{corollary}

\begin{proof}
The probabilities $\mu_j$ are carried by
$K_j\subset[3/8,5/8]^2\subset Q_0$
and satisfy~\eqref{prep:eq:output-ball}
with coefficient $C_F(\eta)\ge1$.
Theorem~\ref{dyad:thm:main}
therefore applies to both measures
with the same thresholds.

The bound~\eqref{dyad:eq:prepared-lower}
follows from~\eqref{dyad:eq:lower-main}
and $\chi\le\eta$.
Since each class probability is a
normalized restriction of $\mu_j$,
its support is contained in the
compact set $K_j$.
The geometric bounds and the
original-weight identities follow.
\end{proof}

The radial $L^p$ bounds of the original
probabilities need not hold with the
same constants after restriction and
normalization by $a_{j,C}$.
Accordingly, the later estimate for
discarded mass uses $\mu_1,\mu_2$,
whereas the directional and Fourier
estimates use the class probabilities
and their associated functions $m_{j,C}$.

Section~\ref{main:profiles} uses these
functions to choose intervals of depths
and the corresponding directional
conditions.
The original-weight identity
\eqref{dyad:eq:theta-weight-identity}
then allows the estimates for individual
class pairs to be recombined.

\clearpage

\section{Dyadic mass estimates and nested index sets}\label{main:profiles}
The estimates below depend on the change in model mass between two depths.
The lower bound for the function $m$ in Section~\ref{main:dyadic} does not
give a corresponding lower bound for its increments on arbitrary intervals.
We first derive estimates for class probabilities from their dyadic mass
comparisons. We then associate a finite set of indices with each interval
of depths, taking the union over its subintervals to obtain nested sets.
These indices will specify the geometric conditions in
Section~\ref{main:directions}.

\subsection{Mass estimates for a single class}\label{profile:sec:data}
Fix $0<\sigma<1$ and an integer $N$ for which
Corollary~\ref{dyad:cor:prepared} applies. Let $\rho$ be one of its class
probabilities, and let $m:[0,1]\to[0,2]$ be the associated nondecreasing,
$2$-Lipschitz function, with $m(0)=0$. Suppressing the class indices, set
\begin{equation}\label{profile:eq:data}
 R=2^N,\qquad M(n)=R^{-m(n/N)},\qquad A_n=m(n/N)-n/N.
\end{equation}
As in Section~\ref{main:dyadic}, let $Q_0=[0,1)^2$ and let $\D_n$ be its
partition into half-open dyadic squares of depth $n$. Write $c_P$ for the
center of a square $P$. Every occupied $P\in\D_n$ satisfies
\begin{equation}\label{profile:eq:mass}
 M(n)R^{-C_{\rm prof}\sigma}\le\rho(P)
 \le M(n)R^{C_{\rm prof}\sigma},\qquad C_{\rm prof}=10.
\end{equation}
For an occupied square $Q$, the conditional probability is
\begin{equation}\label{profile:eq:conditional}
 \rho_Q=\frac{\rho|_Q}{\rho(Q)},
\end{equation}
as in~\eqref{dyad:eq:conditional}. If $Q$ has depth $a$, the measure
$\rho|_Q/M(a)$ need not be a probability: its total mass is between
$R^{-C_{\rm prof}\sigma}$ and $R^{C_{\rm prof}\sigma}$.

For integers $0\le a\le b\le N$, define
\begin{equation}\label{profile:eq:I}
 I(a,b)=m(b/N)-m(a/N).
\end{equation}
Then $M(b)/M(a)=R^{-I(a,b)}$.

\begin{lemma}\label{profile:lem:descendants}
Let $0\le a\le b\le N$. Every occupied $Q\in\D_a$ has at most
\begin{equation}\label{profile:eq:descendants}
 R^{I(a,b)+2C_{\rm prof}\sigma}
\end{equation}
occupied descendants in $\D_b$. For each such descendant $P\subseteq Q$,
\begin{equation}\label{profile:eq:conditional-comparison}
 R^{-2C_{\rm prof}\sigma}\frac{M(b)}{M(a)}
 \le\rho_Q(P)\le
 R^{2C_{\rm prof}\sigma}\frac{M(b)}{M(a)}.
\end{equation}
\end{lemma}
\begin{proof}
The occupied descendants partition $Q$ up to a set of zero $\rho$-measure.
Each has mass at least $M(b)R^{-C_{\rm prof}\sigma}$, while their total
mass is at most $M(a)R^{C_{\rm prof}\sigma}$. Division gives
\eqref{profile:eq:descendants}. Applying~\eqref{profile:eq:mass} to the
numerator and denominator of $\rho_Q(P)=\rho(P)/\rho(Q)$ gives
\eqref{profile:eq:conditional-comparison}.
\end{proof}

By definition,
\begin{equation}\label{profile:eq:dimension-one}
 M(n)=2^{-n}R^{-A_n}.
\end{equation}
Since $0\le m(u)-m(v)\le2(u-v)$ for $0\le v\le u\le1$, the function
$u\mapsto m(u)-u$ is $1$-Lipschitz. Hence
\begin{equation}\label{profile:eq:Alip}
 A_0=0,\qquad |A_n-A_{n'}|\le |n-n'|/N,
 \qquad |A_n|\le1.
\end{equation}
Corollary~\ref{dyad:cor:prepared} also gives
\begin{equation}\label{profile:eq:lower}
 A_n\ge\eta n/N-12\sigma\ge\chi n/N-12\sigma,
 \qquad \chi=\min\{\eta/8,1/16\}.
\end{equation}
This lower bound does not imply that $A_n$ is nondecreasing.

For integers $0\le a\le b\le N$, put
\begin{equation}\label{profile:eq:differences}
 k(a,b)=A_b-\min_{a\le n\le b}A_n,
 \qquad
 \delta_A(a,b)=A_a-\min_{a\le n\le b}A_n.
\end{equation}
Both quantities are nonnegative, and
\begin{equation}\label{profile:eq:difference-identity}
 k(a,b)-\delta_A(a,b)=A_b-A_a.
\end{equation}
All minima over intervals of depths are taken over integers.

\begin{lemma}\label{profile:lem:normalization}
For $0\le a\le b\le N$,
\begin{equation}\label{profile:eq:normalization}
 2^{2(b-a)}\left(\frac{M(b)}{M(a)}\right)^2
 =R^{-2(A_b-A_a)}.
\end{equation}
\end{lemma}
\begin{proof}
Since $I(a,b)=(b-a)/N+A_b-A_a$,
\[
 2^{b-a}\frac{M(b)}{M(a)}=R^{-(A_b-A_a)}.
\]
Squaring proves the result.
\end{proof}

\subsection{Counts of occupied squares and conditional ball estimates}
\label{profile:sec:geometry}
\begin{lemma}\label{profile:lem:geometry}
For $0\le g\le p\le N$, $B_0\ge0$, and $z\in\R^2$,
\begin{equation}\label{profile:eq:center-count}
 \#\{Q\in\D_p:\rho(Q)>0,\ c_Q\in\overline B(z,B_0 2^{-g})\}
 \le (2B_0+4)^2R^{I(g,p)+2C_{\rm prof}\sigma}.
\end{equation}
For an occupied $P\in\D_p$, set
\[
 T_P(z)=2^p(z-c_P),\qquad \lambda_P=(T_P)_\#\rho_P.
\]
Then $\lambda_P$ is a probability carried by $[-1/2,1/2)^2$.
If $p\le t\le N$ and $H=2^{t-p}$, then
\begin{equation}\label{profile:eq:ball-bound}
 \lambda_P(\overline B(z,r))
 \le72R^{\delta_A(p,t)+2C_{\rm prof}\sigma}r,
 \qquad z\in\R^2,\quad H^{-1}\le r\le1.
\end{equation}
\end{lemma}
\begin{proof}
A closed ball of radius $B_0 2^{-g}$ meets at most $(2B_0+4)^2$ squares
of depth $g$: its projection onto each coordinate axis meets at most
$2B_0+4$ dyadic intervals of length $2^{-g}$, including endpoint
intersections. Every center counted in~\eqref{profile:eq:center-count}
belongs to one of these squares, and its depth-$p$ square is a descendant
of that square. Their total mass is therefore at most
\[
 (2B_0+4)^2M(g)R^{C_{\rm prof}\sigma}.
\]
Each counted square has mass at least $M(p)R^{-C_{\rm prof}\sigma}$.
Dividing gives~\eqref{profile:eq:center-count}.

For~\eqref{profile:eq:ball-bound}, take
$k=\lfloor\log_2(1/r)\rfloor$ and $n=p+k$. Then
\[
 p\le n\le t,\qquad r\le2^{-k}<2r.
\]
The inverse image under $T_P$ of the ball has radius
$2^{-p}r\le2^{-n}$ and meets at most $36$ depth-$n$ squares. Thus
\begin{align*}
 \lambda_P(\overline B(z,r))
 &\le36\frac{M(n)R^{C_{\rm prof}\sigma}}
                 {M(p)R^{-C_{\rm prof}\sigma}}\\
 &=36\,2^{-(n-p)}R^{A_p-A_n+2C_{\rm prof}\sigma}\\
 &\le72rR^{\delta_A(p,t)+2C_{\rm prof}\sigma}.
\end{align*}
The last inequality uses $p\le n\le t$ and
\eqref{profile:eq:differences}.
\end{proof}
The ball estimate is asserted only for $2^{-(t-p)}\le r\le1$;
no conclusion at smaller radii is needed here.

\subsection{Two examples}\label{profile:sec:examples}
The following examples concern functions $m$, without asserting their
realization by class probabilities.

If $m(u)=5u/4$, then
\begin{equation}\label{profile:eq:linear-example}
 M(n)=2^{-5n/4},\qquad A_n=\frac{n}{4N},\qquad
 k(a,b)=\frac{b-a}{4N},\qquad \delta_A(a,b)=0.
\end{equation}
In this case $A_b-A_a\ge0$ for every $a\le b$.

In contrast, consider
\begin{equation}\label{profile:eq:uneven}
 m(u)=\begin{cases}
 2u,&0\le u\le\tfrac14,\\[2pt]
 \tfrac12,&\tfrac14\le u\le\tfrac25,\\[2pt]
 \tfrac54u,&\tfrac25\le u\le1.
 \end{cases}
\end{equation}
This function is nondecreasing and $2$-Lipschitz, and satisfies
$m(u)\ge5u/4$. Nevertheless,
\[
 A(u):=m(u)-u=\begin{cases}
 u,&0\le u\le\tfrac14,\\[2pt]
 \tfrac12-u,&\tfrac14\le u\le\tfrac25,\\[2pt]
 \tfrac14u,&\tfrac25\le u\le1
 \end{cases}
\]
decreases on $[1/4,2/5]$.
Suppose that $20$ divides $N$, and set $a=N/4$, $b=3N/5$.
The minimum of $A_n$ on $[a,b]$ is $1/10$, attained at $2N/5$.
Hence
\begin{equation}\label{profile:eq:uneven-differences}
 k(a,b)=\frac1{20},\qquad
 \delta_A(a,b)=\frac3{20},\qquad
 A_b-A_a=-\frac1{10}.
\end{equation}
The factor involving $\delta_A$ in~\eqref{profile:eq:ball-bound} is
therefore $R^{3/20}$, whereas $k(a,b)$ would give $R^{1/20}$. Moreover,
\[
 I(a,b)=\frac14,\qquad \frac{b-a}{N}=\frac7{20},\qquad
 2^{2(b-a)}\left(\frac{M(b)}{M(a)}\right)^2=R^{1/5}.
\]
The factor in~\eqref{profile:eq:normalization} can therefore exceed one
despite the global lower bound for $m$.

\subsection{Selection of depths and indices}\label{lists:sec:definitions}
Fix positive integers $N,q_0$ and a real array $A_0,\ldots,A_N$.
All intervals in this construction have integer endpoints in $[0,N]$;
$[a,t]$ denotes $\{a,a+1,\ldots,t\}$ and has width $t-a$.
For a class probability, we will take
\begin{equation}\label{lists:eq:profile-application}
 q_0=\lceil\kappa N\rceil,\qquad A_n=m(n/N)-n/N,
\end{equation}
where $0<\kappa<1$. The construction itself requires no regularity or
sign condition on the array. The integer $q_0$ is distinct from the
spacing $q_{\rm d}=\lfloor\sigma N\rfloor$ used in
Section~\ref{main:dyadic}.

Partition $[0,N]$ into the blocks
\begin{equation}\label{lists:eq:blocks}
 K_{\rm blk}=\left\lceil\frac{N+1}{q_0}\right\rceil
 =\left\lfloor\frac{N}{q_0}\right\rfloor+1,
 \qquad
 B_j=[jq_0,\min\{(j+1)q_0-1,N\}],\quad 0\le j<K_{\rm blk},
\end{equation}
and choose the smallest minimizing depth in each block:
\begin{equation}\label{lists:eq:block-rep}
 r_j=\min\{n\in B_j:A_n=\min_{u\in B_j}A_u\}.
\end{equation}
The blocks and their representatives remain fixed throughout the
construction.

Call $J=[a,t]$ \emph{terminal} if
\begin{equation}\label{lists:eq:terminal}
 t-a\le100q_0.
\end{equation}
For a nonterminal interval, consider the blocks wholly contained in
\begin{equation}\label{lists:eq:inner-window}
 [a+q_0,t-q_0].
\end{equation}
Among their representatives, choose one minimizing $A_{r_j}$,
resolving ties by taking the smallest depth. Denote this choice by
$p(J)$ and put
\begin{equation}\label{lists:eq:ell}
 \ell(J)=\min\{p(J)-a,t-p(J)\}.
\end{equation}
We refer to $2p(J)\le a+t$ as the \emph{prefix case} and to
$2p(J)>a+t$ as the \emph{suffix case}. The eligible collection is
nonempty, as shown below. The choice $p(J)$ depends only on $J$ and
the given array, not on any larger interval containing $J$.

We use two types of indices, with specified depth $p(\mathfrak t)$
and length $\ell(\mathfrak t)$ given by
\[
 \begin{aligned}
 \mathfrak t=\Par{g}{p},\quad 0\le g<p\le N:
 &\qquad p(\mathfrak t)=p,\quad \ell(\mathfrak t)=p-g,\\
 \mathfrak t=\Proj{p}{t},\quad 0\le p<t\le N:
 &\qquad p(\mathfrak t)=p,\quad \ell(\mathfrak t)=t-p.
 \end{aligned}
\]
The types are distinct, even when the displayed pairs of depths agree.
The corresponding geometric conditions are defined in
Section~\ref{main:directions}.

\subsection{Nested sets of indices}\label{lists:sec:main}
For a nonterminal interval $J=[a,t]$, with $p=p(J)$ and $\ell=\ell(J)$,
define
\begin{equation}\label{lists:eq:generated}
 \G(J)=\{\Par{g}{p}:p-\ell\le g<p\}
 \ \cup\
 \begin{cases}
  \varnothing,&2p\le a+t,\\
  \{\Proj{p}{t}\},&2p>a+t.
 \end{cases}
\end{equation}
Set $\G(J)=\varnothing$ when $J$ is terminal. For each base interval
$\mathcal B=[a,t]$, put
\begin{equation}\label{lists:eq:full-list}
 \M(\mathcal B)=\M(a,t)=\bigcup_{J\subseteq\mathcal B}\G(J).
\end{equation}
The union is over all integer subintervals, including $\mathcal B$ itself.
Thus $\M(\mathcal B)$ is a set of distinct indices: an index generated
by several intervals occurs only once. Repeated factors in amplitudes
will be treated separately in Section~\ref{main:amplitudes}.

We say that $J$ \emph{supplies $\mathfrak t$ in $\mathcal B$} if
$J\subseteq\mathcal B$ and $\mathfrak t\in\G(J)$. This is simply a
subinterval generating the index; it is not part of the index itself.

\begin{theorem}\label{lists:thm:main}
For every real array $A_0,\ldots,A_N$ and positive integers $N,q_0$,
the construction above is well-defined and has the following properties.
\begin{enumerate}[label=(\roman*)]
\item For every nonterminal interval $[a,t]$, the eligible blocks have
nonempty union $[a_*,t_*]$, where
\begin{equation}\label{lists:eq:exact-inner}
 a_*=q_0\left\lceil\frac{a+q_0}{q_0}\right\rceil,
 \qquad
 t_*=q_0\left\lfloor\frac{t-q_0+1}{q_0}\right\rfloor-1.
\end{equation}
In particular,
\begin{equation}\label{lists:eq:sharp-endpoints}
 a+q_0\le a_*\le a+2q_0-1,
 \qquad t-2q_0+1\le t_*\le t-q_0,
\end{equation}
and
\begin{equation}\label{lists:eq:minimum}
 A_n\ge A_{p([a,t])}\qquad(a_*\le n\le t_*).
\end{equation}
\item If $\mathcal B'\subseteq\mathcal B$, then
$\M(\mathcal B')\subseteq\M(\mathcal B)$. Moreover,
$\M(\mathcal B)$ is empty if and only if $\mathcal B$ is terminal.
\item For every $\mathfrak t\in\M(\mathcal B)$,
$p(\mathfrak t)\in\{r_0,\ldots,r_{K_{\rm blk}-1}\}$, and
\begin{equation}\label{lists:eq:sharp-count}
 \#\M(\mathcal B)\le N K_{\rm blk}
 \le2(N+1)\left(1+\frac{N+1}{q_0}\right).
\end{equation}
If $q_0=\lceil\kappa N\rceil$, then $K_{\rm blk}\le\kappa^{-1}+1$.
\item For every $\mathfrak t\in\M(\mathcal B)$, there is a nonterminal
$J=[b_0,u_0]\subseteq\mathcal B$ supplying it. With
$p=p(\mathfrak t)$ and $\ell'=\ell(\mathfrak t)$,
\begin{equation}\label{lists:eq:supplying-interval}
 [p-\ell',p+\ell']\subseteq[b_0,u_0].
\end{equation}
The same $J$ generates every $\Par{g}{p}$ with $p-\ell'\le g<p$.
\item If $\mathcal B=[a,t]$ and $\mathfrak t\in\M(\mathcal B)$, then
\begin{equation}\label{lists:eq:base-locations}
 a+q_0\le p(\mathfrak t)\le t-q_0,
 \qquad 1\le\ell(\mathfrak t)\le\frac{t-a}{2}.
\end{equation}
\end{enumerate}
\end{theorem}
The weaker endpoint bounds $a_*\le a+2q_0$ and $t_*\ge t-2q_0$
will sometimes suffice below.

\begin{corollary}\label{lists:cor:Ncube}
For $N\ge2$, $\#\M(\mathcal B)+1\le N^3$. Thus the bound $N^3$
still holds after one additional indexed condition is included, even
if it is counted as a separate occurrence. For $N=1$, every set
$\M(\mathcal B)$ is empty, so the same conclusion holds.
\end{corollary}
\begin{proof}
Since $K_{\rm blk}\le N+1$,
\[
 \#\M(\mathcal B)+1\le N(N+1)+1\le N^3\qquad(N\ge2).
\]
For $N=1$, all intervals are terminal.
\end{proof}

\subsection{Eligible blocks and minimizing depths}\label{lists:sec:block-proof}
We first prove part~(i) of Theorem~\ref{lists:thm:main}.
Let $[a,t]$ be nonterminal and write $L=a+q_0$, $U=t-q_0$.
The last block of the global partition ends at $N>U$, so it is not
eligible. Every eligible block therefore has exactly $q_0$ elements.
The first candidate begins at $q_0\lceil L/q_0\rceil$ and ends no
later than $L+2q_0-2$. Since
\[
 U-L=t-a-2q_0>98q_0>2q_0-2,
\]
this block lies in $[L,U]$. The eligible blocks are consecutive, and
the last one ends at $q_0\lfloor(U+1)/q_0\rfloor-1$. This proves
\eqref{lists:eq:exact-inner}. The bounds
\[
 0\le q_0\lceil L/q_0\rceil-L\le q_0-1,
 \qquad
 0\le U-\bigl(q_0\lfloor(U+1)/q_0\rfloor-1\bigr)\le q_0-1
\]
give~\eqref{lists:eq:sharp-endpoints}.

If $n\in[a_*,t_*]$ lies in $B_j$, then
$A_n\ge A_{r_j}\ge A_{p([a,t])}$. This proves
\eqref{lists:eq:minimum}. The tie-breaking rules make $p([a,t])$
the smallest minimizer on the eligible union. In particular,
\begin{equation}\label{lists:eq:p-margins}
 a+q_0\le p([a,t])\le t-q_0,
 \qquad \ell([a,t])\ge q_0\ge1.
\end{equation}
Every nonterminal interval therefore generates a parallel index.
Every subinterval of a terminal interval is terminal, so its associated
set $\M(a,t)$ is empty. This proves the nonemptiness assertions in
part~(ii).

If $N\le100q_0$, all intervals are terminal. When
$q_0=\lceil\kappa N\rceil$ with $0<\kappa<1/100$, nonterminal
intervals exist for sufficiently large $N$; for example,
\[
 N>\frac{100}{1-100\kappa}
 \quad\Longrightarrow\quad N>100\lceil\kappa N\rceil.
\]
The theorem includes both cases and requires no lower bound on $N$.

\subsection{Nesting and cardinality}\label{lists:sec:nesting}
Suppose $\mathcal B'\subseteq\mathcal B$ and
$\mathfrak t\in\M(\mathcal B')$. Then
$\mathfrak t\in\G(J)$ for some $J\subseteq\mathcal B'$.
Because the choice of $p(J)$ does not depend on the containing interval,
the same set $\G(J)$ occurs in the union for $\mathcal B$.
Thus $\mathfrak t\in\M(\mathcal B)$, proving nesting.
The individual sets $\G(\mathcal B')$ and $\G(\mathcal B)$ need not
be nested; see Subsection~\ref{lists:subsec:own-only}.

Equivalently, for a nonterminal interval,
\begin{equation}\label{lists:eq:recurrence}
 \M(a,t)=\M(a+1,t)\cup\M(a,t-1)\cup\G([a,t]),
\end{equation}
with empty sets on terminal intervals. Every proper subinterval of
$[a,t]$ lies in $[a+1,t]$ or $[a,t-1]$, so this recurrence agrees
with~\eqref{lists:eq:full-list}.

Fix a representative depth $p$. There are at most $p$ possible indices
$\Par{g}{p}$ and $N-p$ possible indices $\Proj{p}{t}$. Thus at most
$N$ indices have specified depth $p$, and
$\#\M(\mathcal B)\le N K_{\rm blk}$. The remaining bound in
\eqref{lists:eq:sharp-count} follows from
$K_{\rm blk}\le1+(N+1)/q_0$. If $q_0=\lceil\kappa N\rceil$, then
\[
 K_{\rm blk}=\lfloor N/q_0\rfloor+1\le N/q_0+1\le\kappa^{-1}+1.
\]
This proves part~(iii). In particular, the number of possible specified
depths is bounded independently of $N$ when $\kappa$ is fixed.

The initial Fourier estimate also uses $\Par{0}{c}$ for $1\le c<N$.
Its specified depth is $c$, whereas every index in $\M(c,N)$ has
specified depth greater than $c$. It is therefore an additional index,
covered by Corollary~\ref{lists:cor:Ncube}; no generating subinterval
inside $[c,N]$ is needed for it.

\subsection{Subintervals generating an index}\label{lists:sec:providers}
We prove parts~(iv) and~(v) of Theorem~\ref{lists:thm:main}.
Let $\mathfrak t\in\M(\mathcal B)$ and choose a nonterminal
$J=[b_0,u_0]\subseteq\mathcal B$ with $\mathfrak t\in\G(J)$.
Put $p=p(J)$ and $\ell=\min\{p-b_0,u_0-p\}$.

If $\mathfrak t=\Par{g}{p}$, its length
$\ell'=p-g$ satisfies $\ell'\le\ell$. Hence
\[
 b_0\le p-\ell\le p-\ell',\qquad
 p+\ell'\le p+\ell\le u_0.
\]
Every $\Par{g'}{p}$ with $p-\ell'\le g'<p$ also belongs to $\G(J)$.
If $\mathfrak t=\Proj{p}{u_0}$, then $J$ is in the suffix case and
\[
 \ell'=u_0-p=\ell<p-b_0.
\]
Again, $[p-\ell',p+\ell']\subseteq J$, and all the stated parallel
indices belong to $\G(J)$. This proves part~(iv).

Since $[p-\ell',p+\ell']\subseteq J\subseteq\mathcal B=[a,t]$,
we have $2\ell'\le t-a$. Applying~\eqref{lists:eq:p-margins} to $J$
gives $a+q_0\le p\le t-q_0$. Both types of indices have positive
length, proving part~(v).

A generating subinterval need not be unique. If a deterministic choice
is required, take the lexicographically first pair $(b_0,u_0)$ in the
given base. One such interval need not work for every base containing
the same index; see Subsection~\ref{lists:subsec:no-global}.

\subsection{Indices retained after increasing the initial depth}
\label{lists:sec:active}
For a base $\mathcal B=[a_0,t_0]$ and $a_0\le s\le t_0$, define
\begin{equation}\label{lists:eq:active-list}
 \M_{>s}(\mathcal B)=\{\mathfrak t\in\M(\mathcal B):p(\mathfrak t)>s\}.
\end{equation}
The base remains fixed in this definition. For the nonnegative
expressions with independent angular variables introduced later,
Lemma~\ref{amplitude:lem:deletion} controls the effect of removing
factors whose specified depths are at most $s$.

\begin{proposition}\label{lists:prop:active}
If $[s,v]\subseteq\mathcal B$, then
\begin{equation}\label{lists:eq:current-in-active}
 \M(s,v)\subseteq\M_{>s}(\mathcal B).
\end{equation}
In particular, the right-hand side is nonempty if $v-s>100q_0$.
For $s\le s'\le t_0$,
\begin{equation}\label{lists:eq:successive-active}
 \{\mathfrak t\in\M_{>s}(\mathcal B):p(\mathfrak t)>s'\}
 =\M_{>s'}(\mathcal B).
\end{equation}
Every remaining index has a supplying interval inside $\mathcal B$.
\end{proposition}
\begin{proof}
By nesting, $\M(s,v)\subseteq\M(\mathcal B)$, and
\eqref{lists:eq:base-locations} gives $p(\mathfrak t)\ge s+q_0>s$
for each $\mathfrak t\in\M(s,v)$. This proves
\eqref{lists:eq:current-in-active}; if $[s,v]$ is nonterminal,
$\M(s,v)$ is nonempty. Equation~\eqref{lists:eq:successive-active}
follows from $s'\ge s$, and the final assertion follows from
Theorem~\ref{lists:thm:main}(iv).
\end{proof}

If $\M_{>s}(\mathcal B)$ is nonempty and $L$ is its maximum length, then
\begin{equation}\label{lists:eq:active-length}
 L\le(t_0-a_0)/2.
\end{equation}
The bound $L\le(v-s)/2$ need not hold for a smaller interval $[s,v]$;
an example is given in Subsection~\ref{lists:subsec:active-example}.

\begin{corollary}\label{lists:cor:far-list}
If $[b,v]\subseteq\mathcal B$ and $b>s$, then
$\M(b,v)\subseteq\M_{>s}(\mathcal B)$, and every index in $\M(b,v)$
has specified depth greater than $b$.
\end{corollary}
\begin{proof}
Apply nesting and~\eqref{lists:eq:base-locations} to $[b,v]$.
\end{proof}

\subsection{Depth inequalities for the analytic estimates}
\label{lists:sec:analytic-depths}
Throughout this subsection, let $0<\epsilon\le q_0$.
In the application, $\epsilon=\sigma N$ and $q_0=\lceil\kappa N\rceil$
with $0<\sigma\le\kappa$.

\begin{proposition}\label{lists:prop:high-depths}
Let $[a,t]$ be nonterminal, with selected depth $p$ in the prefix case.
Let $v$ be an integer satisfying $0\le t-v<\epsilon$, and let $L$
be the maximum length in $\M(a,t)$. Then $\Par{a}{p}\in\M(a,t)$ and
\begin{gather}
 a<p<v,\qquad v-a>20\epsilon,
 \qquad L\le(v-a)/2+\epsilon,\label{lists:eq:high-inverse}\\
 p-a\le v-p+\epsilon,\qquad L\le v-p+\epsilon,
 \qquad v-p>10\epsilon.\label{lists:eq:high-graph}
\end{gather}
Removing indices with specified depth at most $p$ leaves exactly
$\M_{>p}([a,t])$.
\end{proposition}
\begin{proof}
The prefix inequality gives $\ell([a,t])=p-a$, so
$\Par{a}{p}\in\G([a,t])$. Put $d=t-a>100q_0$. Since
$t-v<\epsilon\le q_0$,
\[
 v-a=d-(t-v)>d-\epsilon>99q_0\ge99\epsilon.
\]
By Theorem~\ref{lists:thm:main},
\[
 L\le d/2=(v-a)/2+(t-v)/2<(v-a)/2+\epsilon.
\]
Also, $p\le(a+t)/2$ implies
\[
 v-p\ge d/2-(t-v)>d/2-\epsilon>49q_0\ge49\epsilon.
\]
Both $p-a$ and $L$ are at most $t-p<v-p+\epsilon$.
These inequalities give the conclusions; the final assertion is
\eqref{lists:eq:active-list}.
\end{proof}

\begin{proposition}\label{lists:prop:direct-depths}
Let $\mathcal B=[a_0,t_0]$ and $[a,v]\subseteq\mathcal B$ satisfy
\[
 v-a>100q_0,\qquad t_0\le v+\epsilon.
\]
Let $L$ be the maximum length in $\M_{>a}(\mathcal B)$, and set
\begin{equation}\label{lists:eq:direct-h}
 L_0=\max\{L,\lceil10\epsilon\rceil\},\qquad h=v-2L_0.
\end{equation}
Assume $h\le a$. Then $L=L_0$. Choose any index of length $L$ in
$\M_{>a}(\mathcal B)$, let $p$ be its specified depth, and choose an
interval $[b_0,u_0]\subseteq\mathcal B$ supplying it. For
$g=\max\{a,p-L\}$, we have
\begin{gather}
 b_0\le p-L,\qquad p+L\le u_0\le v+\epsilon,
 \label{lists:eq:direct-origin}\\
 a<p<v,\qquad a\le g<p,\qquad g\le a+\epsilon,
 \label{lists:eq:direct-g}\\
 p-a\le v-p+2\epsilon,\qquad
 L\le v-p+\epsilon,\qquad v-p>10\epsilon.
 \label{lists:eq:direct-scale}
\end{gather}
The same interval generates $\Par{g}{p}$, which belongs to
$\M_{>a}(\mathcal B)$.
\end{proposition}
\begin{proof}
Proposition~\ref{lists:prop:active} makes $\M_{>a}(\mathcal B)$
nonempty. Since $\epsilon\le q_0$ and $q_0$ is an integer,
\[
 2\lceil10\epsilon\rceil\le20q_0<100q_0<v-a.
\]
If $L_0>L$, then $h\le a$ would imply
$v-a\le2\lceil10\epsilon\rceil$, a contradiction. Thus
$L=L_0$ and $v-a\le2L$.

Theorem~\ref{lists:thm:main}(iv) gives
\eqref{lists:eq:direct-origin}. It follows that
\[
 p-a\le v-a+\epsilon-L\le(v-a)/2+\epsilon.
\]
Hence $p-a\le v-p+2\epsilon$ and
\[
 v-p\ge(v-a)/2-\epsilon>49q_0\ge49\epsilon.
\]
The inequality $L\le v-p+\epsilon$ follows from $p+L\le v+\epsilon$,
and $p>a$ follows from membership in $\M_{>a}(\mathcal B)$.
Finally,
\[
 p-L\le v+\epsilon-2L\le a+\epsilon.
\]
Thus $a\le g\le a+\epsilon$, while $g<p$ because $a<p$ and $L\ge1$.
Since $p-L\le g<p$, Theorem~\ref{lists:thm:main}(iv) gives
$\Par{g}{p}\in\G([b_0,u_0])$. Its specified depth is $p>a$,
so it remains in $\M_{>a}(\mathcal B)$.
\end{proof}
With $\epsilon=\sigma N$, these bounds imply the weaker allowance
$g\le a+2\sigma N$ used later. The generating interval lies in the
original base, not necessarily in the smaller interval $[a,v]$.

For later use, nesting gives $\M(a,v)\subseteq\M(a,t)$ whenever
$a\le v\le t$. Every index in $\M(a,p)$ has specified depth less
than $p$ and length at most $(p-a)/2$.

\subsection{Examples}\label{lists:sec:examples}
These examples concern the finite construction; realization by measures
is not asserted. The functions $m$ used below are nondecreasing and
$2$-Lipschitz.

\subsubsection{Integer endpoints and ties}
Let $q_0=3$, $N=304$, and $A_n=0$ for all $n$. The last block is
$[303,304]$. For $J=[0,301]$, the eligible blocks run from $[3,5]$
through $[294,296]$, so
\[
 a_*=3,\qquad t_*=296.
\]
The tie rule gives $p=3$ and $\ell=3$. Thus
\[
 \G(J)=\{\Par{0}{3},\Par{1}{3},\Par{2}{3}\}.
\]
For $[1,302]$, the eligible union is $[6,299]$ and the choice is $p=6$.
Different intervals can therefore select different depths. The interval
$[0,300]$ is terminal and generates no indices.

\subsubsection{Why all subintervals are included}\label{lists:subsec:own-only}
Take $q_0=1$, $N=103$, and $A_n=-n/103$, corresponding to $m(u)=0$.
For $\mathcal B=[0,103]$, the choice is $p=102$ and
\[
 \G(\mathcal B)=\{\Par{101}{102},\Proj{102}{103}\}.
\]
For $\mathcal B'=[0,102]$, the choice is $p=101$ and
\[
 \G(\mathcal B')=\{\Par{100}{101},\Proj{101}{102}\}.
\]
Thus $\G(\mathcal B')\nsubseteq\G(\mathcal B)$, although
$\mathcal B'\subseteq\mathcal B$. In contrast, $\M(\mathcal B)$
contains $\G(\mathcal B')$ by~\eqref{lists:eq:full-list}.
In particular, the indices $\Par{101}{102}$ and $\Proj{101}{102}$
must be distinguished.

\subsubsection{A generating subinterval may depend on the base}
\label{lists:subsec:no-global}
Let $q_0\ge1$, $N=151q_0$, and
\begin{equation}\label{lists:eq:origin-example}
 A_n=\frac{|n-75q_0|-75q_0}{N}.
\end{equation}
These values arise from $m(u)=\max\{0,2u-150q_0/N\}$. The unique
minimizer is $p=75q_0$, which is also its block representative.
The nonterminal intervals
\[
 B_1=[0,101q_0],\qquad B_2=[50q_0,151q_0]
\]
both select $p$. The first is in the suffix case with $\ell=26q_0$,
and the second is in the prefix case with $\ell=25q_0$.
Both generate $\mathfrak t=\Par{75q_0-1}{75q_0}$.
However, their intersection has width $51q_0$ and contains no
nonterminal interval. No single generating subinterval for this index
can therefore lie in both bases.

\subsubsection{A remaining length can exceed half the current width}
\label{lists:subsec:active-example}
Take $q_0=1$, $N=102$, and
\[
 A_n=\frac{|n-51|-51}{102},\qquad m(u)=\max\{0,2u-1\}.
\]
For $\mathcal B=[0,102]$, the chosen depth is $p=51$. This is the
prefix case, including equality, so the interval generates
$\Par{0}{51}$ of length $51$, but not $\Proj{51}{102}$.
For the smaller interval $[1,102]$, the index $\Par{0}{51}$ remains
in $\M_{>1}(\mathcal B)$, although $51>101/2$.
Thus half the width of the smaller interval need not bound the remaining
lengths. The construction in Proposition~\ref{lists:prop:direct-depths}
uses $g=1$; the index $\Par{1}{51}$ is already present.

\subsection{A common starting depth for two classes}\label{entry:sec:entry}
For a pair of class probabilities, we need a depth after which both
sequences $A_{j,n}$ are bounded below, while $k_j(0,c)$ is controlled
for at least one of the two classes. This depth need not be a block
representative.

Write $m_j,A_{j,n},k_j$ for the corresponding objects in
\eqref{profile:eq:data} and~\eqref{profile:eq:differences}, with $j=1,2$.
The argument applies to any two nondecreasing, $2$-Lipschitz functions
$m_j:[0,1]\to\R$ with $m_j(0)=0$ and
\begin{equation}\label{entry:eq:lower}
 A_{j,n}\ge\chi n/N-C_{\rm low}\sigma,
 \qquad 0\le n\le N,\quad j=1,2,\qquad 0<\chi\le1.
\end{equation}
For the class probabilities, Corollary~\ref{dyad:cor:prepared} gives
this bound with $C_{\rm low}=12$ and
$\chi=\min\{\eta/8,1/16\}$. The estimates~\eqref{profile:eq:Alip}
hold for each $A_{j,n}$.

\begin{lemma}\label{entry:lem:entry}
Assume~\eqref{entry:eq:lower}, and fix
\[
 0<\beta<\min\{\chi/10,1/10\},\qquad
 C_{\rm low}\sigma\le\chi/8,\qquad 0<\kappa<10^{-4}.
\]
Let $q_0=\lceil\kappa N\rceil$. For
\begin{equation}\label{entry:eq:Nentry}
 N\ge N_{\rm ent}:=\left\lceil\max\{200,8/\beta\}\right\rceil,
\end{equation}
define
\begin{equation}\label{entry:eq:entry}
 c=\max\{n\in\{0,\ldots,N\}:\min(A_{1,n},A_{2,n})\le\beta\}.
\end{equation}
Then
\begin{gather}
 \lfloor\beta N/4\rfloor\le c<N/4,\qquad c>0,\notag\\
 A_{j,n}\ge\beta-2/N\quad(j=1,2,\ c\le n\le N),\notag\\
 \min_{j=1,2}k_j(0,c)\le\beta+C_{\rm low}\sigma.
 \label{entry:eq:results}
\end{gather}
Moreover, $N-c>100q_0$. The threshold $N_{\rm ent}$ is independent
of the two functions $m_j$.
\end{lemma}
\begin{proof}
By~\eqref{profile:eq:Alip}, every integer $0\le n\le\beta N/4$
satisfies $A_{j,n}\le n/N\le\beta/4$. Thus the set in
\eqref{entry:eq:entry} is nonempty and
$c\ge\lfloor\beta N/4\rfloor\ge2$.
For $n\ge N/4$, the lower bound gives
\[
 A_{j,n}\ge\chi/4-C_{\rm low}\sigma\ge\chi/8>\beta.
\]
Hence $c<N/4$. By maximality, $A_{j,n}>\beta$ at every integer
depth $n>c$. The one-step bound in~\eqref{profile:eq:Alip} gives
$A_{j,c}>\beta-1/N$, proving the second assertion of
\eqref{entry:eq:results}.

For some $j_*$, $A_{j_*,c}\le\beta$. Since
\[
 \min_{0\le n\le c}A_{j_*,n}\ge-C_{\rm low}\sigma,
\]
we obtain $k_{j_*}(0,c)\le\beta+C_{\rm low}\sigma$.
Finally, $N-c>3N/4$, whereas
\[
 100q_0\le100\kappa N+100<N/100+100<3N/4
\]
for $N\ge200$. This proves $N-c>100q_0$.
\end{proof}
Lemma~\ref{frequency:lem:entry-minima} applies these bounds to the
minimum over two separated depths introduced in
Section~\ref{main:iteration}.

Perform the index construction separately for the two classes, with
the same $N,q_0$ and their respective representatives. Write the
resulting sets as $\M_j(a,t)$ and define
\begin{equation}\label{entry:eq:lists}
 \mathcal A_j=\{\Par{0}{c}\}\cup\M_j(c,N),\qquad j=1,2.
\end{equation}
The additional index is admissible because $0<c<N$. It is not repeated
in $\M_j(c,N)$, whose specified depths are all greater than $c$.
Theorem~\ref{lists:thm:main} and Corollary~\ref{lists:cor:Ncube} give
\begin{equation}\label{entry:eq:list-bounds}
 \#\mathcal A_j\le N(N+1)+1\le N^3,\qquad
 \max_{\mathfrak t\in\mathcal A_j}\ell(\mathfrak t)
 \le\max\{c,(N-c)/2\}\le N/2.
\end{equation}
The two index sets may differ. The index $\Par{0}{c}$ is defined
individually and needs no generating subinterval in $[c,N]$.

As an example for the abstract hypotheses of Lemma~\ref{entry:lem:entry},
take $\chi=1/2$, $\beta=1/40$, and
\[
 m_1(u)=\tfrac32u,\qquad m_2(u)=2u.
\]
Then $A_1(u)=u/2$, $A_2(u)=u$, and both are at least $\chi u$.
If $20$ divides $N$,
\begin{equation}\label{entry:eq:example}
 c=N/20,\qquad k_1(0,c)=1/40,\qquad k_2(0,c)=1/20.
\end{equation}
For sufficiently small $\sigma$, the estimate in~\eqref{entry:eq:results}
would therefore fail if the minimum were replaced by the maximum.
Replacing $A_1(u)=u/2$ by $A_1(u)=3u/4$, while keeping $A_2(u)=u$,
changes $c$ from $N/20$ to $N/30$ when $60$ divides $N$.
The common starting depth is therefore allowed to depend on the class pair.

\subsection{Consequences for the directional estimates}\label{profile:sec:summary}
Lemmas~\ref{profile:lem:descendants} and~\ref{profile:lem:geometry}
give the counts and conditional ball estimates used in
Section~\ref{main:directions}. Theorem~\ref{lists:thm:main} provides
nested index sets with controlled cardinality, while
Lemma~\ref{entry:lem:entry} gives a common starting depth for each
pair of classes. These constructions leave the class probabilities and
their original weights unchanged.

\clearpage

\section{Directional conditions and exceptional sets}\label{main:directions}
We now define the geometric conditions associated with the indices
$\Par{g}{p}$ and $\Proj{p}{t}$ of Section~\ref{main:profiles}.
The first bounds the number of occupied square centers near a line;
the second bounds the probability that two independent points in one
square have nearly equal projections. We estimate the measure of the
directions failing either condition. These estimates will be used in
Section~\ref{main:cutoffs-mass}.

\subsection{The class probability and the measure on directions}
\label{direction:sec:data}
Fix a class probability $\rho$ from Corollary~\ref{dyad:cor:prepared},
with $N$ large enough for that corollary to apply. We retain
$Q_0=[0,1)^2$, the associated function $m$, and the notation
$R=2^N$, $M(n)$, $A_n$, $I(g,p)$, $k(g,p)$, and $\delta_A(p,t)$
from Subsection~\ref{profile:sec:data}. The estimates use only the
comparison~\eqref{profile:eq:mass}, with $C_{\rm prof}=10$, and
apply to any probability satisfying it; the lower
bound~\eqref{profile:eq:lower} is not required. For every occupied
depth-$p$ square $P$, let
\[
 \rho_P=\frac{\rho|_P}{\rho(P)},
\]
as in~\eqref{profile:eq:conditional}. The class $\rho$,
its occupied squares, and these conditional
probabilities remain fixed throughout the construction.

Let $\xi$ be a finite nonnegative Borel measure on $S^1$ with
$\xi(S^1)\le1$. Assume that
\begin{equation}\label{direction:eq:xi}
 \xi(B_{\rm ch}(w,r))\le R^\sigma r
 \qquad(w\in S^1,\ R^{-1}\le r\le1),
\end{equation}
where $B_{\rm ch}(w,r)=\{v\in S^1:|v-w|<r\}$.
We also use angular distance, measured along the shorter arc in radians.
By Lemma~\ref{direction:lem:arcs}, all conclusions below hold with
the same constants if~\eqref{direction:eq:xi} is instead assumed
for open angular balls. No ball estimate is assumed below radius
$R^{-1}$, and $\xi$ need not be nonatomic or have a density.

\subsection{The two geometric conditions}\label{direction:sec:tests}
Fix an integer $J_0\ge1$ and a positive numerical constant
$C_{\rm test}$. Define
\begin{equation}\label{direction:eq:widths}
 W_i(n)=(2-i/J_0)2^{-n}R^{500\sigma},
 \qquad 0\le i\le J_0.
\end{equation}

For $\Par{g}{p}$, where $0\le g<p\le N$, and an occupied
$P\in\D_p$, put
\begin{equation}\label{direction:eq:parallel-stat}
 \mathcal N_{i,g,p,P}(w)
 =\#\left\{\begin{array}{l}
  Q\in\D_p:\ \rho(Q)>0,\quad c_Q\in B(c_P,30\,2^{-g}),\\
  \dist(c_Q,c_P+\R w)\le W_i(p)
 \end{array}\right\}.
\end{equation}
Thus the spatial ball is open, the strip is closed, and $P$
itself is included in the count. Define $G_{i,[g,p]}^\parallel(P)$
by the condition
\begin{equation}\label{direction:eq:parallel-threshold}
 \mathcal N_{i,g,p,P}(w)\le R^{k(g,p)+C_{\rm test}\sigma}.
\end{equation}

For $\Proj{p}{t}$, where $0\le p<t\le N$, and an occupied
$P\in\D_p$, put
\begin{equation}\label{direction:eq:projection-stat}
 \mathcal P_{i,p,t,P}(w)
 =(\rho_P\otimes\rho_P)
 \{(z,z'): |(z-z')\cdot w|\le W_i(t)\}.
\end{equation}
Define $G_{i,[p,t]}^{\rm proj}(P)$ by
\begin{equation}\label{direction:eq:projection-threshold}
 \mathcal P_{i,p,t,P}(w)
 \le 2^{-(t-p)}R^{\delta_A(p,t)+C_{\rm test}\sigma}.
\end{equation}
The event in~\eqref{direction:eq:projection-stat} includes
the diagonal and the boundary of the strip. If the right side
of~\eqref{direction:eq:projection-threshold} is at least one,
the condition holds for every direction.

For either type of index in Subsection~\ref{lists:sec:definitions},
write $G_{i,\mathfrak t}(P)$ for
the corresponding set, and set it equal to $\varnothing$
when $P$ is unoccupied. For $x\in Q_0$, define
$G_{i,\mathfrak t}(x)$ using the unique square containing $x$
at depth $p(\mathfrak t)$; set it equal to $\varnothing$
outside $Q_0$.

Both statistics are Borel functions of $w$. This follows from
the finite sum in~\eqref{direction:eq:parallel-stat} and from
the jointly Borel indicator in~\eqref{direction:eq:projection-stat}.
The finite partition into half-open squares then shows that
$\{(x,w):w\in G_{i,\mathfrak t}(x)\}$ is Borel. Moreover,
\begin{equation}\label{direction:eq:symmetry-levels}
 G_{i,\mathfrak t}(P)=-G_{i,\mathfrak t}(P),\qquad
 G_{i,\mathfrak t}(P)\subseteq G_{i+1,\mathfrak t}(P)
 \quad(0\le i<J_0).
\end{equation}
Indeed, both conditions are unchanged by $w\mapsto-w$,
and the widths decrease with $i$. The sets need not be closed;
see Subsection~\ref{direction:sec:exceptional}.

\subsection{The exceptional-set estimate}\label{direction:sec:main}
\begin{theorem}\label{direction:thm:main}
Assume~\eqref{profile:eq:mass} and~\eqref{direction:eq:xi}.
With $C_{\rm av}=2^{20}$, for all admissible depths and
occupied squares $P$ at the indicated depth,
\begin{align}\label{direction:eq:parallel-mean}
 \int_{S^1}\mathcal N_{0,g,p,P}\,d\xi
 &\le C_{\rm av}(N+1)
 R^{k(g,p)+(501+2C_{\rm prof})\sigma},\\
 \int_{S^1}\mathcal P_{0,p,t,P}\,d\xi
 &\le C_{\rm av}(1+\log_2 H)H^{-1}
 R^{\delta_A(p,t)+(501+2C_{\rm prof})\sigma},
 \qquad H=2^{t-p}.
 \label{direction:eq:projection-mean}
\end{align}
Consequently, for either type of index,
\begin{equation}\label{direction:eq:general-bad}
 \xi(S^1\setminus G_{0,\mathfrak t}(P))
 \le\min\left\{1,C_{\rm av}(N+1)
 R^{-(C_{\rm test}-501-2C_{\rm prof})\sigma}\right\}.
\end{equation}
For $C_{\rm prof}=10$ and $C_{\rm test}\ge1000$, set
\begin{equation}\label{direction:eq:Nexc}
 N_{\rm exc}(\sigma)
 =\left\lceil\max\left\{
 64,\frac{40}{79\sigma},
 \left(\frac{2}{79\sigma}\right)^2
 \right\}\right\rceil.
\end{equation}
If $N\ge N_{\rm exc}(\sigma)$ and the class probability is
defined at scale $N$, then
\begin{equation}\label{direction:eq:headline}
 \xi(S^1\setminus G_{i,\mathfrak t}(P))
 \le R^{-400\sigma}\le R^{-\sigma}
\end{equation}
for every admissible index, every occupied square at its
specified depth, and every $0\le i\le J_0$.
The constants and the threshold are independent of the class,
square, depths, measure $\xi$, and level.
\end{theorem}

The bound~\eqref{direction:eq:headline} is an absolute
$\xi$-mass estimate. It is not a relative estimate, since
no positive lower bound for $\xi(S^1)$ is assumed.

\subsection{An angular estimate}\label{direction:sec:arcs}
\begin{lemma}\label{direction:lem:arcs}
Let $0<\delta\le1$ and $L\ge1$. Suppose that
$\xi(S^1)\le1$ and
\[
 \xi(B_{\rm ch}(w,r))\le Lr
 \qquad(w\in S^1,\ \delta\le r\le1).
\]
Every closed angular arc of radius $u\ge\delta$
has $\xi$-measure at most $\min\{1,2Lu\}$.
If $z\ne0$, $\epsilon>0$, and
\begin{equation}\label{direction:eq:arc-range}
 a=\epsilon/|z|\ge\delta/2,
\end{equation}
then
\begin{align}\label{direction:eq:two-arcs}
 \xi\{w:|z\cdot w|\le\epsilon\}
 &\le\min\{1,8La\},\\
 \xi\{w:\dist(z,\R w)\le\epsilon\}
 &\le\min\{1,8La\}.\notag
\end{align}
The same conclusions hold if the hypothesis is stated
for open angular balls.
\end{lemma}

\begin{proof}
For $\delta\le u\le1/2$, a closed angular arc of radius $u$
is contained in the open chordal ball of radius $2u$
with the same center. Its measure is therefore at most
$2Lu$. For $u>1/2$, the total-mass bound suffices.
Under the angular-ball hypothesis, use the open angular
ball of radius $2u$ instead. In either case the enlargement
includes the endpoints of the closed arc.

If $a\ge1/4$, both bounds in~\eqref{direction:eq:two-arcs}
follow from $\xi(S^1)\le1$. Otherwise, the first set is
contained in two closed arcs centered at the directions
perpendicular to $z$, each of radius
$\arcsin a\le2a$. Since $2a\ge\delta$, each enlarged arc
has measure at most $4La$.
For the second set, use the arcs centered at the two
directions parallel to $z$ and the identity
$\dist(z,\R w)=|z\cdot w^\perp|$.
No rotation invariance of $\xi$ is used.
\end{proof}

The hypotheses of this lemma hold at all scales used below.
For the parallel condition, a nonzero displacement $z$
between two centers in $Q_0$ satisfies $|z|\le\sqrt2$.
With $\epsilon=W_0(p)$,
\begin{equation}\label{direction:eq:parallel-angle-range}
 \frac{\epsilon}{|z|}
 \ge\sqrt2\,2^{-p}R^{500\sigma}\ge R^{-1}.
\end{equation}
For the projection condition, rescaling the depth-$p$
square gives displacements of length at most $\sqrt2$
and tolerance
\begin{equation}\label{direction:eq:eps}
 \epsilon=2^pW_0(t)=2H^{-1}R^{500\sigma},
 \qquad H=2^{t-p}\le R.
\end{equation}
Thus $\epsilon/|z-z'|\ge R^{-1}$ for every nonzero
rescaled displacement. The diagonal is included in
the near-pair estimate below.

\subsection{The parallel condition}\label{direction:sec:parallel}
\begin{proof}[Proof of \eqref{direction:eq:parallel-mean}]
Fix $g,p,P$. The contribution of $c_P$ itself is at most
$\xi(S^1)\le1$. Distinct depth-$p$ centers have separation
at least $2^{-p}$. Partition the other centers into the range
\[
 2^{-g}<|c_Q-c_P|<30\,2^{-g}
\]
and the ranges
\begin{equation}\label{direction:eq:parallel-rings}
 2^{-s-1}<|c_Q-c_P|\le2^{-s},
 \qquad g\le s\le p.
\end{equation}

For a center in the first range, Lemma~\ref{direction:lem:arcs}
and~\eqref{direction:eq:parallel-angle-range} bound the
measure of the directions for which it is counted by
\[
 8R^\sigma W_0(p)2^g
 =16\,2^{-(p-g)}R^{501\sigma}.
\]
By Lemma~\ref{profile:lem:geometry}, with $B_0=30$,
there are at most
$4096R^{I(g,p)+2C_{\rm prof}\sigma}$ such centers.
Their total contribution is at most
\begin{equation}\label{direction:eq:parallel-outer}
 65536R^{A_p-A_g+(501+2C_{\rm prof})\sigma}
 \le65536R^{k(g,p)+(501+2C_{\rm prof})\sigma}.
\end{equation}

For the range indexed by $s$ in
\eqref{direction:eq:parallel-rings}, the corresponding
directional bound for one center is
\[
 8R^\sigma W_0(p)2^{s+1}
 =32\,2^{-(p-s)}R^{501\sigma}.
\]
The number of centers is at most
$36R^{I(s,p)+2C_{\rm prof}\sigma}$, by
\eqref{profile:eq:center-count} with $B_0=1$.
Hence this range contributes at most
\begin{align}
 1152\,2^{-(p-s)}R^{I(s,p)+(501+2C_{\rm prof})\sigma}
 &=1152R^{A_p-A_s+(501+2C_{\rm prof})\sigma}\notag\\
 &\le1152R^{k(g,p)+(501+2C_{\rm prof})\sigma}.
 \label{direction:eq:parallel-ring-bound}
\end{align}
Here $g\le s\le p$ implies $A_p-A_s\le k(g,p)$.
There are at most $N+1$ such ranges, and $k(g,p)\ge0$.
The estimate follows from
\[
 1+65536+1152(N+1)\le C_{\rm av}(N+1).
\]
\end{proof}

\subsection{The projection condition}\label{direction:sec:projection}
\begin{lemma}\label{direction:lem:geometric-average}
Let $H=2^l\le R$, with $l\ge0$ an integer, and let
$\lambda$ be a probability carried by $[-1/2,1/2)^2$.
Suppose that
\[
 \lambda(\overline B(z,r))\le Kr
 \qquad(z\in\R^2,\ H^{-1}\le r\le1),
 \qquad K\ge1.
\]
Let $\xi(S^1)\le1$, and suppose that its chordal or angular
ball bound holds with constant $L\ge1$ for radii in
$[R^{-1},1]$. For every $\epsilon\ge2/H$,
\begin{equation}\label{direction:eq:geometric-average}
 \int_{S^1}(\lambda\otimes\lambda)
 \{|(z-z')\cdot w|\le\epsilon\}\,d\xi(w)
 \le16KL\epsilon(1+\log_2 H).
\end{equation}
\end{lemma}

\begin{proof}
The case $\epsilon\ge1$ follows from the total-mass bounds.
Assume $2/H\le\epsilon<1$ and apply Tonelli's theorem.
Pairs with $|z-z'|\le\epsilon$ have total mass at most
\begin{equation}\label{direction:eq:near-pair}
 \int\lambda(\overline B(z,\epsilon))\,d\lambda(z)
 \le K\epsilon.
\end{equation}
This estimate includes the diagonal and uses a spatial
radius in $[H^{-1},1]$.

Put $J=\lceil\log_2(1/\epsilon)\rceil$. Partition the
remaining pairs at distance at most one into
\[
 2^j\epsilon<|z-z'|
 \le\min\{2^{j+1}\epsilon,1\},
 \qquad 0\le j<J.
\]
The mass of each range is at most $K2^{j+1}\epsilon$.
For each pair in that range,
Lemma~\ref{direction:lem:arcs} gives a directional bound
of at most
\[
 8L\epsilon/|z-z'|\le8L2^{-j}.
\]
Indeed, $|z-z'|\le\sqrt2$ and $\epsilon\ge2/H$
give $\epsilon/|z-z'|\ge R^{-1}$.
Thus each range contributes at most $16KL\epsilon$.

Pairs with $1<|z-z'|\le\sqrt2$ contribute at most
$8L\epsilon$, using only their total mass.
Since $J\le l=\log_2 H$ and $K,L\ge1$, the sum is at most
\[
 (9+16J)KL\epsilon
 \le16(1+\log_2 H)KL\epsilon.
\]
\end{proof}

\begin{proof}[Proof of \eqref{direction:eq:projection-mean}]
Rescale $\rho_P$ by $T_P$ as in
Lemma~\ref{profile:lem:geometry}. The resulting probability
$\lambda_P$ satisfies the preceding lemma with
\[
 H=2^{t-p},\qquad
 K=72R^{\delta_A(p,t)+2C_{\rm prof}\sigma},
 \qquad L=R^\sigma.
\]
By~\eqref{direction:eq:eps}, the rescaled tolerance is
$\epsilon=2H^{-1}R^{500\sigma}\ge2/H$.
Therefore
\[
 \int_{S^1}\mathcal P_{0,p,t,P}\,d\xi
 \le(16\cdot72\cdot2)(1+\log_2 H)H^{-1}
 R^{\delta_A(p,t)+(501+2C_{\rm prof})\sigma}.
\]
Since $16\cdot72\cdot2<C_{\rm av}$, this proves
\eqref{direction:eq:projection-mean}.
\end{proof}

\bigskip \paragraph{\emph{The logarithmic factor.}}
The assumptions of Lemma~\ref{direction:lem:geometric-average}
do not in general give a bound proportional to $\epsilon$.
Let $\lambda$ be uniform length measure on
$[-1/2,1/2)\times\{0\}$, and let $\xi$ be normalized
arclength on $S^1$. Then
$\lambda(B(z,r))\le2r$. A pair at distance $r>\epsilon$
satisfies the projection condition for a fraction
$(2/\pi)\arcsin(\epsilon/r)$ of the directions.
The distance between two independent points has density
$2(1-r)$ on $[0,1]$. Hence, for $0<\epsilon<1/2$,
the averaged projection statistic is at least
\begin{equation}\label{direction:eq:log-example}
 \int_\epsilon^{1/2}2(1-r)\frac{2}{\pi}
 \arcsin(\epsilon/r)\,dr
 \ge\frac{2\epsilon}{\pi}\log\frac{1}{2\epsilon}.
\end{equation}
Here we used $2(1-r)\ge1$ and $\arcsin u\ge u$.

\subsection{The exceptional sets}\label{direction:sec:exceptional}
\begin{proof}[Completion of Theorem~\ref{direction:thm:main}]
The complements of the sets $G_{0,\mathfrak t}(P)$
are strict superlevel sets of the corresponding statistics.
Divide~\eqref{direction:eq:parallel-mean} by the threshold
in~\eqref{direction:eq:parallel-threshold}, and divide
\eqref{direction:eq:projection-mean} by that in
\eqref{direction:eq:projection-threshold}.
Markov's inequality, $1+\log_2 H\le N+1$, and
$\xi(S^1)\le1$ give~\eqref{direction:eq:general-bad}.

For $C_{\rm prof}=10$ and $C_{\rm test}\ge1000$,
\begin{equation}\label{direction:eq:raw-bad}
 \xi(S^1\setminus G_{0,\mathfrak t}(P))
 \le2^{20}(N+1)R^{-479\sigma}.
\end{equation}
The threshold~\eqref{direction:eq:Nexc} ensures that
\begin{equation}\label{direction:eq:absorb}
 2^{20}(N+1)\le R^{79\sigma}.
\end{equation}
Indeed, $\log_2(N+1)\le\sqrt N$ for $N\ge64$,
and the other two lower bounds on $N$ give
\[
 20\le\frac{79\sigma N}{2},
 \qquad
 \sqrt N\le\frac{79\sigma N}{2}.
\]
Thus $20+\log_2(N+1)\le79\sigma N$.
Combining~\eqref{direction:eq:raw-bad} and
\eqref{direction:eq:absorb} proves the bound
$R^{-400\sigma}$ for $i=0$.
The inclusion~\eqref{direction:eq:symmetry-levels}
gives the same estimate for every $i$.
\end{proof}

The ball estimates for conditional probabilities and the
strip events include their boundaries. The closed arcs in
Lemma~\ref{direction:lem:arcs} are covered by larger open
balls. No nonatomicity assumption is used.

\bigskip \paragraph{\emph{A set of admissible directions need not be closed.}}
Let
\[
 z_0=(1/4,1/4),\qquad z_1=(3/4,1/4),\qquad
 \lambda=\tfrac12(\delta_{z_0}+\delta_{z_1}).
\]
At tolerance $3/10$,
\begin{equation}\label{direction:eq:atomic-example}
 (\lambda\otimes\lambda)
 \{|(z-z')\cdot w|\le3/10\}
 =\tfrac12+\tfrac12\one_{\{|w_1|\le3/5\}}.
\end{equation}
The diagonal contributes the first term, and the two
off-diagonal displacements are $\pm(1/2,0)$.
For threshold $3/4$, the set of admissible directions
is $\{|w_1|>3/5\}$, which is not closed.
This example concerns the local statistic; the atomic
measure does not satisfy the ball hypothesis of the
final distance-set theorem.

\subsection{Products and restrictions}\label{direction:sec:retention}
\begin{corollary}\label{direction:cor:retention}
Assume the hypotheses and numerical choices of
Theorem~\ref{direction:thm:main}, with
$N\ge N_{\rm exc}(\sigma)$.
Let $u_1,\ldots,u_M:S^1\to[0,1]$ be Borel functions.
For each $\nu$, suppose that $u_\nu=1$ on
$G_{i_\nu,\mathfrak t_\nu}(P_\nu)$ for an admissible
index and an occupied square at its specified depth.
If all estimates use the same measure $\xi$, then
\begin{equation}\label{direction:eq:product-retention}
 \int_{S^1}
 \left(1-\prod_{\nu=1}^M u_\nu(w)\right)d\xi(w)
 \le MR^{-400\sigma}.
\end{equation}
The empty product is interpreted as one.
\end{corollary}

\begin{proof}
Pointwise,
\[
 1-\prod_{\nu=1}^M u_\nu
 \le\sum_{\nu=1}^M(1-u_\nu)
 \le\sum_{\nu=1}^M
 \one_{S^1\setminus G_{i_\nu,\mathfrak t_\nu}(P_\nu)}.
\]
Integrate and apply Theorem~\ref{direction:thm:main}.
\end{proof}

An unnormalized restriction of $\xi$ to a Borel set
still satisfies~\eqref{direction:eq:xi} and the
total-mass bound. Normalizing a nonzero restriction,
however, divides its ball bound by its mass, so the
same constant is no longer guaranteed. Such restrictions
of $\xi$ do not change the geometric conditions, which
are defined using the original class and its full
conditional probabilities $\rho_P$.

\subsection{Application to smooth cutoffs}\label{direction:sec:summary}
Theorem~\ref{direction:thm:main} applies to every
admissible pair of integer depths, not only to the
indices selected in Section~\ref{main:profiles}.
In particular, it applies to $\Par{0}{c}$ from
\eqref{entry:eq:lists}, even when $c$ is not a block
representative.

For the two families in Corollary~\ref{dyad:cor:prepared},
take $N$ above both the class-construction threshold
and $N_{\rm exc}(\sigma)$. The estimates then hold
for every class and occupied square in either family.
At a fixed point, each index specifies a unique square,
so no additional sum over squares is needed.

Section~\ref{main:cutoffs-mass} uses the gaps between
successive widths to construct smooth cutoffs. Its
total-variation estimate applies the present theorem
to radial projections of unnormalized restrictions
of the probabilities from Theorem~\ref{prep:thm:preparation}.

\clearpage

\section{Smooth directional cutoffs and approximation in total variation}
\label{main:cutoffs-mass}

We construct smooth angular functions that equal one on the sets
$G_{i,\mathfrak t}$ and have support in $G_{i+1,\mathfrak t}$.
The gap between the strip widths at consecutive levels makes this
possible even though the sets need not be closed. These functions
will define the factor $\Theta_N$ from
Section~\ref{rec:conditional-section}.

We then use the probabilities of Theorem~\ref{prep:thm:preparation},
their class decompositions, and Theorem~\ref{direction:thm:main} to
estimate the corresponding measures $\tau_N$ in total variation.
An auxiliary restriction of the original product measure is used
only in this estimate, not in the definition of $\Theta_N$.
Thus the total-variation and Fourier estimates concern the same
measures, with the original class weights.

\subsection{Parameters and angular neighborhoods}
\label{cutoff:sec:data}

Fix a class probability $\rho$ and the sets $G_{i,\mathfrak t}(P)$
defined in Section~\ref{direction:sec:tests}. The counts and
conditional probabilities in those definitions always refer to
the full class. For an admissible index $\mathfrak t$, write
$p(\mathfrak t)$ for its specified depth and
$\ell=\ell(\mathfrak t)$ for its length, as in
Section~\ref{lists:sec:definitions}.

Fix $0<\sigma\le\kappa<1$ and set
\begin{equation}\label{cutoff:eq:starting-scale}
 J_0=\left\lceil100(\kappa^{-1}+\sigma^{-1})+100\right\rceil,
 \qquad
 N\ge N_{\rm cut}:=\max\left\{1,\left\lceil
 \frac{\log_2(2J_0)}{450\sigma}\right\rceil\right\}.
\end{equation}
The cutoffs will be defined for $0\le i<J_0$, using the sets
$G_{i,\mathfrak t}$ at levels $0\le i\le J_0$.

Let $d_{S^1}$ be the shorter angular distance, with values in
$[0,\pi]$. For $w(\theta)=(\cos\theta,\sin\theta)$ and
$w(\theta)^\perp=(-\sin\theta,\cos\theta)$,
\begin{equation}\label{cutoff:eq:metric}
 |w-w'|\le d_{S^1}(w,w'),\qquad
 |w^\perp-(w')^\perp|=|w-w'|.
\end{equation}
The widths in~\eqref{direction:eq:widths} satisfy
\begin{equation}\label{cutoff:eq:width-gap}
 \Delta W(n):=W_i(n)-W_{i+1}(n)
 =J_0^{-1}2^{-n}R^{500\sigma}.
\end{equation}
Both this difference and the thresholds in the definitions of the
sets $G_{i,\mathfrak t}$ are independent of $i$.

\subsection{The cutoff functions}
\label{cutoff:sec:theorem}

For $G\subseteq S^1$ and $r\ge0$, define
\begin{equation}\label{cutoff:eq:closed-enlargement}
 G^{[r]}=\{w:d_{S^1}(w,G)\le r\},
 \qquad d_{S^1}(w,G)=\inf_{u\in G}d_{S^1}(w,u),
\end{equation}
with $d_{S^1}(w,\varnothing)=+\infty$. The set $G^{[r]}$ is
closed, whether or not $G$ is closed.

For an index of length $\ell$, put
\begin{equation}\label{cutoff:eq:angular-scale}
 D_\ell=\max\{1,2^\ell R^{-50\sigma}\},
 \qquad \delta_\ell=(1000D_\ell)^{-1}.
\end{equation}
Fix an even nonnegative function $\phi\in C_c^\infty((-1,1))$
with integral one, and let $\phi_\delta$ be its periodized
dilation, defined in Section~\ref{cutoff:sec:smoothing}.

\begin{theorem}\label{cutoff:thm:cutoffs}
Assume~\eqref{cutoff:eq:starting-scale}. For every admissible
index $\mathfrak t$, every square $P$ at depth $p(\mathfrak t)$,
and every $0\le i<J_0$,
\begin{equation}\label{cutoff:eq:closed-transfer}
 G_{i,\mathfrak t}(P)^{[4\delta_\ell]}
 \subseteq G_{i+1,\mathfrak t}(P).
\end{equation}
For occupied $P$, set $G=G_{i,\mathfrak t}(P)$ and define
\begin{equation}\label{cutoff:eq:cutoff-formula}
 w_{i,\mathfrak t}(x,\cdot)
 =\phi_{\delta_\ell}*\one_{G^{[2\delta_\ell]}}
 \qquad(x\in P).
\end{equation}
Set $w_{i,\mathfrak t}=0$ on unoccupied squares and outside
$Q_0$. Then $w_{i,\mathfrak t}$ and all its angular derivatives
are jointly Borel in $(x,w)$, and the function is constant in
$x$ on each square of depth $p(\mathfrak t)$ and smooth in $w$.
For every $x$,
\begin{equation}\label{cutoff:eq:cutoff-main}
 0\le w_{i,\mathfrak t}\le1,
 \qquad w_{i,\mathfrak t}=1
 \quad\hbox{on }G_{i,\mathfrak t}(x).
\end{equation}
In fact, it equals one on $G_{i,\mathfrak t}(x)^{[\delta_\ell]}$,
and, for every integer $s\ge0$,
\begin{equation}\label{cutoff:eq:cutoff-support}
 \supp_w\partial_\theta^s w_{i,\mathfrak t}(x,\cdot)
 \subseteq\supp_w w_{i,\mathfrak t}(x,\cdot)
 \subseteq G_{i,\mathfrak t}(x)^{[3\delta_\ell]}
 \subseteq G_{i+1,\mathfrak t}(x).
\end{equation}
Moreover,
\begin{equation}\label{cutoff:eq:cutoff-derivatives}
 \sup_x\|\partial_\theta^s w_{i,\mathfrak t}(x,\cdot)\|_\infty
 \le K_sD_\ell^s,
 \qquad K_0=1,\quad
 K_s=1000^s\|\phi^{(s)}\|_{L^1(\R)}\quad(s\ge1),
\end{equation}
and
\begin{equation}\label{cutoff:eq:antipodal}
 w_{i,\mathfrak t}(x,-w)=w_{i,\mathfrak t}(x,w).
\end{equation}
The constants $K_s$ depend only on $s$ and $\phi$.
The threshold $N_{\rm cut}$ depends only on $\kappa,\sigma$,
and is independent of the class, index, square, level, and
derivative order.
\end{theorem}

With $c=1/250$, equation~\eqref{cutoff:eq:closed-transfer} gives
\begin{equation}\label{cutoff:eq:requested-transfer}
 \overline{\{w:d_{S^1}(w,G_{i,\mathfrak t}(P))
                    <cD_\ell^{-1}\}}
 \subseteq G_{i+1,\mathfrak t}(P).
\end{equation}
Indeed, the continuity of the distance function places the set
on the left in $G_{i,\mathfrak t}(P)^{[4\delta_\ell]}$.
The assertion is immediate when $G_{i,\mathfrak t}(P)$ is empty.

\subsection{Stability under changes of direction}
\label{cutoff:sec:transfer}

We first prove~\eqref{cutoff:eq:closed-transfer}. It is trivial
for an unoccupied square, since both sets are empty. For an
occupied square $P$, write $F_{i,\mathfrak t}(P,w)$ for
$\mathcal N_{i,g,p,P}(w)$ or $\mathcal P_{i,p,t,P}(w)$,
according to the type of index.

For a center $c$ counted in the parallel condition,
$|c-c_P|<30\,2^{-g}$, so~\eqref{cutoff:eq:metric} gives
\begin{align}
 &\bigl|| (c-c_P)\cdot w^\perp|
                -|(c-c_P)\cdot(w')^\perp|\bigr|\notag\\
 &\hspace{25mm}\le30\,2^{-g}d_{S^1}(w,w').
 \label{cutoff:eq:parallel-perturbation}
\end{align}
In the projection condition, $z,z'\in P$, and hence
\begin{align}
 &\bigl|| (z-z')\cdot w|-|(z-z')\cdot w'|\bigr|\notag\\
 &\hspace{25mm}\le\sqrt2\,2^{-p}d_{S^1}(w,w').
 \label{cutoff:eq:projection-perturbation}
\end{align}

Write $(u,n)=(g,p)$ for a parallel index and $(u,n)=(p,t)$
for a projection index. In both cases $\ell=n-u$, and
\begin{equation}\label{cutoff:eq:key-scale}
 2^{-u}D_\ell^{-1}\le2^{-n}R^{50\sigma}.
\end{equation}
If $d_{S^1}(w,w')\le8\delta_\ell$, either expression changes
by at most
\begin{equation}\label{cutoff:eq:perturbation-bound}
 30\,2^{-u}\,8\delta_\ell
 \le\frac{240}{1000}\,2^{-n}R^{50\sigma}.
\end{equation}
Since~\eqref{cutoff:eq:starting-scale} implies
$R^{450\sigma}\ge2J_0$, this is less than one quarter of the
width difference:
\begin{equation}\label{cutoff:eq:gap-ratio}
 \frac{\text{perturbation}}{\Delta W(n)}
 \le\frac{6}{25}J_0R^{-450\sigma}
 \le\frac{3}{25}<\frac14.
\end{equation}

Thus any center or pair counted at level $i+1$ in direction $w'$
has tested expression in direction $w$ at most
\begin{equation}\label{cutoff:eq:strict-width}
 W_{i+1}(n)+\tfrac14\Delta W(n)
 =W_i(n)-\tfrac34\Delta W(n)<W_i(n).
\end{equation}
Summing over centers, or integrating against the unchanged
probability $\rho_P\otimes\rho_P$, gives
\begin{equation}\label{cutoff:eq:pointwise-transfer}
 F_{i+1,\mathfrak t}(P,w')\le F_{i,\mathfrak t}(P,w)
 \quad\text{if }d_{S^1}(w,w')\le8\delta_\ell.
\end{equation}
The thresholds defining $G_{i,\mathfrak t}(P)$ are independent
of $i$. Consequently, if $w\in G_{i,\mathfrak t}(P)$ and
$d_{S^1}(w,w')\le8\delta_\ell$, then
$w'\in G_{i+1,\mathfrak t}(P)$.

Now let $G=G_{i,\mathfrak t}(P)\ne\varnothing$ and suppose
$d_{S^1}(w',G)\le4\delta_\ell$. The infimum in the definition
of this distance need not be attained, but there is $w\in G$
with
\[
 d_{S^1}(w',w)<5\delta_\ell<8\delta_\ell.
\]
Equation~\eqref{cutoff:eq:pointwise-transfer} proves
$w'\in G_{i+1,\mathfrak t}(P)$, giving
\eqref{cutoff:eq:closed-transfer}. The empty case follows from
the convention in~\eqref{cutoff:eq:closed-enlargement}.

\subsection{Smoothing on the circle}
\label{cutoff:sec:smoothing}

For definiteness, one may take
\begin{equation}\label{cutoff:eq:bump}
 \phi(t)=c_\phi
 \begin{cases}
 \exp\bigl(-1/(1-4t^2)\bigr),&|t|<1/2,\\
 0,&|t|\ge1/2,
 \end{cases}
 \qquad \int_\R\phi(t)\,dt=1.
\end{equation}
For $0<\delta\le1/1000$, define its $2\pi$-periodic dilation by
\begin{equation}\label{cutoff:eq:periodized-kernel}
 \phi_\delta(\theta)=
 \sum_{k\in\Z}\delta^{-1}\phi((\theta+2\pi k)/\delta).
\end{equation}
This kernel has integral one over a period and is supported
within angular distance $\delta$ of zero. Differentiation and
a change of variable give
\begin{equation}\label{cutoff:eq:kernel-L1}
 \|\partial_\theta^s\phi_\delta\|_{L^1(S^1,d\theta)}
 \le\delta^{-s}\|\phi^{(s)}\|_{L^1(\R)}
 \qquad(s\ge0).
\end{equation}
We use the convolution convention
\begin{equation}\label{cutoff:eq:circle-convolution}
 (\phi_\delta*f)(w(\theta))
 =\int_{-\pi}^{\pi}\phi_\delta(\alpha)
       f(w(\theta-\alpha))\,d\alpha,
\end{equation}
consistent with the arclength measure $dw=d\theta$.

For $G\subseteq S^1$, set
\begin{equation}\label{cutoff:eq:general-cutoff}
 f_G=\phi_\delta*\one_{G^{[2\delta]}}.
\end{equation}
Since $G^{[2\delta]}$ is closed, the indicator is Borel, and
positivity and normalization of the kernel give $0\le f_G\le1$.
If $d_{S^1}(w(\theta),G)\le\delta$, then throughout the
support of the kernel,
\[
 d_{S^1}(w(\theta-\alpha),G)
 \le\delta+|\alpha|\le2\delta.
\]
Therefore
\begin{equation}\label{cutoff:eq:plateau}
 f_G=1\quad\text{on }G^{[\delta]}.
\end{equation}
Similarly, if $d_{S^1}(w,G)>3\delta$, every point sampled in
the convolution has distance greater than $2\delta$ from $G$.
Thus
\begin{equation}\label{cutoff:eq:smooth-support}
 \supp f_G\subseteq G^{[3\delta]}.
\end{equation}

Differentiating the smooth kernel shows that $f_G\in C^\infty(S^1)$
and
\begin{equation}\label{cutoff:eq:derivative-convolution}
 \partial_\theta^s f_G
 =(\partial_\theta^s\phi_\delta)*\one_{G^{[2\delta]}}.
\end{equation}
Consequently,
\begin{equation}\label{cutoff:eq:general-derivative}
 \|\partial_\theta^s f_G\|_\infty
 \le\delta^{-s}\|\phi^{(s)}\|_1.
\end{equation}
All derivatives vanish on the open complement of $\supp f_G$,
so
\begin{equation}\label{cutoff:eq:derivative-support}
 \supp\partial_\theta^s f_G\subseteq\supp f_G
 \subseteq G^{[3\delta]}.
\end{equation}
These statements include $G=\varnothing$ and $G=S^1$, for
which $f_G$ is respectively zero and one.

If $G=-G$, then $G^{[2\delta]}$ is also antipodally symmetric.
Its indicator is $\pi$-periodic, and hence so are $f_G$ and
all its angular derivatives.

\subsection{Dependence on the spatial variable}
\label{cutoff:sec:assembly}

For each occupied square $P$ at depth $p(\mathfrak t)$, apply
\eqref{cutoff:eq:general-cutoff} with $G=G_{i,\mathfrak t}(P)$
and $\delta=\delta_\ell$, and denote the resulting function by
$f_{P,i,\mathfrak t}$. Define
\begin{equation}\label{cutoff:eq:spatial-assembly}
 w_{i,\mathfrak t}(x,w)
 =\sum_{\substack{P\text{ occupied}\\
                   \text{at depth }p(\mathfrak t)}}
       \one_P(x)f_{P,i,\mathfrak t}(w).
\end{equation}
This is a finite sum of Borel functions, as is every angular
derivative. Since the half-open squares are disjoint, at most
one term is nonzero at each $x$. The bounds on each angular
function therefore hold for the assembled function without a
factor depending on the number of squares.

By~\eqref{cutoff:eq:derivative-support} and
\eqref{cutoff:eq:closed-transfer},
\[
 \supp\partial_\theta^s f_{P,i,\mathfrak t}
 \subseteq G_{i,\mathfrak t}(P)^{[3\delta_\ell]}
 \subseteq G_{i+1,\mathfrak t}(P).
\]
Equation~\eqref{cutoff:eq:general-derivative}, with
$\delta_\ell^{-1}=1000D_\ell$, gives
\eqref{cutoff:eq:cutoff-derivatives}.
The remaining assertions of Theorem~\ref{cutoff:thm:cutoffs}
follow from~\eqref{cutoff:eq:plateau} and the antipodal symmetry.

For a fixed index, the smoothing radius is independent of $i$.
The inclusion~\eqref{direction:eq:symmetry-levels} and the
positivity of the kernel thus imply
\[
 w_{i,\mathfrak t}\le w_{i+1,\mathfrak t}
 \qquad(0\le i<J_0-1).
\]
Moreover, $w_{i+1,\mathfrak t}=1$ on
$\supp_w w_{i,\mathfrak t}(x,\cdot)$, so
\[
 w_{i,\mathfrak t}w_{i+1,\mathfrak t}=w_{i,\mathfrak t}
 \qquad(0\le i<J_0-1).
\]
The set at level $J_0$ contains the support of the last cutoff,
at level $J_0-1$; no cutoff at level $J_0$ is needed.

We use the same kernel and formula for every class and admissible
index, writing $w_{\rho,i,\mathfrak t}$ when the class must be
specified. For fixed $N$, this function depends only on $\rho$,
$i$, and $\mathfrak t$, not on an opposite class with which
$\rho$ is paired. The construction is also independent of the
derivative orders and error exponents used later; only the
constants in the derivative estimates depend on the order.
Section~\ref{main:amplitudes} uses these functions and their
angular derivatives to define the amplitudes in the Fourier
argument.

\subsection{The probabilities and the initial index sets}
\label{mass:sec:data}

Fix $0<\eta<1$, and let $\mu_1,\mu_2$ be the probabilities in
Theorem~\ref{prep:thm:preparation}, and set
\begin{equation}\label{mass:eq:parameters}
 \chi=\min\{\eta/8,1/16\},\qquad
 p=p_{\rm rad}=1+\chi/4,\qquad \beta=\chi/20.
\end{equation}
For the remainder of this section, $p$ denotes the radial
integrability exponent; the depth associated with an index is
written $p(\mathfrak t)$. Choose
\begin{equation}\label{mass:eq:small-parameters}
 0<\sigma\le\kappa<10^{-4},\qquad 12\sigma\le\chi/8.
\end{equation}

The probabilities $\mu_j$ are carried by compact sets
$K_j\subset Q_0$ satisfying
\begin{equation}\label{mass:eq:geometry}
 d_0=a_*=7/2048,\qquad d_1=d_*=1/4,\qquad
 d_0\le |x-y|\le d_1\quad(x\in K_1,\ y\in K_2).
\end{equation}
On $K_1\times K_2$, define
\[
 e(x,y)=\frac{x-y}{|x-y|},\qquad
 W(x,y)=|x-y|^{-1/2}.
\]
Both functions are continuous, and
\begin{equation}\label{mass:eq:Wbounds}
 d_1^{-1/2}\le W\le d_0^{-1/2}.
\end{equation}
We use the distance map $\dd(x,y)=|x-y|$ from
\eqref{rec:weighted-measure}. The directions $e(x,y)$ lie in the
arc specified in~\eqref{prep:eq:output-geometry}; its radius is
independent of the input, although its center may vary.

By Lemma~\ref{rad:lem:joint} and
Theorem~\ref{prep:thm:preparation}, there are finite-valued,
nonnegative Borel functions $F_1$ on $K_1\times S^1$ and
$F_2$ on $K_2\times S^1$ such that
\begin{align}\label{mass:eq:density-identities}
 e(x,\cdot)_\#\mu_2&=F_1(x,w)\,dw,\notag\\
 e(\cdot,y)_\#\mu_1&=F_2(y,w)\,dw.
\end{align}
Indeed, $e(x,y)=-\pi_x(y)=\pi_y(x)$, and the antipodal map
preserves arclength and $L^p$ norms. The versions furnished by
the preparation theorem satisfy these identities at every center.
To allow other Borel versions, let $Z_j\subseteq K_j$ be Borel
$\mu_j$-null sets outside which the corresponding identity holds;
for the prepared versions one may take $Z_j=\varnothing$.

The integrated estimates~\eqref{prep:eq:integrated-output} give
\begin{align}\label{mass:eq:radial}
 \int_{K_1}\int_{S^1}F_1(x,w)^p\,dw\,d\mu_1(x)&\le M_*,\notag\\
 \int_{K_2}\int_{S^1}F_2(y,w)^p\,dw\,d\mu_2(y)&\le M_*.
\end{align}
Here $M_*=M_*(\eta)$ is defined in
\eqref{prep:eq:output-constants}. Only these integrated estimates
for the original probabilities are used below; no radial estimate
for a normalized class probability is required.

For $N\ge N_{\rm pair}(\eta,\sigma)$, take the classes from
Corollary~\ref{dyad:cor:prepared} and write
\begin{equation}\label{mass:eq:classes}
 a_{j,C}=\mu_j(C),\qquad
 \rho_{j,C}=\frac{\mu_j|_C}{a_{j,C}},\qquad
 E_{j,N}=\bigcup_{C\in\C_{j,N}}C.
\end{equation}
With $R=2^N$, the class bounds are
\begin{gather}
 \#\C_{j,N}\le Q_N=N^{20/\sigma},\qquad
 a_{j,C}\ge R^{-\sigma},\notag\\
 e_{j,N}:=\mu_j(E_{j,N}^c)
 \le R^{-2}+Q_NR^{-\sigma}.
 \label{mass:eq:class-input}
\end{gather}
Since the classes are disjoint,
\begin{equation}\label{mass:eq:weight-identities}
 \mu_j|_{E_{j,N}}=\sum_C a_{j,C}\rho_{j,C},\qquad
 \sum_Ca_{j,C}=1-e_{j,N}\le1.
\end{equation}
We use these original weights without normalizing their sum.

Set $C_{\rm test}=1000$. Since $C_{\rm prof}=10$,
Theorem~\ref{direction:thm:main} implies
\begin{equation}\label{mass:eq:pertest}
 \xi(S^1\setminus G_{0,\mathfrak t}^{\rho}(P))
 \le R^{-\sigma}
\end{equation}
whenever $N\ge N_{\rm exc}(\sigma)$, $P$ is occupied, and
$\xi$ is a subprobability satisfying~\eqref{direction:eq:xi}.
This weaker consequence of~\eqref{direction:eq:headline} suffices
for the mass estimate.

For an ordered pair $(C,D)\in\C_{1,N}\times\C_{2,N}$, let
\begin{equation}\label{mass:eq:entry}
 c=c(C,D)=\max\left\{n\in\{0,\ldots,N\}:
       \min(A_{1,C,n},A_{2,D,n})\le\beta\right\}.
\end{equation}
For $N\ge N_{\rm ent}$,
Lemma~\ref{entry:lem:entry} applies by
\eqref{dyad:eq:prepared-lower} and
\eqref{mass:eq:small-parameters}, and gives $0<c<N/4$.
Define
\begin{equation}\label{mass:eq:root-lists}
 \mathcal A_{j;C,D}
 =\{\Par{0}{c}\}\cup\M_{\rho_{j,C_j}}(c,N),
 \qquad (C_1,C_2)=(C,D).
\end{equation}
The subscript on $\M$ identifies the class used in the construction.
By Theorem~\ref{lists:thm:main} and
Corollary~\ref{lists:cor:Ncube}, these sets have at most $N^3$
members and satisfy
\begin{equation}\label{mass:eq:list-input}
 \#\mathcal A_{j;C,D}
 \le D_N:=2(N+1)(\kappa^{-1}+2)+1.
\end{equation}
The additional index $\Par{0}{c}$ is distinct from those in
$\M_{\rho_{j,C_j}}(c,N)$, all of which have specified depth
greater than $c$.

Theorem~\ref{cutoff:thm:cutoffs} provides cutoffs
$w_{\rho,0,\mathfrak t}$ satisfying
\begin{equation}\label{mass:eq:cutoff-one}
 w_{\rho,0,\mathfrak t}(z,w)=1
 \quad\text{if }w\in G_{0,\mathfrak t}^{\rho}(z).
\end{equation}
For a fixed class and index, this is the same function in every
class pair. The argument below uses only its Borel measurability,
its values in $[0,1]$, and~\eqref{mass:eq:cutoff-one}.

\subsection{The distance measure and its approximants}
\label{mass:sec:measures}

For the index sets in~\eqref{mass:eq:root-lists}, put
\[
 b^{(0)}_{j;C,D}(z,w)
 =\prod_{\mathfrak t\in\mathcal A_{j;C,D}}
     w_{\rho_{j,C_j},0,\mathfrak t}(z,w).
\]
Define, on $K_1\times K_2$,
\begin{equation}\label{mass:eq:Theta}
 \Theta_N(x,y)=\sum_{C,D}\one_C(x)\one_D(y)
 b^{(0)}_{1;C,D}(x,e(x,y))
 b^{(0)}_{2;C,D}(y,e(x,y)).
\end{equation}
The sum and products are finite, and the sets $C\times D$ are
disjoint. Hence $\Theta_N$ is Borel, takes values in $[0,1]$,
and vanishes outside $E_{1,N}\times E_{2,N}$.

Let
\begin{equation}\label{mass:eq:measures}
 \alpha=\dd_\#(W\,\mu_1\otimes\mu_2),\qquad
 \tau_N=\dd_\#(W\Theta_N\,\mu_1\otimes\mu_2).
\end{equation}
The measure $\alpha$ is the fixed measure in
\eqref{rec:weighted-measure}. Both $\alpha$ and $\tau_N$ are
carried by the compact set
\[
 D(K_1,K_2)=\{|x-y|:x\in K_1,\ y\in K_2\}.
\]
The classes and index sets may vary with $N$, so the sequence
$\tau_N$ need not be monotone. Neither measure is renormalized.

\subsection{The total-variation bound}
\label{mass:sec:statement}

The preceding constructions apply for $N\ge N_{\rm data}$, where
\begin{equation}\label{mass:eq:Ndata}
 N_{\rm data}=\max\{2,N_{\rm pair}(\eta,\sigma),N_{\rm ent},
                      N_{\rm cut},N_{\rm exc}(\sigma)\}.
\end{equation}
The thresholds are those in
\eqref{dyad:eq:prepared-threshold}, \eqref{entry:eq:Nentry},
\eqref{cutoff:eq:starting-scale}, and~\eqref{direction:eq:Nexc}.
To absorb the polynomial factors, set
\begin{equation}\label{mass:eq:poly-constants}
 C_\kappa=2(\kappa^{-1}+3),\qquad b_\sigma=20/\sigma+1,
\end{equation}
and
\begin{equation}\label{mass:eq:Npoly}
 N_{\rm pol}=\left\lceil\max\left\{
 64,\left(\frac{4b_\sigma}{\sigma}\right)^2,
 \frac{4\log_2 C_\kappa}{\sigma}\right\}\right\rceil.
\end{equation}
Finally, let
\begin{equation}\label{mass:eq:Ne}
 N_e=\max\{N_{\rm data},\lceil4/\sigma\rceil,N_{\rm pol}\}.
\end{equation}

\begin{theorem}\label{mass:thm:mass}
Let $\mu_1,\mu_2$ and the parameters be as in
Subsection~\ref{mass:sec:data},
and define $\Theta_N,\alpha,\tau_N$ as above.
For every $N\ge\max\{N_{\rm data},\lceil4/\sigma\rceil\}$,
\begin{align}\label{mass:eq:raw}
 \iint(1-\Theta_N)\,d\mu_1d\mu_2
 &\le2R^{-2}+2Q_NR^{-\sigma}+2Q_ND_NR^{-\sigma}\notag\\
 &\quad+2M_*R^{-\sigma(p-1)/2}.
\end{align}
For every $N\ge N_e$,
\begin{gather}
 0\le\tau_N\le\alpha,\qquad
 d_1^{-1/2}\le\alpha(\R)\le d_0^{-1/2},\notag\\
 \|\alpha-\tau_N\|_{\TV}\le C_eR^{-a_e},\qquad
 a_e=\min\{\sigma/2,\sigma(p-1)/2\},\notag\\
 C_e=(6+2M_*)d_0^{-1/2}.
 \label{mass:eq:conclusion}
\end{gather}
Moreover,
\begin{equation}\label{mass:eq:exact-TV}
 \|\alpha-\tau_N\|_{\TV}
 =\iint W(x,y)(1-\Theta_N(x,y))\,d\mu_1(x)d\mu_2(y).
\end{equation}
The constants and thresholds depend only on
$\eta,\kappa,\sigma$ and the fixed conventions, not on the
particular input probability or the choice of density versions.
\end{theorem}

\subsection{The joint laws of an endpoint and a direction}
\label{mass:sec:density}

We first justify evaluating the densities at the pair direction
$e(x,y)$. Define
\[
 T_1(x,y)=(x,e(x,y)),\qquad
 \Lambda_1=(T_1)_\#(\mu_1\otimes\mu_2).
\]
For every nonnegative Borel function $H$ on $K_1\times S^1$,
\eqref{mass:eq:density-identities} and Tonelli's theorem give
\begin{align}\label{mass:eq:joint-law}
 \int H\,d\Lambda_1
 &=\int_{K_1}\int_{K_2}H(x,e(x,y))\,d\mu_2(y)d\mu_1(x)\notag\\
 &=\int_{K_1}\int_{S^1}H(x,w)F_1(x,w)\,dw\,d\mu_1(x).
\end{align}
Thus
\[
 \Lambda_1=F_1(\mu_1\otimes dw)\ll\mu_1\otimes dw.
\]
Similarly, for $T_2(x,y)=(y,e(x,y))$,
\[
 \Lambda_2:=(T_2)_\#(\mu_1\otimes\mu_2)
 =F_2(\mu_2\otimes dw).
\]

If $\widetilde F_1$ is a finite-valued nonnegative Borel version
of $F_1$ and $B\subseteq K_1\times S^1$ is the
$(\mu_1\otimes dw)$-null set on which they differ, then
\begin{equation}\label{mass:eq:version-null}
 (\mu_1\otimes\mu_2)
 \{(x,y):(x,e(x,y))\in B\}
 =\Lambda_1(B)=0.
\end{equation}
The same holds at the second endpoint. Consequently, the threshold
events used below are unchanged up to product null sets when the
Borel density versions are changed.

We also exclude null sets arising from unoccupied squares. For a
class $C\in\C_{j,N}$ and a square $P$,
\[
 \mu_j(C\cap P)=a_{j,C}\rho_{j,C}(P).
\]
If $P$ is unoccupied for $\rho_{j,C}$, its intersection with $C$
therefore has zero $\mu_j$-measure. At each fixed $N$, there are
only finitely many classes, depths, and indices. Outside the
resulting finite union of null sets, every specified square used
below is occupied. The rectangles $Z_1\times K_2$ and
$K_1\times Z_2$ also have zero product measure.

\subsection{An auxiliary restriction of the pair measure}
\label{mass:sec:truncation}

Set $T_N=R^{\sigma/2}$ and define
\begin{equation}\label{mass:eq:Gamma}
 \begin{split}
 \Gamma_N=\{(x,y)\in(K_1\setminus Z_1)\times(K_2\setminus Z_2):\ &
 F_1(x,e(x,y))\le T_N,\\
 &F_2(y,e(x,y))\le T_N\}.
 \end{split}
\end{equation}
This set is Borel. It is used only in the proof of the mass
estimate and is not included in the definition of $\Theta_N$
or $\tau_N$.

\begin{lemma}\label{mass:lem:tail}
We have
\begin{equation}\label{mass:eq:tail}
 (\mu_1\otimes\mu_2)(\Gamma_N^c)
 \le2M_*T_N^{1-p}
 =2M_*R^{-\sigma(p-1)/2}.
\end{equation}
\end{lemma}

\begin{proof}
By~\eqref{mass:eq:joint-law}, the product measure of the event
$F_1(x,e(x,y))>T_N$ equals
\[
 \int_{K_1}\int_{\{w:F_1(x,w)>T_N\}}
       F_1(x,w)\,dw\,d\mu_1(x).
\]
Since $F_1\le T_N^{1-p}F_1^p$ on this event,
\eqref{mass:eq:radial} bounds the integral by
$M_*T_N^{1-p}$. The same argument applies to $F_2$.
The union bound gives~\eqref{mass:eq:tail}.
\end{proof}

For each endpoint, define the directional subprobabilities
\begin{align}\label{mass:eq:xi}
 \xi_{1,x,N}
 &=e(x,\cdot)_\#\bigl(\one_{\Gamma_N}(x,\cdot)\mu_2\bigr),\notag\\
 \xi_{2,y,N}
 &=e(\cdot,y)_\#\bigl(\one_{\Gamma_N}(\cdot,y)\mu_1\bigr).
\end{align}
No normalization is performed after restriction.

\begin{lemma}\label{mass:lem:cap}
For $x\in K_1\setminus Z_1$, as measures on $S^1$,
\begin{equation}\label{mass:eq:dominated-density}
 \xi_{1,x,N}
 \le\one_{\{F_1(x,\cdot)\le T_N\}}F_1(x,w)\,dw
 \le T_N\,dw.
\end{equation}
The analogous assertion holds for $\xi_{2,y,N}$ when
$y\in K_2\setminus Z_2$. For $N\ge\lceil4/\sigma\rceil$,
these measures satisfy~\eqref{direction:eq:xi}.
\end{lemma}

\begin{proof}
For a Borel set $A\subseteq S^1$, omitting the restriction
involving $F_2$ in the definition of $\xi_{1,x,N}(A)$ gives
\[
 \xi_{1,x,N}(A)
 \le\int_A\one_{\{F_1(x,w)\le T_N\}}F_1(x,w)\,dw.
\]
This proves~\eqref{mass:eq:dominated-density}.

For $0<r\le1$, a chordal ball of radius $r$ has arclength
$4\arcsin(r/2)\le\pi r<4r$; an angular ball has arclength
$2r$. Hence either type of ball satisfies
\[
 \xi_{1,x,N}(B(w,r))\le4T_Nr.
\]
If $N\ge4/\sigma$, then $4\le R^{\sigma/2}$, so
$4T_Nr\le R^\sigma r$. Since $\xi_{1,x,N}$ has total mass
at most one, the hypotheses of~\eqref{direction:eq:xi} hold.
The second endpoint is treated identically.
\end{proof}

\subsection{Summation over the class pairs}
\label{mass:sec:weights}

Fix $C\in\C_{1,N}$, $D\in\C_{2,N}$, and
$\mathfrak t\in\mathcal A_{1;C,D}$. For
$\mu_1$-almost every $x\in C$, the specified square is occupied.
Write
\[
 B_{C,\mathfrak t,x}
 =S^1\setminus G_{0,\mathfrak t}^{\rho_{1,C}}(x).
\]
Lemma~\ref{mass:lem:cap} and~\eqref{mass:eq:pertest} imply
\begin{equation}\label{mass:eq:one-fibre-failure}
 \xi_{1,x,N}(B_{C,\mathfrak t,x})\le R^{-\sigma}.
\end{equation}
Restricting the second endpoint further to $D$ can only decrease
this measure. Integrating with respect to $\mu_1|_C$ gives
\begin{equation}\label{mass:eq:one-failure}
 \begin{aligned}
 &(\mu_1\otimes\mu_2)
 \{(x,y)\in(C\times D)\cap\Gamma_N:
          e(x,y)\in B_{C,\mathfrak t,x}\}\\
 &\hspace{35mm}\le a_{1,C}R^{-\sigma}.
 \end{aligned}
\end{equation}
The event is Borel, since the specified squares form a finite
Borel partition and the corresponding good sets are Borel.
The estimate contains only $a_{1,C}$: restricting the original
opposite probability to $D$ does not by itself produce a factor
$a_{2,D}$.

There are at most $D_N$ first-endpoint conditions for each pair
$(C,D)$. Their total contribution is therefore bounded by
\begin{equation}\label{mass:eq:first-total}
 \sum_{D\in\C_{2,N}}\sum_{C\in\C_{1,N}}
       D_Na_{1,C}R^{-\sigma}
 \le Q_ND_NR^{-\sigma}.
\end{equation}
The second endpoint gives the same bound. This summation uses
the original weights at one endpoint and the number of classes
at the other. No sum over specified squares is needed, since
$x$ and $\mathfrak t$ determine the square.

For a pair in $C\times D$, every factor in
\eqref{mass:eq:Theta} equals one whenever all its selected
conditions hold. Thus, up to a product null set,
\begin{align}\label{mass:eq:pointwise-union}
 1-\Theta_N
 &\le\one_{(E_{1,N}\times E_{2,N})^c}
       +\one_{\Gamma_N^c}\notag\\
 &\quad+\sum_{j=1}^2\sum_{C,D}
       \sum_{\mathfrak t\in\mathcal A_{j;C,D}}
       \one_{\{(x,y)\in(C\times D)\cap\Gamma_N:
           e(x,y)\notin G_{0,\mathfrak t}^{\rho_{j,C_j}}(z_j)\}},
\end{align}
where $(z_1,z_2)=(x,y)$ and $(C_1,C_2)=(C,D)$.
The product measure of the first event is
\[
 1-(1-e_{1,N})(1-e_{2,N})
 =e_{1,N}+e_{2,N}-e_{1,N}e_{2,N}
 \le e_{1,N}+e_{2,N}.
\]
Integrating~\eqref{mass:eq:pointwise-union} and using
\eqref{mass:eq:class-input}, Lemma~\ref{mass:lem:tail}, and
\eqref{mass:eq:first-total} in both orientations proves
\eqref{mass:eq:raw}.

\subsection{The threshold and the total-variation identity}
\label{mass:sec:tv}

\begin{lemma}\label{mass:lem:polynomial}
For $N\ge N_{\rm pol}$,
\begin{equation}\label{mass:eq:polynomial}
 Q_N(1+D_N)\le R^{\sigma/2}.
\end{equation}
\end{lemma}

\begin{proof}
For $N\ge1$, we have $1+D_N\le C_\kappa(N+1)$, whence
\[
 \log_2\bigl(Q_N(1+D_N)\bigr)
 \le b_\sigma\log_2(N+1)+\log_2 C_\kappa.
\]
For $N\ge64$, $\log_2(N+1)\le\sqrt N$.
The definition of $N_{\rm pol}$ also gives
\[
 b_\sigma\sqrt N\le\frac{\sigma N}{4},\qquad
 \log_2 C_\kappa\le\frac{\sigma N}{4}.
\]
Adding these inequalities proves~\eqref{mass:eq:polynomial}.
\end{proof}

To complete the proof of Theorem~\ref{mass:thm:mass}, let
\[
 \zeta_N=W(1-\Theta_N)\,\mu_1\otimes\mu_2.
\]
This measure is positive, and
\[
 \alpha-\tau_N=\dd_\#\zeta_N\ge0.
\]
By the convention~\eqref{main:variation} and preservation of
total mass under pushforward,
\[
 \|\alpha-\tau_N\|_{\TV}
 =(\alpha-\tau_N)(\R)
 =\zeta_N(K_1\times K_2),
\]
which is~\eqref{mass:eq:exact-TV}.
For $N\ge N_e$, equations~\eqref{mass:eq:raw} and
\eqref{mass:eq:polynomial} yield
\begin{align*}
 \iint(1-\Theta_N)\,d\mu_1d\mu_2
 &\le2R^{-2}+2R^{-\sigma/2}
       +2M_*R^{-\sigma(p-1)/2}\\
 &\le(6+2M_*)R^{-a_e}.
\end{align*}
Here $R^{-2}\le R^{-\sigma/2}$, since $0<\sigma<1$.
Combining this with~\eqref{mass:eq:Wbounds} proves the
total-variation bound. The same bounds for $W$ give the asserted
upper and lower bounds for $\alpha(\R)$.

Define
\begin{equation}\label{mass:eq:Npositive}
 N_{\rm pos}=\max\left\{N_e,
 \left\lceil\frac{\log_2(2C_e\sqrt{d_1})}{a_e}\right\rceil
 \right\}.
\end{equation}
Then, for $N\ge N_{\rm pos}$,
\begin{equation}\label{mass:eq:positive-mass}
 \tau_N(\R)
 =\alpha(\R)-\|\alpha-\tau_N\|_{\TV}
 \ge\tfrac12d_1^{-1/2}>0.
\end{equation}
The geometric series also gives, for every integer $K\ge N_e$,
\begin{equation}\label{mass:eq:TV-tail}
 \sum_{N\ge K}\|\alpha-\tau_N\|_{\TV}
 \le\frac{C_e2^{-a_eK}}{1-2^{-a_e}}.
\end{equation}
Neither conclusion requires the approximants to be nested.

\subsection{The measures associated with class pairs}
\label{mass:sec:components}

For each ordered pair $(C,D)$, define
\[
 \tau_{C,D,N}=\dd_\#\left[
 W(x,y)b^{(0)}_{1;C,D}(x,e(x,y))
 b^{(0)}_{2;C,D}(y,e(x,y))
 \,d\rho_{1,C}(x)d\rho_{2,D}(y)\right].
\]
Substituting~\eqref{mass:eq:classes} into
\eqref{mass:eq:Theta} gives
\begin{equation}\label{mass:eq:component-identity}
 \tau_N=\sum_{C,D}a_{1,C}a_{2,D}\tau_{C,D,N}.
\end{equation}
Here the depth $c$ and both index sets are those of the pair
$(C,D)$. Since
\[
 \sum_{C,D}a_{1,C}a_{2,D}\le1,
\]
a uniform Fourier estimate for the measures $\tau_{C,D,N}$
passes to $\tau_N$ without a factor counting class pairs.
This differs from~\eqref{mass:eq:one-failure}, where the bound
contains only one class weight.

\begin{corollary}\label{mass:cor:prepared}
For the probabilities in Theorem~\ref{prep:thm:preparation},
the measures~\eqref{mass:eq:measures} satisfy
\begin{equation}\label{mass:eq:prepared-rate}
 0\le\tau_N\le\alpha,\qquad
 \|\alpha-\tau_N\|_{\TV}
 \le C_e2^{-\sigma\chi N/8}
 \quad(N\ge N_e),
\end{equation}
where $\chi=\min\{\eta/8,1/16\}$ and
\[
 C_e=(6+2M_*(\eta))\sqrt{2048/7}.
\]
The constants and thresholds are independent of the original
normalized input once $\eta,\kappa,\sigma$ are fixed.
Moreover,
\[
 2\le\alpha(\R)\le\sqrt{2048/7},\qquad
 \tau_N(\R)\ge1\quad(N\ge N_{\rm pos}).
\]
\end{corollary}

\begin{proof}
Since $p-1=\chi/4$ and $0<\chi\le1/16$,
\[
 a_e=\min\{\sigma/2,\sigma(p-1)/2\}=\sigma\chi/8.
\]
Apply Theorem~\ref{mass:thm:mass} and
\eqref{mass:eq:positive-mass}, using $d_1=1/4$ and
$d_0=7/2048$.
\end{proof}

This verifies the total-variation hypothesis of
Proposition~\ref{rec:conditional} for the measures
\eqref{mass:eq:measures}. The Fourier estimate for these same
measures is proved in Theorem~\ref{frequency:thm:frequency};
Corollary~\ref{frequency:cor:common} gives a common starting scale.

\subsection{Examples concerning restriction and normalization}
\label{mass:sec:examples}

The following examples concern general measures and are not
asserted to satisfy all the hypotheses of the preparation theorem.

The exponent $1-p$ in~\eqref{mass:eq:tail} cannot be replaced by
$-p$ under an $L^p$ bound alone. For $T\ge1$ and $p>1$, choose
two disjoint arcs of length one. Let $F$ equal $2T$ on a subarc
of the first of length $(2T)^{-p}$, equal
$1-(2T)^{1-p}$ on the second arc, and vanish elsewhere. Then
\[
 \int F\,dw=1,\qquad
 \int F^p\,dw=1+\bigl(1-(2T)^{1-p}\bigr)^p\le2,
\]
whereas
\[
 \int_{\{F>T\}}F\,dw=(2T)^{1-p}.
\]
Thus the arclength of the tail and its probability under
$F\,dw$ have different decay rates.

Normalization also changes a directional ball bound. Arclength
restricted to an open arc of length $\ell<1$ has mass $\ell$
and density one. After normalization its density is $1/\ell$,
and the open angular ball of radius $\ell/2$ centered at the
midpoint has mass one. This is why the measures in
\eqref{mass:eq:xi} are not normalized.

Nor does restriction to the opposite class automatically supply
its weight in~\eqref{mass:eq:one-failure}. Let $\xi$ be uniform
measure on an interval of length one, let $D$ be a subinterval
of length $\delta\in(0,1)$, and set $B=D$. Then
\[
 \xi(B)=\delta,\qquad a_D=\delta,\qquad
 \xi|_D(B)=\delta>\delta^2=a_D\delta.
\]
The estimate $\xi(B)\le\delta$ survives restriction, but does
not acquire the factor $a_D$.

Finally, inserting the auxiliary restriction into the approximant
would define a different measure:
\[
 \widetilde\tau_N
 =\dd_\#(W\one_{\Gamma_N}\Theta_N\,\mu_1\otimes\mu_2).
\]
Although $0\le\widetilde\tau_N\le\tau_N$, domination of
positive measures does not imply domination of their Fourier
transforms in absolute value. The Fourier argument must therefore
use the same $\tau_N$ as~\eqref{mass:eq:measures}.

\subsection{Consequences for the Fourier argument}
\label{mass:sec:summary}

Theorem~\ref{cutoff:thm:cutoffs} provides the angular support
and derivative bounds, while Theorem~\ref{mass:thm:mass} gives
summable total-variation errors for the measures
\eqref{mass:eq:measures}. Their decomposition
\eqref{mass:eq:component-identity} uses the original product
weights. Section~\ref{main:amplitudes} develops the differentiated
products of these cutoffs needed for the Fourier estimate.

\clearpage
\section{Allowed amplitudes and angular quantities}\label{main:amplitudes}

The Fourier argument uses products of the smooth cutoffs from Section~\ref{main:cutoffs-mass} and their normalized angular derivatives. We call these products, multiplied by a bounded function of the spatial variable, \emph{amplitudes}. Each factor is associated with a condition index and a prescribed derivative order. For a fixed bound on the combined prescribed derivative order, the amplitudes have uniform pointwise bounds independent of the number of factors, even when that number grows with $N$. Further angular derivatives used to estimate an amplitude are controlled by the cutoff derivative bounds; these estimates do not change the derivative orders prescribed in the original amplitude.

We then compare two angular expressions. The first takes the squared modulus of an angular integral in which the two endpoint transforms share the same direction. The second takes the product of the separate angular integrals of the squared moduli of the two endpoint transforms. Cauchy--Schwarz bounds the first by the second. In the second expression, nonnegativity allows us to bound and remove cutoff factors that are constant on the spatial squares under consideration, at a fixed multiplicative cost. We will specify when this removal preserves the allowed form of the remaining amplitudes; the circle expansion and the geometric counting estimate follow in later sections.

\subsection{Classes, cutoff factors, and prescribed orders}\label{amplitude:sec:data}
Let $\rho_1,\rho_2$ be class probabilities with profiles $m_1,m_2$ from Section~\ref{main:dyadic}. The same class may be used at both endpoints. We keep
\[
 R=2^N,\qquad M_j(n)=R^{-m_j(n/N)},\qquad C_{\rm prof}=10.
\]
For every occupied depth-$n$ square, we use the mass comparison in~\eqref{dyad:eq:prepared-comparison}. Restricting the domains of integration does not change the class probabilities or their profiles.

At endpoint $j$, let $\mathcal A_j$ index a finite family of at most $N^3$ cutoff factors. Each $t\in\mathcal A_j$ specifies one factor with condition index $\mathfrak t_{j,t}$ from Section~\ref{lists:sec:definitions}. Its specified depth and length are
\[
 p_{j,t}=p(\mathfrak t_{j,t}),\qquad
 \ell_{j,t}=\ell(\mathfrak t_{j,t}).
\]
The same condition may occur in several factors, each indexed separately, even when they use the same cutoff function. The objects $\M(\mathcal B)$ of Section~\ref{main:profiles} are finite sets of distinct condition indices. When we form an amplitude from such a set, we initially use one factor for each member. The size bound for these condition sets, including one extra initial condition, was proved there.

We fix levels $0\le i_j<J_0$ and write
\begin{equation}\label{amplitude:eq:cutoff-data}
 w_{j,t}(z,w)=w_{\rho_j,i_j,\mathfrak t_{j,t}}(z,w),\qquad
 D_{j,t}=D_{\ell_{j,t}}=\max\{1,2^{\ell_{j,t}}R^{-50\sigma}\}.
\end{equation}
The cutoff functions $w_{j,t}$ are those constructed in Theorem~\ref{cutoff:thm:cutoffs}. In particular,
\begin{gather}
 0\le w_{j,t}\le1,\qquad
 |\partial_\theta^s w_{j,t}(z,w)|\le K_sD_{j,t}^{s}\quad(s\ge1),\notag\\
 \supp_w\partial_\theta^s w_{j,t}(z,\cdot)
 \subseteq\supp_w w_{j,t}(z,\cdot)
 \subseteq G_{i_j+1,\mathfrak t_{j,t}}^{\rho_j}(z)\quad(s\ge0).\label{amplitude:eq:individual-bounds}
\end{gather}
The angular variable is $w=w(\theta)=(\cos\theta,\sin\theta)$, and $\partial_\theta$ denotes differentiation with respect to this angular parameter. Arclength is $dw=d\theta$. Each $w_{j,t}$ and all its angular derivatives are constant as functions of $z$ on each dyadic square at depth $p_{j,t}$. The numerical constants $K_s$ are common to all classes and occurrences. The endpoint levels may differ here; the quantities involving a single profile considered later will use a common level.

\subsubsection{Allowed amplitudes}
An amplitude multiplies an oscillatory exponential. At endpoint $j$, its allowed form is
\begin{equation}\label{amplitude:eq:form}
 b_{j,z}(w)=f_j(z)
 \prod_{t\in\mathcal A_j}D_{j,t}^{-\alpha_{j,t}}
                  \partial_\theta^{\alpha_{j,t}}w_{j,t}(z,w),
 \qquad |f_j|\le1.
\end{equation}
Here $f_j$ is Borel, may be complex valued, and is independent of both the angle and the radial frequency. The integers $\alpha_{j,t}\ge0$ are fixed as part of the chosen amplitude. The sum of these orders is the \emph{prescribed derivative order} of the amplitude at that endpoint. We use one bound $B\in\{0,1,2,\ldots\}$ for the combined order at the two endpoints:
\begin{equation}\label{amplitude:eq:combined-order}
 \sum_{j=1}^2\sum_{t\in\mathcal A_j}\alpha_{j,t}\le B.
\end{equation}
In particular, $B$ bounds the order at either endpoint. The bound $B$ is fixed independently of $N$. When a statement concerns just one amplitude, it is enough that its own sum is at most $B$.

Every indexed occurrence appears once in~\eqref{amplitude:eq:form}, including those with order zero. An empty product is one, so an empty family gives $b_{j,z}(w)=f_j(z)$. Scalar coefficients in a linear combination stay outside the amplitudes. The allowed form contains no further angular multiplier or radial-frequency dependence.

These functions are jointly Borel in $(z,w)$ and smooth in $w$ at each fixed $z$. The proofs use only boundedness and measurability of $f_j$ in the spatial variable.

The prescribed orders $\alpha_{j,t}$ describe the amplitude that occurs in the integral. We also estimate $\partial_\theta^s b_{j,z}$ for a further fixed order $s$, for example when integration by parts is used to bound a remainder. Such an estimate leaves the original amplitude and its prescribed order unchanged. If a differentiated product itself occurs as an amplitude in a later integral, its new prescribed order must satisfy the bound stated for that integral. A derivative bound alone does not establish membership in the specified family.

\subsection{Product bounds}\label{amplitude:sec:products}
The number of factors may grow with $N$, but at most $B$ have positive prescribed derivative order. All the other factors are bounded by one. This yields a pointwise bound for the amplitudes independent of the number of factors.

For $B\ge1$, we set
\begin{equation}\label{amplitude:eq:UB}
 U_B=\bigl(\max\{1,K_1,\ldots,K_B\}\bigr)^B,
 \qquad U_0=1.
\end{equation}
For one endpoint, we put
\begin{equation}\label{amplitude:eq:S}
 S_j=\max\left\{1,\sum_{t\in\mathcal A_j}D_{j,t}\right\}.
\end{equation}
The scale $S_j$ is only used in the derivative estimate. It is not an additional factor in the definition of an amplitude.

\Needspace{140pt}
\begin{lemma}\label{amplitude:lem:products}
Every amplitude of prescribed order at most $B$, as well as every subproduct of its normalized cutoff factors, has modulus at most $U_B$. For every fixed integer $s\ge0$,
\begin{equation}\label{amplitude:eq:product-bound}
 \sup_z\|\partial_\theta^s b_{j,z}\|_\infty
 \le C_{B,s}S_j^s.
\end{equation}
The constants depend only on $B,s$ and $K_1,\ldots,K_{B+s}$. They do not depend on the number of undifferentiated factors, the class, the positions, or the chosen bounded spatial multiplier. At order zero one may use $U_B$, without a factor depending on $N$.
\end{lemma}
\begin{proof}
There are at most $B$ occurrences with positive prescribed order. For each one,
\[
 D_{j,t}^{-\alpha_{j,t}}
 |\partial_\theta^{\alpha_{j,t}}w_{j,t}|
 \le K_{\alpha_{j,t}}.
\]
All other factors have modulus at most one. This proves the zero-order assertion, including the same bound for every subproduct.

For $s\ge1$, the product rule gives, with the endpoint subscript omitted,
\begin{equation}\label{amplitude:eq:Leibniz}
 \partial_\theta^s\left[f(z)\prod_tD_t^{-\alpha_t}
                            \partial_\theta^{\alpha_t}w_t\right]
 =f(z)\sum_{\substack{u_t\ge0\\\sum_tu_t=s}}
 \frac{s!}{\prod_tu_t!}
 \prod_tD_t^{-\alpha_t}\partial_\theta^{\alpha_t+u_t}w_t.
\end{equation}
The multiplier $f(z)$ is independent of $\theta$, and each factor of positive total order satisfies
\begin{equation}\label{amplitude:eq:factor-bound}
 D_t^{-\alpha_t}|\partial_\theta^{\alpha_t+u_t}w_t|
 \le K_{\alpha_t+u_t}D_t^{u_t}.
\end{equation}
A factor with $\alpha_t=u_t=0$ has bound one. At most $B+s$ factors satisfy $\alpha_t+u_t>0$, and each corresponding order is at most $B+s$. Hence the product of their numerical constants is bounded in terms of $B,s$ and $K_1,\ldots,K_{B+s}$. For a nonempty family, the multinomial theorem gives
\begin{equation}\label{amplitude:eq:multinomial}
 \sum_{\sum_tu_t=s}\frac{s!}{\prod_tu_t!}\prod_tD_t^{u_t}
 =\left(\sum_tD_t\right)^s.
\end{equation}
This proves~\eqref{amplitude:eq:product-bound}. For an empty family, the amplitude is $f(z)$ and all positive angular derivatives vanish. The argument includes $B=0$ and repeated factors.
\end{proof}

The bound $0\le w_{j,t}\le1$ for undifferentiated cutoff factors is essential. If each of $n$ such factors were bounded only by a constant $K_0>1$, the product estimate would incur a factor $K_0^n$ even at derivative order zero.

\subsection{A common derivative scale}\label{amplitude:sec:scale}
Let $L$ be the largest condition length $\ell_{j,t}$ among the factors under consideration, with $L=0$ if there are no factors. We set
\begin{equation}\label{amplitude:eq:Dstar}
 D_{\max}=\max\{1,2^LR^{-50\sigma}\},\qquad
 D_*=R^\sigma D_{\max}.
\end{equation}
Since each family has at most $N^3$ factors,
\begin{equation}\label{amplitude:eq:single-scale}
 |b_{j,z}|\le U_B,\qquad
 |\partial_\theta^s b_{j,z}|
 \le C_{B,s}(N^3D_{\max})^s\quad(s\ge1).
\end{equation}
We also need uniform bounds for angular derivatives of products evaluated at possibly different spatial points.

\begin{lemma}\label{amplitude:lem:copies}
Suppose each constituent amplitude has prescribed order at most $B$. A product of two copies, with complex conjugation allowed and with possibly different spatial points, satisfies
\begin{equation}\label{amplitude:eq:pair-scale}
 \left|\partial_\theta^s
 \bigl(b_{j,z}\overline{b_{j,z'}}\bigr)\right|
 \le C_{B,s}D_*^s\qquad(s\ge0)
\end{equation}
once $2N^3\le R^\sigma$. A product of at most four such constituent amplitudes obeys the same form of bound once
\begin{equation}\label{amplitude:eq:common-threshold}
 4N^3\le R^\sigma.
\end{equation}
For any fixed $s$, only the individual cutoff constants through order $B+s$ are needed. The constants can depend on $B,s$ and the number of copies, but not on $N$ or on their spatial points.
\end{lemma}
\begin{proof}
We keep the factors from different copies separately indexed. Two copies have at most $2N^3$ factors and combined prescribed order at most $2B$. After $s$ further derivatives, at most $2B+s$ factors have positive total derivative order, and the derivative order of each factor is at most $B+s$. The argument of Lemma~\ref{amplitude:lem:products} therefore gives
\[
 C_{B,s}(2N^3D_{\max})^s.
\]
The hypothesis $2N^3\le R^\sigma$ gives the claimed bound. At order zero, we use $U_B^2\le C_{B,0}$. The same argument applies to at most four copies under~\eqref{amplitude:eq:common-threshold}; each individual order is still at most $B+s$. Since all cutoff derivatives are real-valued, complex conjugation affects only the spatial multiplier.
\end{proof}

For fixed $\sigma>0$, the lower threshold on $N$ required by~\eqref{amplitude:eq:common-threshold} can be chosen independently of $B$ and $s$, although the constants may grow with those orders.

The prescribed orders still have to be counted when a product occurs in an integral. In particular, two copies of an endpoint of order $B$ have combined order $2B$. Repeated occurrences of the same cutoff must also be retained as separate factors: if its value is $1/2$, two occurrences contribute $1/4$.

\subsection{Supports, restrictions, and complex conjugation}\label{amplitude:sec:support}
The support of each amplitude retains all the conditions attached to its factors, including after angular differentiation.

\begin{lemma}\label{amplitude:lem:support}
For every allowed amplitude and every integer $s\ge0$,
\begin{equation}\label{amplitude:eq:support}
 \supp_w\partial_\theta^s b_{j,z}(\cdot)
 \subseteq
 \bigcap_{t\in\mathcal A_j}\supp_w w_{j,t}(z,\cdot)
 \subseteq
 \bigcap_{t\in\mathcal A_j}G_{i_j+1,\mathfrak t_{j,t}}^{\rho_j}(z).
\end{equation}
An empty intersection means the whole circle. Restriction of the spatial variable keeps every remaining listed condition and every prescribed derivative order. Complex conjugation preserves the allowed form. Each amplitude and each of its angular derivatives are antipodally symmetric.
\end{lemma}
\begin{proof}
Outside the support of any cutoff factor, that cutoff and all its derivatives vanish in a neighborhood. The amplitude and its derivatives therefore vanish there as well. This proves the first inclusion in~\eqref{amplitude:eq:support}; the second follows from~\eqref{amplitude:eq:individual-bounds}.

Spatial restriction keeps the same cutoff functions and orders and restricts only $f_j$. Conjugation replaces $f_j$ by $\overline{f_j}$ because the cutoff derivatives are real. Finally, the cutoffs are $\pi$-periodic in $\theta$, and derivatives and products preserve this period.
\end{proof}

In particular, a nonzero amplitude imposes each of its listed next-level conditions. The good sets need not be closed: the closed cutoff supports already lie inside them by Theorem~\ref{cutoff:thm:cutoffs}.

The support formulation matters: a derivative can be nonzero where the function itself vanishes. For example, $w(\theta)=1-\cos\theta$ satisfies $w(0)=0$ and $w''(0)=1$, while zero lies in its support.

All conditions refer to the original full class. After restriction, the occupied squares and the actual conditional probabilities in the projection tests remain those fixed in Section~\ref{main:directions}.

\subsection{Endpoint transforms}\label{angular:sec:transforms}
For integers $0\le a\le v\le N$ and occupied depth-$a$ squares $X,Y$ of the respective classes, we define
\begin{align}\label{angular:eq:transforms}
 U_X(r,w)&=M_1(a)^{-1}\int_Xe^{\ii rw\cdot x}b_{1,x}(w)\,d\rho_1(x),\notag\\
 V_Y(r,w)&=M_2(a)^{-1}\int_Ye^{\ii rw\cdot y}b_{2,y}(w)\,d\rho_2(y).
\end{align}
The exponents have a positive sign, in contrast to the Fourier transform convention in~\eqref{main:fourier}. This difference will be taken into account in the distance formulas.

The integrals are taken with respect to the fixed class probabilities $\rho_j$ and are divided by the model masses $M_j(a)$, rather than by the actual masses of $X$ and $Y$. The mass comparison gives
\begin{equation}\label{angular:eq:transform-bound}
 |U_X(r,w)|\le U_B\frac{\rho_1(X)}{M_1(a)}
 \le U_BR^{C_{\rm prof}\sigma},\qquad
 |V_Y(r,w)|\le U_BR^{C_{\rm prof}\sigma}.
\end{equation}
Boundedness and joint Borel measurability of the integrands imply that the transforms are Borel in $(r,w)$. No density of the class measures or spatial smoothness of the multipliers is needed.

The squares $X,Y$ may coincide or touch, and the two class measures may be identical. A zero-mass square has zero transform; later maxima use only occupied squares. The subsequent conversion of the distance phase into circle integrals additionally requires separation of the actual supports.

For $a\le d\le N$, we define $U_P$ on a depth-$d$ descendant using the same amplitude and denominator $M_1(d)$, and define $V_Q$ in the same way. Additivity gives
\begin{equation}\label{angular:eq:split}
 U_X(r,w)=\frac{M_1(d)}{M_1(a)}
          \sum_{\substack{P\subseteq X\\P\in\D_d}}U_P(r,w),
 \qquad
 V_Y(r,w)=\frac{M_2(d)}{M_2(a)}
          \sum_{\substack{Q\subseteq Y\\Q\in\D_d}}V_Q(r,w).
\end{equation}
The depth-$d$ descendants form finite disjoint partitions of $X$ and $Y$. Every amplitude factor remains present, and the ratios are $R^{-I_j(a,d)}$ in the notation of Section~\ref{main:profiles}.

Squaring the sums in~\eqref{angular:eq:split} produces cross terms. Controlling them requires the later geometric counting argument; the partition identity alone does not bound the resulting squared moduli by the sums of the corresponding diagonal terms.

\subsection{Shared-angle and independent-angle expressions}\label{angular:sec:definitions}
We fix once and for all
\begin{equation}\label{angular:eq:radial-cutoff}
 \Psi\in C_c^\infty((1/4,4)),\qquad
 \Psi\ge0,\qquad \Psi\ge1\text{ on }[1/2,2],\qquad
 d\lambda_v(r)=\Psi(r/2^v)\,\frac{dr}{2^v}.
\end{equation}
The measure $\lambda_v$ is positive and finite on $(0,\infty)$. Its total mass is the fixed number $\int_0^\infty\Psi(s)\,ds$, independent of $v$. It need not be a probability.

For a smooth real function $\psi$ on the circle with $0\le\psi\le1$, we set
\begin{equation}\label{angular:eq:shared}
 \mathsf H_F^{X,Y}(a,v;b_1,b_2)
 =2^{2(v-a)}\int_0^\infty
 \left|\int_{S^1}\psi(w)U_X(r,w)\overline{V_Y(r,w)}\,dw\right|^2
 \,d\lambda_v(r).
\end{equation}
This is the \emph{shared-angle expression}: the same $w$ appears in both transforms inside the inner integral. The subscript $F$ is simply part of this notation.

The \emph{independent-angle expression} is
\begin{equation}\label{angular:eq:independent}
 \mathsf H^{X,Y}(a,v;b_1,b_2)
 =2^{2(v-a)}\int_0^\infty
 \left(\int_{S^1}|U_X(r,w)|^2\,dw\right)
 \left(\int_{S^1}|V_Y(r,w)|^2\,dw\right)\,d\lambda_v(r).
\end{equation}
The two angular integrals have independent variables. The radial frequency $r$ remains common. Thus ``independent'' describes the two integrations; it is not a probabilistic independence hypothesis.

The function $\psi$ stays outside the amplitudes. In the later conversion of a separated distance phase, it will equal one near every actual cross-direction and vanish near the antipodal directions. Preparation gives such a cutoff, with uniform margins, for the initial pair; the local conversion will state the geometry it requires. For the definition and comparison here, only $0\le\psi\le1$ is needed. The later conversion requires more, since that bound alone would allow $\psi=0$.

Both expressions are nonnegative. Cancellation can still occur in the spatial integrals and, for the shared-angle expression, in the inner angular integral. Thus terms obtained by expanding either expression need not be individually nonnegative.

\begin{lemma}\label{angular:lem:comparison}
With the same measures, amplitudes, depths, radial cutoff, and model masses,
\begin{equation}\label{angular:eq:comparison}
 0\le\mathsf H_F^{X,Y}(a,v;b_1,b_2)
 \le\mathsf H^{X,Y}(a,v;b_1,b_2).
\end{equation}
Moreover both are finite, and
\begin{equation}\label{angular:eq:crude}
 \mathsf H^{X,Y}(a,v;b_1,b_2)
 \le (2\pi)^2\!\left(\int_0^\infty\Psi(s)\,ds\right)
 U_B^4\,2^{2(v-a)}R^{4C_{\rm prof}\sigma}.
\end{equation}
\end{lemma}
\begin{proof}
For each $r$, Cauchy--Schwarz and $0\le\psi\le1$ give
\[
 \left|\int_{S^1}\psi U_X\overline{V_Y}\,dw\right|^2
 \le\left(\int_{S^1}|U_X|^2\,dw\right)
       \left(\int_{S^1}|V_Y|^2\,dw\right).
\]
Integration against $d\lambda_v$ proves~\eqref{angular:eq:comparison}, with $\psi$ still outside the endpoint amplitudes. The bound~\eqref{angular:eq:crude} follows from~\eqref{angular:eq:transform-bound}, the circle's arclength $2\pi$, and the total mass of $\lambda_v$.
\end{proof}

The estimate~\eqref{angular:eq:crude} gives finiteness; frequency decay will require further arguments. Both definitions use exactly the prefactor $2^{2(v-a)}$, arclength measure on the circle, and the radial measure $d\lambda_v$.

\subsection{Removing spatially constant factors}\label{amplitude:sec:deletion}
A cutoff factor whose specified depth is at most $d$ is constant in the spatial variable on a depth-$d$ square. It may still depend on direction. In the independent-angle expression, we can factor out such a product from the spatial integral and bound it inside an integral of a squared modulus.

Let $P,Q$ be depth-$d$ squares for the two class measures, with $d\le v$. At each endpoint, we set
\begin{equation}\label{amplitude:eq:remaining-occurrences}
 \mathcal A_{j,\le d}=\{t\in\mathcal A_j:p_{j,t}\le d\},\qquad
 \mathcal A_{j,>d}=\{t\in\mathcal A_j:p_{j,t}>d\}.
\end{equation}
These subsets retain the original factor indices and prescribed orders. We define $b_{j,\rm rem}$ by keeping only the factors indexed by $\mathcal A_{j,>d}$ in~\eqref{amplitude:eq:form}, with the same spatial multiplier.

\begin{lemma}\label{amplitude:lem:deletion}
For the positive independent-angle expression on $P,Q$,
\begin{equation}\label{amplitude:eq:deletion}
 \mathsf H^{P,Q}(d,v;b_1,b_2)
 \le U_B^4\,
 \mathsf H^{P,Q}(d,v;b_{1,\rm rem},b_{2,\rm rem}).
\end{equation}
No measure, model denominator, radial cutoff, depth, or level changes. Every retained occurrence keeps its original prescribed derivative, and the combined order does not increase. No additional angular or radial multiplier is introduced into a remaining amplitude. In particular, when $B=0$ the scalar factor is one.
\end{lemma}
\begin{proof}
Every factor indexed by $\mathcal A_{1,\le d}$ is independent of $x$ on $P$. Their product is an angular function $F_P(w)$ with $|F_P|\le U_B$, by Lemma~\ref{amplitude:lem:products}. We can therefore write
\begin{equation}\label{amplitude:eq:factorization}
 U_P(r,w)=F_P(w)U_{P,\rm rem}(r,w).
\end{equation}
The denominator remains $M_1(d)$ on both sides. Consequently,
\begin{equation}\label{amplitude:eq:positive-delete}
 \int_{S^1}|U_P(r,w)|^2\,dw
 \le U_B^2\int_{S^1}|U_{P,\rm rem}(r,w)|^2\,dw.
\end{equation}
The same argument applies on $Q$, with denominator $M_2(d)$. Multiplying the two nonnegative inequalities and integrating against $d\lambda_v$ proves~\eqref{amplitude:eq:deletion}. The remaining amplitudes have exactly the stated form because only the specified factors have been removed; all other data and orders stay fixed.
\end{proof}

The case $p_{j,t}=d$ is included because the factor is spatially constant on that depth-$d$ square. A factor specified at a deeper depth does not have this property merely because it is bounded.

The remaining amplitude imposes only the conditions of its remaining factors. Any geometric argument that needs a condition must therefore use it before its factor is removed. The lemma gives an inequality for positive independent-angle quantities; it gives no identity for amplitudes elsewhere in a signed expansion. Its cost $U_B^4$ is a fixed scalar, independent of the number of factors and of $N$.

\subsubsection{Cancellation in the shared-angle expression}
The distinction is already visible on a two-point direction space with counting measure. We take
\begin{equation}\label{amplitude:eq:cancellation-example}
 U=(1,-1),\qquad V=(1,1),\qquad F=(1,0).
\end{equation}
Then
\[
 \sum U\overline V=0,\qquad
 \sum FU\overline V=1,\qquad |F|\le1.
\]
Thus one cannot bound the shared-angle absolute square with the factor present by the corresponding square after removing it. In contrast,
\[
 \sum|FU|^2\le\sum|U|^2,
\]
as in~\eqref{amplitude:eq:positive-delete}. This finite example shows why boundedness alone cannot justify removal in the shared-angle expression.

\subsection{The base interval and the remaining conditions}\label{angular:sec:base}
For the one-profile quantities used in the later iteration, a base interval $\mathcal B=[a_0,t_0]$ supplies the conditions, while $[a,v]$ is the current interval. We have
\begin{equation}\label{angular:eq:base}
 a_0\le a\le v\le t_0\le N,\qquad t_0\le v+\sigma N.
\end{equation}
We use one class $\rho$, its model masses $M$, a common level $i<J_0$, and exactly one occurrence of each member of the set
\begin{equation}\label{angular:eq:base-list}
 \M_{>a}(\mathcal B)
 =\{\mathfrak t\in\M(\mathcal B):p(\mathfrak t)>a\}
\end{equation}
at both endpoints. With the squares and amplitudes understood, we write the independent-angle expression as $\mathsf H_i^{\mathcal B}(a,v)$.

The base interval stays fixed as the current interval changes. Section~\ref{lists:sec:active} proves that the condition set for $[a,v]$ is contained in~\eqref{angular:eq:base-list}. It also bounds the maximum remaining condition length by half the width of $\mathcal B$; the width of $[a,v]$ alone need not give that bound.

If $a\le d\le v$, spatial restriction to depth-$d$ squares first keeps every factor. Applying Lemma~\ref{amplitude:lem:deletion} to the resulting positive independent-angle expressions then leaves exactly
\begin{equation}\label{angular:eq:base-update}
 \{\mathfrak t\in\M_{>a}(\mathcal B):p(\mathfrak t)>d\}
 =\M_{>d}(\mathcal B).
\end{equation}
The base $\mathcal B$, level $i$, and remaining prescribed orders stay fixed, and the combined order can only decrease. The condition $t_0\le v+\sigma N$ is also unchanged.

The set identity~\eqref{angular:eq:base-update} describes which factors remain after the inequality in Lemma~\ref{amplitude:lem:deletion}. It does not justify removal in a shared-angle expression or in a signed part of a four-point expansion. The remaining set is $\M_{>d}(\mathcal B)$; replacing it by $\M(d,v)$ could discard further conditions.

For two different profiles, the endpoint families may differ. The expression~\eqref{angular:eq:independent} retains their separate measures, profiles, factors, and model denominators until the endpoints are separated by the following comparison.

\subsection{Radial Cauchy--Schwarz and duplication of one endpoint}\label{angular:sec:duplication}
Cauchy--Schwarz in the radial variable bounds the two-profile expression by quantities that each use two copies of the same endpoint. For fixed squares and amplitudes, we write
\[
 E_1(r)=\int_{S^1}|U_X(r,w)|^2\,dw,\qquad
 E_2(r)=\int_{S^1}|V_Y(r,w)|^2\,dw.
\]
Both functions are nonnegative and bounded. We define
\begin{equation}\label{angular:eq:duplicates}
 H_1=2^{2(v-a)}\int E_1(r)^2\,d\lambda_v(r),\qquad
 H_2=2^{2(v-a)}\int E_2(r)^2\,d\lambda_v(r).
\end{equation}
Cauchy--Schwarz in the common radial measure gives
\begin{equation}\label{angular:eq:radial-CS}
 \mathsf H^{X,Y}(a,v;b_1,b_2)\le(H_1H_2)^{1/2}.
\end{equation}
The prefactor agrees with~\eqref{angular:eq:independent}, so no further radial normalization is needed.

The quantity $H_j$ uses endpoint $j$ twice, with the same measure, square, model denominator, amplitude, factor family, and level. If the original pair has combined prescribed order at most $B$, either duplicated pair has combined order at most $2B$. The number of factors at each endpoint is unchanged. A later estimate for $H_j$ must therefore hold with order bound $2B$. The comparison itself supplies no decay estimate for $H_j$.

\subsection{Fixed-order suprema and finite maxima}\label{angular:sec:suprema}
Later estimates will be uniform over the bounded spatial multipliers and prescribed derivative orders. We therefore fix the class measures, profiles, factor families, levels, depths, $\Psi$, the geometric cutoff when used, the constants $K_s$, and the combined order bound $B$. Let $\mathcal F_B$ be the family of pairs in~\eqref{amplitude:eq:form}--\eqref{amplitude:eq:combined-order} with these data. The bound~\eqref{angular:eq:crude} makes the corresponding suprema finite; no maximizing amplitude or measurable choice is needed.

For one profile and fixed parent squares $X,Y$, we use the notation
\begin{equation}\label{angular:eq:supremum}
 \mathcal H_{i,B}^{\mathcal B}(d,v\mid X,Y)
 =\max_{\substack{P\subseteq X,\ Q\subseteq Y\\
                   P,Q\text{ occupied at depth }d}}
     \sup_{b_1,b_2}\mathsf H^{P,Q}(d,v;b_1,b_2).
\end{equation}
The supremum uses one factor for each member of $\M_{>d}(\mathcal B)$ at both endpoints, at level $i$ and with combined order at most $B$. The class, model masses, base, and numerical cutoff bounds stay fixed. The maximum is over occupied depth-$d$ descendants of $X,Y$, a finite nonempty family when both parents are occupied. An empty family is assigned value zero.

For a finite child set $\mathcal P$ and any fixed parameter family $\mathcal F$, the elementary identity
\begin{equation}\label{angular:eq:sup-max}
 \sup_{b\in\mathcal F}\max_{P\in\mathcal P}F(P,b)
 =\max_{P\in\mathcal P}\sup_{b\in\mathcal F}F(P,b)
\end{equation}
holds whenever the displayed suprema are finite. Interchanging these two operations introduces no factor $\#\mathcal P$.

For the base-dependent family above, restriction followed by the removal in Lemma~\ref{amplitude:lem:deletion} gives an amplitude with exactly the factors indexed by $\M_{>d}(\mathcal B)$. It therefore belongs to the specified child family. Taking the supremum over that family gives an upper bound, with the same level and order bound. Fixed scalar coefficients in a preceding inequality remain outside the supremum. A uniform norm bound alone does not allow further angular multipliers or an order bound $B$ that grows with $N$.

\subsection{Use in the Fourier estimate}\label{angular:sec:summary}
The circle expansion and its finite inverse will use the amplitude family just defined. Every factor in the given amplitude must remain present, the bounded spatial multiplier must stay independent of angle and radial frequency, and the geometric cutoff $\psi$ must stay outside the amplitudes. Later geometric estimates must check the conditions and depth inequalities they need before any factor is removed from a positive expression.

Section~\ref{main:circle} uses these amplitudes in the conversion from the distance phase to the shared-angle expression. The independent-angle bound then gives the expression estimated in Section~\ref{main:graph}.

\clearpage

\section{Circle expansions and inverse conversion}\label{main:circle}
We use circle integrals to separate the two endpoints of a distance phase. For a fixed direction $w$, the plane wave factors as
\[
 e^{\ii r w\cdot(x-y)}=e^{\ii r w\cdot x}\overline{e^{\ii r w\cdot y}}.
\]
The distance phase $e^{\ii r|x-y|}$ has no such factorization. We express its weighted transform as a sum of the shared-angle integrals from Section~\ref{main:amplitudes}, with a controlled error. Every original cutoff remains in the resulting amplitudes, so their supports continue to impose the required geometric conditions.

We first expand a circle integral near the directions where the angular derivative of its phase vanishes. A finite inverse then expresses the distance phase in terms of differentiated circle integrals and a remainder. The estimates account for amplitudes whose derivatives grow with the scale. We separate the remaining spatial coefficients and check the frequency normalization. Finally, an expansion at both stationary directions gives the estimate for two displacements that will be used in the far contribution.

These calculations use the cutoff and amplitude bounds already proved. The later counting and recursive arguments will supply the decay estimate for the angular expressions.

\subsection{The local conversion and its data}\label{circle:sec:data}
We use the Fourier convention~\eqref{main:fourier}, arclength $dw=d\theta$, the radial cutoff~\eqref{angular:eq:radial-cutoff}, and the amplitude family~\eqref{amplitude:eq:form}. The conversion also requires the following separation of the actual supports.

We fix $\sigma>0$, an integer $N\ge1$, and $R=2^N$. For integers $0\le a\le v\le N$, we put
\begin{equation}\label{circle:eq:scale-data}
 q_{\rm osc}=v-a\ge20\sigma N.
\end{equation}
Here $q_{\rm osc}$ denotes the difference of depths; the reconstruction kernel and the block spacings have separate notation. The range in~\eqref{circle:eq:scale-data} can be nonempty only if $\sigma\le1/20$.

Let $\lambda_1,\lambda_2$ be finite positive measures, with positive normalizing numbers $M_1=M_1(a)$ and $M_2=M_2(a)$. In the applications, $\lambda_1=\rho_1|_X$ and $\lambda_2=\rho_2|_Y$, and $M_j(a)$ is the original class's model mass. The proof uses only
\begin{equation}\label{circle:eq:mass-data}
 \frac{\lambda_j(\R^2)}{M_j}\le C_mR^\mu
 \qquad(j=1,2),\qquad C_m\ge1,\quad\mu\ge0,
\end{equation}
with $C_m,\mu$ fixed before $N$. The proof therefore applies to any finite positive measures satisfying these bounds. If either measure vanishes, the conclusion is immediate.

Each support is contained in a translate of a $2^{-a}$ dilate of a fixed bounded box. Assume that on the actual product support,
\begin{equation}\label{circle:eq:separation}
 d_0\le2^a|x-y|\le d_1,\qquad 0<d_0<d_1<\infty.
\end{equation}
We write $D(x,y)=|x-y|$ and $e(x,y)=(x-y)/|x-y|$ in this section. Suppose all these directions lie in one fixed small arc. We choose a smooth real function $\psi$, with $0\le\psi\le1$, equal to one on a uniform neighborhood of the cross-directions and zero on a uniform neighborhood of their antipodes. The neighborhood widths and the required finite derivative bounds are common to the family; rotations of a fixed cutoff are allowed.

The enclosing boxes may coincide, as they do at the initial scale when both prepared compact supports lie in the root square. Separation and the choice of $\psi$ concern the actual supported pairs. The plateau of $\psi$ is needed for the conversion: the size bound alone would allow $\psi=0$.

At endpoint $j$, we keep the family $\mathcal A_j$ of at most $N^3$ separately indexed cutoff factors from Section~\ref{amplitude:sec:data}. All factors have one common level $i<J_0$ and satisfy~\eqref{amplitude:eq:individual-bounds}. The proof uses their derivative and support bounds; spatial constancy and the particular formulas for the conditions are not needed here. Thus the conversion applies before a finer spatial partition is chosen.

We write the undifferentiated products as
\begin{equation}\label{circle:eq:input-products}
 \omega_{1,x}(w)=\prod_{t\in\mathcal A_1}w_{1,t}(x,w),\qquad
 \omega_{2,y}(w)=\prod_{t\in\mathcal A_2}w_{2,t}(y,w).
\end{equation}
The empty product is one. Let $L$ be the largest condition length in the two families, with $L=0$ when both are empty. We assume
\begin{equation}\label{circle:eq:length-data}
 L\le q_{\rm osc}/2+\sigma N.
\end{equation}
Repeated occurrences are indexed separately, and every factor remains in the conversion.

For the general finite measures just described, use the same transform definitions as~\eqref{angular:eq:transforms}:
\begin{equation}\label{circle:eq:finite-transforms}
 U_X(r,w)=M_1^{-1}\int e^{\ii rw\cdot x}b_{1,x}(w)\,d\lambda_1(x),\qquad
 V_Y(r,w)=M_2^{-1}\int e^{\ii rw\cdot y}b_{2,y}(w)\,d\lambda_2(y).
\end{equation}
Here $X,Y$ label the enclosing boxes. For class restrictions, these are exactly the previous transforms. We keep the definitions of $\mathsf H_F^{X,Y}$ and $\mathsf H^{X,Y}$, including their radial measure and prefactor $2^{2q_{\rm osc}}$. The mass bound~\eqref{circle:eq:mass-data} and Lemma~\ref{amplitude:lem:products} ensure that they are finite.

For bounded complex Borel functions $|f|,|g|\le1$, we define
\begin{equation}\label{circle:eq:nu}
 \nu_{f,g}=\dd_\#\left[
 \frac{(2^aD(x,y))^{-1/2}f(x)\overline{g(y)}
 \omega_{1,x}(e(x,y))\omega_{2,y}(e(x,y))}{M_1M_2}
 \,d\lambda_1(x)d\lambda_2(y)\right].
\end{equation}
The map $\dd(x,y)=|x-y|$ is the distance map from Section~\ref{main:cutoffs-mass}. The measure is finite and complex, and is positive when $f=g=1$. Its normalization uses the given $M_j$, which need not equal the actual masses of the restrictions. We use $T$ below for the desired exponent in the error bound $R^{-T}$.

\Needspace{360pt}
\begin{theorem}\label{circle:thm:conversion}
Under (9.1)--(9.5) and the stated geometric and
cutoff hypotheses, let $T>0$. Set
\begin{equation}\label{circle:eq:orders}
 P=T+2\mu+2,\qquad
 J_{\rm c}=\left\lfloor\frac{P+9}{10\sigma}\right\rfloor+1,
 \qquad B_{\rm inv}=2(J_{\rm c}-1).
\end{equation}
There are common constants $C_T,N_T$ such that, for every $N\ge N_T$,
\begin{equation}\label{circle:eq:target}
 2^{-a}\int_{2^{v-1}}^{2^{v+1}}|\widehat\nu_{f,g}(r)|^2\,dr
 \le C_T\sup_{b_1,b_2}\mathsf H_F^{X,Y}(a,v;b_1,b_2)+C_TR^{-T}.
\end{equation}
The supremum is over exactly~\eqref{amplitude:eq:form}, with the given lists, their level and cutoff functions fixed, and combined prescribed order at most $B_{\rm inv}$. Every original occurrence remains present. The spatial multipliers are bounded by one and independent of angle and radial frequency. No extra angular factor is introduced into an endpoint amplitude.

The conversion uses numerical cutoff bounds through order $2J_{\rm c}+4$. For any subsequently fixed integer $k\ge0$, its output amplitudes satisfy
\begin{equation}\label{circle:eq:output-seminorm}
 \sup_z\|b_{j,z}\|_{C^k(S^1)}\le C_{B_{\rm inv},k}A^k,
 \qquad A=D_*=R^\sigma\max\{1,2^LR^{-50\sigma}\}.
\end{equation}
At order zero one may use $U_{B_{\rm inv}}$. Additional numerical estimates through order $k$ use cutoff derivatives through $B_{\rm inv}+k$; they do not increase the prescribed order.

The constants may depend on $T,\sigma,C_m,\mu$, the fixed boxes, $d_0,d_1$, the geometric margins, $\psi,\Psi$, and the indicated finite numerical cutoff bounds. They are independent of the particular measures, positions, model masses, lists, depths, and spatial factors satisfying those data.
\end{theorem}

In~\eqref{circle:eq:output-seminorm}, $A$ is the angular derivative scale. It is separate from the excess profile $A_n$ and the ball coefficient in Section~\ref{main:localized-section}. The integer $J_{\rm c}$ counts terms in the circle expansion.

The proof gives a representation before taking the squared norm. For $s=r/2^v\in[1/2,2]$,
\begin{equation}\label{circle:eq:strong-representation}
 \widehat\nu_{f,g}(-2^vs)
 =2^{q_{\rm osc}/2}\sum_\gamma c_\gamma(s)\mathcal I_\gamma(s)
   +E(2^vs),
\end{equation}
where $\mathcal I_\gamma(s)$ is the shared-angle integral for one allowed pair, and
\begin{equation}\label{circle:eq:strong-coeff}
 \sum_\gamma\|c_\gamma\|_{C^k([1/2,2])}\le C_{T,k},\qquad
 \sup_{s\in[1/2,2]}|E(2^vs)|\le C_TR^{-P+2\mu}.
\end{equation}
The coefficient sum is bounded independently of $N$ for every fixed $k$. The error estimate is pointwise; the proof does not differentiate that error in the radial variable.
\subsection{The circle expansion}\label{circle:sec:circle}
A stationary direction is one where the derivative of the angular phase vanishes. For $w\cdot e$, these directions are $e$ and $-e$. The function $\psi$ equals one near $e$ and vanishes near $-e$. Near $e$, an exact change of variable makes the phase quadratic.

We begin with a smooth compactly supported function $a$ on a fixed interval. In this calculation, $a$ is a local amplitude rather than a spatial depth. We use the Fourier transform $\widehat a(\xi)=\int e^{-\ii u\xi}a(u)\,du$ and set
\[
 c_{\rm sp}=\sqrt{2\pi}\,e^{-\ii\pi/4}.
\]
For $\Lambda>0$, we have
\begin{equation}\label{circle:eq:gaussian}
 \int_{\R}e^{-\ii\Lambda u^2/2}a(u)\,du
 =\frac{c_{\rm sp}}{2\pi}\Lambda^{-1/2}
 \int_{\R}\widehat a(\xi)e^{\ii\xi^2/(2\Lambda)}\,d\xi.
\end{equation}
To justify the identity, we insert $e^{-\varepsilon u^2}$, with $\varepsilon>0$, and apply Fourier inversion and Fubini. The inner Gaussian integral is
\[
 \sqrt{\frac{\pi}{\varepsilon+\ii\Lambda/2}}
 \exp\!\left(-\frac{\xi^2}{4(\varepsilon+\ii\Lambda/2)}\right).
\]
The square root is chosen continuously from the right half-plane. The exponential has modulus at most one, and the square-root factor is at most $(2\pi/\Lambda)^{1/2}$. Dominated convergence as $\varepsilon\downarrow0$ gives~\eqref{circle:eq:gaussian}.

For an integer $m\ge1$, Taylor's theorem gives, for real $t$,
\[
 \left|e^{\ii t}-\sum_{j=0}^{m-1}\frac{(\ii t)^j}{j!}\right|
 \le \frac{|t|^m}{m!}.
\]
Moreover,
\begin{equation}\label{circle:eq:fourier-moment}
 \int |\xi|^{2m}|\widehat a(\xi)|\,d\xi
 \le C_m\bigl(\|a\|_1+\|a^{(2m+2)}\|_1\bigr).
\end{equation}
Indeed, the bound $|\widehat a|\le\|a\|_1$ applies for $|\xi|\le1$, while $2m+2$ integrations by parts give an integrable bound proportional to $|\xi|^{-2}$ after multiplication by $|\xi|^{2m}$ on $|\xi|>1$. Taylor's theorem and~\eqref{circle:eq:fourier-moment} now yield
\begin{equation}\label{circle:eq:gaussian-expansion}
 \int e^{-\ii\Lambda u^2/2}a(u)\,du
 =c_{\rm sp}\Lambda^{-1/2}
 \sum_{j=0}^{m-1}\frac{\Lambda^{-j}}{(2\ii)^j j!}a^{(2j)}(0)
 +\mathcal E_m,
\end{equation}
where
\begin{equation}\label{circle:eq:gaussian-error}
 |\mathcal E_m|\le C_m\Lambda^{-m-1/2}
 \bigl(\|a\|_1+\|a^{(2m+2)}\|_1\bigr).
\end{equation}
The sign in~\eqref{circle:eq:gaussian-expansion} follows from
$(2\pi)^{-1}\int\xi^{2j}\widehat a(\xi)\,d\xi=(-1)^j a^{(2j)}(0)$.
Only this finite expansion and its remainder will be used.

For a unit vector $e=(\cos\theta_e,\sin\theta_e)$, define
\[
 I_\Lambda(b;e)=\int_{S^1}e^{\ii\Lambda w\cdot e}\psi(w)b(w)\,dw.
\]
We choose a fixed smooth angular cutoff $\vartheta_{\rm loc}$ equal to one near zero and supported in a small interval. The common plateau of $\psi$ ensures that $\vartheta_{\rm loc}(\theta-\theta_e)$ is supported where $\psi=1$ for every permitted $e$. Near $e$, we set
\begin{equation}\label{circle:eq:coordinate}
 \theta-\theta_e=\varphi(u)=2\arcsin(u/2).
\end{equation}
Then
\[
 \cos(\theta-\theta_e)=1-u^2/2,\qquad
 \varphi'(u)=(1-u^2/4)^{-1/2}.
\]
The localized integral is therefore
\[
 e^{\ii\Lambda}\int e^{-\ii\Lambda u^2/2}a_e(u)\,du,
 \qquad
 a_e(u)=\vartheta_{\rm loc}(\varphi(u))\varphi'(u)b(\theta_e+\varphi(u)).
\]
The support of $a_e$ lies in a fixed compact subinterval of $(-2,2)$.
If $A\ge1$ and
\begin{equation}\label{circle:eq:generic-derivatives}
 \|\partial_\theta^k b\|_\infty\le Q_kA^k
\end{equation}
through a specified finite order, the chain rule gives
$\|a_e^{(k)}\|_\infty\le C_kA^k$ through that order. The constants
use only the fixed coordinate, cutoff, and numerical bounds $Q_k$.
They are uniform in $e$.

\subsubsection{The constant-coefficient operators}
We next express $a_e^{(2j)}(0)$ by a differential operator in the global angular coordinate. Constant coefficients let us expand its action into amplitudes of the form~\eqref{amplitude:eq:form}, with scalar coefficients outside and no additional angular multiplier.

For a formal parameter $z$, put
\[
 h_z(u)=\varphi'(u)e^{z\varphi(u)}.
\]
Direct differentiation gives
\begin{equation}\label{circle:eq:h-ode}
 (1-u^2/4)h_z''-(3u/4)h_z'-(z^2+1/4)h_z=0.
\end{equation}
For completeness, if $t=1-u^2/4$, then $\varphi'=t^{-1/2}$ and
\[
 \frac{h_z'}{h_z}=\frac{u}{4t}+zt^{-1/2}.
\]
Differentiating and adding the square gives $h_z''/h_z$; substitution verifies~\eqref{circle:eq:h-ode}.

Taking $n$ derivatives of~\eqref{circle:eq:h-ode} at zero gives
\begin{equation}\label{circle:eq:h-recurrence}
 h_z^{(n+2)}(0)
 =\left(z^2+\frac{(n+1)^2}{4}\right)h_z^{(n)}(0).
\end{equation}
Indeed the two differentiated coefficient terms contribute
$-n(n-1)h_z^{(n)}(0)/4$ and $-3nh_z^{(n)}(0)/4$.
Since $h_z(0)=1$, iteration along the even integers yields
\[
 h_z^{(2j)}(0)
 =\prod_{k=0}^{j-1}\bigl(z^2+(k+1/2)^2\bigr).
\]
By the chain rule, the derivative at zero of $\varphi'(u)b(\theta_e+\varphi(u))$ is a constant-coefficient differential operator applied to $b$ at $\theta_e$. Testing it on $e^{z\theta}$ identifies its polynomial by the preceding identity. Since the local cutoff is one near zero, it contributes no derivatives there. Thus
\begin{equation}\label{circle:eq:Lj}
 L_0=c_{\rm sp}\operatorname{Id},\qquad
 L_j=\frac{c_{\rm sp}}{(2\ii)^j j!}
 \prod_{k=0}^{j-1}\left(\partial_\theta^2+(k+1/2)^2\right)
 \quad(j\ge1).
\end{equation}
For example,
\begin{equation}\label{circle:eq:first-L}
 \frac{L_1}{c_{\rm sp}}=-\frac{\ii}{2}\left(\partial_\theta^2+\frac14\right),
 \qquad
 \frac{L_2}{c_{\rm sp}}=-\frac18
 \left(\partial_\theta^2+\frac14\right)
 \left(\partial_\theta^2+\frac94\right).
\end{equation}
The operators use the fixed global angular coordinate; only the point of evaluation depends on $e$.

For a direct check of the signs, take $b(\theta)=e^{\ii n\theta}$, with $n\in\Z$. Set
\begin{equation}\label{circle:eq:mode-check}
 a_j(n)=\frac{\prod_{k=1}^j(4n^2-(2k-1)^2)}{j!8^j},\quad a_0(n)=1,
 \qquad (L_jb)(e)=c_{\rm sp}\,\ii^j a_j(n)e^{\ii n\theta_e}.
\end{equation}
This follows from $\partial_\theta^2b=-n^2b$. At $n=0$, the first correction is $-\ii/(8\Lambda)$ times the leading term, in agreement with~\eqref{circle:eq:first-L}.

It remains to estimate the part with angular factor $\psi(w)-\vartheta_{\rm loc}(\theta-\theta_e)$. This factor vanishes near both stationary directions, so the phase derivative is uniformly bounded away from zero on its support. We integrate by parts using
\[
 e^{\ii\Lambda\cos(\theta-\theta_e)}
 =\frac{1}{-\ii\Lambda\sin(\theta-\theta_e)}
 \partial_\theta e^{\ii\Lambda\cos(\theta-\theta_e)}.
\]
A smooth extension of the reciprocal coefficient has common derivative bounds on this support. After $m+1$ integrations by parts, the integral is bounded by $C_m\Lambda^{-m-1}A^{m+1}$. Periodicity removes the boundary terms. Together with~\eqref{circle:eq:gaussian-expansion}, this proves the following uniform estimate.
\begin{lemma}\label{circle:lem:direct}
Let $m\ge1$ be an integer and $\Lambda,A\ge1$. Assume~\eqref{circle:eq:generic-derivatives} through
order $2m+4$. Then
\begin{equation}\label{circle:eq:direct}
 I_\Lambda(b;e)
 =\Lambda^{-1/2}\left[
 e^{\ii\Lambda}\sum_{j=0}^{m-1}\Lambda^{-j}(L_jb)(e)
 +\mathcal R_m\right],\qquad
 |\mathcal R_m|\le C_m\Lambda^{-m}A^{2m+4}.
\end{equation}
The constant depends on the fixed geometric margins and the numerical
bounds $Q_k$, but not on $A,\Lambda,e$, or the particular amplitude.
\end{lemma}
\begin{proof}
The local remainder is bounded by
$C_m\Lambda^{-m-1/2}A^{2m+2}$ using the fixed support of $a_e$.
After extracting $\Lambda^{-1/2}$ it is smaller than the displayed
bound. The nonstationary remainder, after the same extraction, is
$C_m\Lambda^{-m-1/2}A^{m+1}$, also smaller. We retain the convenient
order $2m+4$ uniformly in both estimates. The coefficient calculation
is~\eqref{circle:eq:Lj}.
\end{proof}

\subsection{A finite inverse with an explicit error}\label{circle:sec:inverse}
The leading term of the direct expansion is $c_{\rm sp}\Lambda^{-1/2}e^{\ii\Lambda}b(e)$. We cancel successive correction terms by adding finitely many differentiated circle integrals. The resulting inverse formula has an explicit remainder.

We define
\begin{equation}\label{circle:eq:Bj}
 B_0=L_0^{-1},\qquad
 B_j=-L_0^{-1}\sum_{h=1}^jL_hB_{j-h}\quad(j\ge1).
\end{equation}
Each $B_j$ is a constant-coefficient differential operator of order at
most $2j$. The recursion gives the exact identities
\begin{equation}\label{circle:eq:inverse-identity}
 \sum_{h+j=n}L_hB_j=
 \begin{cases}\operatorname{Id},&n=0,\\0,&1\le n<J_{\rm c}.\end{cases}
\end{equation}
For reference, writing $t=\partial_\theta^2$, the first three are
\[
 B_0=c_{\rm sp}^{-1}\operatorname{Id},\qquad
 B_1=c_{\rm sp}^{-1}\frac{\ii}{2}(t+1/4),\qquad
 B_2=c_{\rm sp}^{-1}\frac18(t+1/4)(7/4-t).
\]
These first operators also check the signs for our Fourier convention.

\begin{lemma}\label{circle:lem:inverse}
Let $J_{\rm c}\ge1$ be an integer. Assume~\eqref{circle:eq:generic-derivatives} through order $2J_{\rm c}+4$.
For $\Lambda,A\ge1$,
\begin{equation}\label{circle:eq:inverse}
 e^{\ii\Lambda}b(e)
 =\Lambda^{1/2}\sum_{j=0}^{J_{\rm c}-1}\Lambda^{-j}I_\Lambda(B_jb;e)
 +O_{J_{\rm c}}\!\left(A^4(A^2/\Lambda)^{J_{\rm c}}\right).
\end{equation}
The error is uniform with the same fixed data as in
Lemma~\ref{circle:lem:direct}.
\end{lemma}
\begin{proof}
At inverse order $j$, apply Lemma~\ref{circle:lem:direct} to $B_jb$ through
$m=J_{\rm c}-j$ direct terms. The function $A^{-2j}B_jb$ has derivative
bounds of the form~\eqref{circle:eq:generic-derivatives}, since $A\ge1$ and
$B_j$ has order at most $2j$. Thus the normalized remainder for
$B_jb$ is bounded by
\[
 C_{J_{\rm c}}\Lambda^{-(J_{\rm c}-j)}A^{2(J_{\rm c}-j)+4+2j}
 =C_{J_{\rm c}}\Lambda^{-(J_{\rm c}-j)}A^{2J_{\rm c}+4}.
\]
Multiplication by the factor $\Lambda^{-j}$ in the inverse sum makes
this $C_{J_{\rm c}}\Lambda^{-J_{\rm c}}A^{2J_{\rm c}+4}$. There are only $J_{\rm c}$ such terms.
The principal terms of total degree $n<J_{\rm c}$ cancel according to
\eqref{circle:eq:inverse-identity}, leaving $e^{\ii\Lambda}b(e)$.
The largest derivative order used in the calculation is
\[
 2j+2(J_{\rm c}-j)+4=2J_{\rm c}+4.
\]
This proves the finite remainder and the asserted derivative order.
\end{proof}

The useful ratio is $A^2/\Lambda$. We will bound it by a fixed negative power of $R$ and choose $J_{\rm c}$ large enough to control the additional factor $A^4$. This accounts for the growth of the amplitude derivatives with $N$.

\subsection{The derivative scale and the oscillation}\label{circle:sec:scales}
We now compare the derivative scale of the cutoff products with the size of the oscillation. We index all factors in the two endpoint families together and put
\[
 n=\max\{1,\#\mathcal A_1+\#\mathcal A_2\},\qquad
 D_{\max}=\max\{1,2^LR^{-50\sigma}\},\qquad A=R^\sigma D_{\max}.
\]
For $b_{x,y}=\omega_{1,x}\omega_{2,y}$, Lemmas~\ref{amplitude:lem:products} and~\ref{amplitude:lem:copies} give
\begin{equation}\label{circle:eq:product-bound}
 \|b_{x,y}\|_\infty\le1,\qquad
 \|\partial_\theta^k b_{x,y}\|_\infty
 \le C_k(nD_{\max})^k\le C_kA^k
\end{equation}
once $2N^3\le R^\sigma$. The combined family has at most $2N^3$ factors, and the numerical constants use only $K_1,\ldots,K_k$. The same product bounds give~\eqref{circle:eq:output-seminorm} after the prescribed order is fixed.

For $r=2^v s$, $s\in[1/2,2]$, set
\[
 \Lambda=rD(x,y).
\]
The geometry and the scale assumptions give
\begin{align}\label{circle:eq:scale-estimates}
 \Lambda&\ge(d_0/2)2^{q_{\rm osc}},\notag\\
 A&\le\max\{R^\sigma,2^{q_{\rm osc}/2}R^{-48\sigma}\},\notag\\
 2^{-q_{\rm osc}}A^2&\le\max\{R^{2\sigma}2^{-q_{\rm osc}},R^{-96\sigma}\}
             \le R^{-18\sigma},\\
 A^2/\Lambda&\le(2/d_0)R^{-18\sigma}.\notag
\end{align}
The second line uses $L\le q_{\rm osc}/2+\sigma N$, and the third uses $2^{q_{\rm osc}}\ge R^{20\sigma}$. We choose a common starting scale such that
\begin{equation}\label{circle:eq:thresholds}
 2N^3\le R^\sigma,\qquad
 R^{8\sigma}\ge\max\{1,2/d_0\}.
\end{equation}
Then
\begin{equation}\label{circle:eq:small-ratio}
 \Lambda\ge R^{10\sigma},\qquad
 A^2/\Lambda\le R^{-10\sigma},\qquad A\le R^2.
\end{equation}
The first inequality follows from the stronger bound $\Lambda\ge R^{12\sigma}$. The last follows from $\sigma\le1/20$, $q_{\rm osc}\le N$, and~\eqref{circle:eq:scale-estimates}.

By~\eqref{circle:eq:orders},
\[
 10\sigma J_{\rm c}>P+9.
\]
Consequently the inverse remainder in~\eqref{circle:eq:inverse} satisfies
\begin{equation}\label{circle:eq:point-error}
 A^4(A^2/\Lambda)^{J_{\rm c}}
 \le R^{8-10\sigma J_{\rm c}}\le R^{-P-1}\le R^{-P}.
\end{equation}
The first $J_{\rm c}$ inverse terms have combined prescribed derivative order at most $2(J_{\rm c}-1)=B_{\rm inv}$. Bounding the remainders uses cutoff derivatives through $2J_{\rm c}+4$, without increasing the prescribed orders in those terms. The chosen $J_{\rm c}$ is independent of $N$ and also allows for the normalized measure masses and frequency integration, as checked in~\eqref{circle:eq:integrated-error}.

\subsection{Separating the spatial coefficient}\label{circle:sec:spatial}
The inverse introduces distance powers that depend on both endpoints. We separate them by a smooth Fourier series on a fixed product box. This changes only the bounded spatial multipliers in the amplitudes.

\begin{lemma}\label{circle:lem:spatial}
Under~\eqref{circle:eq:separation} and the fixed-box assumptions, for every
fixed integer $j\ge0$ there are functions $F_{j,\ell}(x)$ and
$G_{j,\ell}(y)$ of modulus at most one and scalar coefficients
$\gamma_{j,\ell}$ such that, on the actual product support,
\begin{equation}\label{circle:eq:spatial-series}
 (2^aD(x,y))^{-j}
 =\sum_\ell\gamma_{j,\ell}F_{j,\ell}(x)
       \overline{G_{j,\ell}(y)},\qquad
 \sum_\ell|\gamma_{j,\ell}|\le C_j.
\end{equation}
The series converges absolutely and uniformly. The bounds are common
to the family, even when the enclosing boxes overlap. The functions
$F_{j,\ell},G_{j,\ell}$ have no angular or radial dependence.
\end{lemma}
\begin{proof}
For $j=0$, take a single constant term. For $j>0$, use the coordinates
\[
 x=c_X+2^{-a}u,\qquad y=c_Y+2^{-a}z,\qquad
 h=2^a(c_X-c_Y).
\]
The variables $u,z$ range over fixed bounded boxes. Choose smooth
compactly supported functions $\chi_1(u),\chi_2(z)$ equal to one on
those boxes. Also choose a smooth function $\chi_D$ on $[0,\infty)$,
equal to one on $[d_0,d_1]$ and vanishing near zero and outside a fixed
bounded interval. On the whole product space define
\[
 F_h(u,z)=\chi_1(u)\chi_2(z)
          \chi_D(|h+u-z|)|h+u-z|^{-j},
\]
where the final product is set to zero near $h+u-z=0$. Since $\chi_D$ vanishes there, $F_h$ is smooth. It has support in a fixed product box and derivative bounds of every fixed order independent of $h$, because translation preserves the bounds for the smooth function of $h+u-z$.

On the actual product support, $\chi_1=\chi_2=\chi_D=1$, so $F_h$
is the desired coefficient. On a sufficiently large fixed four-torus,
it extends periodically as a smooth function vanishing near the
boundary of a fundamental box. Applying $(1-\Delta_{u,z})^3$ to its
Fourier coefficient integral gives
\[
 |\gamma_{j,k,l}|\le C_j(1+|k|^2+|l|^2)^{-3},
 \qquad k,l\in\Z^2.
\]
The bound is summable on $\Z^4$, so the Fourier series converges absolutely and uniformly to $F_h$. Each exponential separates into factors of $u$ and $z$ of modulus one. Taking the negative exponent for $G$ gives the form $F\overline G$; pulling these factors back to $x,y$ proves~\eqref{circle:eq:spatial-series}.
\end{proof}

When the enclosing boxes coincide, $h=0$ is allowed. The smooth extension vanishes near $u=z$ and agrees with the required distance power on the separated actual supports.

\subsection{Proof of the weighted conversion}\label{circle:sec:conversion}
We now combine the expansions. The derivative calculation must retain every cutoff factor, and the coefficient sum must include the frequency powers that compensate for the number of differentiated terms.

For an actual pair $(x,y)$, we apply Lemma~\ref{circle:lem:inverse} to
$b_{x,y}=\omega_{1,x}\omega_{2,y}$, with $\Lambda=rD(x,y)$. The scale calculation gives
a uniform pointwise error $O(R^{-P})$. Multiplication by the distance
weight uses the exact identity
\begin{equation}\label{circle:eq:weight}
 (2^aD)^{-1/2}(rD)^{1/2-j}
 =2^{q_{\rm osc}/2}s^{1/2-j}2^{-j q_{\rm osc}}(2^aD)^{-j}.
\end{equation}
This cancels the leading distance power. The factor $2^{q_{\rm osc}/2}$ will be squared in the Fourier norm. We write
\[
 B_j=\sum_{d=0}^{2j}\beta_{j,d}\partial_\theta^d.
\]
Indexing all factors in $\omega_{1,x}\omega_{2,y}$ by $t$, the product rule gives
\begin{equation}\label{circle:eq:leibniz}
 \partial_\theta^d\prod_tw_t
 =\sum_{|\alpha|=d}\frac{d!}{\prod_t\alpha_t!}
       \left(\prod_tD_t^{\alpha_t}\right)
       \prod_t\left(D_t^{-\alpha_t}
                      \partial_\theta^{\alpha_t}w_t\right).
\end{equation}
Every factor appears once in each term, including factors with derivative order zero. For an empty family, positive derivatives vanish and the zeroth derivative is one. We split the last product between the endpoints and apply Lemma~\ref{circle:lem:spatial}. Its spatial factors give the new multipliers $fF_{j,\ell}$ and $gG_{j,\ell}$, still bounded by one and independent of angle and radial frequency.

The normalized cutoff derivatives are real, so the second endpoint product has the required form inside $\overline{V_Y}$. The scalar coefficients $\beta_{j,d}$ and $\gamma_{j,\ell}$ remain outside the integral, and $\psi$ remains outside the endpoint amplitudes.

\subsubsection{The coefficient sum}
We estimate the coefficient sum with the derivative scales included. We put
\[
 S=\max\{1,\sum_tD_t\}\le nD_{\max}\le A.
\]
For a nonempty list, the multinomial theorem gives
\begin{equation}\label{circle:eq:multinomial}
 \sum_{|\alpha|=d}\frac{d!}{\prod_t\alpha_t!}
               \prod_tD_t^{\alpha_t}
 =\left(\sum_tD_t\right)^d.
\end{equation}
For an empty list, the same upper bound by $S^d$ remains valid.
Since $d\le2j$ and $S\ge1$, the total absolute scalar coefficient
at inverse order $j$, including the spatial Fourier series, is at most
\begin{equation}\label{circle:eq:coefficient-bound}
 C_j2^{-j q_{\rm osc}}S^{2j}
       \|s^{1/2-j}\|_{C^k([1/2,2])}
 \le C_{j,k}(2^{-q_{\rm osc}}A^2)^j\le C_{j,k}.
\end{equation}
The norm of $s^{1/2-j}$ is finite for fixed $j,k$, and~\eqref{circle:eq:scale-estimates} gives the last inequality. For $j\ge1$ the bound includes the gain $R^{-18\sigma j}$. Summing over the fixed range $0\le j<J_{\rm c}$ gives
\[
 \sum_\gamma\|c_\gamma\|_{C^k([1/2,2])}\le C_{T,k}.
\]
The frequency powers have therefore bounded the full coefficient sum independently of $N$ and of the number of factors.

For example, a second derivative can give on the order of $n^2$ terms, but the first inverse coefficient also contains $2^{-q_{\rm osc}}$. The total is bounded by $2^{-q_{\rm osc}}(nD_{\max})^2$, which is small by the scale separation. The higher inverse orders use the same calculation.

We multiply the pointwise expansion by $f(x)\overline{g(y)}$ and integrate against $d\lambda_1(x)d\lambda_2(y)/(M_1M_2)$. The phase separates as
\[
 e^{\ii r w\cdot(x-y)}
 =e^{\ii r w\cdot x}\overline{e^{\ii r w\cdot y}}.
\]
Thus the main terms are precisely the integrals
$\mathcal I_\gamma(s)=\int\psi U_{X,\gamma}\overline{V_{Y,\gamma}}\,dw$
in~\eqref{circle:eq:strong-representation}. Their endpoint amplitudes have
combined prescribed order at most $2(J_{\rm c}-1)$.

The pointwise inverse error is bounded by $C_TR^{-P}$. The distance
weight is at most $d_0^{-1/2}$, and the total normalized product mass
is at most $C_m^2R^{2\mu}$. Hence
\begin{equation}\label{circle:eq:measure-error}
 |E(2^v s)|\le C_TR^{-P+2\mu}.
\end{equation}
The spatial Fourier series makes the coefficient family countable. Its absolute coefficient bound, the uniform bound for the amplitudes, and the finite normalized masses justify interchanging the sums and integrals.

Lemma~\ref{amplitude:lem:support} places every differentiated factor in the support of its original cutoff. Since~\eqref{circle:eq:leibniz} retains every factor, a nonzero resulting amplitude imposes all the original level-$(i+1)$ conditions, whether or not the good sets are closed. Further angular derivatives preserve these supports and satisfy Lemma~\ref{amplitude:lem:products} with $B=B_{\rm inv}$. This proves~\eqref{circle:eq:output-seminorm} for the same amplitude and prescribed orders.

The definition of $\mathsf H_F$ and the bound $\Psi\ge1$ on $[1/2,2]$
give
\[
 \|\mathcal I_\gamma\|_{L^2([1/2,2])}
 \le 2^{-q_{\rm osc}}\bigl(\mathsf H_{F,\gamma}^{X,Y}(a,v)\bigr)^{1/2}.
\]
Minkowski's inequality and absolute convergence yield
\begin{equation}\label{circle:eq:minkowski}
 \left\|\sum_\gamma c_\gamma\mathcal I_\gamma\right\|_2
 \le C_T2^{-q_{\rm osc}}\sup_\gamma\bigl(\mathsf H_{F,\gamma}^{X,Y}(a,v)\bigr)^{1/2}.
\end{equation}
The frequency change of variable and the squared leading factor in
\eqref{circle:eq:strong-representation} give
\begin{equation}\label{circle:eq:normalization}
 2^{-a}\,2^v\,(2^{q_{\rm osc}/2})^2=2^{2q_{\rm osc}}.
\end{equation}
This is exactly the factor already present in~\eqref{angular:eq:shared}.

For the error, use~\eqref{circle:eq:measure-error}, $q_{\rm osc}\le N$, and the
choice of $P$:
\begin{align}\label{circle:eq:integrated-error}
 2^{-a}\int_{2^{v-1}}^{2^{v+1}}|E(r)|^2\,dr
 &\le C_T2^{q_{\rm osc}}R^{-2P+4\mu}\notag\\
 &\le C_TR^{1-2P+4\mu}
   =C_TR^{-2T-3}\le C_TR^{-T}.
\end{align}
Combining the main term and the error by $|u+v|^2\le2|u|^2+2|v|^2$ proves the desired bound on $\widehat\nu_{f,g}(-r)$.

To obtain the bound at positive frequencies, we account for the possible complex values of $f,g$. The real cutoffs and distance weight give
\[
 \overline{\nu_{f,g}}=\nu_{\overline f,\overline g},\qquad
 \widehat\nu_{f,g}(r)
 =\overline{\widehat\nu_{\overline f,\overline g}(-r)}.
\]
The conjugated factors still have modulus at most one. Applying the negative-frequency estimate to them proves~\eqref{circle:eq:target}, with the same amplitude supremum. When $f=g=1$, this reduces to conjugate symmetry for the positive measure. Theorem~\ref{circle:thm:conversion} and~\eqref{circle:eq:strong-representation} are proved.

The expansion order $J_{\rm c}$ and the finite derivative orders are fixed before we choose a common starting scale for the cutoff bounds,~\eqref{circle:eq:thresholds}, and~\eqref{circle:eq:small-ratio}. All these choices depend only on the stated common data. Further angular estimates through a fixed order $k$ use the cutoff constants through $B_{\rm inv}+k$ in the product bounds. They change the multiplicative constants while leaving the prescribed orders and the powers of $R$ unchanged.

\subsection{The expansion at both stationary directions}\label{circle:sec:far}
For the later far contribution, we need the direct expansion at both stationary directions. We establish it here, then integrate the explicit coefficients for two displacements in their common radial variable.

For a general complex smooth function $b$ on $S^1$, define
\[
 L_{j,+}=L_j,\qquad
 L_{j,-}b=\overline{L_j(\overline b)}.
\]
The second formula conjugates the operator coefficients. Thus both $L_{j,+}$ and $L_{j,-}$ are complex-linear, constant-coefficient polynomials in $\partial_\theta^2$.

\begin{corollary}\label{circle:cor:full-circle}
Let $J_{\rm c}\ge1$ be an integer and let $\Lambda,A\ge1$. Assume the derivative bounds~\eqref{circle:eq:generic-derivatives} through order $2J_{\rm c}+4$. Uniformly in all unit vectors $e$,
\begin{equation}\label{circle:eq:full-circle}
 \begin{split}
 \int_{S^1}e^{\ii\Lambda w\cdot e}b(w)\,dw
 &=\Lambda^{-1/2}\sum_{\varepsilon\in\{+1,-1\}}
 e^{\ii\varepsilon\Lambda}
 \sum_{j=0}^{J_{\rm c}-1}\Lambda^{-j}(L_{j,\varepsilon}b)(\varepsilon e)
 +\mathcal E_{J_{\rm c}},\\
 |\mathcal E_{J_{\rm c}}|&\le C_{J_{\rm c}}\Lambda^{-J_{\rm c}-1/2}A^{2J_{\rm c}+4}.
 \end{split}
\end{equation}
\end{corollary}
\begin{proof}
Use two smooth cutoffs supported near $e$ and $-e$, respectively,
and equal to one on smaller neighborhoods. They are rotations of
fixed cutoffs. Near $e$, Section~\ref{circle:sec:circle} applies. Near $-e$,
write $\theta=\theta_e+\pi+\varphi(u)$. The phase is
$-1+u^2/2$. The quadratic Gaussian calculation is the conjugate of
\eqref{circle:eq:gaussian}, yielding exactly $L_{j,-}$ acting on $b$ at
$-e$. On the remaining angular support, the phase derivative is
bounded away from zero. The same integration-by-parts estimate gives
the stated error. Rotation makes all constants uniform in $e$.
\end{proof}

Let $s$ range in a fixed compact interval $[s_0,s_1]\subset(0,\infty)$,
and write $\Lambda=st$ with $t=2^v d>0$. The amplitude $b$ and
its evaluation direction are independent of $s$. For every fixed
$k,j$,
\begin{equation}\label{circle:eq:radial-coefficients}
 \left\|s\longmapsto(st)^{-j}(L_{j,\varepsilon}b)(\varepsilon e)
      \right\|_{C^k([s_0,s_1])}
 \le C_{j,k}t^{-j}A^{2j}.
\end{equation}
Only the power of $s$ is differentiated. If $A^2/t\le1$, the
right side is uniformly bounded. Including $(st)^{-1/2}$ gives a
bound $C_{J_{\rm c},k}t^{-1/2}$ for the full nonoscillatory main coefficient
summed over $j<J_{\rm c}$.

These radial derivative estimates apply to the explicit main coefficients. For the error in~\eqref{circle:eq:full-circle}, we use only its pointwise bound when integrating.

\subsection{Two displacements with a common radial frequency}\label{circle:sec:radial-kernel}
We now consider two displacements and their angular amplitudes. Let $b_1,b_2$ satisfy
\eqref{circle:eq:generic-derivatives} with common constants and derivative
scale $A$. They can be products of endpoint amplitudes and their
complex conjugates. Let $\Delta_1,\Delta_2\in\R^2$, put
$d_j=|\Delta_j|$, and write $d_+=\max(d_1,d_2)$ and
$d_-=\min(d_1,d_2)$. Consider
\begin{equation}\label{circle:eq:radial-integral}
 \mathcal K=\int\Psi(s)
 \left(\int_{S^1}e^{\ii2^v s w\cdot\Delta_1}b_1(w)\,dw\right)
 \left(\int_{S^1}e^{\ii2^v s w\cdot\Delta_2}b_2(w)\,dw\right)ds.
\end{equation}
Assume, for fixed positive $c_1,c_2$,
\begin{equation}\label{circle:eq:far-hypotheses}
 2^v d_+\ge R^{c_1\sigma},\qquad
 A^2/(2^v d_+)\le R^{-c_2\sigma},\qquad 1\le A\le R^2.
\end{equation}
Fixed multiplicative constants can be absorbed by slightly reducing $c_1,c_2$ and increasing the common starting scale.

\begin{proposition}\label{circle:prop:radial-kernel}
For every $Q>0$, one can choose a finite expansion
order, independent of $N$, such that the following
conclusions hold. If $d_+>2d_-$, then $\mathcal K=O(R^{-Q})$.
If $d_+\le2d_-$, then $\mathcal K$ is a finite sum of main terms
whose absolute values sum to at most
\begin{equation}\label{circle:eq:kernel-bound}
 C_Q(2^v d_1)^{-1/2}(2^v d_2)^{-1/2}
      (1+2^v|d_1-d_2|)^{-2},
\end{equation}
plus $O(R^{-Q})$. Every main coefficient is obtained by differentiating
the original amplitudes angularly and evaluating them at their
stationary directions. Thus all cutoff support conditions present
in those amplitudes remain in the coefficients.
\end{proposition}
\begin{proof}
By~\eqref{circle:eq:far-hypotheses}, the direct error for the larger length is bounded by
$CA^4(A^2/(2^v d_+))^{J_{\rm c}}\le CR^{8-c_2\sigma J_{\rm c}}$.
We choose a finite $J_{\rm c}$ to make this smaller than the required negative power of $R$. The same bound applies to the smaller length when the lengths are comparable. These estimates are uniform on the fixed support of $\Psi$, after absorbing its fixed endpoints into the starting scale. An exact angular integral is bounded by $2\pi Q_0$, and the main angular sums are bounded by~\eqref{circle:eq:radial-coefficients}. Thus every product containing a remainder is controlled by its pointwise bound.

If $d_+>2d_-$, expand only the larger-length integral. Keep the
other exponential in the phase. For each of its directions $w$,
that phase, as a function of $s$, has derivative
\[
 2^v(\varepsilon d_++w\cdot\Delta_-),
 \qquad
 |2^v(\varepsilon d_++w\cdot\Delta_-)|
 \ge 2^{v-1}d_+.
\]
Here $\Delta_-$ is the smaller displacement. The nonoscillatory
coefficients have bounded derivatives of every fixed order in $s$,
by~\eqref{circle:eq:radial-coefficients}. The remaining angular amplitude
is independent of $s$. Repeated integration by parts in $s$, followed
by its bounded angular integral, gives $O(R^{-Q})$. The compact
support of $\Psi$ removes boundary terms. This also covers $d_-=0$;
no direction for the zero displacement is needed.

If the lengths are comparable, expand both integrals. A main term
has phase
\[
 s2^v(\varepsilon_1d_1+\varepsilon_2d_2).
\]
Its coefficient and its first two $s$ derivatives are bounded by a
constant times $(2^vd_1)^{-1/2}(2^vd_2)^{-1/2}$. The integral is
therefore bounded by that factor times
$C(1+2^v|\varepsilon_1d_1+\varepsilon_2d_2|)^{-2}$, using the
trivial bound when the phase derivative is small and two integrations
by parts otherwise. For opposite signs the absolute length is
$|d_1-d_2|$; for equal signs it is $d_1+d_2\ge|d_1-d_2|$.
The finite coefficient sum gives~\eqref{circle:eq:kernel-bound}.

Finally, the product rule expands $L_{j,\varepsilon}$ on the original cutoff products. Every factor remains, and its angular derivatives have support inside its original support. This proves the support assertion for the main coefficients. The remainders are controlled separately by their absolute bounds.
\end{proof}

\subsection{Application to the far contribution}\label{circle:sec:far-scales}
We check the scales at which the two-displacement estimate will be used. For an integer $L\ge0$, we define
\begin{equation}\label{circle:eq:far-scales}
 L_0=\max\{L,\lceil10\sigma N\rceil\},\qquad h=v-2L_0>a.
\end{equation}
Suppose the larger displacement satisfies $d_+>(30-\sqrt2)2^{-h}$. Then
\[
 2^vd_+>(30-\sqrt2)2^{2L_0},\qquad 2^{2L_0}\ge R^{20\sigma}.
\]
With the actual derivative scale $A=D_*=R^\sigma\max\{1,2^LR^{-50\sigma}\}$,
\begin{equation}\label{circle:eq:far-ratios}
 \frac{A^2}{2^vd_+}\lesssim
 \begin{cases}
 R^{-18\sigma},&2^LR^{-50\sigma}\le1,\\
 R^{-98\sigma},&2^LR^{-50\sigma}>1.
 \end{cases}
\end{equation}
In the first case, $A^2=R^{2\sigma}$ and the oscillation is at least a fixed multiple of $R^{20\sigma}$. In the second, $A^2=2^{2L}R^{-98\sigma}$ and $L\le L_0$. The fixed constants can then be absorbed into a common starting scale.

For products such as $b_{1,x}\overline{b_{1,x'}}$ and $b_{2,y}\overline{b_{2,y'}}$, Lemma~\ref{amplitude:lem:copies} supplies the derivative bounds with this $A$. The four-copy bound under~\eqref{amplitude:eq:common-threshold} covers every product used here. The smoothness estimate counts the total derivative order of the copies together; each original endpoint amplitude and all its factors remain unchanged.

Proposition~\ref{circle:prop:radial-kernel} now gives the kernel bound and its support conditions. Antipodal symmetry makes the geometric conditions at the negative stationary direction agree with those at the positive one. In the later recombination, we must also count compatible square pairs and justify the removal of restrictions inside positive integrals.

\subsection{Use in the graph estimate}\label{circle:sec:output}
For the amplitude pairs supplied by Theorem~\ref{circle:thm:conversion}, Lemma~\ref{angular:lem:comparison} gives
\[
 \mathsf H_F^{X,Y}(a,v;b_1,b_2)\le
 \mathsf H^{X,Y}(a,v;b_1,b_2).
\]
Here $\psi$ is external and bounded by one. The next section estimates the independent-angle expression on the right. The circle conversion retains every amplitude factor. To apply the graph estimate, the retained families must supply the required parallel factors and satisfy the depth inequalities.

\clearpage

\section{The two-profile graph estimate}\label{main:graph}
The independent-angle expression contains cross terms between transforms on different smaller squares. We control them by first using oscillation and then counting the terms that remain. The parallel-direction condition gives the needed count. The argument uses the amplitudes of Section~\ref{main:amplitudes} and allows the parent squares, or the two angular directions, to coincide.

When we pass from depth $a$ to depth $p$, each endpoint transform splits over its depth-$p$ descendants. Expanding the two squared transforms gives four square labels. We first estimate terms whose phase has a derivative bounded away from zero in an angular variable or in the common radial variable. Integration by parts applies to these full integrals before absolute values are taken.

For the remaining terms, a finite graph has rows $(P,Q)$ and columns $(P',Q')$, with an edge for each retained four-square term. The degree of a row or column is its number of possible partners. With $(P,Q)$ fixed, the parallel condition bounds the choices of $P'$, and two perpendicular coordinates then locate $Q'$. We can begin with either endpoint, which gives the smaller of the two profile bounds. The same argument bounds column degrees.

The graph inequality then gives a sum of nonnegative expressions on smaller squares, with every original amplitude factor retained. The profile comparisons subsequently replace that sum by a maximum. We keep the model-mass factors explicit throughout this calculation.

\subsection{Classes, amplitudes, and depths}\label{graph:sec:data}
We fix $R=2^N$, $\sigma>0$, $C_{\rm prof}=10$, and $C_{\rm test}\ge1000$. Let $\rho_1,\rho_2$ be class probabilities with the mass comparisons from Section~\ref{main:profiles}. They may be the same class. In the application, they are supplied by Corollary~\ref{dyad:cor:prepared}. The proof also applies to translated dyadic grids in unit root squares, including coincident roots.

For $j=1,2$ we retain separate profiles and model masses:
\begin{equation}\label{graph:eq:mass-comparison}
 M_j(n)R^{-C_{\rm prof}\sigma}\le\rho_j(P)
 \le M_j(n)R^{C_{\rm prof}\sigma},\qquad
 M_j(n)=R^{-m_j(n/N)}
\end{equation}
for each occupied depth-$n$ square. The functions $m_j$ are nondecreasing and $2$-Lipschitz, with $m_j(0)=0$. Thus
\begin{equation}\label{graph:eq:model-range}
 R^{-2}\le M_j(n)\le1.
\end{equation}
The notation from Section~\ref{main:profiles} is applied separately:
\[
 A_{j,n}=m_j(n/N)-n/N,\qquad
 I_j(a,p)=m_j(p/N)-m_j(a/N),\qquad
 k_j(a,p)=A_{j,p}-\min_{a\le n\le p}A_{j,n}.
\]
Every minimum here is over integer depths. Lemma~\ref{profile:lem:descendants} gives
\begin{equation}\label{graph:eq:descendants}
 \#\{P\subset X:P\text{ occupied at depth }p\}
 \le R^{I_j(a,p)+2C_{\rm prof}\sigma}
\end{equation}
when $X$ is an occupied depth-$a$ square of the $j$th class. Also, directly from the definition of the minimum,
\begin{equation}\label{graph:eq:k-containment}
 k_j(g,p)\le k_j(a,p)\qquad(a\le g<p).
\end{equation}
These facts do not require the lower profile bound obtained from a ball exponent greater than one.

At endpoint $j$, the family $\mathcal A_j$ has at most $N^3$ separately indexed cutoff factors, at a fixed level $0\le i_j<J_0$. We use the functions from Theorem~\ref{cutoff:thm:cutoffs}. The endpoint levels and families may differ. Repeated occurrences and factors for other condition types remain present, although the count below uses one specified parallel condition.

We recall that a parallel occurrence $\Par{g}{p}$, evaluated at a point $z$ in an occupied depth-$p$ square $P$, controls
\begin{equation}\label{graph:eq:parallel-test}
 \begin{split}
 \#\{P':{}&P'\text{ occupied for }\rho_j\text{ at depth }p,\\
 &c_{P'}\in B(c_P,30\,2^{-g}),\\
 &\dist(c_{P'},c_P+\R w)\le W_{i_j+1}(p)\}
 \le R^{k_j(g,p)+C_{\rm test}\sigma}.
 \end{split}
\end{equation}
The cutoff support imposes this condition at level $i_j+1$. As in Section~\ref{main:directions}, it counts all occupied centers of the original class in the given ball and strip. Spatial restriction does not redefine the count. The ball is open, the strip inequality is closed, and
\[
 W_{i_j+1}(p)\ge2^{-p}R^{500\sigma}.
\]
The square center $c_P$, rather than the particular point $z$, is the center used in the test.

We use the amplitudes~\eqref{amplitude:eq:form}, with combined prescribed order at most $B$ as in~\eqref{amplitude:eq:combined-order}. Their bounded Borel spatial multipliers are independent of angle and radial frequency, and every indexed factor appears, including those of order zero. The proof only needs the order at each endpoint to be at most $B$. It therefore also applies under separate endpoint bounds $B$, with combined bound $2B$ for that larger family.

We write $L$ for the largest condition length $\ell_{j,t}$ in the two families and use
\[
 D_{\max}=\max\{1,2^LR^{-50\sigma}\},\qquad D_*=R^\sigma D_{\max}
\]
from~\eqref{amplitude:eq:Dstar}. The theorem below requires a parallel occurrence at each endpoint, so its lists are nonempty.

Let $X,Y$ be occupied depth-$a$ squares for the respective classes. The transforms $U_X,V_Y$ in~\eqref{angular:eq:transforms} use model denominators $M_1(a),M_2(a)$. We abbreviate their independent-angle expression~\eqref{angular:eq:independent} as $\mathsf H^{X,Y}(a,v)$. For depth-$p$ descendants $P,Q$, the child expression $\mathsf H^{P,Q}(p,v)$ uses the same amplitudes restricted to those squares, model denominators $M_1(p),M_2(p)$, and prefactor $2^{2(v-p)}$. The radial measure is still $d\lambda_v(r)=\Psi(r/2^v)\,dr/2^v$. No factor is removed in this restriction.

The two angular variables remain independent and the radial variable remains common. The parent squares may coincide, touch, or be far apart.

\subsection{The graph estimate}\label{graph:sec:theorem}
The numerical coefficient in the maximum estimate is
\begin{equation}\label{graph:eq:constant}
 \boxed{C_G=C_{\rm test}+64+4C_{\rm prof}.}
\end{equation}
At the fixed choices $C_{\rm test}=1000$ and $C_{\rm prof}=10$, this is $1104$. Its value comes from counting squares, not from a differentiation order.

\begin{theorem}\label{graph:thm:graph}
Use the class data and the combined order bound $B$ specified above. Suppose that the integer
depths satisfy $0\le a<p\le v\le N$ and
\begin{equation}\label{graph:eq:depth-assumptions}
 p-a\le v-p+2\sigma N,\qquad
 L\le v-p+2\sigma N,\qquad v-p\ge10\sigma N.
\end{equation}
At endpoint $j$, suppose the unchanged list contains a parallel
occurrence $[g_j,p]$ with
\begin{equation}\label{graph:eq:test-depths}
 a\le g_j<p,\qquad g_j\le a+2\sigma N.
\end{equation}
Its cutoff has the support inclusion
\eqref{amplitude:eq:individual-bounds} for the test \eqref{graph:eq:parallel-test}.
For every fixed $B$ and $T>0$, there are common constants $C,C_T,N_0$ such that, for $N\ge N_0$,
\begin{align}\label{graph:eq:sum-theorem}
 \mathsf H^{X,Y}(a,v)
 &\le C\,2^{2(p-a)}R^{-2(I_1(a,p)+I_2(a,p))}
 R^{\min_j k_j(a,p)+(C_{\rm test}+64)\sigma}\notag\\
 &\qquad\cdot\sum_{\substack{P\subset X,\ Q\subset Y\\
                       P,Q\text{ occupied at depth }p}}
              \mathsf H^{P,Q}(p,v)+C_TR^{-T}.
\end{align}
Consequently,
\begin{align}\label{graph:eq:max-theorem}
 \mathsf H^{X,Y}(a,v)
 &\le C R^{-\sum_{j=1}^2(A_{j,p}-A_{j,a})
                 +\min_j k_j(a,p)+(C_{\rm test}+64+4C_{\rm prof})\sigma}\notag\\
 &\qquad\cdot\max_{\substack{P\subset X,\ Q\subset Y\\
                       P,Q\text{ occupied at depth }p}}
              \mathsf H^{P,Q}(p,v)+C_TR^{-T}.
\end{align}
Every child retains all original cutoff occurrences, their levels, and
their prescribed derivative orders. In particular, no factor whose
specified depth is $p$ is removed. No angular or radial factor is added
to a child amplitude.

The exponent coefficients in these estimates are independent of $B$,
$T$, and all integration-by-parts orders. Ordinary constants and the
starting value of $N$ may depend on the fixed parameters, $B,T$, the
fixed radial function, and the required finite set of common numbers
$K_s$. They do not depend on a particular measure, profile, square,
spatial multiplier, or list satisfying the displayed assumptions.
\end{theorem}

We first prove~\eqref{graph:eq:sum-theorem} for fixed amplitudes. Uniformity will then allow us to take the supremum needed after Theorem~\ref{circle:thm:conversion}; see Section~\ref{graph:sec:compatibility}.

The parallel condition contributes $C_{\rm test}$ to the degree exponent. The remaining $64=4+60$ comes from at most $R^{4\sigma}$ relevant depth-$g_j$ squares and, after the first candidate is chosen, at most $CR^{60\sigma}$ grid centers for the other candidate. Counting child pairs contributes the further $4C_{\rm prof}$ in the maximum estimate. These coefficients come from geometry and mass comparisons; differentiation orders affect the ordinary constants.

\subsection{Derivative and support bounds}\label{graph:sec:bounds}
Lemmas~\ref{amplitude:lem:products} and~\ref{amplitude:lem:copies} give, for each fixed additional order $s$,
\begin{gather}\label{graph:eq:pair-bounds}
 |b_{j,z}|\le U_B,\qquad
 \left|\partial_\theta^s\bigl(b_{j,z}\overline{b_{j,z'}}\bigr)\right|
 \le C_{B,s}D_*^s,
\end{gather}
once $2N^3\le R^\sigma$. Each individual cutoff derivative has order at most $B+s$, so only $K_1,\ldots,K_{B+s}$ are needed, even when two copies have combined prescribed order $2B$.

By Lemma~\ref{amplitude:lem:support}, a nonzero amplitude imposes every original next-level condition, and angular differentiation preserves the support inclusions. We will use these inclusions when choosing a direction in a candidate square. They hold without closedness of the good sets or nonvanishing of the undifferentiated cutoff at that point.

The additional derivatives occur only in estimates for discarded integrals. The child amplitudes keep their original prescribed orders.

\subsection{The angular partition and the four-square expansion}\label{graph:sec:partition}
We divide the circle into small arcs to decide which full integrals to retain. The partition weights will sum to one in the principal estimate, leaving the original child amplitudes.

We set
\begin{equation}\label{graph:eq:local-scales}
 {h_{\rm g}}=p-a,\qquad {q_{\rm g}}=v-p,\qquad
 \delta=\min\{c_0,2^{-{h_{\rm g}}}R^{10\sigma}\},\qquad
 \tau=2^{-p}R^{30\sigma},
\end{equation}
where $c_0>0$ is a fixed small number. Here $\delta$ is an angular
radius. The quantity $\tau$ is a spatial tolerance, larger than the
side length of a depth-$p$ square. Their different powers of $R$
provide room to replace a point by a square center and a direction by
an arc center.

There is a nonnegative smooth partition of unity
$\{\eta_\nu\}_{\nu\in\mathcal V}$ on $S^1$ such that
\begin{equation}\label{graph:eq:partition}
 \sum_\nu\eta_\nu=1,\quad
 \supp\eta_\nu\subset B_{S^1}(w_\nu,C\delta),\quad
 \|\partial_\theta^s\eta_\nu\|_\infty\le C_s\delta^{-s},\quad
 \#\mathcal V\le C\delta^{-1}.
\end{equation}
To construct it, we place equally spaced centers at separation comparable to $\delta$ and use translates of one rescaled nonnegative smooth bump, positive on an interval wider than the center spacing. The bumps have bounded overlap and their sum is bounded above and below by positive constants. Dividing by that sum gives~\eqref{graph:eq:partition}, including the derivative bounds. Angular and chordal distances are comparable on these small arcs.

For an occupied depth-$p$ square $P$ of endpoint $j$, we define
\begin{equation}\label{graph:eq:eligibility}
 E_j(P,\nu)=\int_P\int_{S^1}\eta_\nu(w)|b_{j,z}(w)|^2\,dw\,d\rho_j(z).
\end{equation}
Only pairs $(P,\nu)$ with $E_j(P,\nu)>0$ can contribute. Indeed, if $E_j(P,\nu)=0$, then for almost every $w$ with $\eta_\nu(w)>0$, the amplitude vanishes for $\rho_j$-almost every $z\in P$. The transform is therefore zero at every radial frequency for these angles. The angular null set is independent of $r$, because the amplitude vanishes before multiplication by the exponential. Every term involving this square and arc has integral zero.

If $E_j(P,\nu)>0$, we can choose $z_*\in P$ and $w_*$ with $\eta_\nu(w_*)>0$ and $b_{j,z_*}(w_*)\ne0$, outside any common spatial null set already removed. Lemma~\ref{amplitude:lem:support} then gives the parallel condition for that square at $w_*$. At fixed $N$ only finitely many such choices are needed. No quantitative lower bound for $E_j(P,\nu)$ is required.

The identity defining the child transforms gives
\begin{equation}\label{graph:eq:splitting}
 U_X=\zeta_1\sum_{P\subset X}U_P,\qquad
 V_Y=\zeta_2\sum_{Q\subset Y}V_Q,\qquad
 \zeta_j=\frac{M_j(p)}{M_j(a)}=R^{-I_j(a,p)}.
\end{equation}
All square sums below are over occupied depth-$p$ descendants unless otherwise stated. We expand the two absolute squares in~\eqref{angular:eq:independent} and insert~\eqref{graph:eq:partition} in each angular variable. A term has row label $(P,Q)$, column label $(P',Q')$, and arc labels $\nu_1,\nu_2$. Its transform product is
\[
 U_P(r,w_1)\overline{U_{P'}(r,w_1)}
 V_Q(r,w_2)\overline{V_{Q'}(r,w_2)}.
\]
In the spatial integrals, we write
\begin{equation}\label{graph:eq:displacements}
 \Delta_1=x-x',\quad\Delta_2=y-y',\quad
 d_1=c_P-c_{P'},\quad d_2=c_Q-c_{Q'}.
\end{equation}
The phase is
\begin{equation}\label{graph:eq:phase}
 \Phi=r(w_1\cdot\Delta_1+w_2\cdot\Delta_2).
\end{equation}
Each displacement is internal to its own parent square. Therefore
\begin{equation}\label{graph:eq:displacement-size}
 |\Delta_j|+|d_j|\le C2^{-a},\qquad
 |\Delta_j-d_j|\le C2^{-p}.
\end{equation}
Thus the relative location of the parents does not enter the phase estimate. Bounded amplitudes, finite spatial and angular measures, and the compact radial support justify Fubini throughout.

For fixed arc centers $w_1^0=w_{\nu_1}$ and $w_2^0=w_{\nu_2}$, we retain a row-column term when all four corresponding values of $E_j$ are positive and
\begin{equation}\label{graph:eq:edges}
 |(w_j^0)^\perp\cdot d_j|\le\tau\quad(j=1,2),\qquad
 |w_1^0\cdot d_1+w_2^0\cdot d_2|\le\tau.
\end{equation}
The first two inequalities bound the transverse displacements, and the last bounds their signed longitudinal sum. They depend only on square labels and arc centers. Each retained term keeps its full smooth angular integrals, with no new indicator depending on the running angle.

\subsection{Estimating the nonstationary terms}\label{graph:sec:nonstationary}
A nonstationary integral has a phase derivative bounded away from zero on the amplitude's support. We show that a term failing any inequality in~\eqref{graph:eq:edges} has this property in an angular variable or in the radial variable. Terms with a zero value of $E_j$ have already vanished.

Changing a square center to an actual point changes either displacement
by at most $C2^{-p}$. Changing an arc center to an angle in its support
changes a scalar product with an internal displacement by at most
$C2^{-a}\delta$. Consequently each expression in~\eqref{graph:eq:edges}
changes by at most
\begin{equation}\label{graph:eq:perturbation}
 C(2^{-p}+2^{-a}\delta)
 \le C2^{-p}R^{10\sigma}.
\end{equation}
Its ratio to $\tau$ is at most $CR^{-20\sigma}$, and is less than $1/2$ for all sufficiently large $N$, uniformly on slightly larger arcs. The same calculation applies when $\delta=c_0$, since then $c_0\le2^{-{h_{\rm g}}}R^{10\sigma}$ and hence $2^{-a}c_0\le2^{-p}R^{10\sigma}$.

If the last inequality in~\eqref{graph:eq:edges} fails, then throughout the
term,
\[
 |w_1\cdot\Delta_1+w_2\cdot\Delta_2|\ge\tau/2.
\]
With $s=r/2^v$, the phase is linear in $s$. Integration by parts differentiates only $\Psi(s)$, since the amplitudes and arc cutoffs are independent of $s$. Its compact support removes the boundary terms. After $m$ integrations, the angular-radial integral has modulus at most
\begin{equation}\label{graph:eq:radial-removal}
 C_{B,m}(2^v\tau)^{-m},\qquad
 2^v\tau=2^{q_{\rm g}}R^{30\sigma}\ge R^{40\sigma}.
\end{equation}
The spatial variables stay fixed in this calculation.

Suppose a transverse inequality fails, for example the one involving
$d_1$. On a local angular interval containing the first arc, the
perturbation estimate gives
\begin{equation}\label{graph:eq:phase-derivatives}
 |\partial_{\theta_1}\Phi|\ge c2^v\tau=:F_0,
 \qquad
 |\partial_{\theta_1}^u\Phi|\le C_u2^{v-a}=:C_uG_0
 \quad(u\ge1).
\end{equation}
The first bound holds on a neighborhood of the amplitude support.
Here $s$ stays in the fixed support of $\Psi$, bounded above and
below by positive constants. The same argument works in $\theta_2$
when the other transverse condition fails.

We keep the integration-by-parts bound explicit because the amplitude derivatives depend on $N$. Let $a_0$ be the product of the two first-angle amplitudes, their arc cutoff, and the remaining bounded factors. Lemma~\ref{amplitude:lem:copies} and~\eqref{graph:eq:partition} give
\[
 |\partial_{\theta_1}^u a_0|\le C_{B,u}H^u,
 \qquad H=1+D_*+\delta^{-1}.
\]
The constant includes the undifferentiated second-angle factors. For $b=(\partial_{\theta_1}\Phi)^{-1}$, reciprocal differentiation and~\eqref{graph:eq:phase-derivatives} give
\begin{equation}\label{graph:eq:reciprocal}
 |\partial_{\theta_1}^u b|
 \le C_uF_0^{-1}(1+G_0/F_0)^u.
\end{equation}
Each term after $u$ derivatives has a bound $C_uF_0^{-1}(G_0/F_0)^k$ for some $0\le k\le u$, which proves~\eqref{graph:eq:reciprocal}.

We set $Z=H+1+G_0/F_0$ and define recursively
\[
 a_{m+1}=-\partial_{\theta_1}(a_m b/\ii).
\]
The product rule and~\eqref{graph:eq:reciprocal} prove by induction that
\begin{equation}\label{graph:eq:ibp-induction}
 |\partial_{\theta_1}^u a_m|
 \le C_{B,m,u}F_0^{-m}Z^{m+u}.
\end{equation}
For the induction step, distribute $u+1$ derivatives between $a_m$
and $b$. A term with $t$ derivatives on $a_m$ has bound
$C F_0^{-m}Z^{m+t}F_0^{-1}Z^{u+1-t}$, which is exactly the
required bound at $m+1$. To perform $m$ integrations only the
initial bounds through order $m$ are needed.

We choose the local interval so that the arc cutoff is supported in its interior and~\eqref{graph:eq:phase-derivatives} holds wherever the reciprocal is used. Boundary terms then vanish. All remaining integration intervals have bounded length, so~\eqref{graph:eq:ibp-induction} with $u=0$ gives
\begin{equation}\label{graph:eq:ibp-ratio}
 C_{B,m}\left(
 \frac{1+D_*+\delta^{-1}+2^{-a}/\tau}{2^v\tau}
 \right)^m
\end{equation}
for the angular-radial integral at fixed spatial points.

The scale assumptions~\eqref{graph:eq:depth-assumptions} imply
\begin{align}\label{graph:eq:scale-bounds}
 D_*&\le R^\sigma+2^{q_{\rm g}}R^{-47\sigma},\notag\\
 \delta^{-1}&\le c_0^{-1}+2^{q_{\rm g}}R^{-8\sigma},\\
 2^{-a}/\tau&=2^{h_{\rm g}}R^{-30\sigma}\le2^{q_{\rm g}}R^{-28\sigma}.\notag
\end{align}
The first bound follows from $D_*\le R^\sigma+2^LR^{-49\sigma}$ and $L\le q_{\rm g}+2\sigma N$. The other two use $h_{\rm g}\le q_{\rm g}+2\sigma N$ in the definitions of $\delta$ and $\tau$.

Since $2^{q_{\rm g}}\ge R^{10\sigma}$, the numerator in~\eqref{graph:eq:ibp-ratio} is at most $C2^{q_{\rm g}}R^{-8\sigma}$. Division by $2^{q_{\rm g}}R^{30\sigma}$ gives $CR^{-38\sigma}$, so in particular
\begin{equation}\label{graph:eq:gain}
 \frac{1+D_*+\delta^{-1}+2^{-a}/\tau}{2^v\tau}
 \le CR^{-26\sigma}.
\end{equation}
The numerical constant raised to the fixed power $m$ is absorbed into $C_{B,m}$. Thus~\eqref{graph:eq:radial-removal} and~\eqref{graph:eq:gain} bound every excluded term, after angular and radial integration at fixed spatial points, by $C_{B,m}R^{-26\sigma m}$.

We now sum these errors and include the normalizations. Each unit root contains at most $R^2$ depth-$p$ squares, giving at most $R^8$ four-square labels. Since $\delta^{-1}\le CR$, there are at most $CR^2$ arc pairs. The splitting factors and child denominators satisfy
\begin{equation}\label{graph:eq:denominator-cancel}
 \zeta_1^2\zeta_2^2M_1(p)^{-2}M_2(p)^{-2}
 =M_1(a)^{-2}M_2(a)^{-2}\le R^8.
\end{equation}
The last inequality uses~\eqref{graph:eq:model-range}. The energy prefactor is at most $R^2$, and the four-point measure restricted to a given choice of squares has mass at most one. Together these bounds give a factor $CR^{20}$ for all sums and normalizations, without any smallness assumption on $C_{\rm prof}\sigma$.

We choose the integer
\begin{equation}\label{graph:eq:integration-order}
 m=\left\lfloor\frac{T+21}{26\sigma}\right\rfloor+1.
\end{equation}
The total excluded contribution is bounded by
\begin{equation}\label{graph:eq:total-error}
 C_{B,m}R^{20-26\sigma m}\le C_TR^{-T-1}
 \le C_TR^{-T}.
\end{equation}
This uses cutoff constants only through the finite order $B+m$.
We have proved the exact comparison
\begin{equation}\label{graph:eq:retained-equality}
 \mathsf H^{X,Y}(a,v)
 =\text{sum of the integrated retained terms}+E_N,
 \qquad |E_N|\le C_TR^{-T}.
\end{equation}
The retained sum still contains oscillatory integrals. Their absolute values will be taken after we bound the number of retained terms.

\subsection{Counting the retained terms}\label{graph:sec:degrees}
The retained terms satisfy geometric conditions on the square centers. The parallel test bounds one unknown square, and two coordinates locate the other. The argument remains uniform when the two directions coincide.

We fix the two arcs and a row $(P,Q)$. To count columns $(P',Q')$, we may begin with either endpoint. Starting with $P'$, its center satisfies
\[
 \dist(c_{P'},c_P+\R w_1^0)\le\tau.
\]
We only need candidates with $E_1(P',\nu_1)>0$. We group them by their depth-$g_1$ ancestors. Since all are descendants of $X$, there are at most
\begin{equation}\label{graph:eq:ancestor-groups}
 2^{2(g_1-a)}\le R^{4\sigma}
\end{equation}
groups, even when unoccupied ancestors are counted.

In each nonempty group, we choose a candidate $P_*'$ and, by positivity in~\eqref{graph:eq:eligibility}, a point and direction $(z_*,w_*)$ where its amplitude is nonzero. With $c_*=c_{P_*'}$, the parallel condition $[g_1,p]$ holds at center $c_*$ in direction $w_*$. We use a candidate's center because the fixed row square $P$ may lie outside this depth-$g_1$ ancestor; a test centered at $c_P$ would not give the required ball restriction.

For any other candidate center $c$ in the same group,
\begin{equation}\label{graph:eq:within-group}
 |c-c_*|\le\sqrt2\,2^{-g_1}.
\end{equation}
Both candidates lie within $\tau$ of the line through $c_P$ in
direction $w_1^0$. Subtracting these two transverse bounds gives
$|(w_1^0)^\perp\cdot(c-c_*)|\le2\tau$. Also
$|w_*-w_1^0|\le C\delta$. Therefore
\begin{align}\label{graph:eq:witness-strip}
 \dist(c,c_*+\R w_*)
 &=|w_*^\perp\cdot(c-c_*)|\notag\\
 &\le2\tau+C\delta\,|c-c_*|\notag\\
 &\le2\tau+C2^{-g_1}\delta
 \le C2^{-p}R^{30\sigma}.
\end{align}
In the last step we used $g_1\ge a$ and
$2^{-a}\delta\le2^{-p}R^{10\sigma}$.

Equation~\eqref{graph:eq:within-group} places every candidate center inside
$B(c_*,30\,2^{-g_1})$. For a common sufficiently large $N$, equation
\eqref{graph:eq:witness-strip} is strictly smaller than
$2^{-p}R^{500\sigma}\le W_{i_1+1}(p)$. Every candidate in the group
is therefore counted by the full parallel test at $(z_*,w_*)$.
That test gives at most
\[
 R^{k_1(g_1,p)+C_{\rm test}\sigma}
 \le R^{k_1(a,p)+C_{\rm test}\sigma}
\]
candidates, using~\eqref{graph:eq:k-containment}. After summing the groups,
there are at most
\begin{equation}\label{graph:eq:first-unknown}
 R^{k_1(a,p)+(C_{\rm test}+4)\sigma}
\end{equation}
choices for $P'$. Replacing endpoint $1$ by endpoint $2$ proves the
same statement with $k_2$ when $Q'$ is chosen first.

The chosen point and direction supply one valid test. Its center is the square center $c_*$, and no restriction of the class measure is normalized in this argument.

Continue with endpoint $1$ chosen first. Once $P'$ is fixed,
$d_1=c_P-c_{P'}$ is fixed. The last condition in~\eqref{graph:eq:edges}
requires
\[
 \left|w_2^0\cdot c_{Q'}-
       \bigl(w_2^0\cdot c_Q+w_1^0\cdot d_1\bigr)\right|\le\tau.
\]
The second transverse condition requires
\[
 |(w_2^0)^\perp\cdot c_{Q'}-(w_2^0)^\perp\cdot c_Q|\le\tau.
\]
These are two perpendicular coordinates of the \emph{same} unknown
center. Equivalently, $c_{Q'}$ lies within Euclidean distance
$\sqrt2\tau$ of
\begin{equation}\label{graph:eq:partner-center}
 c_Q+(w_1^0\cdot d_1)w_2^0.
\end{equation}
A ball of that radius contains at most
\begin{equation}\label{graph:eq:second-unknown}
 C(1+2^p\tau)^2\le C R^{60\sigma}
\end{equation}
depth-$p$ grid centers: an enclosing axis-aligned square has at most $2\sqrt2\,2^p\tau+2$ grid coordinates in each direction, and $2^p\tau=R^{30\sigma}\ge1$.

The coordinates $w_2^0,(w_2^0)^\perp$ are orthonormal regardless of $w_1^0$. Thus the bound remains uniform for coincident, opposite, or nearly coincident directions. If we count $Q'$ first, we use the corresponding coordinates of endpoint $1$ to locate $P'$.

Combining~\eqref{graph:eq:first-unknown} and~\eqref{graph:eq:second-unknown},
and choosing which endpoint is counted first, gives the row bound
\begin{equation}\label{graph:eq:degree}
 D_{\mathrm{gr}}
 =C R^{\min_{j=1,2}k_j(a,p)+(C_{\rm test}+64)\sigma}.
\end{equation}
Here $64=4+60$ comes from~\eqref{graph:eq:ancestor-groups} and~\eqref{graph:eq:second-unknown}. The column bound is the same: exchanging row and column changes $d_1,d_2$ to $-d_1,-d_2$, preserving~\eqref{graph:eq:edges} and positivity of the four values of $E_j$. The row and column label sets are identical, so the graph is symmetric and both degrees satisfy~\eqref{graph:eq:degree}.

For example, when $w_1^0=w_2^0=(1,0)$, the three conditions are
\[
 |(d_1)_2|\le\tau,\qquad |(d_2)_2|\le\tau,\qquad
 |(d_1)_1+(d_2)_1|\le\tau.
\]
Once $d_1$ is fixed, both coordinates of $d_2$ lie in intervals of length $2\tau$. This shows directly why equal angular directions cause no loss.

\subsection{Summing the retained terms}\label{graph:sec:summation}
With the nonstationary errors estimated and both degrees bounded, we now take absolute values in the retained sum. The cancellation needed for~\eqref{graph:eq:total-error} has already been used.

For fixed $r,w_1,w_2$, we put $F_{P,Q}=U_P(r,w_1)V_Q(r,w_2)$. With row label $\alpha$ and column label $\beta$, we have
\begin{equation}\label{graph:eq:graph-inequality}
 \begin{split}
 \sum_{(\alpha,\beta)\text{ retained}}|F_\alpha\overline{F_\beta}|
 &\le\frac12\sum_{(\alpha,\beta)\text{ retained}}
        (|F_\alpha|^2+|F_\beta|^2)\\
 &\le D_{\mathrm{gr}}\sum_{P,Q}|U_P(r,w_1)|^2|V_Q(r,w_2)|^2.
 \end{split}
\end{equation}
The first inequality is $2|uv|\le|u|^2+|v|^2$. In the second, each squared term occurs at most its row or column degree, and we enlarge the sum to all occupied child pairs. Both degree bounds are needed: row degrees may all equal one while many edges enter a single column. Here symmetry gives the required column bound.

The enlarged right side of~\eqref{graph:eq:graph-inequality} is independent of the arc centers. We multiply by the nonnegative weights $\eta_{\nu_1}(w_1)\eta_{\nu_2}(w_2)$ and sum over arc labels. Pointwise in $w_1,w_2$, we have
\begin{equation}\label{graph:eq:partition-cancel}
 \sum_{\nu_1,\nu_2}\eta_{\nu_1}(w_1)\eta_{\nu_2}(w_2)=1.
\end{equation}
Thus the principal bound has no factor counting arcs. Such a factor was used only in the error estimate, where repeated integration by parts supplied arbitrary decay.

We integrate the resulting nonnegative expression against the two circle measures and $\Psi(r/2^v)\,dr/2^v$. For a child pair, the integral of the product of squared transforms is
\[
 2^{-2(v-p)}\mathsf H^{P,Q}(p,v).
\]
Combining~\eqref{graph:eq:retained-equality},~\eqref{graph:eq:splitting}, and
\eqref{graph:eq:graph-inequality}, we obtain
\begin{equation}\label{graph:eq:before-sum}
 \mathsf H^{X,Y}(a,v)
 \le 2^{2(v-a)}\zeta_1^2\zeta_2^2D_{\mathrm{gr}}
        2^{-2(v-p)}\sum_{P,Q}\mathsf H^{P,Q}(p,v)+C_TR^{-T}.
\end{equation}
Since $\zeta_j=R^{-I_j(a,p)}$, this proves~\eqref{graph:eq:sum-theorem} with the original amplitudes restricted to each child square.

Equation~\eqref{graph:eq:descendants} gives at most
\begin{equation}\label{graph:eq:child-pairs}
 R^{I_1(a,p)+I_2(a,p)+4C_{\rm prof}\sigma}
\end{equation}
child pairs. Replacing their sum by this count times their maximum
leaves the following principal factor, with $I_j=I_j(a,p)$:
\begin{align}\label{graph:eq:principal-cancel}
 2^{2(p-a)}R^{-2(I_1+I_2)}R^{I_1+I_2}
 &=R^{2(p-a)/N-I_1-I_2}\notag\\
 &=R^{-(A_{1,p}-A_{1,a})-(A_{2,p}-A_{2,a})}.
\end{align}
Here $I_j=A_{j,p}-A_{j,a}+(p-a)/N$ for each of the two separate profiles. The degree bound contributes $(C_{\rm test}+64)\sigma$, and~\eqref{graph:eq:child-pairs} contributes $4C_{\rm prof}\sigma$. This proves~\eqref{graph:eq:max-theorem} with the coefficient~\eqref{graph:eq:constant}. The error term is unchanged when the child sum is bounded by its count times a maximum, completing the proof of Theorem~\ref{graph:thm:graph}.

The four factors in this calculation are collected below.
\begin{center}
\begin{tabular}{@{}P{67mm}P{76mm}@{}}
\toprule
Operation & Contribution\\\midrule
Split both transforms and square both sums & $R^{-2(I_1+I_2)}$\\
Replace the parent energy prefactor by the child prefactor & $2^{2(p-a)}$\\
Replace the sum by the number of child pairs times a maximum & $R^{I_1+I_2+4C_{\rm prof}\sigma}$\\
Use the two degree bounds & $R^{\min_j k_j(a,p)+(C_{\rm test}+64)\sigma}$\\
\bottomrule
\end{tabular}
\end{center}
After the graph inequality, the remaining sum is over the two labels $(P,Q)$. The degree bound has already accounted for the other two.

\subsection{Applying the estimate after circle conversion}\label{graph:sec:compatibility}
The graph estimate applies to the amplitudes in Theorem~\ref{circle:thm:conversion} with $B=B_{\rm inv}$. They have the form~\eqref{amplitude:eq:form}, with all original cutoff occurrences and levels, bounded spatial multipliers independent of angle and radial frequency, and combined order at most $B_{\rm inv}$. Lemma~\ref{angular:lem:comparison} first gives $\mathsf H_F\le\mathsf H$ for those same amplitudes, keeping the geometric cutoff external.

Each application also needs the parallel occurrences and depth inequalities~\eqref{graph:eq:depth-assumptions}--\eqref{graph:eq:test-depths}. An empty list allowed in circle conversion cannot provide these occurrences. Propositions~\ref{lists:prop:high-depths} and~\ref{lists:prop:direct-depths} give the depth information for two later applications, with their specified classes, lists, levels, and geometric tests.

The child transforms use the original integrands restricted to the child squares, with model denominators at the child depth. Their amplitudes retain every cutoff occurrence, level, and prescribed order; only the spatial multipliers are restricted. Derivatives taken to estimate discarded integrals do not become derivatives of these child amplitudes.

In particular, factors specified at depth $p$ remain present. Once the graph estimate has used their conditions, Lemma~\ref{amplitude:lem:deletion} may remove them from the positive independent-angle child quantities. This is a further inequality, and is unavailable in advance in the shared-angle integral or a signed four-point term.

For fixed measures, profiles, factor families, levels, and depths, the constants are uniform over the allowed amplitudes. We may therefore take the supremum needed after the inverse theorem. Finiteness of the child set and~\eqref{angular:eq:sup-max} give
\[
 \sup_{b_1,b_2}\max_{P,Q}
 \mathsf H^{P,Q}(p,v;b_1|_P,b_2|_Q)
 =\max_{P,Q}\sup_{b_1,b_2}
 \mathsf H^{P,Q}(p,v;b_1|_P,b_2|_Q).
\]
No maximizing amplitude is needed. Enlarging to the full child family gives an upper bound with the same factors, levels, and fixed order bound. A larger family with arbitrary angular multipliers or an order growing with $N$ is not covered by this estimate.

\subsection{Uniform constants}\label{graph:sec:uniformity}
We finish by specifying one starting scale for all the data in the theorem. The differentiation order affects this choice and the ordinary constants, while the coefficient $C_G$ remains the value in~\eqref{graph:eq:constant}.

We first fix $\sigma$, $J_0$, $C_{\rm prof}$, $C_{\rm test}$, the cutoff kernel, the radial cutoff, and the circle and dyadic-grid conventions. For fixed prescribed-order bound $B$ and error exponent $T$, we choose $m$ by~\eqref{graph:eq:integration-order}. The nonstationary estimate uses only the cutoff constants $K_s$ through order $B+m$, along with derivatives of the trigonometric phase and the fixed radial cutoff.

We then take $N_0$ large enough that $2N^3\le R^\sigma$, the perturbation~\eqref{graph:eq:perturbation} is less than a fixed fraction of $\tau$ on the enlarged arcs, and the strip bound~\eqref{graph:eq:witness-strip} is less than $2^{-p}R^{500\sigma}$. All derivative and geometric constants have already been fixed, so this choice is uniform over the measures, profiles, factor families, squares, spatial multipliers, and depths in the hypotheses. Although $m$ and the ordinary constants may be large when $\sigma$ is small, the counting argument and $C_G$ are independent of $m$.

The proof includes zero displacements, since it never divides by their lengths, and the endpoint case $g_j=a$, for which the ancestor count is one. Complex spatial multipliers are only restricted or conjugated; no spatial derivative is taken. Half-open dyadic boundaries continue to give exact disjoint partitions even when they carry mass. An identically zero endpoint amplitude makes the parent expression zero. The other cases, including coincident parent squares and directions, repeated factors, and nonclosed good sets, were included in the arguments above.

\subsection{Use in the local reductions}\label{graph:sec:summary}
The graph estimate bounds the independent-angle expression by a sum of nonnegative child expressions, or by their maximum, with error $C_TR^{-T}$ after all sums and normalizations. Every child keeps the original amplitude factors and prescribed orders. The choice of which endpoint to count first gives $\min_j k_j(a,p)$, and the angular partition contributes no count of arcs to the principal bound.

Section~\ref{main:reductions} uses this graph estimate both directly on angular quantities and after circle conversion of a positive distance count.

\clearpage

\section{Recursive reductions}\label{main:reductions}
The circle conversion bounds Fourier integrals of distance measures by angular expressions. Cauchy--Schwarz and the graph estimate then bound these expressions in terms of corresponding quantities on smaller dyadic squares. The iteration also uses positive counts measuring how often two pairs of points determine nearly equal distances. This section establishes six reductions to shorter depth intervals and specifies the measures, normalizations, and directional conditions in the resulting quantities.

For a positive count, the suffix case uses the projection condition to obtain a count on a shorter interval. In the prefix case, low-frequency contributions are bounded by positive counts on shorter intervals, while high-frequency contributions are bounded by angular quantities through the circle conversion and graph estimate. An angular quantity is treated either by a direct application of the graph estimate or by splitting it into near and far contributions. The far contribution is bounded by positive counts at the next level. Theorem~\ref{reduce:thm:six}
collects these six estimates.

We first establish the scalar and geometric estimates needed for the positive counts, then prove the six reductions. We conclude with the terminal estimates and the choice of derivative orders and a common starting scale. Section~\ref{main:iteration} controls the iteration of these estimates.

\subsection{The measures and positive distance counts}\label{positive:sec:data}
We fix a class probability $\rho$ and its profile $m$ from Corollary~\ref{dyad:cor:prepared}, suppressing the class subscript. We use the notation of Section~\ref{main:profiles}:
\[
R=2^N,\qquad 0<\sigma\le\kappa<1,\qquad q_0=\lceil\kappa N\rceil,\qquad M(n)=R^{-m(n/N)},\qquad A_n=m(n/N)-n/N.
\]
The arguments below use the occupied-square mass comparisons~\eqref{profile:eq:mass}. We retain $C_{\rm prof}$ in the estimates, although $C_{\rm prof}=10$ for the constructed classes. Conditional probabilities are normalized by the actual square masses, as in~\eqref{profile:eq:conditional}.

The condition sets $\M(a,t)$ and selected depths are defined in Theorem~\ref{lists:thm:main}. These sets are nested under inclusion of intervals, and every condition in $\M(a,t)$ has specified depth strictly between $a$ and $t$ and length at most $(t-a)/2$. An interval $[a,t]$ is nonterminal if $t-a>100q_0$ and terminal otherwise. For a nonterminal interval with selected depth $p$, a prefix choice generates $\Par ap$, while a suffix choice generates $\Proj pt$. By Corollary~\ref{lists:cor:Ncube}, each condition set contains at most $N^3$ indices.

The good sets of Section~\ref{direction:sec:tests} are defined using the full class. The parallel condition counts all occupied square centers, while the projection condition uses the full conditional probability $\rho_P$. These conditions, including their closed strip inequalities, remain unchanged when the pair measure is restricted.

We use the cutoffs of Theorem~\ref{cutoff:thm:cutoffs} at levels $0\le i<J_0$. Their angular derivatives are supported in the corresponding level-$(i+1)$ good sets. By~\eqref{cutoff:eq:closed-transfer}, the closed $1/(250D_\ell)$-neighborhood of each level-$i$ good set is contained in its level-$(i+1)$ good set, where $D_\ell=\max\{1,2^\ell R^{-50\sigma}\}$. The cutoff hypothesis $R^{450\sigma}\ge2J_0$ also implies $R^{500\sigma}\ge9$, which will be used below.

Let $K=4096$. We say that depth-$a$ squares $X,Y$ are separated if
\begin{equation}\label{positive:eq:separated}
 (K/8)2^{-a}\le |c_X-c_Y|\le4K2^{-a}.
\end{equation}
Set
\[
 c_-=K/8-\sqrt2>1,\qquad c_+=4K+\sqrt2,
 \qquad \omega_-=c_+^{-1/2},\quad \omega_+=c_-^{-1/2}.
\]
For every $x\in\overline X$ and $y\in\overline Y$,
\begin{equation}\label{positive:eq:distance-weight}
 c_-2^{-a}\le|x-y|\le c_+2^{-a},\qquad
 \omega_-\le(2^a|x-y|)^{-1/2}\le\omega_+.
\end{equation}
The directions $(x-y)/|x-y|$ lie in a small arc about $(c_X-c_Y)/|c_X-c_Y|$. Rotating a fixed smooth function yields a cutoff $0\le\psi\le1$ that equals one near these directions and vanishes near their antipodes, with angular margins and derivative bounds independent of $X,Y$ and $a$. The cutoff remains outside the endpoint amplitudes.

Let
\[
\dd(x,y)=|x-y|,\qquad e(x,y)=\frac{x-y}{|x-y|}.
\]
For $0\le i\le J_0$, let $\theta_i^{X,Y}(a,t)$ be the restriction of $\rho|_X\otimes\rho|_Y$ to pairs $(x,y)$ satisfying
\[
e(x,y)\in G_{i,\mathfrak t}(x)\cap G_{i,\mathfrak t}(y)
\quad\hbox{for every }\mathfrak t\in\M(a,t).
\]
Fix $C_0=300$, and write $\theta_i=\theta_i^{X,Y}(a,t)$ when the other parameters are understood. Define
\begin{equation}\label{positive:eq:positive-count}
\mathsf D_i^{X,Y}(a,t)=2^{t-a}M(a)^{-4}
\iint \one_{\{||x-y|-|x'-y'||\le C_0 2^{-t}\}}
      d\theta_i(x,y)d\theta_i(x',y').
\end{equation}
This quantity measures near coincidences between the distances determined by two ordered pairs, with neither endpoint fixed. The measure $\theta_i$ is positive and is not normalized by its total mass; the normalization in~\eqref{positive:eq:positive-count} uses the model mass $M(a)^4$.

The positive count is defined also at level $J_0$, whereas a further circle conversion requires a cutoff at a level $i<J_0$.

The supports of $\rho|_X$ and $\rho|_Y$ lie in $\overline X$ and $\overline Y$, respectively, so~\eqref{positive:eq:distance-weight} gives the separation required by Theorem~\ref{circle:thm:conversion}. For $x\in\overline X$ and $y\in\overline Y$, we have $|(x-y)-(c_X-c_Y)|\le\sqrt2\,2^{-a}$. Since both vectors have length at least $c_-2^{-a}$, their unit directions differ by at most $2\sqrt2/c_-$ in chordal distance and $\pi\sqrt2/c_-<1/64$ in angular distance. Thus rotating a fixed smooth function that equals one on an arc of radius $1/32$ and is supported in an arc of radius $1/8$ gives the required cutoff $\psi$.

For angular quantities, we use the independent-angle expression~\eqref{angular:eq:independent}, with the same class at both endpoints and model denominator $M(b)$ on depth-$b$ squares. These squares need not be separated. We also use
\[
\mathcal H_{i,B}^{[a,t]}(p,v\mid X,Y)
\]
as defined in~\eqref{angular:eq:supremum}. The maximum is over occupied depth-$p$ descendants of $X,Y$, and the supremum is over amplitudes with exactly the conditions $\M_{>p}([a,t])$ at both endpoints, at level $i$ and with combined prescribed order at most $B$. The class, model masses, base, and cutoff derivative bounds remain fixed.

\subsection{Scalar estimates for finite positive measures}\label{scalar:sec:scalar}
For a finite positive Borel measure $\lambda$ on $\R$ and $\delta>0$, define
\[
I_k^\delta=[k\delta,(k+1)\delta),\qquad
B_\delta(\lambda)=\sum_{k\in\Z}\lambda(I_k^\delta)^2.
\]
Since these intervals partition $\R$, we have $B_\delta(\lambda)\le\lambda(\R)^2<\infty$.

\begin{lemma}\label{scalar:lem:bins}
Let $\lambda,\nu$ be finite positive Borel measures on $\R$, and let $\delta>0$ and $C\ge1$. Then
\begin{align}
 B_\delta(\lambda)&\le(\lambda\otimes\lambda)
                      \{|s-s'|\le\delta\},\label{scalar:eq:bins-lower}\\
 \sup_{u\in\R}(\lambda\otimes\nu)
       \{|s-s'-u|\le C\delta\}
 &\le n_C\sqrt{B_\delta(\lambda)B_\delta(\nu)},
 \qquad n_C=2\lceil C\rceil+3,\label{scalar:eq:bins-shift}\\
 \iint\frac{d\lambda(s)d\nu(s')}{(1+|s-s'|/\delta)^2}
 &\le5\sqrt{B_\delta(\lambda)B_\delta(\nu)}.\label{scalar:eq:bins-kernel}
\end{align}
\end{lemma}
\begin{proof}
The first inequality follows because two points in the same interval $I_k^\delta$ differ by less than $\delta$. Set $b_k=\lambda(I_k^\delta)$ and $c_k=\nu(I_k^\delta)$. For every integer $h$, Cauchy--Schwarz gives
\[
\sum_k b_kc_{k-h}\le\|b\|_2\|c\|_2.
\]
If $s\in I_k^\delta$ and $s'\in I_l^\delta$, then
\[
(s-s')/\delta\in(k-l-1,k-l+1).
\]
Thus $|s-s'-u|\le C\delta$ implies $|(k-l)-u/\delta|<C+1$. There are at most $n_C$ possible integer offsets $h=k-l$. Summing the preceding Cauchy--Schwarz bound over these offsets proves~\eqref{scalar:eq:bins-shift}, uniformly in $u\in\R$.

For the last estimate, if $s\in I_k^\delta$ and $s'\in I_l^\delta$, with $h=k-l$, then
\[
(1+|s-s'|/\delta)^{-2}\le\max\{1,|h|\}^{-2}.
\]
Since $1+2\sum_{h\ge1}h^{-2}\le5$, summing the same Cauchy--Schwarz bound over $h\in\Z$ proves~\eqref{scalar:eq:bins-kernel}.
\end{proof}

In particular, if $(\lambda\otimes\lambda)\{|s-s'|\le\delta\}\le B$, then $(\lambda\otimes\lambda)\{|s-s'-u|\le C\delta\}\le n_CB$ for every $u\in\R$. We use this after fixing $x,x'$ in the suffix reduction.

\begin{lemma}\label{scalar:lem:fourier-majorant}
For every finite positive Borel measure $\lambda$ on $\R$, $\delta>0$, and $C\ge1$,
\begin{equation}\label{scalar:eq:fourier-majorant}
 (\lambda\otimes\lambda)\{|s-s'|\le C\delta\}
 \le2(C+1)\delta\int_{|r|\le\delta^{-1}}
                          |\widehat\lambda(r)|^2dr.
\end{equation}
\end{lemma}
\begin{proof}
Set $b=1/(2(C+1))$ and
\[
 K_b(u)=\left(\frac{\sin(bu)}{bu}\right)^2,
\]
where $K_b(0)=1$. Since $|bu|\le1/2$ for $|u|\le C$, the inequality $\sin z/z\ge1-z^2/6$ gives $K_b(u)\ge1/2$ on this interval. Using $\sin(bu)/(bu)=(2b)^{-1}\int_{-b}^be^{\ii ru}\,dr$, we obtain
\begin{equation}\label{scalar:eq:triangular}
 K_b(u)=\int_{-2b}^{2b}T_b(r)e^{\ii ru}\,dr,
 \qquad T_b(r)=\frac{(2b-|r|)_+}{4b^2}.
\end{equation}
The function $T_b$ is nonnegative, has integral one, is supported in $[-2b,2b]\subset[-1,1]$, and satisfies $\|T_b\|_\infty=C+1$.

Replacing $u$ by $(s-s')/\delta$ and applying Fubini's theorem gives
\[
 \iint K_b((s-s')/\delta)\,d\lambda(s)d\lambda(s')
 =\delta\int T_b(\delta r)|\widehat\lambda(r)|^2dr.
\]
Since $K_b(u)\ge1/2$ for $|u|\le C$, the bounds on $T_b$ yield~\eqref{scalar:eq:fourier-majorant}.
\end{proof}

We use the fixed nonnegative function $\Psi$ from~\eqref{angular:eq:radial-cutoff}, extended by zero to $\R$. Recall that $\Psi\ge1$ on $[1/2,2]$. Set
\[
 K_\Psi(u)=\int_\R\Psi(r)e^{-\ii ru}\,dr,
 \qquad C_\Psi=4\max\{\|\Psi\|_1,\|\Psi''\|_1\}.
\]
The bound $|K_\Psi(u)|\le\|\Psi\|_1$ and two integrations by parts give
\begin{equation}\label{scalar:eq:smooth-kernel}
 |K_\Psi(u)|\le C_\Psi(1+|u|)^{-2}.
\end{equation}
For a finite positive measure $\zeta$ and $v\ge0$, Fubini's theorem gives
\begin{equation}\label{scalar:eq:kernel-identity}
 \int\Psi(r/2^v)|\widehat\zeta(r)|^2dr
 =2^v\iint K_\Psi(2^v(s-s'))\,d\zeta(s)d\zeta(s').
\end{equation}
The interchange of integrals is justified by $2^v\|\Psi\|_1\zeta(\R)^2<\infty$. Since $\Psi\ge1$ on $[1/2,1]$, taking the absolute value of $K_\Psi$ and applying~\eqref{scalar:eq:smooth-kernel} yields
\begin{equation}\label{scalar:eq:shell-positive-kernel}
 \int_{2^{v-1}}^{2^v}|\widehat\zeta(r)|^2dr
 \le C_\Psi2^v\iint
       (1+2^v|s-s'|)^{-2}\,d\zeta(s)d\zeta(s').
\end{equation}
The kernel on the right is nonnegative, so the double integral is monotone under domination of either positive measure.

If $\lambda=\sum_{j=1}^m\lambda_j$ is a finite sum of positive measures, the triangle inequality in $\ell^2$ gives
\begin{equation}\label{scalar:eq:bin-minkowski}
 B_\delta(\lambda)^{1/2}\le\sum_{j=1}^mB_\delta(\lambda_j)^{1/2}.
\end{equation}
If each bin $I_k^\delta$ meets the supports of at most $L$ of the measures $\lambda_j$, Cauchy--Schwarz in each bin gives
\begin{equation}\label{scalar:eq:bin-overlap}
 B_\delta\left(\sum_j\lambda_j\right)
 \le L\sum_jB_\delta(\lambda_j).
\end{equation}
We verify the required overlap bound for whole bins in the suffix reduction.

\subsection{Geometry inside a separated pair of squares}\label{positive:sec:geometry}
The suffix choice makes the quadratic error in the distance smaller than the fine distance tolerance.

\begin{lemma}\label{positive:lem:geometry}
Let $X,Y$ satisfy~\eqref{positive:eq:separated}, and let $P\subset X$, $Q\subset Y$ be depth-$p$ descendants with $p>a$. Set $d_c=c_P-c_Q$ and $e_0=d_c/|d_c|$. For $x\in P$ and $y\in Q$,
\begin{align}
 \bigl||x-y|-e_0\cdot(x-y)\bigr|&\le2^{a-2p},\label{positive:eq:taylor}\\
 |e(x,y)-e_0|&\le4\,2^{a-p},\label{positive:eq:center-direction}\\
 \bigl||x-y|-|c_P-c_Q|\bigr|&\le\sqrt2\,2^{-p}<2\,2^{-p}.\label{positive:eq:center-distance}
\end{align}
The angular distance between the directions of any two pairs in $P\times Q$ is at most $16\,2^{a-p}$.
\end{lemma}
\begin{proof}
Set $u=(x-c_P)-(y-c_Q)$, so that $x-y=d_c+u$ and $|u|\le\sqrt2\,2^{-p}$. For every $0\le\tau\le1$, the vector $d_c+\tau u$ is the difference of two points in $\overline X$ and $\overline Y$. Hence~\eqref{positive:eq:distance-weight} gives $|d_c+\tau u|\ge c_-2^{-a}$. The Hessian of $z\mapsto|z|$ has operator norm at most $|z|^{-1}$. Taylor's theorem therefore bounds the remainder by
\[
 \tfrac12c_-^{-1}2^a|u|^2\le c_-^{-1}2^{a-2p}\le2^{a-2p}.
\]
Since $e_0\cdot d_c=|d_c|$, this proves~\eqref{positive:eq:taylor}.

For nonzero vectors $z,z'$,
\[
 \left|\frac z{|z|}-\frac{z'}{|z'|}\right|
 \le\frac{2|z-z'|}{\min\{|z|,|z'|\}}.
\]
Applying this inequality to $x-y$ and $d_c$ proves~\eqref{positive:eq:center-direction}. For two pairs in $P\times Q$, the triangle inequality then gives a chordal distance of at most $8\,2^{a-p}$ between their directions. Since angular distance on $S^1$ is at most $\pi/2$ times chordal distance, the asserted angular bound follows from $4\pi<16$. Finally, the reverse triangle inequality gives~\eqref{positive:eq:center-distance}.
\end{proof}

If $2p>a+t$, then $2^{a-2p}\le2^{-t}$. Let $(x_0,y_0)\in P\times Q$ satisfy the conditions in $\M(a,t)$ at level $i$, and set $\omega_0=e(x_0,y_0)$. For $y,y'\in Q$,
\begin{equation}\label{positive:eq:projection-direction-change}
 |(y-y')\cdot(e_0-\omega_0)|
 \le8\,2^{a-2p}\le8\,2^{-t}.
\end{equation}
If $R^{500\sigma}\ge9$, the event $|(y-y')\cdot e_0|\le2^{-t}$ is therefore contained in the event $|(y-y')\cdot\omega_0|\le W_i(t)$.

\begin{lemma}\label{positive:lem:full-pair-transfer}
Let $X,Y$ be depth-$a$ squares satisfying~\eqref{positive:eq:separated}. Let $p$ be the selected depth for a suffix choice on a nonterminal interval $[a,t]$, and let $P\subset X$, $Q\subset Y$ be depth-$p$ descendants. Suppose $0\le i<J_0$. If one pair in $P\times Q$ satisfies all conditions in $\M(a,t)$ at both endpoints at level $i$, then every pair in $P\times Q$ satisfies all conditions in $\M(a,p)$ at both endpoints at level $i+1$.
\end{lemma}
\begin{proof}
By Theorem~\ref{lists:thm:main}, $\M(a,p)\subseteq\M(a,t)$, and every index in $\M(a,p)$ has specified depth less than $p$. Thus its specified squares are constant throughout $P$ and $Q$. Its length satisfies $\ell'\le(p-a)/2$. Set $d=p-a$. Since $2p>a+t$ and $t-a>100q_0$, we have $d>50q_0\ge50$. By Lemma~\ref{positive:lem:geometry}, the angular distance between the directions of any two pairs in $P\times Q$, multiplied by $D_{\ell'}$, is at most
\[
 16\,2^{-d}\max\{1,2^{d/2}R^{-50\sigma}\}
 \le16(2^{-d}+2^{-d/2})<\frac1{250}.
\]
Hence the direction of every pair in $P\times Q$ lies within $4\delta_{\ell'}$ of $\omega_0$. Since the specified squares are unchanged, the inclusion~\eqref{cutoff:eq:closed-transfer} applies at both endpoints to every condition in $\M(a,p)$, proving the assertion.
\end{proof}

\subsection{The three positive reductions}\label{positive:sec:theorem}
In the prefix case, we use a single positive measure on all the dyadic frequency intervals. Set $\mathcal A=\M(a,t)$ and, for $0\le i<J_0$, define
\begin{equation}\label{positive:eq:zeta}
 \zeta=\dd_\#\left[
 M(a)^{-2}(2^a|x-y|)^{-1/2}
 \prod_{\mathfrak t\in\mathcal A}
 w_{i,\mathfrak t}(x,e(x,y))w_{i,\mathfrak t}(y,e(x,y))
 \,d\rho|_X(x)d\rho|_Y(y)\right].
\end{equation}
The pushforward is understood as in Subsection~\ref{main:measure-conventions}. For the given parent data, $\zeta$ is independent of the frequency index $v$. It is auxiliary to this estimate; $\tau_N$ remains the measure defined in~\eqref{mass:eq:measures}.

For each integer $a<v\le t$, set
\begin{equation}\label{positive:eq:Sv}
 S_v=2^{-a}\int_{2^{v-1}}^{2^v}|\widehat\zeta(r)|^2dr.
\end{equation}
Define the low- and high-frequency index sets
\[
 \mathcal L=\{v:a<v\le t,\ t-v\ge\sigma N\},\qquad
 \mathcal U=\{v:a<v\le t,\ 0\le t-v<\sigma N\}.
\]
These sets partition the integers $a<v\le t$, with equality $t-v=\sigma N$ included in $\mathcal L$.

\Needspace{490pt}
\begin{theorem}\label{positive:thm:three}
Assume the data of Section~\ref{positive:sec:data}. Let $X,Y$ be separated occupied depth-$a$ squares, let $[a,t]$ be nonterminal with selected depth $p$, and let $0\le i<J_0$. For every fixed $T>0$, the following estimates hold for all sufficiently large $N$, uniformly over the permitted classes, squares, intervals, and levels.

If $2p>a+t$, then
\begin{equation}\label{positive:eq:suffix-main}
 \mathsf D_i^{X,Y}(a,t)
 \le C_{\rm suf}
 R^{A_p-\min_{p\le n\le t}A_n+C_{\rm test}\sigma}
 \mathsf D_{i+1}^{X,Y}(a,p),
 \qquad C_{\rm suf}=48843.
\end{equation}
Assume now that $2p\le a+t$, and let $\zeta$ be the measure in~\eqref{positive:eq:zeta}. For every $v\in\mathcal L$,
\begin{equation}\label{positive:eq:low-main}
 S_v\le C_{\rm lf}\mathsf D_{i+1}^{X,Y}(a,v),
 \qquad C_{\rm lf}=5C_\Psi\omega_+^2.
\end{equation}
There is no $R^{C\sigma}$ factor and no additive error in this comparison. The constant is independent of $T$.

For the high-frequency estimate, we use Theorems~\ref{circle:thm:conversion} and~\ref{graph:thm:graph}. Set $S=T+2$, $\mu=C_{\rm prof}\sigma$, and
\begin{equation}\label{positive:eq:inverse-order}
\begin{gathered}
 P_{\mathrm{inv}}=S+2\mu+2,\quad
 J_{\rm pos}=\left\lfloor\frac{P_{\mathrm{inv}}+9}{10\sigma}\right\rfloor+1,
 \\ B=2(J_{\rm pos}-1),\quad C_G=C_{\rm test}+64+4C_{\rm prof}.
\end{gathered}
\end{equation}
For every $v\in\mathcal U$,
\begin{equation}\label{positive:eq:high-main}
 S_v\le C_T
 R^{-2(A_p-A_a)+k(a,p)+C_G\sigma}
       \mathcal H_{i,B}^{[a,t]}(p,v\mid X,Y)+C_TR^{-S}.
\end{equation}
The angular quantities on the right have level $i$, base $[a,t]$, and remaining condition set $\M_{>p}([a,t])$. Their amplitudes have the form~\eqref{amplitude:eq:form}, with no additional angular multiplier or radial-frequency dependence.

In particular, the complete prefix estimate is
\begin{align}\label{positive:eq:combined}
 \mathsf D_i^{X,Y}(a,t)
 &\le C_{\rm base}\left(
 R^{4C_{\rm prof}\sigma}+\sum_{v\in\mathcal L}\mathsf D_{i+1}^{X,Y}(a,v)\right)\notag\\
 &\quad+C_T R^{-2(A_p-A_a)+k(a,p)+C_G\sigma}
       \sum_{v\in\mathcal U}\mathcal H_{i,B}^{[a,t]}(p,v\mid X,Y)
       +C_TR^{-T}.
\end{align}
The coefficients in the powers of $R$ are independent of $T$, $B$, and the orders used in integration by parts. Multiplicative constants and the starting scale may depend on the fixed parameters, geometric functions, and finitely many common cutoff derivative bounds. The constants $C_{\rm suf}$, $C_{\rm lf}$, and $C_{\rm base}$ can be chosen independently of $T$.
\end{theorem}

\subsection{The suffix reduction}\label{positive:sec:suffix}
We apply the projection condition to the full conditional probability on a square, then use the directional neighborhood inclusion to place the selected product of squares in the child good-pair set.

Write $\theta=\theta_i^{X,Y}(a,t)$ and $\delta=2^{-t}$. Select the products $P\times Q$ of depth-$p$ descendants $P\subset X$, $Q\subset Y$ for which $\theta(P\times Q)>0$. Their union carries $\theta$. In each of these finitely many products, choose a pair $(x_0,y_0)$ satisfying the conditions in $\M(a,t)$ at level $i$.

Set
\[
 \omega_0=e(x_0,y_0),\qquad e_0=e(c_P,c_Q),\qquad
 {B_{\rm proj}}=2^{-(t-p)}R^{A_p-\min_{p\le n\le t}A_n+C_{\rm test}\sigma}.
\]
The suffix choice generates $\Proj pt$ in the parent list. The chosen pair therefore satisfies~\eqref{direction:eq:projection-threshold} on $Q$, with the full conditional probability $\rho_Q$.

Let $\lambda_Q=(y\mapsto e_0\cdot y)_\#\rho_Q$. By~\eqref{positive:eq:projection-direction-change},
\[
 (\lambda_Q\otimes\lambda_Q)\{|s-s'|\le\delta\}
 \le (\rho_Q\otimes\rho_Q)
       \{|(y-y')\cdot\omega_0|\le9\delta\}
 \le{B_{\rm proj}}.
\]
Since $9\delta\le W_i(t)$, the last inequality follows from the projection condition, with its closed strip inequality. Lemma~\ref{scalar:lem:bins} gives
\begin{equation}\label{positive:eq:translated-projection}
 \sup_u(\lambda_Q\otimes\lambda_Q)
        \{|s-s'-u|\le3\delta\}\le9{B_{\rm proj}}.
\end{equation}
For fixed $x,x'\in P$, the condition $||x-y|-|x'-y'||\le\delta$ and the two Taylor errors in~\eqref{positive:eq:taylor} imply
\[
 |e_0\cdot(y-y')-e_0\cdot(x-x')|\le3\delta.
\]
Thus~\eqref{positive:eq:translated-projection}, followed by integration in $x,x'$, gives
\begin{equation}\label{positive:eq:pair-fine-bins}
 B_\delta\bigl(\dd_\#(\rho|_P\otimes\rho|_Q)\bigr)
 \le9{B_{\rm proj}}\,\rho(P)^2\rho(Q)^2.
\end{equation}
Here~\eqref{scalar:eq:bins-lower} bounds the bin sum by the corresponding distance-collision integral. The denominator in $\rho_Q$ is the actual mass $\rho(Q)$, whereas the pair measure on the left is unnormalized.

Put $\Delta=2^{-p}$. Assign a selected pair $(P,Q)$ to group $k$ if $|c_P-c_Q|\in[k\Delta,(k+1)\Delta)$. Define
\[
 W_k=\sum_{(P,Q)\text{ in group }k}\rho(P)\rho(Q),\qquad
 \lambda_k=\sum_{(P,Q)\text{ in group }k}
              \dd_\#(\rho|_P\otimes\rho|_Q).
\]
By the triangle inequality~\eqref{scalar:eq:bin-minkowski} and~\eqref{positive:eq:pair-fine-bins},
\begin{equation}\label{positive:eq:group-bins}
 B_\delta(\lambda_k)\le9{B_{\rm proj}} W_k^2.
\end{equation}
The triangle inequality is applied to the square roots, so no factor counting the square pairs in a group is introduced.

By~\eqref{positive:eq:center-distance}, the distances in group $k$ lie in
\[
 [(k-2)\Delta,(k+3)\Delta].
\]
Since $\delta\le\Delta$, a fine bin $[u,u+\delta)$ can meet group $k$ only if
\[
 u/\Delta-3\le k\le(u+\delta)/\Delta+2.
\]
The interval of possible $k$ has length at most six. Thus at most nine group supports can meet any fine bin, including at their boundaries. Equation~\eqref{scalar:eq:bin-overlap} therefore gives
\[
 B_\delta\Bigl(\sum_k\lambda_k\Bigr)
 \le9\sum_kB_\delta(\lambda_k)
 \le81{B_{\rm proj}}\sum_kW_k^2.
\]
Since $\dd_\#\theta\le\sum_k\lambda_k$, the parent distance-collision integral at tolerance $C_0\delta$ is at most
\begin{equation}\label{positive:eq:parent-fine-bound}
 603\cdot81\,{B_{\rm proj}}\sum_kW_k^2,
\end{equation}
because $n_{300}=603$ in~\eqref{scalar:eq:bins-shift}. This also bounds comparisons between different groups, since the overlap estimate applies to the sum of the measures.

By Lemma~\ref{positive:lem:full-pair-transfer}, every selected full product $P\times Q$ lies in the child good-pair set at level $i+1$.

Any two distances arising from one group differ by at most $5\Delta<C_0\Delta$. The products $P\times Q$ are pairwise disjoint because the squares form partitions. Hence
\begin{equation}\label{positive:eq:coarse-child}
 \sum_kW_k^2
 \le 2^{-(p-a)}M(a)^4\mathsf D_{i+1}^{X,Y}(a,p).
\end{equation}

Multiply~\eqref{positive:eq:parent-fine-bound} by the parent normalization $2^{t-a}M(a)^{-4}$ and use~\eqref{positive:eq:coarse-child}. The factors simplify as follows:
\begin{align*}
 &2^{t-a}M(a)^{-4}
 \left(2^{-(t-p)}R^{A_p-\min_{[p,t]}A+C_{\rm test}\sigma}\right)
 2^{-(p-a)}M(a)^4\\
 &\hspace{25mm}=R^{A_p-\min_{[p,t]}A+C_{\rm test}\sigma}.
\end{align*}
Since $603\cdot81=48843$, this proves~\eqref{positive:eq:suffix-main}. The parent and child have the same two spatial squares and the same model mass $M(a)$.

\subsection{Fourier comparison and the low-frequency reduction}\label{positive:sec:low}
We first compare the positive measures and then apply the Fourier majorant. Domination of positive measures does not imply a pointwise bound for the moduli of their Fourier transforms.

Let $\lambda=\dd_\#\theta_i^{X,Y}(a,t)$ and let $\zeta$ be~\eqref{positive:eq:zeta}. On the set defining $\theta_i^{X,Y}(a,t)$, every cutoff equals one. By~\eqref{positive:eq:distance-weight},
\begin{equation}\label{positive:eq:positive-domination}
 \lambda\le\omega_-^{-1}M(a)^2\zeta
\end{equation}
as positive measures. Domination also holds on their positive product measures, so Lemma~\ref{scalar:lem:fourier-majorant} with $C=C_0$ gives
\begin{equation}\label{positive:eq:parent-Fourier}
 \mathsf D_i^{X,Y}(a,t)
 \le C_{\rm comp}\,2^{-a}\int_{|r|\le2^t}|\widehat\zeta(r)|^2dr,
 \qquad C_{\rm comp}=2(C_0+1)\omega_-^{-2}.
\end{equation}
The factor $2^{-t}$ from the scalar majorant combines with $2^{t-a}$ to give $2^{-a}$, while the two factors $M(a)^2$ from~\eqref{positive:eq:positive-domination} cancel $M(a)^{-4}$.

The cutoff product is at most one, and~\eqref{profile:eq:mass} gives
\[
 \zeta(\R)\le\omega_+M(a)^{-2}\rho(X)\rho(Y)
 \le\omega_+R^{2C_{\rm prof}\sigma}.
\]
Consequently
\begin{equation}\label{positive:eq:stopping}
 2^{-a}\int_{|r|\le2^a}|\widehat\zeta(r)|^2dr
 \le2\omega_+^2R^{4C_{\rm prof}\sigma}.
\end{equation}
This contributes one additive term to the parent bound.

Since $\zeta$ is a positive measure, $|\widehat\zeta(-r)|=|\widehat\zeta(r)|$. The rest of the integral in~\eqref{positive:eq:parent-Fourier} is therefore exactly twice the sum of the $S_v$ over $a<v\le t$, up to interval endpoints of Lebesgue measure zero.

Let $\theta'$ be $\rho|_X\otimes\rho|_Y$ restricted to pairs whose direction lies in the support of every parent cutoff at the corresponding endpoint. This is a Borel restriction: the specified spatial partitions are finite, and the angular supports are closed.

Set $\lambda'=\dd_\#\theta'$. The pointwise bounds on the factors imply the positive-measure domination
\[
 \zeta\le\omega_+M(a)^{-2}\lambda'.
\]
Applying~\eqref{scalar:eq:shell-positive-kernel} and then the preceding domination gives
\begin{equation}\label{positive:eq:low-positive}
 S_v\le C_\Psi\omega_+^2 2^{v-a}M(a)^{-4}
       \iint(1+2^v|s-s'|)^{-2}\,d\lambda'(s)d\lambda'(s').
\end{equation}
This use of measure domination is valid because the kernel is nonnegative.

If $v\le t$, then $\M(a,v)\subseteq\M(a,t)$. Each cutoff is supported in its corresponding level-$(i+1)$ good set. Thus
\[
 \theta'\le\theta_{i+1}^{X,Y}(a,v).
\]
Let $\lambda_{\rm child}=\dd_\#\theta_{i+1}^{X,Y}(a,v)$. Lemma~\ref{scalar:lem:bins}, now with $\delta=2^{-v}$, implies
\begin{align*}
 \iint(1+2^v|s-s'|)^{-2}\,d\lambda'(s)d\lambda'(s')
 &\le5B_{2^{-v}}(\lambda_{\rm child})\\
 &\le5(\lambda_{\rm child}\otimes\lambda_{\rm child})
                   \{|s-s'|\le C_0 2^{-v}\}.
\end{align*}
Substituting this in~\eqref{positive:eq:low-positive} proves~\eqref{positive:eq:low-main}.

This scalar comparison holds for every $a<v\le t$. We use it when $t-v\ge\sigma N$, so the interval shortens by at least $\sigma N$. For the high frequencies, an analytic estimate will instead increase the starting depth.

The exact prefactor in~\eqref{positive:eq:low-positive} is
\[
 2^{v-a}M(a)^{-4}.
\]
This is the normalization of the shorter count. Since the argument uses the actual measures and the bound one for the cutoff product,~\eqref{positive:eq:low-main} has no loss from mass comparisons or the number of conditions. Bounding a sum of these contributions by their maximum would introduce a factor of at most $N$.

\subsection{The high-frequency reduction}\label{positive:sec:analytic}\label{positive:sec:high}
For the high frequencies, the prefix choice gives both the frequency separation needed for the circle conversion and the parallel condition needed for the graph estimate. We write $L$ for the largest condition length in $\M(a,t)$ and first state the two estimates in the form used here.

For $\zeta$ in~\eqref{positive:eq:zeta}, Theorem~\ref{circle:thm:conversion} applies whenever
\begin{equation}\label{positive:eq:I-scales}
 v-a\ge20\sigma N,\qquad L\le(v-a)/2+\sigma N.
\end{equation}
For any $S>0$, Theorem~\ref{circle:thm:conversion} gives
\begin{equation}\label{positive:eq:I}
 2^{-a}\int_{2^{v-1}}^{2^{v+1}}|\widehat\zeta(r)|^2dr
 \le C_S\sup_{b_1,b_2}\mathsf H_F^{X,Y}(a,v;b_1,b_2)+C_SR^{-S}.
\end{equation}
Here $\lambda_1=\rho|_X$, $\lambda_2=\rho|_Y$, and both denominators are $M(a)$. The mass comparison~\eqref{profile:eq:mass} gives $C_m=1$ and $\mu=C_{\rm prof}\sigma$. The distance weight is exactly that in~\eqref{positive:eq:zeta}; both spatial multipliers are one, and both condition sets are $\M(a,t)$ at level $i$.

The supremum in~\eqref{positive:eq:I} is over amplitudes of the form~\eqref{amplitude:eq:form}, with every original cutoff occurrence retained and combined prescribed order bounded by~\eqref{positive:eq:inverse-order}. The function $\psi$ remains outside the endpoint amplitudes. The conversion uses cutoff derivative bounds through order $2J_{\rm pos}+4$. Subsequent derivative estimates may require higher cutoff bounds, but do not change the prescribed orders. The coefficient sum is bounded independently of $N$.

With the same profile at both endpoints, Theorem~\ref{graph:thm:graph} gives
\begin{equation}\label{positive:eq:G}
 \mathsf H^{X,Y}(a,v;b_1,b_2)
 \le C R^{-2(A_p-A_a)+k(a,p)+C_G\sigma}
       \max_{P,Q}\mathsf H^{P,Q}(p,v;b_1|_P,b_2|_Q)
       +C_SR^{-S},
\end{equation}
provided
\begin{equation}\label{positive:eq:G-scales}
 p-a\le v-p+2\sigma N,\quad
 L\le v-p+2\sigma N,\quad v-p\ge10\sigma N,
\end{equation}
and both endpoint families contain a parallel factor $\Par gp$ with $a\le g<p$ and $g\le a+2\sigma N$. Since the profiles agree, the minimum in~\eqref{graph:eq:max-theorem} reduces to $k(a,p)$. The child amplitudes are the original amplitudes restricted to the child squares, with every factor specified at $p$ still present.

We assume the prefix case and $v\in\mathcal U$, and write $W=t-a$. The largest condition length satisfies $L\le W/2$ by Theorem~\ref{lists:thm:main}. Nonterminality and $\sigma\le\kappa$ give
\[
 v-a>100q_0-\sigma N>20\sigma N,
 \qquad L\le(v-a)/2+\sigma N.
\]
Thus~\eqref{positive:eq:I-scales} holds. The prefix choice has $p-a\le t-p$ and supplies $\Par ap$ on both sides. Also
\begin{align*}
 p-a&\le t-p<v-p+\sigma N,\\
 L&\le(t-a)/2\le t-p<v-p+\sigma N,\\
 v-p&>(t-a)/2-\sigma N>50q_0-\sigma N>10\sigma N.
\end{align*}
These inequalities, also recorded in Proposition~\ref{lists:prop:high-depths}, imply~\eqref{positive:eq:G-scales} with $g=a$. They also hold in the midpoint case $2p=a+t$.

The remaining hypotheses also hold: Corollary~\ref{lists:cor:Ncube} gives at most $N^3$ factors, the cutoffs have their common derivative bounds, and the separated supports give the required $\psi$. All measures and model denominators remain unchanged.

The interval in $S_v$ is contained in the interval on the left of~\eqref{positive:eq:I}. Hence Theorem~\ref{circle:thm:conversion} gives
\[
 S_v\le C_S\sup_{b,c}\mathsf H_F^{X,Y}(a,v;b,c)+C_SR^{-S}.
\]
For each amplitude pair in this supremum, we apply Lemma~\ref{angular:lem:comparison}, in the form~\eqref{angular:eq:comparison}, and then Theorem~\ref{graph:thm:graph}. Every cutoff still occurs, so its support imposes the required level-$(i+1)$ condition. This includes the parallel condition specified at $p$.

The resulting depth-$p$ integrals have the same amplitudes restricted to their squares. The error bound is uniform over this amplitude family and therefore remains valid after taking the supremum.

We now apply Lemma~\ref{amplitude:lem:deletion} to each nonnegative independent-angle expression on a child pair of squares. The remaining indices are exactly
\[
 \{\mathfrak t\in\M(a,t):p(\mathfrak t)>p\}
 =\M_{>p}([a,t]).
\]
The level, spatial multipliers, and remaining prescribed orders stay fixed. The removed factors contribute only $U_B^4$.

The new current interval and earlier base satisfy
\[
 a\le p\le v\le t,\qquad t<v+\sigma N.
\]
These quantities therefore belong to the family in~\eqref{angular:eq:supremum}, with combined prescribed-order bound $B$ from~\eqref{positive:eq:inverse-order}. This proves~\eqref{positive:eq:high-main}.

\subsection{Completion of the positive-count reductions}\label{positive:sec:uniformity}
We choose the finite differentiation orders before the starting scale, so that the preceding estimates hold uniformly for all classes and squares.

Fix $T>0$ and use $S=T+2$ in both analytic theorems. First choose the inverse order in~\eqref{positive:eq:inverse-order}, with $\mu=C_{\rm prof}\sigma$, and then an integration-by-parts order for the graph estimate with prescribed-order bound $B$ and error $O(R^{-S})$. The cutoff construction supplies all the required finite derivative bounds.

Both analytic errors include the relevant normalizing factors. Composing the two estimates multiplies each error only by a fixed constant: the power of $R$ in the graph estimate and the factor $U_B^4$ from removing cutoffs multiply only the main term. Each high-frequency interval therefore contributes an error of at most $C_TR^{-S}$.

There are at most $N$ high-frequency intervals. Since $N\le2^N=R$ for $N\ge1$,
\[
 \sum_{v\in\mathcal U}C_TR^{-S}
 \le C_TR^{-T-1}\le C_TR^{-T}.
\]
The effect of coefficients from earlier reductions will be estimated in Section~\ref{main:iteration}.

We take $N$ above the cutoff threshold, the thresholds for the inverse and graph estimates with error exponent $S$ (the latter at prescribed order $B$), and the threshold for $2N^3\le R^\sigma$. These depend only on the fixed structural constants, $\kappa,\sigma,T$, the finite cutoff bounds, and the geometry. They are common to all classes, squares, permitted levels, and frequency indices.

The cutoff threshold already gives $R^{500\sigma}\ge9$. The suffix transfer follows from $p-a>50q_0\ge50$, and the scalar estimates hold at every positive scale.

We combine~\eqref{positive:eq:parent-Fourier}, the stopping bound~\eqref{positive:eq:stopping}, the symmetry between positive and negative frequencies, and~\eqref{positive:eq:low-main}--\eqref{positive:eq:high-main}. The summed error is at most $C_TR^{-T}$. Absorbing the fixed constants gives~\eqref{positive:eq:combined} and completes the proof of Theorem~\ref{positive:thm:three}.

\Needspace{245pt}
\begin{longtable}{@{}P{38mm}P{64mm}P{47mm}@{}}
\toprule Estimate & Main coefficient & Normalization\\\midrule\endfirsthead
\toprule Estimate & Main coefficient & Normalization\\\midrule\endhead
Suffix & $C_{\rm suf}R^{A_p-\min_{[p,t]}A+C_{\rm test}\sigma}$ & $M(a)^{-4}$ cancels exactly with the coarse positive count; $2^{t-p}$ cancels the projected tolerance.\\
Low frequency & $C_{\rm lf}$ & The positive kernel gives exactly $2^{v-a}M(a)^{-4}$. No comparison with model masses is used.\\
High frequency & $C_T R^{-2(A_p-A_a)+k(a,p)+C_G\sigma}$ & The inverse step has a fixed coefficient sum. The displayed power is exactly the graph power.\\
Very low frequencies & $C R^{4C_{\rm prof}\sigma}$, as an additive stopping term & Bound the total mass of $\zeta$, then square it. This is an additive term, not a multiplier of a low-frequency count.\\
\bottomrule
\end{longtable}
For the constructed classes, $4C_{\rm prof}=40$ and $C_G=C_{\rm test}+104$, as in~\eqref{graph:eq:constant}.

\subsection{Examples for the positive-count estimates}\label{positive:sec:examples}
The following examples explain the use of positive kernels, full conditional probabilities, and the exact low-frequency normalization.

Take $\lambda=\delta_0$ and $\nu=\delta_0+\delta_1$. Then $0\le\lambda\le\nu$, but at $r=\pi$,
\[
 |\widehat\lambda(\pi)|=1,\qquad |\widehat\nu(\pi)|=0.
\]
This is why the proofs of~\eqref{positive:eq:parent-Fourier} and the low-frequency estimate compare positive integrals before passing to, or away from, the Fourier expression.

The example~\eqref{amplitude:eq:cancellation-example} shows why factors cannot be removed from a shared-angle expression using only their pointwise bounds. The graph estimate also requires the parallel condition at the child depth. Its factor is retained through the conversion and graph estimate, and removed only from the resulting nonnegative child expression.

When $x,x'$ are fixed, the distance comparison asks whether
\[
 e_0\cdot(y-y')\approx e_0\cdot(x-x').
\]
The projection condition controls a probability of nearly equal projections at zero shift, whereas the required translation can be much larger than the fine tolerance. The bin estimate gives the needed bound for every translation.

Normalizing a restriction of small mass inside $P\times Q$ could destroy the projection bound. We instead estimate the full product using the conditional probability $\rho_Q$. A pair satisfying the parent conditions supplies a direction, and the directional neighborhood inclusion then places all of $P\times Q$ in the child good-pair set. This allows us to sum the actual masses $\rho(P)\rho(Q)$.

A loss $R^{c\sigma}$ repeated about $1/\sigma$ times could accumulate to a fixed power $R^{O(c)}$. The low-frequency estimate avoids this by having a fixed multiplicative constant. The actual number of reductions will be bounded in Section~\ref{main:iteration}.

The positive-count reductions are summarized by~\eqref{positive:eq:suffix-main} and~\eqref{positive:eq:combined}. The prefix estimate retains the full frequency sum and one additive stopping term. We next establish the direct, near, and far reductions for angular quantities; their repeated application is controlled in Section~\ref{main:iteration}.

\subsection{The direct angular reduction}\label{reduce:sec:direct}
When a remaining condition is sufficiently long relative to the current interval, it supplies the parallel condition needed for the graph estimate.

Throughout the angular reductions, we keep one class $\rho$, its profile, and its model masses. The depth-$a$ parent squares $X,Y$ need not be separated and may coincide. The parent quantity is the independent-angle expression~\eqref{angular:eq:independent}.

The base $\mathcal B=[a_0,t_0]$ and current interval $[a,v]$ satisfy
\begin{equation}\label{reduce:eq:base}
 a_0\le a\le v\le t_0\le N,\qquad t_0\le v+\sigma N.
\end{equation}
Both endpoint families contain one occurrence of each condition in $\M_{>a}(\mathcal B)$, at level $0\le i<J_0$, with combined prescribed order at most a fixed integer $B$. We write $\mathsf H_i^{\mathcal B}(a,v)$ when the squares and amplitudes are understood.

We will use the following exact identity at both the near and far scales:
\begin{equation}\label{reduce:eq:cancellation}
 2^{2(d-a)}\left(\frac{M(d)}{M(a)}\right)^4
 R^{2I(a,d)+4C_{\rm prof}\sigma}
 =R^{-2(A_d-A_a)+4C_{\rm prof}\sigma}.
\end{equation}
It follows from $M(d)/M(a)=R^{-I(a,d)}$ and $A_d-A_a=I(a,d)-(d-a)/N$. The factor $R^{2I(a,d)+4C_{\rm prof}\sigma}$ will arise from counting pairs of occupied descendants using Lemma~\ref{profile:lem:descendants}.

We assume that $\mathsf H_i^{\mathcal B}(a,v)$ is nonterminal. Let $L$ be the largest condition length in $\M_{>a}(\mathcal B)$, and set
\begin{equation}\label{reduce:eq:h}
 L_0=\max\{L,\lceil10\sigma N\rceil\},\qquad h=v-2L_0.
\end{equation}
The remaining list contains $\M(a,v)$, which is nonempty because $[a,v]$ is nonterminal; its indices have specified depths greater than $a$.

Proposition~\ref{lists:prop:direct-depths}, with $\epsilon=\sigma N$, applies to these data. If $h\le a$, it gives $L_0=L$. Choose a longest remaining condition, let $p$ be its specified depth, and choose a supplying interval $[b_0,u_0]\subseteq\mathcal B$ for that condition. The same interval supplies $\Par gp$, where $g=\max\{a,p-L\}$, and
\begin{gather}
 a<p<v,\qquad a\le g<p,\qquad g\le a+\sigma N,\label{reduce:eq:direct-g}\\
 p-a\le v-p+2\sigma N,\qquad
 L\le v-p+\sigma N,\qquad v-p>10\sigma N.\label{reduce:eq:direct-depths}
\end{gather}
In particular, $\Par gp$ is still in $\M_{>a}(\mathcal B)$. The supplying interval gives
\begin{equation}\label{reduce:eq:direct-origin}
 b_0\le p-L,\qquad p+L\le u_0\le t_0\le v+\sigma N,
 \qquad v-a\le2L.
\end{equation}
These inequalities verify the hypotheses of the graph estimate, using a parallel condition already present in the parent set.

The supplying interval lies in $\mathcal B$ and may extend outside $[a,v]$. It identifies the needed condition without changing the base or the parent quantity.

\begin{proposition}\label{reduce:prop:direct}
Under the data above, if $h\le a$, then for every fixed $S>0$,
\begin{equation}\label{reduce:eq:direct}
 \mathsf H_i^{\mathcal B}(a,v)
 \le C_{B,S}R^{-2(A_p-A_a)+k(a,p)+C_G\sigma}
          \max_{P\subset X,Q\subset Y}\mathsf H_i^{\mathcal B}(p,v;P,Q)
       +C_{B,S}R^{-S}.
\end{equation}
Here the maximum is over occupied depth-$p$ descendants. For the fixed input, the child amplitudes are precisely the restricted old amplitudes with factors specified at depth at most $p$ removed. The child base is still $\mathcal B$, its level is $i$, and its combined prescribed order is at most the input order $B$.
\end{proposition}
\begin{proof}
The inequalities recalled above verify the scale and parallel-index requirements of Theorem~\ref{graph:thm:graph} with $g_1=g_2=g$. The two profiles coincide, so its exponent becomes $-2(A_p-A_a)+k(a,p)+C_G\sigma$.

The parallel factor is present, and all its angular derivatives have the same support restriction. The graph estimate therefore applies and initially gives the unchanged amplitudes on the child squares.

Apply Lemma~\ref{amplitude:lem:deletion} to these nonnegative independent-angle expressions to remove the factors specified at depths at most $p$. The resulting list is exactly
\[
 \{\mathfrak t\in\M_{>a}(\mathcal B):p(\mathfrak t)>p\}
 =\M_{>p}(\mathcal B).
\]
The scalar $U_B^4$ is absorbed into the fixed constant. The base inequalities $a_0\le p<v\le t_0\le v+\sigma N$ hold, and the spatial multipliers, levels, and remaining orders stay fixed. The graph error is already normalized at the parent and is not multiplied by the main power of $R$.
\end{proof}
The bound is uniform over amplitudes with fixed order bound and the same class, profile, base, condition set, and cutoff bounds.

\subsection{The near contribution}\label{reduce:sec:near}
For the near contribution, counting neighboring square centers is sufficient; no additional directional condition is needed.

Assume now that $h>a$ in~\eqref{reduce:eq:h}. In particular $a<h<v$. Restrict the original amplitudes, without deleting any factors, to occupied depth-$h$ descendants. Let $U_P,V_Q$ be their transforms normalized by $M(h)$. The splitting is exact:
\begin{equation}\label{reduce:eq:splitting}
 U_X=\zeta_h\sum_{P\subset X}U_P,\qquad
 V_Y=\zeta_h\sum_{Q\subset Y}V_Q,\qquad
 \zeta_h=\frac{M(h)}{M(a)}.
\end{equation}

On expanding the two squared moduli, assign a term to the near contribution if
\begin{equation}\label{reduce:eq:near-condition}
 |c_P-c_{P'}|\le30\,2^{-h},\qquad
 |c_Q-c_{Q'}|\le30\,2^{-h}.
\end{equation}
Let $\mathcal N_h$ and $\mathcal F_h$ denote the integrated near and complementary far contributions, including the original prefactor and the four factors from~\eqref{reduce:eq:splitting}. Then
\begin{equation}\label{reduce:eq:split-H}
 \mathsf H_i^{\mathcal B}(a,v)=\mathcal N_h+\mathcal F_h.
\end{equation}
The expansion is finite, and all integrands are bounded against finite spatial, angular, and radial measures.

Both contributions are real: interchanging $(P,Q)$ and $(P',Q')$ conjugates the integrand and preserves the selected terms. Neither contribution is necessarily nonnegative, so we use
\[
 \mathsf H_i^{\mathcal B}(a,v)\le |\mathcal N_h|+|\mathcal F_h|.
\]

Depth-$h$ square centers lie on a translate of $2^{-h}\Z^2$. Within distance $30\,2^{-h}$, each coordinate has at most $61$ integer offsets, so there are at most $61^2$ neighboring centers, including boundary equality.

For row label $(P,Q)$ and column label $(P',Q')$, the two conditions in~\eqref{reduce:eq:near-condition} give both row and column degrees at most
\begin{equation}\label{reduce:eq:near-degree}
 D_{\rm near}=(61^2)^2=61^4.
\end{equation}
Restriction to occupied descendants can only decrease these counts. No separation of $X,Y$ or condition on the angular directions is needed.

For fixed $r,w_1,w_2$, put $F_{P,Q}=U_P(r,w_1)V_Q(r,w_2)$. Then
\begin{align}\label{reduce:eq:near-finite}
 \left|\sum_{\text{near rows and columns}}F_{P,Q}\overline{F_{P',Q'}}\right|
 &\le\sum_{\text{near}}|F_{P,Q}F_{P',Q'}|\notag\\
 &\le \frac12\sum_{\text{near}}
          (|F_{P,Q}|^2+|F_{P',Q'}|^2)\notag\\
 &\le D_{\rm near}\sum_{P,Q}|F_{P,Q}|^2.
\end{align}
The last inequality uses both degree bounds and is valid for complex transforms, regardless of the sign of the near sum.

\Needspace{235pt}
\begin{proposition}\label{reduce:prop:near}
Let $h>a$. Write $\mathsf H_{\rm old}^{P,Q}(h,v)$ for the child energy with all original factors restricted to $P$ and $Q$. Then
\begin{equation}\label{reduce:eq:near-sum}
 |\mathcal N_h|\le
 61^4\,2^{2(h-a)}\left(\frac{M(h)}{M(a)}\right)^4
 \sum_{P\subset X,Q\subset Y}\mathsf H_{\rm old}^{P,Q}(h,v).
\end{equation}
After deleting spatially constant factors from these positive child expressions,
\begin{equation}\label{reduce:eq:near-max}
 |\mathcal N_h|\le
 61^4U_B^4
 R^{-2(A_h-A_a)+4C_{\rm prof}\sigma}
 \max_{P\subset X,Q\subset Y}\mathsf H_i^{\mathcal B}(h,v;P,Q).
\end{equation}
The quantities in the second bound have condition set $\M_{>h}(\mathcal B)$ at both endpoints, level $i$, and combined prescribed order at most $B$. There is no additive error in either estimate.
\end{proposition}
\begin{proof}
Integrate~\eqref{reduce:eq:near-finite} over both angular variables with respect to arclength and against $\Psi(r/2^v)dr/2^v$, including the factors from~\eqref{reduce:eq:splitting}. The parent energy prefactor is $2^{2(v-a)}$ and the child one is $2^{2(v-h)}$. Their ratio is exactly $2^{2(h-a)}$, proving~\eqref{reduce:eq:near-sum}.

Apply Lemma~\ref{amplitude:lem:deletion} to these nonnegative independent-angle child energies. The removed indices are exactly those in the old list with specified depth at most $h$. The remaining list is $\M_{>h}(\mathcal B)$, so it defines an angular child at $[h,v]$ with the same base. The base inequalities hold since $a_0\le a<h<v\le t_0\le v+\sigma N$.

By Lemma~\ref{profile:lem:descendants}, there are at most
\[
 R^{2I(a,h)+4C_{\rm prof}\sigma}
\]
child pairs. Bound the child sum by this number times its maximum, and use~\eqref{reduce:eq:cancellation}. This gives~\eqref{reduce:eq:near-max}.
\end{proof}

\subsection{The far contribution}\label{reduce:sec:far}
For the far contribution, we return to the exact splitting~\eqref{reduce:eq:splitting}. Every factor in the original amplitudes is still present. Their supports will impose the conditions for the new positive counts.

\bigskip \paragraph{\emph{The scales and the original supports.}}
For fixed points in a far term, set
\[
 \Delta_1=x-x',\quad\Delta_2=y-y',\quad
 d_j=|\Delta_j|,\quad d_+=\max d_j,\quad d_-=\min d_j.
\]
At least one pair of depth-$h$ centers has distance greater than $30\,2^{-h}$. Each displacement joins two points in the same depth-$a$ square, so
\begin{equation}\label{reduce:eq:far-lengths}
 d_+>(30-\sqrt2)2^{-h},\qquad d_+\le\sqrt2\,2^{-a}.
\end{equation}
The angular amplitudes in the expanded four-point integral are $b_x(w)\overline{b_{x'}(w)}$ and $c_y(w)\overline{c_{y'}(w)}$. Their angular derivative bounds use the scale $D_*$ from~\eqref{amplitude:eq:Dstar}. For $r=2^vs$ on the support of $\Psi$,
\[
 rd_+\gtrsim2^{v-h}=2^{2L_0}\ge R^{20\sigma}.
\]
If $2^LR^{-50\sigma}\le1$, then $D_*^2/(rd_+)\lesssim R^{-18\sigma}$; otherwise this ratio is at most a fixed multiple of $R^{-98\sigma}$. Also $D_*\le R^2$. These bounds verify~\eqref{circle:eq:far-hypotheses} for all sufficiently large $N$, uniformly in the relative positions of $X$ and $Y$.

If $d_+>2d_-$, Proposition~\ref{circle:prop:radial-kernel} gives a radial integral of size $O(R^{-Q})$ for every fixed $Q>0$, including when $d_-=0$. For comparable lengths, its main terms satisfy~\eqref{circle:eq:kernel-bound} and retain the support conditions at the stationary directions. The full-circle expansion of Corollary~\ref{circle:cor:full-circle} is applied to the original amplitudes.

\bigskip \paragraph{\emph{Length bands and positive counts.}}
Assume throughout the angular split that
\begin{equation}\label{reduce:eq:far-threshold}
 20\sigma N>9.
\end{equation}
For comparable lengths choose the unique integer $n$ with $2^{-n-1}<d_+\le2^{-n}$, and set $b=n+12$. The relevant bands are
\begin{gather}
 \mathcal J(a,h,v)=\{b\in\Z:a+11\le b\le h+7,\ b<v\},\label{reduce:eq:J}\\
 \Omega_b=\{K2^{-b}/2<d_+\le K2^{-b},\ d_+\le2d_-\}.\label{reduce:eq:Omega}
\end{gather}
Indeed~\eqref{reduce:eq:far-lengths} gives $n\ge a-1$ and $n\le h-5$. By~\eqref{reduce:eq:far-threshold}, $h+7<v$. Both lengths on $\Omega_b$ lie in $(K2^{-b}/4,K2^{-b}]$.

Replacing both endpoints of either displacement by their depth-$b$ square centers changes its length by at most $\sqrt2\,2^{-b}$. The two internal square pairs therefore satisfy~\eqref{positive:eq:separated}, with uniform geometric margins. In particular, the directions between each pair of squares lie in an arc small enough for a later inverse conversion.

Since $[b,v]\subseteq\mathcal B$, every index in $\M(b,v)$ belongs to $\M(\mathcal B)$ and has specified depth greater than $b>a$. It is therefore present in the original parent amplitudes.

For all $(z,z')$, define the Borel function
\begin{equation}\label{reduce:eq:chi}
 \chi_b(z,z')=
 \begin{cases}
 \displaystyle\prod_{\mathfrak t\in\M(b,v)}
 \mathbf1_{G_{i+1,\mathfrak t}(z)\cap G_{i+1,\mathfrak t}(z')}
       \bigl(e(z,z')\bigr),&z\ne z',\\[4pt]
 0,&z=z'.
 \end{cases}
\end{equation}
This is the good-pair restriction defining the positive count on $[b,v]$ at level $i+1$. Every product-rule term in a main stationary coefficient retains all original cutoff occurrences, possibly differentiated. By Lemma~\ref{amplitude:lem:support}, a nonzero coefficient at $\pm\Delta_j/d_j$ imposes every child condition at both endpoints of that displacement. Antipodal symmetry makes the two signs equivalent. Thus each nonzero main term has $\chi_b(x,x')\chi_b(y,y')=1$. No such support assertion is needed for the remainder.

The positive child measures are the original pair measures restricted by $\chi_b$, without the parent spatial multipliers or differentiated cutoff factors. Those factors are bounded in the coefficients.

\bigskip \paragraph{\emph{Recombination and normalization.}}
For $Z=X$ or $Y$, let $\mathcal E_b(Z)$ be the ordered occupied depth-$b$ square pairs in $Z$ satisfying~\eqref{positive:eq:separated}, and set
\[
 S_b(Z)=\sum_{(P,P')\in\mathcal E_b(Z)}
       \sqrt{\mathsf D_{i+1}^{P,P'}(b,v)},\qquad
 D_b(Z)=\max_{(P,P')\in\mathcal E_b(Z)}\mathsf D_{i+1}^{P,P'}(b,v).
\]
An empty sum or maximum is zero. Set $D_b^{\max}=\max\{D_b(X),D_b(Y)\}$.

\Needspace{270pt}
\begin{proposition}\label{reduce:prop:far}
For each fixed prescribed-order bound $B$ and each $S>0$, there are fixed constants and a common threshold such that
\begin{align}\label{reduce:eq:far-sum}
 |\mathcal F_h|
 &\le C_{B,S}\sum_{b\in\mathcal J(a,h,v)}
  2^{2(b-a)}\left(\frac{M(b)}{M(a)}\right)^4 S_b(X)S_b(Y)
  +C_{B,S}R^{-S},\\
 |\mathcal F_h|
 &\le C_{B,S}G_K^2\sum_{b\in\mathcal J(a,h,v)}
  R^{-2(A_b-A_a)+4C_{\rm prof}\sigma}\sqrt{D_b(X)D_b(Y)}
  +C_{B,S}R^{-S},\label{reduce:eq:far-max}\\
 G_K&=(8K+1)^2.\notag
\end{align}
The geometric mean may be replaced by $D_b^{\max}$. The coefficient $4C_{\rm prof}$ is independent of $B,S$ and all temporary differentiation orders.
\end{proposition}
\begin{proof}
On expanding the four spatial integrals, the model denominators of the depth-$h$ pieces cancel:
\[
 \left(\frac{M(h)}{M(a)}\right)^4 M(h)^{-4}=M(a)^{-4}.
\]
On $\Omega_b$, the stationary factor in~\eqref{circle:eq:kernel-bound} is at most $C2^{-(v-b)}$. Apart from an error $O(R^{-S})$, $|\mathcal F_h|$ is therefore bounded by the nonnegative expression
\begin{align}\label{reduce:eq:positive-far}
 C_{B,S}2^{2(v-a)}M(a)^{-4}\sum_b2^{-(v-b)}
 \int_{\text{far}\,\cap\Omega_b}
 (1+2^v|d_1-d_2|)^{-2}
 \chi_b(x,x')\chi_b(y,y')\,d\rho^{\otimes4}_{X,X,Y,Y}.
\end{align}
The four measures are the unnormalized restrictions to $X,X,Y,Y$, in that order.

The integrand in~\eqref{reduce:eq:positive-far} is nonnegative, so we can first remove the restriction to far labels. Summing the disjoint depth-$h$ pieces restores the original four-point measure. We keep $\Omega_b$ while partitioning the four points into depth-$b$ square pairs; its length bounds select pairs satisfying~\eqref{positive:eq:separated}. Only after selecting those pairs do we remove $\Omega_b$ from their positive integrals. This partitions the original measure rather than the depth-$h$ pieces, so it applies for both $b<h$ and $b>h$.

For a fixed depth-$b$ square, the center of a separated partner differs by at most $4K2^{-b}$. Each rescaled coordinate difference is an integer in $[-4K,4K]$, giving at most $G_K=(8K+1)^2$ partners. Requiring occupancy can only reduce this number. Lemma~\ref{profile:lem:descendants} therefore gives
\[
 \#\mathcal E_b(Z)\le G_KR^{I(a,b)+2C_{\rm prof}\sigma}.
\]
For fixed selected square pairs, after removing $\Omega_b$ as above, write $\lambda_1=\dd_\#\theta_{i+1}^{P,P'}(b,v)$ and $\lambda_2=\dd_\#\theta_{i+1}^{Q,Q'}(b,v)$. Lemma~\ref{scalar:lem:bins} gives
\begin{align}\label{reduce:eq:bin-far}
 \iint(1+2^v|s-t|)^{-2}\,d\lambda_1(s)d\lambda_2(t)
 &\le5\sqrt{B_{2^{-v}}(\lambda_1)B_{2^{-v}}(\lambda_2)}\notag\\
 &\le5M(b)^4 2^{-(v-b)}
 \sqrt{\mathsf D_{i+1}^{P,P'}(b,v)\mathsf D_{i+1}^{Q,Q'}(b,v)}.
\end{align}
The first inequality is~\eqref{scalar:eq:bins-kernel}. The second follows from~\eqref{scalar:eq:bins-lower} and the exact definition~\eqref{positive:eq:positive-count}, since $C_0\ge1$.

The product of the parent prefactor, stationary factor, and scalar factor is exactly
\[
 2^{2(v-a)}M(a)^{-4}\,2^{-(v-b)}\,
 M(b)^4 2^{-(v-b)}
 =2^{2(b-a)}\left(\frac{M(b)}{M(a)}\right)^4.
\]
Summation proves~\eqref{reduce:eq:far-sum}. Use the bound on $\#\mathcal E_b(Z)$ and~\eqref{reduce:eq:cancellation} to obtain~\eqref{reduce:eq:far-max}.

For the pointwise radial remainder, take
\[
Q=S+3+4C_{\rm prof}\sigma.
\]
The parent prefactor times its normalized four-point mass is at most $R^{2+4C_{\rm prof}\sigma}$, so the remainder is $O(R^{-S-1})$. Each four-point tuple belongs to a unique depth-$h$ product and, when the lengths are comparable, a unique band. The remainder is therefore integrated once, with no factor counting bands or squares. The main terms still require the sum over $b$.
\end{proof}
The depth-$h$ denominators cancel and the corresponding pieces are recombined before the depth-$b$ partition is made. Each resulting positive count uses the original class measure, condition set $\M(b,v)$, and level $i+1$.

\subsection{Terminal estimates and the low-frequency stopping term}\label{reduce:sec:terminal}
On a terminal current interval, the total mass gives the required bound. The Fourier comparison also contributes an additive term from frequencies below the starting depth.

\begin{proposition}\label{reduce:prop:terminal}
For every positive quantity and every angular quantity of combined prescribed order at most $B$,
\begin{align}
 \mathsf D_i^{X,Y}(a,t)&\le2^{t-a}R^{4C_{\rm prof}\sigma},\label{reduce:eq:Dtrivial}\\
 \mathsf H_i^{\mathcal B}(a,v)&\le C_{B,\Psi}2^{2(v-a)}R^{4C_{\rm prof}\sigma},
 \quad C_{B,\Psi}=(2\pi)^2\|\Psi\|_1U_B^4.\label{reduce:eq:Htrivial}
\end{align}
If the current interval is terminal, these give
\begin{align}
 \mathsf D_i^{X,Y}(a,t)&\le2^{100}R^{100\kappa+4C_{\rm prof}\sigma},\label{reduce:eq:Dterminal}\\
 \mathsf H_i^{\mathcal B}(a,v)&\le2^{200}C_{B,\Psi}R^{200\kappa+4C_{\rm prof}\sigma}.\label{reduce:eq:Hterminal}
\end{align}
Terminality concerns the current interval, not the width of $\mathcal B$.
\end{proposition}
\begin{proof}
The pair measure has mass at most $\rho(X)\rho(Y)$, so dropping the collision indicator and applying~\eqref{profile:eq:mass} gives~\eqref{reduce:eq:Dtrivial}. The transform bounds $|U_X|,|V_Y|\le U_BR^{C_{\rm prof}\sigma}$, the two circle masses $2\pi$, and the radial mass $\|\Psi\|_1$ give~\eqref{reduce:eq:Htrivial}. Finally, a terminal interval has width at most $100q_0\le100\kappa N+100$, which yields the remaining bounds.
\end{proof}
The profile factor comes from four normalized masses and is independent of the number of cutoff factors.

For a prefix choice,~\eqref{positive:eq:stopping} gives
\[
 2^{-a}\int_{|r|\le2^a}|\widehat\zeta(r)|^2\,dr
 \le2\omega_+^2R^{4C_{\rm prof}\sigma}.
\]
This is the additive stopping term in~\eqref{positive:eq:combined}.

\subsection{The six reductions}\label{reduce:sec:common}
We now collect the six estimates, retaining the frequency and depth sums.

We use $\mathcal H_{i,B}^{\mathcal B}(d,v\mid X,Y)$ from~\eqref{angular:eq:supremum}, with the exact condition set $\M_{>d}(\mathcal B)$ at both endpoints and combined order at most $B$. The class, model masses, cutoffs, level, base, and numerical derivative bounds are fixed. Its finiteness follows from~\eqref{angular:eq:crude}.

The spatial maxima are over finitely many occupied children, and the amplitude suprema are over the fixed family. Each frequency or depth sum has at most $N+1$ terms. For fixed parent amplitudes, restriction and factor removal in the direct and near estimates produce amplitudes in the corresponding child family, so taking its supremum gives an upper bound.

For a nonterminal positive count on $[a,t]$, let $p$ be the selected depth. In the prefix case, define
\[
 \mathcal L=\{v\in\Z:a<v\le t,\ t-v\ge\sigma N\},\qquad
 \mathcal U=\{v\in\Z:a<v\le t,\ 0\le t-v<\sigma N\}.
\]
These sets partition its frequency indices. Theorem~\ref{positive:thm:three} gives the following complete prefix bound in the present notation:
\begin{align}\label{reduce:eq:positive-full}
 \mathsf D_i^{X,Y}(a,t)
 &\le C_+\left(R^{4C_{\rm prof}\sigma}
       +\sum_{v\in\mathcal L}\mathsf D_{i+1}^{X,Y}(a,v)\right)\notag\\
 &\quad+C_T R^{-2(A_p-A_a)+k(a,p)+C_G\sigma}
       \sum_{v\in\mathcal U}
       \mathcal H_{i,B_{\rm inv}}^{[a,t]}(p,v\mid X,Y)
       +C_TR^{-T}.
\end{align}
In the suffix case, the bound is
\begin{equation}\label{reduce:eq:suffix}
 \mathsf D_i^{X,Y}(a,t)
 \le C_+R^{A_p-\min_{p\le n\le t}A_n+C_{\rm test}\sigma}
         \mathsf D_{i+1}^{X,Y}(a,p).
\end{equation}
The constant $C_+$ is independent of $T$ and the prescribed-order bound. The suffix and low-frequency estimates have no oscillatory error, and the low coefficient is fixed. For the high-frequency estimate, we apply circle conversion, Cauchy--Schwarz in the angular variable, and the graph estimate, and only then remove spatially constant factors. Here $B_{\rm inv}$ is chosen with $S=T+2$ in~\eqref{positive:eq:inverse-order}.

In~\eqref{reduce:eq:positive-full}, all frequency contributions are summed; an individual child need not bound the parent.

\Needspace{470pt}
\begin{theorem}\label{reduce:thm:six}
Fix the original class, its profile, and the condition lists, cutoffs, and geometric data specified above. Use the combined prescribed-order convention of Section~\ref{amplitude:sec:data}. For every $T>0$ and initial finite prescribed-order bound $B_{\rm in}$, there exist a finite bound $B_*\ge B_{\rm in}$, common cutoff derivative bounds through finitely many orders, and a common starting scale such that the following estimates hold.

Terminal positive quantities at levels $0\le i\le J_0$ and terminal angular quantities of combined prescribed order at most $B_*$ at levels $0\le i<J_0$ satisfy the bounds in Proposition~\ref{reduce:prop:terminal}. A nonterminal positive quantity at level $0\le i<J_0$ satisfies~\eqref{reduce:eq:suffix} or~\eqref{reduce:eq:positive-full}, according to its suffix or prefix choice, with the high-frequency suprema taken over amplitudes of combined prescribed order at most $B_*$.

For a nonterminal angular quantity at level $0\le i<J_0$ satisfying~\eqref{reduce:eq:base} with combined prescribed order at most $B_*$, the following alternatives hold.

If $h\le a$, then
\begin{equation}\label{reduce:eq:common-direct}
 \mathsf H_i^{\mathcal B}(a,v)
 \le C_T R^{-2(A_p-A_a)+k(a,p)+C_G\sigma}
       \mathcal H_{i,B_*}^{\mathcal B}(p,v\mid X,Y)+C_TR^{-T},
\end{equation}
where $p$ is obtained from a longest remaining condition as in Proposition~\ref{lists:prop:direct-depths}.

If $h>a$, then
\begin{align}\label{reduce:eq:common-angular-split}
 \mathsf H_i^{\mathcal B}(a,v)
 &\le C_T R^{-2(A_h-A_a)+4C_{\rm prof}\sigma}
       \mathcal H_{i,B_*}^{\mathcal B}(h,v\mid X,Y)\notag\\
 &\quad+C_T\sum_{b\in\mathcal J(a,h,v)}
       R^{-2(A_b-A_a)+4C_{\rm prof}\sigma}D_b^{\max}
       +C_TR^{-T}.
\end{align}
Each far contribution in $D_b^{\max}$ is a positive count on separated squares at level $i+1$, with condition set $\M(b,v)$ and no inherited derivatives. The near angular quantity retains the base $\mathcal B$, level $i$, and exactly the remaining condition set $\M_{>h}(\mathcal B)$.

The exponent coefficients $C_{\rm test}$, $C_G=C_{\rm test}+64+4C_{\rm prof}$, and $4C_{\rm prof}$ are independent of $B_{\rm in}$, $T$, and the temporary differentiation orders. The multiplicative constants and the starting scale may depend on these fixed parameters. These are one-step estimates; repeated reductions are controlled separately.
\end{theorem}

\begin{proof}
Choose the common orders as in Subsection~\ref{reduce:sec:uniformity}. The positive-count estimates follow from Theorem~\ref{positive:thm:three} by enlarging the high-frequency amplitude suprema to the family with combined prescribed order at most $B_*$. In the direct case, apply Proposition~\ref{reduce:prop:direct} and bound its child amplitudes by the defining supremum. In the split case use the exact identity~\eqref{reduce:eq:split-H}, the triangle inequality, Proposition~\ref{reduce:prop:near}, and Proposition~\ref{reduce:prop:far}. The factors $61^4U_{B_*}^4$, $G_K^2$, and the kernel constant are absorbed into $C_T$ without changing the powers of $R$. The terminal and stopping terms were proved in Subsection~\ref{reduce:sec:terminal}.
\end{proof}

\Needspace{365pt}
The table records the changes in each contribution. Except in the low case, the exponent has the form $-2(A_{a'}-A_a)+e+C\sigma$, where $a'$ is the new starting depth and $e$ is listed below. The frequency and depth sums are unchanged.
\begin{longtable}{@{}P{17mm}P{42mm}P{42mm}P{39mm}@{}}
\toprule Case & Change of interval and type & Additional exponent $e$ & List and level\\\midrule\endfirsthead
\toprule Case & Change of interval and type & Additional exponent $e$ & List and level\\\midrule\endhead
Suffix & positive $[a,t]$ to positive $[a,p]$ & $A_p-\min_{[p,t]}A$ & $\M(a,p)$, level $i+1$\\[4pt]
Low & positive $[a,t]$ to positive $[a,v]$ & zero; fixed coefficient, no $C\sigma$ & $\M(a,v)$, level $i+1$\\[4pt]
High & positive $[a,t]$ to angular $[p,v]$ & $k(a,p)$ & base $[a,t]$, remaining list, level $i$\\[4pt]
Direct & angular $[a,v]$ to angular $[p,v]$ & $k(a,p)$ & same base, remaining list, level $i$\\[4pt]
Near & angular $[a,v]$ to angular $[h,v]$ & zero & same base, remaining list, level $i$\\[4pt]
Far & angular $[a,v]$ to positive $[b,v]$ & zero & $\M(b,v)$, level $i+1$, no derivatives\\
\bottomrule
\end{longtable}
The structural coefficients are $C_{\rm test}$ for suffix, $C_G$ for high and direct, and $4C_{\rm prof}$ for near and far. These rows can all use the upper bound $C_G$. The low-frequency coefficient remains a fixed constant, as needed for the iteration.

\subsection{Choosing the orders and the starting scale}\label{reduce:sec:uniformity}
We first fix $T$ and choose the necessary finite derivative orders, then take a starting scale common to all permitted measures and amplitudes.

\bigskip \paragraph{\emph{Prescribed orders.}}
We fix $T>0$ and set
\begin{equation}\label{reduce:eq:common-orders}
 S=T+2,\qquad
 B_* =\max\{B_{\rm in},B_{\rm inv}(S)\},
\end{equation}
where $B_{\rm inv}(S)=2(J_{\rm pos}-1)$, with $J_{\rm pos}$ chosen in~\eqref{positive:eq:inverse-order} for $S$ and mass-ratio exponent $\mu=C_{\rm prof}\sigma$. This is the combined prescribed order in the conversion~\eqref{circle:eq:orders}; it is finite and independent of $N$.

A high-frequency conversion starts with undifferentiated cutoffs and produces combined order at most $B_{\rm inv}(S)$. The direct and near estimates restrict the original amplitudes, and removing factors can only decrease their combined prescribed order. The far estimate uses additional derivatives but produces positive counts with no prescribed derivatives. A subsequent conversion therefore starts with undifferentiated cutoffs again. All resulting amplitude pairs have combined prescribed order at most $B_*$.

These are amplitudes of the exact form~\eqref{amplitude:eq:form}, with no additional angular multiplier.

\bigskip \paragraph{\emph{Temporary derivatives and the starting scale.}}
For the direct and high-frequency reductions, we apply the graph estimate with error exponent $S$. Choose an integer integration-by-parts order
\[
 m_{\rm gr}>\frac{S+21}{26\sigma}.
\]
These derivatives estimate the discarded terms and do not enter the child amplitudes.

For the far contribution, use a normalized error $O(R^{-S})$ and a pointwise radial error $O(R^{-Q})$, where
\[
Q=S+3+4C_{\rm prof}\sigma.
\]
Let $J_{\rm dir}$ be the number of terms chosen in the direct circle expansion for the far contribution to give a pointwise remainder $O(R^{-Q})$ for parent amplitudes of combined prescribed order at most $B_*$. We also choose the radial integration order to obtain the same remainder bound. The ratios estimated in Section~\ref{reduce:sec:far} are bounded by fixed negative powers of $R$, so both orders can be chosen independently of $N$. It suffices to take the original cutoff bounds through the largest of $2J_{\rm pos}+4$, $B_*+m_{\rm gr}$, and $2B_*+2J_{\rm dir}+4$.

Finally, we take $N$ above the thresholds for the class and cutoff estimates, the analytic theorems at these orders, $4N^3\le R^\sigma$, and~\eqref{reduce:eq:far-threshold}. All constants and thresholds are then common to the allowed classes, profiles, squares, condition sets, spatial multipliers, and levels.

\bigskip \paragraph{\emph{Errors and smoothing levels.}}
A high-frequency contribution has normalized error $C_SR^{-S}$. There are at most $N$ frequency indices, so
\[
 NR^{-S}\le R^{-T-1}\le R^{-T}.
\]
The direct graph error is already normalized and is $O(R^{-S})$. The near estimates have no additive remainder. In the far argument,
\[
 R^{2+4C_{\rm prof}\sigma}R^{-Q}=R^{-S-1}.
\]

Each far four-point tuple contributes to the remainder once; for comparable lengths, the associated band is unique. Thus no factor counting bands multiplies the remainder.

When the circle conversion is followed by the graph estimate, only a fixed constant multiplies the graph remainder. The power of $R$ and the bound for removing factors multiply only the main term, not the additive error.

In the iteration, earlier coefficients multiply later errors. Section~\ref{main:iteration} bounds these products and chooses $T$ accordingly, before fixing the derivative orders and starting scale.

A nonterminal quantity can be reduced only when $i<J_0$. Positive counts at level $J_0$ are defined by their good sets, but no cutoff at that level is available for a further circle conversion. The iteration must therefore also bound the number of level increases.

The estimates are uniform over individual amplitudes with $B_*$ and the cutoff derivative bounds fixed, so they pass to the corresponding suprema without an additional multiplicity factor or a maximizing amplitude.

\subsection{Examples for the angular reductions}\label{reduce:sec:examples}
The following examples explain the sign of the near contribution and the need to retain the base interval and original amplitudes.

In one coordinate, take centers at $0,30,60$ in units of the smaller-square spacing. The matrix indicating which pairs of centers are at distance at most $30$ is
\[
 E=\begin{pmatrix}1&1&0\\1&1&1\\0&1&1\end{pmatrix}.
\]
For $z=(1/2,-1,1/2)$,
\[
 z^*Ez=-\frac12.
\]
Thus a symmetric selection that includes the diagonal can have a negative quadratic form. Taking just one square at the other endpoint gives the same quadratic form in the four-square expansion. This example concerns the sign of a selected sum; it need not satisfy all the distance-construction hypotheses. The bound~\eqref{reduce:eq:near-finite} holds regardless of the sign.

The cancellation example~\eqref{amplitude:eq:cancellation-example} concerns removal inside a shared-angle expression. In the present decomposition, factors are removed only from the nonnegative bound for the near contribution. The far contribution still uses the original amplitudes from~\eqref{reduce:eq:splitting}.

The far contribution uses depths $a+11\le b\le h+7$, which may exceed $h$. For example, if $a=0$, $h=110$, $v=130$, $L_0=10$, and comparable lengths have $d_+=32\,2^{-110}=2^{-105}$, the chosen band has $b=117=h+7$. The lengths lie well within a unit parent's diameter, and a center displacement of $32\,2^{-110}$ can occur between depth-$110$ descendants. Parameters with $\sigma N=\kappa N=1$ give $q_0=1$, a nonterminal current interval, and the stipulated far margin. An increasing excess array has first eligible depth $a+1$ when $q_0=1$, so its generated parallel indices have length one. On the base $[0,130]$ its largest list length is therefore one, and the lower bound $\lceil10\sigma N\rceil=10$ in $L_0$ gives the displayed splitting depths.

This example illustrates that the band selection can give $b>h$; it is not intended to satisfy every large-$N$ analytic hypothesis. The proof handles either ordering by canceling the depth-$h$ denominators and recombining the original measure in a nonnegative bound before partitioning at depth $b$.

Theorem~\ref{reduce:thm:six} provides the six local reductions with their model normalizations, condition sets, and smoothing-level changes. The positive counts and auxiliary measure $\zeta$ are used to estimate the class-pair components of $\tau_N$. Section~\ref{main:iteration} controls the number of reductions, the accumulated profile exponents, and the coefficients preceding the local errors.

\clearpage

\section{Control of the iteration}\label{main:iteration}
We combine the six local reductions of Theorem~\ref{reduce:thm:six} to obtain a uniform bound for one class. A minimum over pairs of separated depths controls the additional exponent along each branch. We bound the number of reductions and smoothing levels, then sum the main contributions and the errors multiplied by the preceding coefficients.

The minimum comparison and the counting argument apply to arbitrary finite real arrays, with the condition sets of Section~\ref{main:profiles}. Profile regularity is used only in the endpoint comparison and the analytic estimates. The minimum is always taken over the current interval, even when an angular quantity retains conditions from a larger base interval.

In Subsection~\ref{summation:subsec:envelopes}, we apply these results to the analytic inequalities, keeping the profile exponents, structural losses in powers of $R$, and multiplicative constants separate. The local error exponent is chosen before the derivative orders and common starting scale. Section~\ref{main:frequency} then combines the bounds for two classes.

\subsection{Fixed data and properties of the condition sets}\label{iteration:sec:lists}
Fix integers $N\ge1$ and $q_0\ge1$, a real number $0<\epsilon\le q_0$, and a real array $A_0,\ldots,A_N$. All intervals have integer endpoints in $[0,N]$, and their width is the difference of the endpoints. As in Section~\ref{main:profiles}, an interval is terminal when its width is at most $100q_0$. In the application,
\begin{equation}\label{iteration:eq:parameters}
 q_0=\lceil\kappa N\rceil,\qquad
 \epsilon=\sigma N,\qquad 0<\sigma\le\kappa<1.
\end{equation}
Here $\epsilon$ is a gap between frequency depths, not the dimension margin $\varepsilon$ in the introduction. The integer ceiling in $q_0$ is retained throughout.

The blocks, representatives, and interval choices are those in~\eqref{lists:eq:blocks} and~\eqref{lists:eq:block-rep}. We write
\begin{equation}\label{iteration:eq:representatives}
 K_{\rm blk}=\left\lceil\frac{N+1}{q_0}\right\rceil
             =\left\lfloor\frac N{q_0}\right\rfloor+1,
 \qquad \mathcal R=\{r_0,\ldots,r_{K_{\rm blk}-1}\}.
\end{equation}
The array, blocks, and representatives remain fixed along a branch. In the application, a change from an angular quantity to a positive count also retains the class.

For a nonterminal interval $[a,t]$, let $[a_*,t_*]$ be the union of its eligible blocks and let $p$ be its selected depth. Theorem~\ref{lists:thm:main} gives
\begin{gather}
 a+q_0\le a_*\le a+2q_0-1,\qquad
 t-2q_0+1\le t_*\le t-q_0,\notag\\
 a_*\le p\le t_*,\qquad A_n\ge A_p\quad(a_*\le n\le t_*).\label{iteration:eq:inner-recall}
\end{gather}
The selected depth minimizes $A$ over the eligible blocks. Smaller values near the ends of $[a,t]$ can produce an additional exponent in a reduction.

If a condition of length $\ell'$ specified at $p$ belongs to $\M(\mathcal B)$, the same theorem gives a supplying interval $[b_0,u_0]\subseteq\mathcal B$ satisfying
\[
 [p-\ell',p+\ell']\subseteq[b_0,u_0].
\]
Increasing $a$ in $\M_{>a}(\mathcal B)$ removes conditions without changing the lengths of those that remain. Proposition~\ref{lists:prop:active} also gives
\begin{equation}\label{iteration:eq:remaining-nonempty}
 \varnothing\ne\M(a,v)\subseteq\M_{>a}(\mathcal B)
 \quad\hbox{when }[a,v]\subseteq\mathcal B\hbox{ is nonterminal}.
\end{equation}
These are the properties of the condition sets needed for the minimum and counting arguments below.

\subsection{The six transitions and their depth bounds}\label{iteration:sec:transitions}\label{iteration:sec:depths}
A branch follows one contribution from each local bound. We first control such branches; the complete estimate will require summing all of them.

A positive node represents a positive count and is specified by a current interval $[a,t]$ and a smoothing level $i$. An angular node represents an angular quantity and is specified by a current interval $[a,v]$, a level $i$, and a base $\mathcal B=[a_0,t_0]$ satisfying
\begin{equation}\label{iteration:eq:base}
 a_0\le a\le v\le t_0,\qquad t_0\le v+\epsilon.
\end{equation}
Its remaining condition set is $\M_{>a}(\mathcal B)$. These data describe the quantities in Theorem~\ref{reduce:thm:six}; only their interval, type, and level changes enter the present finite argument.

For the finite argument, we initially allow every nonnegative integer $i$. Subsection~\ref{iteration:sec:levels} shows that a branch starting at level zero uses only the available smoothing levels.

At a nonterminal positive node, let $p$ be the selected depth. If $2p>a+t$, the suffix child is positive on $[a,p]$, at level $i+1$, with additional exponent
\begin{equation}\label{iteration:eq:suffix-e}
 e_{\rm suffix}=A_p-\min_{p\le n\le t}A_n.
\end{equation}

If $2p\le a+t$, the frequency indices are the integers $a<v\le t$. Those with $t-v\ge\epsilon$ give positive low children on $[a,v]$, at level $i+1$, with $e=0$. Those with $0\le t-v<\epsilon$ give angular high children on $[p,v]$, with base $[a,t]$, unchanged level $i$, and
\begin{equation}\label{iteration:eq:high-e}
 e_{\rm high}=k(a,p):=A_p-\min_{a\le n\le p}A_n.
\end{equation}
A prefix node also has one additive low-frequency stopping term. Choosing it ends the branch without a transition or level increase.

At a nonterminal angular node, we let $L$ be the largest condition length in $\M_{>a}(\mathcal B)$ and set
\begin{equation}\label{iteration:eq:h}
 L_0=\max\{L,\lceil10\epsilon\rceil\},\qquad h=v-2L_0.
\end{equation}
If $h\le a$, choose a remaining index of maximal length, let $p$ be its specified depth, and choose a supplying interval $[b_0,u_0]\subseteq\mathcal B$. The direct child is angular on $[p,v]$, with the same base and level, and additional exponent
\begin{equation}\label{iteration:eq:direct-e}
 e_{\rm direct}=k(a,p).
\end{equation}

If $h>a$, the near child is angular on $[h,v]$, with the same base and level, and $e=0$. The far children are positive on $[b,v]$, at level $i+1$, with condition set $\M(b,v)$ and $e=0$. They have no earlier base. Their integer starting depths satisfy
\begin{equation}\label{iteration:eq:far-range}
 a+11\le b\le h+7,\qquad b<v.
\end{equation}
This is the range~\eqref{reduce:eq:J} from Proposition~\ref{reduce:prop:far}. For the counting argument, we allow every depth in this range, even if its spatial child family is empty. This can only add branches.

The rules preserve containment of the current intervals and the base condition. A high step gives a base ending less than $\epsilon$ beyond the child endpoint; direct and near steps keep that endpoint and base fixed. A main-term branch ends at a terminal node or a low-frequency stopping term. We also count finite prefixes, including those ending at an additive error. No transition follows a terminal node.

Along every branch the starting depth is nondecreasing and the ending depth is nonincreasing, with at least one strict change at each transition. The depth estimates below verify these inequalities. Hence widths strictly decrease through integers, so every branch is finite at fixed $N$.

A change of type may change the base and condition set, but retains the array and representatives, and in the application the class as well. Every condition in a new positive set is still specified strictly beyond its current starting depth.

\paragraph*{Depth inequalities.}
We verify the containment and width bounds needed for the two-depth minimum below.

\begin{lemma}\label{iteration:lem:positive-widths}
A suffix child has width greater than $50q_0$. A high child has width greater than $49q_0$. Both are contained in their current parent.
\end{lemma}
\begin{proof}
For a suffix child, $p-a>(t-a)/2>50q_0$. Also $p\le t-q_0<t$. In a prefix case, $p\le(a+t)/2$, so for a high child
\[
 v-p\ge(t-a)/2-(t-v)>50q_0-\epsilon\ge49q_0.
\]
Thus $p<v$. The inequalities $a<p\le v\le t$ prove containment.
\end{proof}

Proposition~\ref{lists:prop:direct-depths} applies to the same earlier base and current interval. In a nonterminal direct case it gives $L=L_0$, $a<p<v$, and
\begin{gather}\label{iteration:eq:direct-geometry}
 v-a\le2L,\qquad v-p>49q_0,\\
 b_0\le p-L,\qquad p+L\le u_0\le v+\epsilon.
 \label{iteration:eq:direct-provider}
\end{gather}
Indeed, the geometric inequalities give
\[
 p-a\le v-a+\epsilon-L\le (v-a)/2+\epsilon,
 \qquad
 v-p\ge(v-a)/2-\epsilon>50q_0-\epsilon\ge49q_0.
\]
The supplying interval lies in the earlier base and may extend beyond the current interval. It identifies a profile minimum, but the comparison below concerns the current parent and child.

For a near child, $a<h<v$ and its width is exactly $2L_0$. A far child obeys $a<b<v$ by \eqref{iteration:eq:far-range}. A low child obeys $a<v<t$ because $t-v\ge\epsilon>0$. Together these facts establish that the interval width strictly decreases at every transition.

\subsection{The minimum of two separated profile values}
\label{iteration:sec:minimum}
The minimum of $A$ on an interval may remain unchanged when a reduction incurs a positive additional exponent. We instead minimize the sum of two profile values at separated depths.

For an interval $J=[a,t]$ with $t-a\ge10q_0$, set
\begin{equation}\label{iteration:eq:V}
 V(J)=\min\{A_r+A_s:a\le r<s\le t,\ s-r\ge10q_0\}.
\end{equation}
The admissible set is finite and nonempty, including when $t-a=10q_0$. We leave $V(J)$ undefined on shorter intervals.

If $J'\subseteq J$ and both minima are defined, the admissible pairs for $J'$ form a subset of those for $J$. Therefore
\begin{equation}\label{iteration:eq:V-containment}
 V(J)\le V(J').
\end{equation}
This proves the comparison for a zero-exponent transition whenever the child minimum is defined.

\Needspace{170pt}
\begin{theorem}\label{iteration:thm:minimum}
For every suffix, high, or direct transition described in Subsection~\ref{iteration:sec:transitions}, the current child interval has width greater than $49q_0$, and
\begin{equation}\label{iteration:eq:minimum-step}
 e\le V(J')-V(J).
\end{equation}
For a low, near, or far transition, $e=0$ and \eqref{iteration:eq:minimum-step} holds whenever $V(J')$ is defined. A transition to width less than $10q_0$ must have $e=0$ and is already terminal.
\end{theorem}
\begin{proof}
The width bounds follow from Subsection~\ref{iteration:sec:depths}, and containment handles the zero-exponent cases. For each other type, we compare an arbitrary admissible child pair with a suitable parent pair.

\bigskip \paragraph{\emph{Suffix.}}
Here $J=[a,t]$ and $J'=[a,p]$. Let $(r,s)$ be admissible in $J'$. Since
\[
 s\ge a+10q_0\ge a_*,\qquad s\le p\le t_*,
\]
the inner-interval minimum gives $A_s\ge A_p$. Choose $z\in[p,t]$ attaining $\min_{[p,t]}A$. Then $z\ge p\ge s$, so $(r,z)$ is admissible in $J$, and
\[
 A_r+A_z\le A_r+A_s-\left(A_p-\min_{[p,t]}A\right).
\]
Applying this to a minimizing child pair proves~\eqref{iteration:eq:minimum-step}.

\bigskip \paragraph{\emph{High.}}
Here $J=[a,t]$ and $J'=[p,v]$. An admissible child pair satisfies
\[
 p\le r\le v-10q_0\le t-10q_0\le t_*.
\]
It also has $r\ge p\ge a_*$, so $A_r\ge A_p$. Choose $z\in[a,p]$ attaining $\min_{[a,p]}A$. Then $z\le p\le r$, hence $(z,s)$ is admissible in $J$, and
\[
 A_z+A_s\le A_r+A_s-k(a,p).
\]
A minimizing child pair again proves the claim.

\bigskip \paragraph{\emph{Direct.}}
Here $J=[a,v]$ and $J'=[p,v]$. If $k(a,p)=0$, use containment. Suppose instead that $k(a,p)>0$ and choose $z\in[a,p]$ with
\[
 A_z=\min_{[a,p]}A<A_p.
\]
Let $[b_*,u_*]$ be the union of the eligible blocks in the supplying interval $[b_0,u_0]$. On this union, $A_n\ge A_p$, and $p\in[b_*,u_*]$. Since $z\le p$ and $A_z<A_p$, we must have $z<b_*$, whether or not $z$ lies in the supplying interval. By \eqref{lists:eq:sharp-endpoints},
\begin{equation}\label{iteration:eq:left-low-value}
 a\le z<b_0+2q_0.
\end{equation}
Using $v\le a+2L$ and the length of the supplying interval, we get
\begin{equation}\label{iteration:eq:right-control}
 v\le a+2L<b_0+2q_0+2L\le u_0+2q_0.
\end{equation}
Let $(r,s)$ be admissible in the current child $[p,v]$. By~\eqref{iteration:eq:right-control},
\[
 p\le r\le v-10q_0<u_0-8q_0.
\]
Since $b_*\le p$ and $u_*\ge u_0-2q_0+1$, these inequalities place $r$ in $[b_*,u_*]$. Hence $A_r\ge A_p$.

The pair $(z,s)$ lies in the current parent interval $[a,v]$, since $z\ge a$, $z\le p\le r$, and $s\le v$. Its separation is at least that of $(r,s)$, and
\[
 A_z+A_s\le A_r+A_s-k(a,p).
\]
Taking a minimizing child pair proves \eqref{iteration:eq:minimum-step} with $V([a,v])$, not $V(\mathcal B)$ or $V([b_0,u_0])$.
\end{proof}

The additional exponents are nonnegative. No bound on the size or regularity of $A$ was used in Theorem~\ref{iteration:thm:minimum}.

\subsection{The last defined minimum along a branch}
\label{iteration:sec:endpoint}
If the last interval is too short to define $V$, we stop the minimum comparison at the preceding interval. The final transition then has exponent zero.

Consider a branch from a nonterminal interval $J_{\rm in}=J^{(0)}$ through current intervals $J^{(1)},\ldots,J^{(m)}$, with additional exponents $e_j$ on its transitions. We write $J_*=[u,w]$ for the last interval of width at least $10q_0$.

The initial interval has width at least $10q_0$. Any later shorter interval is terminal and is reached by a zero-exponent transition. Theorem~\ref{iteration:thm:minimum} therefore telescopes to
\begin{equation}\label{iteration:eq:sum-e}
 \sum_{j=0}^{m-1}e_j\le V(J_*)-V(J_{\rm in}).
\end{equation}
If a low-frequency term or an error ends the branch at a long interval, that interval is $J_*$. The estimate also holds when there are no transitions.

We now use the profile of a class. Equations~\eqref{profile:eq:data} and~\eqref{profile:eq:Alip} give
\[
 A_n=m(n/N)-n/N,\qquad
 |A_s-A_r|\le |s-r|/N,\qquad A_0=0,\qquad |A_n|\le1.
\]
The Lipschitz constant one follows from monotonicity and the $2$-Lipschitz bound of $m$. The excess array $A$ itself need not be monotone.

\begin{lemma}\label{iteration:lem:endpoint}
Let $[u,w]$ have width at least $10q_0$, and let $a_f\in[u,w]$ be any integer. Under \eqref{profile:eq:Alip},
\begin{equation}\label{iteration:eq:endpoint}
 V([u,w])\le2A_{a_f}+\frac{10q_0}{N}.
\end{equation}
\end{lemma}
\begin{proof}
Set $d=10q_0$ and choose
\[
 r=\min\{a_f,w-d\},\qquad s=r+d.
\]
Since $w-d\ge u$ and $a_f\ge u$, we have
\[
 u\le r\le a_f\le s\le w,\qquad s-r=d.
\]
Thus $(r,s)$ is admissible, and
\[
 |r-a_f|+|s-a_f|=d.
\]
Using \eqref{profile:eq:Alip} gives
\[
 V([u,w])\le A_r+A_s
 \le2A_{a_f}+\frac{|r-a_f|+|s-a_f|}{N},
\]
which is \eqref{iteration:eq:endpoint}.
\end{proof}

Let $a_{\rm in}$ and $a_f$ be the initial and final starting depths. Every transition has principal exponent
\[
 -2(A_{a_{j+1}}-A_{a_j})+e_j;
\]
for low and suffix transitions the difference of starting values is zero. Nested intervals imply $a_f\in J_*$. Combining \eqref{iteration:eq:sum-e} with Lemma~\ref{iteration:lem:endpoint} gives
\begin{equation}\label{iteration:eq:principal-branch}
 \sum_j\bigl[-2(A_{a_{j+1}}-A_{a_j})+e_j\bigr]
 \le2A_{a_{\rm in}}-V(J_{\rm in})+\frac{10q_0}{N}.
\end{equation}
It remains to estimate the structural $C\sigma$ losses, terminal factors, multiplicative constants, number of contributions, and additive errors. For this, we first bound the branch length uniformly in $N$.

\subsection{A uniform bound on the number of reductions}
\label{iteration:sec:branch-count}
The left endpoint is nondecreasing and the right endpoint is nonincreasing. Their total movement bounds most transition counts. To count near steps, we also use the longest remaining condition.

For a complete branch or a finite prefix, let
\[
 n_S,n_L,n_H,n_D,n_N,n_F
\]
denote the numbers of suffix, low, high, direct, near, and far transitions. Put $U=n_H+n_D$. A stopping term is not counted as a transition.

\Needspace{280pt}
\begin{theorem}\label{iteration:thm:count}
For the finite rules of Subsection~\ref{iteration:sec:transitions}, every branch and every finite prefix satisfy
\begin{align}
 U&\le K_{\rm blk},\label{iteration:eq:count-U}\\
 n_S&\le\left\lfloor N/q_0\right\rfloor=K_{\rm blk}-1,\label{iteration:eq:count-S}\\
 n_L&\le\left\lfloor N/\lceil\epsilon\rceil\right\rfloor,\label{iteration:eq:count-L}\\
 n_F&\le n_H+1,\label{iteration:eq:count-F}\\
 n_N&\le K_{\rm blk}+n_D+1.\label{iteration:eq:count-N}
\end{align}
Consequently,
\begin{equation}\label{iteration:eq:nonlow}
 n_S+n_H+n_D+n_N+n_F\le4K_{\rm blk}+1.
\end{equation}
In the application \eqref{iteration:eq:parameters}, these imply
\begin{equation}\label{iteration:eq:original-count}
 n_{\rm nonlow}<20/\kappa,\qquad n_L\le1/\sigma.
\end{equation}
\end{theorem}

\begin{proof}
A high or direct step strictly increases the starting depth to a member of the fixed set $\mathcal R$. Since the starting depth never decreases, no representative can be selected twice in this way. This proves~\eqref{iteration:eq:count-U}.

A far step does not change this conclusion: its positive condition set uses the same representatives, and its starting depth is larger.

Each suffix choice satisfies $p\le t-q_0$ and reduces the ending depth by at least $q_0$. Since that endpoint is nonincreasing and its total decrease is at most $N$, we obtain~\eqref{iteration:eq:count-S}.

A low step has the integer decrease $t-v\ge\epsilon$, hence $t-v\ge\lceil\epsilon\rceil$. The same total-decrease argument proves \eqref{iteration:eq:count-L}. With $\epsilon=\sigma N$, this is at most $N/\epsilon=1/\sigma$.

Each maximal consecutive sequence of angular nodes starts at the initial node or after a high step. Direct and near steps remain within the sequence; a far step ends it by producing a positive node. There are at most $n_H+1$ such sequences and at most one far step in each, proving~\eqref{iteration:eq:count-F}. A terminal far child ends the branch; a nonterminal one may later lead to another high step.

\bigskip \paragraph{\emph{Counting the near steps.}}
For~\eqref{iteration:eq:count-N}, consider a near transition from $[a,v]$ to $[h,v]$, where $h=v-2L_0>a$. If its child is nonterminal, then
\[
 v-h=2L_0>100q_0.
\]
Since $2\lceil10\epsilon\rceil\le20q_0$, the lower bound $\lceil10\epsilon\rceil$ cannot determine $L_0$, and $L_0=L$.

We let $L'$ be the largest condition length in $\M_{>h}(\mathcal B)$. This set is nonempty because the child is nonterminal. Removing conditions does not increase the largest length, so $L'\le L$.

If $L'<L$, every old index of length $L$ was removed. Choose one such index and let $p_*$ be its specified depth. It satisfies
\begin{equation}\label{iteration:eq:passed-rep}
 a<p_*\le h,\qquad p_*\in\mathcal R.
\end{equation}
Associate this near step with $p_*$. Every subsequent starting depth is at least $h$, and every remaining or newly selected condition is specified above that depth. Thus no later near step can be assigned to $p_*$. There are at most $K_{\rm blk}$ steps of this type.

If $L'=L$, recomputing the split at the child gives
\[
 v-2\max\{L',\lceil10\epsilon\rceil\}=v-2L=h.
\]
The child's next transition, if taken, is therefore direct. We associate the near step with that following direct step; different near steps give different direct steps.

A terminal near child ends the branch. A finite prefix may instead end with a nonterminal near child whose following direct step is not included. There is at most one of these two final possibilities. Thus the second form, together with a possible terminal or unmatched final near step, contributes at most $n_D+1$.

Combining the two cases gives $n_N\le K_{\rm blk}+n_D+1$ for branches and finite prefixes.

By \eqref{iteration:eq:count-F}--\eqref{iteration:eq:count-N},
\begin{align*}
 n_{\rm nonlow}
 &=U+n_S+n_N+n_F\\
 &\le U+n_S+(K_{\rm blk}+n_D+1)+(n_H+1)\\
 &=2U+n_S+K_{\rm blk}+2\\
 &\le2K_{\rm blk}+(K_{\rm blk}-1)+K_{\rm blk}+2=4K_{\rm blk}+1.
\end{align*}
In the application, $q_0\ge\kappa N$ gives
\[
 K_{\rm blk}=\lfloor N/q_0\rfloor+1\le\kappa^{-1}+1.
\]
Hence
\[
 4K_{\rm blk}+1\le4\kappa^{-1}+5<9\kappa^{-1}<20\kappa^{-1}
 \qquad(0<\kappa<1).
\]
This proves~\eqref{iteration:eq:original-count}.
\end{proof}

\subsection{All required smoothing levels are available}
\label{iteration:sec:levels}
Only suffix, low, and far transitions increase the smoothing index. High, direct, and near transitions preserve it.

\Needspace{230pt}
\begin{corollary}\label{iteration:cor:levels}
Starting from index $i_{\rm in}$, the largest reached index is at most
\begin{equation}\label{iteration:eq:index-increase}
 i_{\rm in}+2K_{\rm blk}+\left\lfloor\frac N{\lceil\epsilon\rceil}\right\rfloor.
\end{equation}
For \eqref{iteration:eq:parameters} and $i_{\rm in}=0$, every used index and every next support level are strictly below
\[
 J_0=\left\lceil100(\kappa^{-1}+\sigma^{-1})+100\right\rceil,
\]
which is the number already fixed in~\eqref{cutoff:eq:starting-scale}. A common bound for the total number of transitions, with an allowance for the final node, is
\begin{equation}\label{iteration:eq:Jrec}
 J_{\rm rec}=\left\lceil20/\kappa+1/\sigma\right\rceil+2
\end{equation}
\end{corollary}
\begin{proof}
The total increase is $n_S+n_L+n_F$. The detailed bounds give
\[
 n_S+n_L+n_F
 \le(K_{\rm blk}-1)+n_L+(n_H+1)
 \le2K_{\rm blk}+n_L,
\]
proving \eqref{iteration:eq:index-increase}. The next support level is at most one larger. At $i_{\rm in}=0$ it is therefore at most
\[
 2K_{\rm blk}+1/\sigma+1\le2/\kappa+1/\sigma+3<J_0.
\]
The total number of transitions is less than $20/\kappa+1/\sigma$, so~\eqref{iteration:eq:Jrec} also allows for the final node.
\end{proof}

The counting argument did not restrict the integer levels. It now ensures that every cutoff required along a branch starting at zero has level below $J_0$.

For an arbitrary starting index, the condition
\[
 i_{\rm in}+2K_{\rm blk}+\left\lfloor N/\lceil\epsilon\rceil\right\rfloor+1<J_0
\]
is sufficient, as is the coarser condition $i_{\rm in}+J_{\rm rec}+1<J_0$. The condition $i_{\rm in}<J_0$ alone need not leave enough levels for the full branch.

\Needspace{310pt}
\subsection{Examples for the minimum comparison}
\label{iteration:sec:examples}
The first two examples use arrays obtained from nondecreasing $2$-Lipschitz functions $m$. They concern the finite argument, without asserting that these profiles arise from class measures.

To see why one minimum is insufficient, take $N=121$, $q_0=1$, and
\begin{equation}\label{iteration:eq:single-example}
 A_n=\frac{\min\{2n,121-n\}}{242}.
\end{equation}
The corresponding function is
\[
 m(u)=u+\min\{u,(1-u)/2\}.
\]
It is nondecreasing and $2$-Lipschitz, with slopes $2$ and $1/2$. The eligible depths of $[0,121]$ are $[1,120]$. Their unique minimum occurs at $p=120$, so this is a suffix choice. Its additional exponent is
\[
 e=A_{120}-A_{121}=1/242>0.
\]
But
\[
 \min_{0\le n\le121}A_n=\min_{0\le n\le120}A_n=0,
\]
so the increase of the ordinary one-depth minimum is zero. It cannot bound $e$.

For the two-depth minimum,
\[
 V([0,121])=0,\qquad V([0,120])=1/242.
\]
The admissible pairs $(0,121)$ and $(0,120)$ attain these values. Thus the two-depth comparison is an equality in this example, whereas the minimum over a single depth does not control the additional exponent.

To distinguish the base from the current interval, take $N=220$, $q_0=2$, $\epsilon=1$, and
\begin{equation}\label{iteration:eq:direct-example}
 A_n=\frac{\min\{n,1+|n-110|\}}{220}.
\end{equation}
Since $|A_{n+1}-A_n|\le1/220$, piecewise linear interpolation of $m(n/220)=A_n+n/220$ gives a nondecreasing $2$-Lipschitz profile. Let the earlier base be $[0,220]$ and the current angular interval be $[0,219]$. The base's selected depth is $p=110$, and it supplies $\Par0{110}$ of length $110$.

That length is maximal, since every list length is at most half the base width. Thus $h=219-220=-1\le0$, and the direct child is $[110,219]$. Its additional exponent is $1/220$, since $A_0=0$ and $A_{110}=1/220$.

A supplying interval for this condition must contain $[0,220]$, so it extends one depth beyond the current interval. Nevertheless, the direct argument in Theorem~\ref{iteration:thm:minimum} constructs its parent pair inside $[0,219]$.

When $t-a=10q_0$, the only admissible pair in \eqref{iteration:eq:V} is $(a,t)$, so the minimum is defined. At width $10q_0-1$ it is not defined. Both intervals are already terminal under the threshold $100q_0$.

Equation~\eqref{iteration:eq:sum-e} stops at the last defined minimum, before a final zero-exponent transition. Assigning zero to the minimum on a shorter interval could invalidate the endpoint comparison when profile values are negative.

Finally, a nonterminal near child need not have a shorter longest condition. If that length remains $L$, its current interval is $[v-2L,v]$ and the next step is direct. If the length decreases, a representative in $(a,h]$ has been passed permanently. These alternatives give the near-step count.

\subsection{Global suprema and the local inequalities}\label{summation:subsec:envelopes}\label{summation:sec:local}
We now apply the branch estimates to the full local inequalities. A parent may have several frequency or depth contributions, all of which must be included.

We fix $0<\sigma\le\kappa<1$ and a class probability $\rho$ with profile $m$, and use $R=2^N$, $q_0=\lceil\kappa N\rceil$, and $\epsilon=\sigma N$. The class may depend on $N$ but remains fixed throughout the expansion at that scale. We use the mass comparison~\eqref{profile:eq:mass}, the condition sets of Section~\ref{main:profiles}, and the common cutoffs of Theorem~\ref{cutoff:thm:cutoffs}. The constants are uniform over these constructed families. The lower bound $A_n\ge\chi n/N-C_{\rm low}\sigma$ is not needed here.

The positive counts are those of~\eqref{positive:eq:positive-count}; their squares satisfy~\eqref{positive:eq:separated}. The angular quantities are those of~\eqref{angular:eq:independent}, with the base and current interval related by~\eqref{iteration:eq:base}. Their two squares need not be separated. The prescribed-order bound always concerns the combined order at both endpoints.

With $\rho$, $m$, the tests, cutoffs, and their numerical bounds fixed, we define
\begin{equation}\label{summation:eq:d-global}
 d_i(a,t)=\max_{X,Y\text{ occupied and separated at depth }a}\mathsf D_i^{X,Y}(a,t),
\end{equation}
with the empty maximum equal to zero, and
\begin{equation}\label{summation:eq:h-global}
 h_{i,B}^{\mathcal B}(a,v)=\max_{X,Y\text{ occupied at depth }a}
 \ \sup_{b,c}\mathsf H^{X,Y}(a,v;b,c).
\end{equation}
Here the maxima range over all admissible parent squares at the indicated depth, unlike $\mathcal H$ in~\eqref{angular:eq:supremum}, whose parent squares are fixed. The supremum is over amplitudes of the form~\eqref{amplitude:eq:form}, with the indicated base and level and combined prescribed order at most $B$. The class, profile, cutoffs, and cutoff derivative bounds remain fixed.

Proposition~\ref{reduce:prop:terminal} gives finite bounds at every interval width:
\begin{equation}\label{summation:eq:crude}
 d_i(a,t)\le2^{t-a}R^{4C_{\rm prof}\sigma},\qquad
 h_{i,B}^{\mathcal B}(a,v)\le C_{B,\Psi}2^{2(v-a)}R^{4C_{\rm prof}\sigma}.
\end{equation}
Here one may take $C_{B,\Psi}=(2\pi)^2\|\Psi\|_1U_B^4$.

If $F(x)\le c+\sum_j a_jG_j(x)$ for every $x$, with $a_j\ge0$, then
\begin{equation}\label{summation:eq:sup}
 \sup_xF(x)\le c+\sum_j a_j\sup_xG_j(x).
\end{equation}
Thus each child maximum or supremum is bounded by the corresponding $d$ or $h$, with no factor counting squares and no need to choose a maximizing amplitude. The class and interval rules are unchanged.

\paragraph*{The local inequalities.}
Only the explicit frequency and depth sums contribute to the number of terms.

We retain $C_G=C_{\rm test}+64+4C_{\rm prof}$ from~\eqref{graph:eq:constant}. For each direct step, fix a rule selecting a condition of maximal length and one of its supplying intervals, for example by lexicographic order. The choice depends only on the condition set and profile, so the same local estimate applies to all amplitudes and square pairs.

For each fixed $T>0$ and initial combined prescribed-order bound $B_{\rm in}$, Theorem~\ref{reduce:thm:six} gives the following estimates in a common family with order bound $B_*$, for all sufficiently large $N$. Choose $C_{\rm it}\ge1$ to bound their multiplicative constants. It may depend on the fixed parameters, $T$, $B_{\rm in}$, and finitely many common cutoff derivative bounds, but not on $N$, the class, profile, squares, or spatial multipliers. It is distinct from the distance tolerance $C_0=300$.

For a nonterminal positive node on $[a,t]$, let $p$ be its selected depth. In the suffix case,
\begin{equation}\label{summation:eq:local-suffix}
 d_i(a,t)\le C_{\rm it}R^{A_p-\min_{[p,t]}A+C_{\rm test}\sigma}d_{i+1}(a,p).
\end{equation}
In the prefix case define
\[
 \mathcal L=\{v\in\Z:a<v\le t,\ t-v\ge\epsilon\},\qquad
 \mathcal U=\{v\in\Z:a<v\le t,\ 0\le t-v<\epsilon\}.
\]
Then
\begin{align}\label{summation:eq:local-prefix}
 d_i(a,t)&\le C_{\rm it}R^{4C_{\rm prof}\sigma}
 +C_{\rm it}\sum_{v\in\mathcal L}d_{i+1}(a,v)\notag\\
 &\quad+C_{\rm it}R^{-2(A_p-A_a)+k(a,p)+C_G\sigma}
       \sum_{v\in\mathcal U}h_{i,B_*}^{[a,t]}(p,v)
 +C_{\rm it}R^{-T}.
\end{align}
In particular, each low-frequency child has coefficient $C_{\rm it}$, with no power of $R$.

For a nonterminal angular interval, set
\[
 L_0=\max\{L,\lceil10\epsilon\rceil\},\qquad h=v-2L_0.
\]
When $h\le a$, the selected condition of maximal length is specified at a depth $p$ for which
\begin{equation}\label{summation:eq:local-direct}
 h_{i,B_*}^{\mathcal B}(a,v)
 \le C_{\rm it}R^{-2(A_p-A_a)+k(a,p)+C_G\sigma}
 h_{i,B_*}^{\mathcal B}(p,v)+C_{\rm it}R^{-T}.
\end{equation}
When $h>a$, put
\[
 \mathcal J=\{b\in\Z:a+11\le b\le h+7,\ b<v\}.
\]
Then
\begin{align}\label{summation:eq:local-split}
 h_{i,B_*}^{\mathcal B}(a,v)
 &\le C_{\rm it}R^{-2(A_h-A_a)+4C_{\rm prof}\sigma}h_{i,B_*}^{\mathcal B}(h,v)\notag\\
 &\quad+C_{\rm it}\sum_{b\in\mathcal J}
 R^{-2(A_b-A_a)+4C_{\rm prof}\sigma}d_{i+1}(b,v)+C_{\rm it}R^{-T}.
\end{align}
The near child keeps its base. A far child is a positive count with condition set $\M(b,v)$ and no prescribed derivatives.

At a terminal current interval we use~\eqref{summation:eq:crude}. The common order $B_*$ will be chosen in Subsection~\ref{summation:sec:orders}. A high conversion produces amplitudes within this bound; graph restriction and factor removal do not increase it, and the far calculation returns a positive count without prescribed derivatives.

To obtain the displayed inequalities, apply~\eqref{summation:eq:sup} to Theorem~\ref{reduce:thm:six} and bound each spatial child maximum by the global $d$ or $h$. The local normalization losses are already included, and no count of squares is introduced.

The local remainders already include the relevant sums and normalizations; in particular,~\eqref{summation:eq:local-split} has one remainder for the full split. The factor $R^{4C_{\rm prof}\sigma}$ in a terminal or low-frequency stopping value occurs only at the end of a main-term branch.

\subsection{The uniform single-profile bound}\label{summation:sec:theorem}
The following theorem bounds the complete angular quantity, including all branches and errors.

\Needspace{385pt}
\begin{theorem}\label{summation:thm:iteration}
Use the class probabilities, profiles, tests, and common cutoff family constructed in the preceding sections, with their fixed structural constants $C_{\rm prof},C_{\rm test}$ and $C_G$ in~\eqref{graph:eq:constant}. Fix $0<\sigma\le\kappa<1$, the geometric cutoff functions, and an initial finite combined prescribed-order bound $B_{\rm in}$. Let $J_{\rm rec}$ be~\eqref{iteration:eq:Jrec}. Then there are $C,C_E,N_{\rm it}<\infty$, common to the allowed measures, profiles, squares, and spatial multipliers, such that the following holds for every integer $N\ge N_{\rm it}$ and every $c\in\{0,\ldots,N\}$ satisfying $N-c>100\lceil\kappa N\rceil$. Set
\begin{equation}\label{summation:eq:theorem-data}
 E_c=2A_c-V([c,N]),\qquad
 C_{\rm loss}=C_{\rm test}+3C_G+48C_{\rm prof},\qquad
 P_{\rm it}=J_{\rm rec}+1.
\end{equation}
Then
\begin{equation}\label{summation:eq:theorem-strong}
 h_{0,B_{\rm in}}^{[c,N]}(c,N)
 \le C(1+N)^{P_{\rm it}}
 R^{E_c+240\kappa+C_{\rm loss}\sigma/\kappa}+C_E R^{-20}.
\end{equation}
In particular,
\begin{equation}\label{summation:eq:theorem}
 h_{0,B_{\rm in}}^{[c,N]}(c,N)
 \le C'(1+N)^{P_{\rm it}}
 R^{E_c+C_*(\kappa+\sigma/\kappa)},
 \qquad C_*=240+C_{\rm test}+3C_G+48C_{\rm prof}.
\end{equation}
The coefficient $C_*$ and polynomial exponent $P_{\rm it}$ are independent of the derivative orders and the individual class. The constants $C,C',C_E,N_{\rm it}$ may depend on the fixed parameters, the initial order, and the finitely many common cutoff derivative bounds fixed after the choice of the local error exponent.
\end{theorem}

The theorem concerns nonterminal intervals $[c,N]$ and requires no lower profile bound at $c$. Such a bound is used in Section~\ref{main:frequency} to choose a common starting depth for two profiles.

\subsection{The finite expansion of the local inequalities}\label{summation:sec:expansion}
We now expand the scalar inequalities to include all main contributions and local errors.

At fixed $N$, only finitely many intervals, bases, levels, and types are reached. Together with the fixed profile, these determine the coefficients for $d$ and $h$. The squares and amplitudes have already been included in their suprema.

There are at most $N+1$ successors at each step. A prefix has at most $N$ frequency indices, partitioned into the low and high sets. An angular split has one near child and at most $N$ far depths. Direct and suffix steps have one child. The low-frequency stopping term is additive and has no successor.

The direct choice is fixed in Subsection~\ref{summation:sec:local} and contributes no additional sum. Bounding an empty spatial child family by a global maximum may add paths, but they obey the same profile and interval rules. Their lengths are therefore still bounded by Theorem~\ref{iteration:thm:count}.

\begin{lemma}\label{summation:lem:tree}
Suppose nonnegative finite numbers $F_s$ obey a finite system
\begin{equation}\label{summation:eq:tree-recursion}
 F_s\le C_{\rm it} z_s+C_{\rm it}\sum_{e:s\to s'}R^{\gamma_e}F_{s'}+C_{\rm it}R^{-T},
\end{equation}
where there are at most $N+1$ successors, $z_s\ge0$, and every successor path has at most $J_{\rm rec}$ edges. Terminal numbers obey $F_s\le C_{\rm it}z_s$. Terms absent from a particular local inequality may be omitted. Then the value at the root is bounded by a sum of main and error terms. Each main term has the form
\begin{equation}\label{summation:eq:main-tree-term}
 C_{\rm it}^{d+1} R^{\sum_{e\in\pi}\gamma_e}z_s,
 \qquad d=\#\pi\le J_{\rm rec},
\end{equation}
and each error term has the form
\begin{equation}\label{summation:eq:error-tree-term}
 C_{\rm it}^{d+1}R^{\sum_{e\in\pi}\gamma_e-T}.
\end{equation}
Each of the two collections contains at most $(N+3)^{J_{\rm rec}+1}$ terms. Here a path can end at an early stopping term or at an intermediate error, not only at a terminal node.
\end{lemma}
\begin{proof}
Repeatedly substitute the inequality for each successor. All coefficients are nonnegative. A quantity reached by different paths is expanded separately along each path, producing a finite tree. Each node has at most $N+1$ successors, so the total number of nodes is at most
\[
 \sum_{d=0}^{J_{\rm rec}}(N+1)^d\le(N+3)^{J_{\rm rec}+1}.
\]
Each node contributes at most one stopping or terminal main term and at most one local error. Multiplying the coefficients along its unique route gives~\eqref{summation:eq:main-tree-term} and~\eqref{summation:eq:error-tree-term}. The expansion ends after at most $J_{\rm rec}$ edges.
\end{proof}

The factor $(N+3)^{J_{\rm rec}+1}$ counts the paths, including distinct paths reaching the same quantity. It does not count squares or amplitudes, which have already been included in the suprema.

For each edge other than a low-frequency edge, set
\[
 \gamma_e=-2(A_{a'}-A_a)+e+c_e\sigma,
\]
with $c_e=C_{\rm test}$ for suffix, $c_e=C_G$ for high and direct, and $c_e=4C_{\rm prof}$ for near and far. For low, $\gamma_e=0$. At a nonterminal prefix stop use $z_s=R^{4C_{\rm prof}\sigma}$; at other nonterminal nodes use $z_s=0$. For terminal nodes use the appropriate bound from \eqref{summation:eq:crude}, putting its fixed prefactor into $C_{\rm it}$.

The branches satisfy Theorem~\ref{iteration:thm:count}. Each main contribution ends at a terminal node or a low-frequency stopping term. Each error is multiplied only by coefficients preceding the node where it occurs.

\subsection{The main contributions and their exponent}\label{summation:sec:main}
We first bound one main contribution, then multiply by the number of possible contributions.

Let a main path start at $[c,N]$, end at starting depth $a_f$, and stop either at a terminal node or a low-frequency term. Equation~\eqref{iteration:eq:principal-branch} bounds the sum of the principal edge exponents by
\[
 E_c+10q_0/N\le E_c+40q_0/N.
\]
The endpoint comparison also covers a branch with no transitions.

A terminal angular node has width at most $100q_0$, so its width factor is at most $R^{200q_0/N}$. A terminal positive node needs only $R^{100q_0/N}$. A low-frequency stopping term has no width factor. Thus principal exponents and the largest possible terminal width contribution together are at most
\begin{equation}\label{summation:eq:width-main}
 E_c+240q_0/N\le E_c+240\kappa+240/N.
\end{equation}
The factor $R^{240/N}=2^{240}$ is absorbed into the ordinary multiplicative constant.

Along a path, the sum of the structural exponent losses is
\begin{equation}\label{summation:eq:structural-path}
 \sigma\{C_{\rm test} n_S+C_G(n_H+n_D)+4C_{\rm prof}(n_N+n_F)\}.
\end{equation}
There is no $n_L\sigma$ term.

Using the convenient weaker bounds from Theorem~\ref{iteration:thm:count},
\[
 n_S\le\kappa^{-1},\quad n_H+n_D\le K_{\rm blk},
 \quad n_N\le2K_{\rm blk}+1,\quad n_F\le K_{\rm blk}+1,
\]
we obtain
\begin{align*}
 \eqref{summation:eq:structural-path}
 &\le\sigma\{C_{\rm test}/\kappa+C_G K_{\rm blk}+4C_{\rm prof}(3K_{\rm blk}+2)\}\\
 &\le(C_{\rm test}+3C_G+44C_{\rm prof})\sigma/\kappa.
\end{align*}
The second inequality uses $K_{\rm blk}\le\kappa^{-1}+2\le3/\kappa$ and $3K_{\rm blk}+2\le3/\kappa+8\le11/\kappa$, since $0<\kappa<1$.

There is at most one final terminal or low-frequency stopping factor $R^{4C_{\rm prof}\sigma}$. Since $\sigma\le\sigma/\kappa$, its inclusion gives
\begin{equation}\label{summation:eq:loss-total}
 C_{\rm loss}\frac{\sigma}{\kappa},
 \qquad C_{\rm loss}=C_{\rm test}+3C_G+48C_{\rm prof}.
\end{equation}
Every main contribution is therefore at most
\begin{equation}\label{summation:eq:one-main}
 2^{240}C_{\rm it}^{J_{\rm rec}+1}
 R^{E_c+240\kappa+C_{\rm loss}\sigma/\kappa}.
\end{equation}

By Lemma~\ref{summation:lem:tree}, there are at most $(N+3)^{J_{\rm rec}+1}$ main terms. As $N+3\le3(N+1)$, summing \eqref{summation:eq:one-main} gives
\begin{equation}\label{summation:eq:all-main}
 C_{\rm main}(1+N)^{J_{\rm rec}+1}
 R^{E_c+240\kappa+C_{\rm loss}\sigma/\kappa},
 \qquad C_{\rm main}=2^{240}(3C_{\rm it})^{J_{\rm rec}+1}.
\end{equation}
Any fixed terminal prefactor is included in this constant. The polynomial in $N$ is retained as displayed.

Since $C_*=240+C_{\rm loss}$ is at least both $240$ and $C_{\rm loss}$, it bounds both coefficients in~\eqref{summation:eq:all-main}. With $C_{\rm prof}=10$, $C_{\rm test}=1000$, and $C_G=1104$,
\[
 C_{\rm loss}=1000+3\cdot1104+48\cdot10=4792,
 \qquad C_*=5032.
\]
These coefficients are independent of differentiation orders.

\subsection{The accumulated errors}\label{summation:sec:errors}
We bound the coefficients preceding each local error by an estimate valid for every finite branch prefix.

By \eqref{profile:eq:Alip}, $|A_n|\le1$. Every additional exponent satisfies $0\le e\le2$, and $-2(A_{a'}-A_a)\le4$. Thus every principal edge exponent is at most six. Since $C_G\ge C_{\rm test},4C_{\rm prof}$ and $\sigma<1$, all complete edge exponents satisfy
\begin{equation}\label{summation:eq:Gamma}
 \gamma_e\le\Gamma,\qquad \Gamma=6+C_G.
\end{equation}
A low edge has exponent zero. The bound~\eqref{summation:eq:Gamma} is independent of the local error exponent, prescribed-order bound, and integration-by-parts orders.

For a prefix of length at most $J_{\rm rec}$, the power of $R$ multiplying a local error is at most $R^{\Gamma J_{\rm rec}}$. The product of its ordinary coefficients, including the coefficient of the local error, is at most $C_{\rm it}^{J_{\rm rec}+1}$.

Before applying the local analytic estimates, we choose
\begin{equation}\label{summation:eq:T}
 \boxed{T_{\rm loc}=(\Gamma+1)(J_{\rm rec}+1)+20.}
\end{equation}
The local estimates are available for every fixed error exponent, so this choice is possible.

Lemma~\ref{summation:lem:tree} bounds the sum of all error contributions by
\begin{equation}\label{summation:eq:error-polynomial}
 C_{\rm it}^{J_{\rm rec}+1}(N+3)^{J_{\rm rec}+1}R^{\Gamma J_{\rm rec}-T_{\rm loc}}.
\end{equation}
This choice of $T_{\rm loc}$ also absorbs the polynomial number of error terms. Since $N+1\le2^N=R$ for integers $N\ge1$, we obtain
\begin{align}\label{summation:eq:all-errors}
 \eqref{summation:eq:error-polynomial}
 &\le(3C_{\rm it})^{J_{\rm rec}+1}R^{\Gamma J_{\rm rec}+J_{\rm rec}+1-T_{\rm loc}}\notag\\
 &=(3C_{\rm it})^{J_{\rm rec}+1}R^{-\Gamma-20}
 \le C_E R^{-20}.
\end{align}
This bounds all additive errors, including those at intermediate nodes and those reached along different routes.

Suppose two successive bounds are
\[
 F_0\le aF_1+e_0,\qquad F_1\le bF_2+e_1,
 \qquad a,b\ge0.
\]
Then $F_0\le abF_2+ae_1+e_0$: the error $e_1$ receives the preceding factor $a$ and no outgoing factor $b$. Formula~\eqref{summation:eq:error-tree-term} gives exactly this product of preceding coefficients. The choice~\eqref{summation:eq:T} controls every such product without a terminal-width estimate or a two-depth minimum for an error contribution.

\subsection{The order of choices and proof of the theorem}\label{summation:sec:orders}\label{summation:sec:examples}\label{summation:sec:summary}
The structural coefficients in the powers of $R$ are fixed before the small parameters. The subsequent choices of derivative orders and multiplicative constants do not change them.

Fix $C_{\rm prof},C_{\rm test},C_G$ and the geometry, then $\kappa,\sigma$ and $B_{\rm in}$. Equations~\eqref{iteration:eq:Jrec} and~\eqref{summation:eq:Gamma} determine $J_{\rm rec}$ and $\Gamma$ independently of derivative bounds. Choose $T_{\rm loc}$ by~\eqref{summation:eq:T}.

For the local inverse conversion, Subsection~\ref{reduce:sec:uniformity} gives the choices
\begin{gather}
 S=T_{\rm loc}+2,\qquad \mu=C_{\rm prof}\sigma,\qquad P_{\rm inv}=S+2\mu+2,\notag\\
 J_{\rm inv}=\left\lfloor\frac{P_{\rm inv}+9}{10\sigma}\right\rfloor+1,
 \qquad B_{\rm inv}=2(J_{\rm inv}-1),\qquad
 B_* =\max\{B_{\rm in},B_{\rm inv}\}.\label{summation:eq:inverse-orders}
\end{gather}
These are the orders in~\eqref{reduce:eq:common-orders} and Theorem~\ref{circle:thm:conversion} for the chosen $T_{\rm loc}$. Here $J_{\rm inv}$ denotes the order $J_{\rm pos}$ in~\eqref{positive:eq:inverse-order}, evaluated at $S=T_{\rm loc}+2$. The bound $B_*$ concerns the combined prescribed order of both amplitudes, and therefore also bounds each endpoint order.

Each positive-to-angular conversion starts with undifferentiated cutoffs and produces order at most $B_{\rm inv}$. Graph restriction and factor removal do not increase the order. A far calculation produces a positive count, so a later conversion starts with undifferentiated cutoffs again. Thus the same family with bound $B_*$ suffices throughout the iteration.

With $B_*$ fixed, we choose the graph and far integration orders. A sufficient graph order is an integer
\[
 m_{\rm gr}>\frac{S+21}{26\sigma}.
\]
For the far estimate, choose the circle expansion and radial integration-by-parts orders to give a pointwise remainder $O(R^{-Q})$, where $Q=S+3+4C_{\rm prof}\sigma$. After normalization, the remainder is $O(R^{-S})$. These estimates may use angular derivatives of order greater than $B_*$, but they do not increase the prescribed orders of the resulting amplitudes.

Only finitely many cutoff bounds $K_j$ are needed. Theorem~\ref{reduce:thm:six} then gives common prefactors and starting scales. Choose $C_{\rm it}\ge1$ to bound every local and terminal prefactor. Choose $N_{\rm it}$ above the thresholds for the class construction, cutoffs, amplitude products, analytic estimates, and~\eqref{reduce:eq:far-threshold}. Each threshold is already uniform over the family, so their finite maximum is uniform as well.

The order is
\begin{equation}\label{summation:eq:order-chain}
\begin{split}
 &(C_{\rm prof},C_{\rm test},C_G,\text{fixed geometry})\ \longrightarrow\ (\kappa,\sigma,J_{\rm rec},\Gamma)\\
 &\longrightarrow\ T_{\rm loc}\ \longrightarrow\ B_*\
 \longrightarrow\ \text{temporary orders and numerical bounds}\\
 &\longrightarrow\ C_{\rm it},N_{\rm it}.
\end{split}
\end{equation}
The constants $C$ and $C_E$ depend on $C_{\rm it}$ through the fixed power $J_{\rm rec}+1$. This does not change $\Gamma$, $T_{\rm loc}$, or $C_*$. The bound $N+1\le R$ absorbs the polynomial number of errors independently of $C_{\rm it}$.

\begin{proof}[Proof of Theorem~\ref{summation:thm:iteration}]
Make the choices above and enlarge the initial amplitude family from order bound $B_{\rm in}$ to $B_*$, which can only increase the supremum. By Corollary~\ref{iteration:cor:levels}, all required smoothing levels are available, so the same local inequalities and constants apply at every node.

We apply Lemma~\ref{summation:lem:tree} to the scalar inequalities in Subsection~\ref{summation:sec:local}. Theorem~\ref{iteration:thm:count} bounds every branch length by $J_{\rm rec}$. The main bound~\eqref{summation:eq:all-main} and total error bound~\eqref{summation:eq:all-errors} prove~\eqref{summation:eq:theorem-strong}, with $P_{\rm it}=J_{\rm rec}+1$.

Since $|A_n|\le1$ and $V([c,N])\le2$, we have $E_c\ge-4$. All displayed structural losses are nonnegative, so for $R\ge1$ the term $C_ER^{-20}$ is bounded by $C_ER^{E_c+C_*(\kappa+\sigma/\kappa)}$. Absorb its fixed coefficient into $C'$ and use
\[
 240\kappa+C_{\rm loss}\sigma/\kappa
 \le(240+C_{\rm loss})(\kappa+\sigma/\kappa).
\]
This proves~\eqref{summation:eq:theorem}, uniformly over the class, squares, amplitudes, and eligible starting depths $c$. No sum over $c$ is needed.
\end{proof}

\paragraph*{Remarks on the estimates.}
The estimates above keep spatial maxima, frequency sums, and products of preceding coefficients distinct. We summarize the reasons for these distinctions.

Replacing a spatial maximum by a sum could introduce a number of squares growing as a power of $R$. We instead bound the maximum by the global quantity. Inequality~\eqref{summation:eq:sup} also takes amplitude suprema without selecting maximizing functions.

A factor $R^{c\sigma}$ at each of $1/\sigma$ low steps could accumulate to $R^c$, whose exponent does not tend to zero with $\sigma/\kappa$. The low-frequency estimate has no such exponent loss. Its multiplicative constants occur a bounded number of times for fixed parameters and are included in $C$.

A main branch ending at an interval too short for $V$ uses its last defined minimum in~\eqref{iteration:eq:principal-branch}. An error may occur before terminality; its preceding coefficients are bounded by~\eqref{summation:eq:Gamma} and the prefix length.

The earlier base determines the remaining cutoff factors, while the current interval determines terminality and $V$. A supplying interval may extend beyond the current endpoint, but the direct comparison still uses $V$ on the current parent and child. Both intervals remain fixed when taking the amplitude suprema.

\clearpage

\section{Fourier decay of the approximating distance measures}\label{main:frequency}
We prove a Fourier bound for the measures $\tau_N$ that approximate $\alpha$ in total variation. Each component of $\tau_N$ pairs classes from the two prepared probabilities, whose profiles and condition sets may differ. The two-profile graph estimate and Cauchy--Schwarz in the radial variable reduce the bound to two quantities covered by Theorem~\ref{summation:thm:iteration}.

The common starting depth $c$ gives an upper bound for the smaller profile rise in the first graph estimate and lower bounds for both two-depth minima on $[c,N]$. The excess profile values at $c$ cancel when the graph and single-profile estimates are combined. Summing with the original product weights then bounds $\tau_N$ without a factor counting classes.

\subsection{The measures and their class components}\label{frequency:sec:data}
We fix $0<\eta<1$ and the probabilities $\mu_1,\mu_2$ supplied by Theorem~\ref{prep:thm:preparation}. Their compact supports lie in $Q_0$, satisfy $d_0\le |x-y|\le d_1$, and determine directions $(x-y)/|x-y|$ contained in a uniformly small arc. Choose $\psi$ by rotating a fixed smooth cutoff so that it equals one near these directions and vanishes near their antipodes. The arc may depend on the prepared pair, but the angular margins and derivative bounds are uniform. Both enclosing squares may be $Q_0$, since Theorem~\ref{circle:thm:conversion} requires separation of the supports themselves.

We keep $\dd(x,y)=|x-y|$, $e(x,y)=(x-y)/|x-y|$, and $W(x,y)=|x-y|^{-1/2}$. The function $\Theta_N$ and measures $\alpha,\tau_N$ are those of~\eqref{mass:eq:Theta}--\eqref{mass:eq:measures}, with $\beta=\chi/20$ as in~\eqref{mass:eq:parameters}. Below we choose $\kappa$ and $\sigma$ within the range allowed by the mass theorem, and use these definitions at every scale.

At scale $R=2^N$, Corollary~\ref{dyad:cor:prepared} supplies the class probabilities and their original weights. For an ordered pair $(C,D)$ we abbreviate
\[
 \rho_1=\rho_{1,C},\quad \rho_2=\rho_{2,D},\quad
 M_j(n)=R^{-m_j(n/N)},\quad A_{j,n}=m_j(n/N)-n/N.
\]
Their profile comparisons and lower bounds are
\begin{gather}
 M_j(n)R^{-C_{\rm prof}\sigma}\le\rho_j(P)
       \le M_j(n)R^{C_{\rm prof}\sigma},\notag\\
 A_{j,n}\ge\chi n/N-C_{\rm low}\sigma,
 \qquad \chi=\min\{\eta/8,1/16\},\quad C_{\rm low}=12.\label{frequency:eq:profile-data}
\end{gather}
The mass comparison holds on squares occupied by the corresponding class. Each class probability is carried by its class, and its topological support lies in the prepared compact set, even though the class itself need not be compact.

The common starting depth $c=c(C,D)$ and factor families $\mathcal A_{j;C,D}$ are defined in~\eqref{mass:eq:entry}--\eqref{mass:eq:root-lists}. We recall the component in~\eqref{mass:eq:component-identity}:
\begin{equation}\label{frequency:eq:component}
 \tau_{C,D,N}=\dd_\#\left[
 W(x,y)b^{(0)}_{1;C,D}(x,e(x,y))
 b^{(0)}_{2;C,D}(y,e(x,y))
 \,d\rho_{1,C}(x)d\rho_{2,D}(y)\right].
\end{equation}
Each individual cutoff depends on its class, condition, and level. The selected factor family may depend on the opposite class through $c(C,D)$. The weight of this component is $a_{1,C}a_{2,D}$.

We use the uniform profile estimates and the fact that the original product weights sum to at most one. The auxiliary radial restriction used in the mass estimate is not inserted into~\eqref{frequency:eq:component}.

\subsection{The entry depth and the two-depth minima}\label{frequency:sec:entry}
The starting depth from Lemma~\ref{entry:lem:entry} gives the following comparison between the two-depth minima and the graph correction.

\begin{lemma}\label{frequency:lem:entry-minima}
Assume the hypotheses of Lemma~\ref{entry:lem:entry}, and take $N\ge N_{\rm ent}$. For each of the two profiles define, using \eqref{iteration:eq:V},
\begin{equation}\label{frequency:eq:V}
 V_j([c,N])=
 \min_{\substack{c\le r<s\le N\\s-r\ge10q_0}}
 (A_{j,r}+A_{j,s}),\qquad q_0=\lceil\kappa N\rceil.
\end{equation}
Both minima are defined, and
\begin{equation}\label{frequency:eq:entry-gain}
 V_j([c,N])\ge2\beta-4/N,\qquad
 \min_j k_j(0,c)-\tfrac12\{V_1([c,N])+V_2([c,N])\}
 \le-\beta+C_{\rm low}\sigma+4/N.
\end{equation}
\end{lemma}
\begin{proof}
Lemma~\ref{entry:lem:entry} gives $N-c>100q_0$, so admissible pairs separated by $10q_0$ exist. Every profile value at a depth at least $c$ is at least $\beta-2/N$. Summing the two values in any admissible pair gives the first inequality. For the second inequality, combine $\min_j k_j(0,c)\le\beta+C_{\rm low}\sigma$ from the same lemma with the lower bounds for $V_j$.
\end{proof}

The estimate uses the smaller of the two rises; the example in Subsection~\ref{entry:sec:entry} explains why their maximum cannot be used. The same lemma gives $0<c<N/4$. Every condition in $\M_{\rho_j}(c,N)$ is specified deeper than $c$ and has length at most $(N-c)/2$. We write $L$ for the largest condition length in the two factor families. Including the extra $\Par{0}{c}$ factor, we obtain
\begin{equation}\label{frequency:eq:root-list-bounds}
 L\le\max\{c,(N-c)/2\}\le N/2,
 \qquad \#\mathcal A_{j;C,D}\le N(N+1)+1\le N^3.
\end{equation}
The count follows from Corollary~\ref{lists:cor:Ncube}. The extra $\Par{0}{c}$ is separate from $\M_{\rho_j}(c,N)$. Its condition and cutoff are defined for every admissible pair of depths, so $c$ need not be a block representative.

\subsection{The Fourier estimate}\label{frequency:sec:theorem}
The exponent coefficients from Sections~\ref{main:directions}, \ref{main:graph}, and~\ref{main:iteration} are
\begin{equation}\label{frequency:eq:numerical}
 C_{\rm prof}=10,\qquad C_{\rm test}=1000,\qquad
 C_G=1104,\qquad C_*=5032.
\end{equation}
Here $C_{\rm low}=12$ is the lower-profile coefficient in~\eqref{frequency:eq:profile-data}; $C_{\rm lf}$ denotes the separate low-frequency prefactor in Section~\ref{main:reductions}. We use $P_{\rm it}=J_{\rm rec}+1$ from~\eqref{summation:eq:theorem-data}.

\Needspace{360pt}
\begin{theorem}\label{frequency:thm:frequency}
Use the prepared probabilities and the classes, lists, cutoffs, and measures constructed in Sections~\ref{main:preparation}--\ref{main:cutoffs-mass}. Keep $\beta=\chi/20$ as in \eqref{mass:eq:parameters}, and assume
\[
 0<\sigma\le\kappa<10^{-4},\qquad C_{\rm low}\sigma\le\chi/8.
\]
For all $N$ above a common threshold, every component formed from the class probabilities satisfies
\begin{align}\label{frequency:eq:component-raw}
 \int_{R/2}^{2R}|\widehat{\tau_{C,D,N}}(r)|^2\,dr
 &\le C(1+N)^{P_{\rm it}}
 R^{-\beta+(C_{\rm low}+C_G)\sigma+C_*(\kappa+\sigma/\kappa)}
 +CR^{-20}.
\end{align}
The constants do not depend on the class pair, its original masses, or the number of classes. For the successive choices
\begin{gather}
 \beta=\chi/20,\qquad
 0<\kappa<\min\left\{10^{-4},\frac{\beta}{16(1+C_*)}\right\},\notag\\
 0<\sigma<\min\left\{
 \frac{\kappa}{1000},\frac{\chi}{8(1+C_{\rm low})},
 \frac{\beta}{16(1+C_{\rm low}+C_G)},
 \frac{\beta\kappa}{16(1+C_*)}\right\},\label{frequency:eq:parameters}
\end{gather}
there are common finite constants $C_s,N_s$ such that the approximant $\tau_N$ defined in~\eqref{mass:eq:measures} satisfies
\begin{equation}\label{frequency:eq:frequency-conclusion}
 \int_{R/2\le |r|\le2R}|\widehat{\tau_N}(r)|^2\,dr
 \le C_sR^{-\beta/2}\qquad(N\ge N_s).
\end{equation}
The proof gives $CR^{-3\beta/4}+CR^{-20}$ before weakening the exponent. Once the displayed parameters and required finite derivative bounds are fixed, all constants and thresholds are uniform over the prepared family.
\end{theorem}

Choosing $\kappa$ and then $\sigma$ as half of their respective upper bounds fixes them in terms of $\eta$ and satisfies the hypotheses of the entry and mass theorems. The exponent coefficients are fixed before the derivative orders used below.

\subsection{Conversion at depth zero and the first graph estimate}\label{frequency:sec:component}
We fix an ordered class pair and suppress its labels. At depth zero the actual and model masses are all one:
\begin{equation}\label{frequency:eq:root-normalization}
 M_1(0)=M_2(0)=1,\qquad \rho_j(Q_0)=1.
\end{equation}
The mass-ratio exponent in Theorem~\ref{circle:thm:conversion} is therefore zero. For an error $O(R^{-20})$, the formulas~\eqref{circle:eq:orders} give
\begin{equation}\label{frequency:eq:root-order}
 P_{\rm out}=22,\qquad
 J_{\rm out}=\left\lfloor\frac{31}{10\sigma}\right\rfloor+1,
 \qquad B_{\rm root}=2(J_{\rm out}-1).
\end{equation}
These orders depend on $\sigma$, but not on the class pair or $N$. Lemma~\ref{entry:lem:entry} gives $0<c<N/4$ and a nonterminal interval $[c,N]$. By~\eqref{frequency:eq:root-list-bounds}, we have $L\le N/2$ and at most $N^3$ factors on each side. The inequality $\sigma\le\kappa<10^{-4}$ gives $N\ge20\sigma N$.

The class probabilities remain supported in $K_1,K_2$. At $a=0$, the inverse distance weight is
\[
 (2^a|x-y|)^{-1/2}=|x-y|^{-1/2}=W(x,y).
\]
Together with~\eqref{frequency:eq:root-normalization} and initial spatial factors equal to one, this verifies that Theorem~\ref{circle:thm:conversion} applies to~\eqref{frequency:eq:component}:
\begin{equation}\label{frequency:eq:root-I-applied}
 F_{C,D}:=\int_{R/2}^{2R}|\widehat{\tau_{C,D,N}}(r)|^2\,dr
 \le C\sup\mathsf H_F^{Q_0,Q_0}(0,N)+CR^{-20}.
\end{equation}
The resulting amplitudes may be complex, although the measure on the left is positive. Every level-zero cutoff occurrence is retained, with any prescribed angular derivative normalized by the corresponding power of its scale factor.

For these amplitudes, the comparison from Lemma~\ref{angular:lem:comparison} follows from
\[
 \left|\int\psi U\overline V\,dw\right|^2
 \le\left(\int|U|^2\,dw\right)\left(\int|V|^2\,dw\right).
\]
The radial weight and prefactor are the same in both expressions, so $\mathsf H_F\le\mathsf H$ with unchanged endpoint amplitudes; $\psi$ remains outside those amplitudes.

We next apply Theorem~\ref{graph:thm:graph} at
\[
 a=0,\qquad p=c,\qquad v=N,\qquad g_1=g_2=0.
\]
The extra $\Par0c$ factor supplies the parallel condition at each endpoint. Its angular derivatives are supported in the corresponding level-one good set for the original class. The depth inequalities are
\begin{gather*}
 c<N/4<N-c+2\sigma N,\qquad
 L\le N/2<N-c+2\sigma N,\\
 N-c>3N/4>10\sigma N.
\end{gather*}
The last inequality uses $\sigma<10^{-4}$. The graph estimate allows both enclosing squares to be $Q_0$.

Because $A_{1,0}=A_{2,0}=0$, its exponent is
\begin{equation}\label{frequency:eq:graph-factor}
 G_c=-A_{1,c}-A_{2,c}+\min_jk_j(0,c)+C_G\sigma.
\end{equation}
After the two uniform estimates are combined,
\begin{equation}\label{frequency:eq:after-graph}
 F_{C,D}\le C R^{G_c}
 \sup_{\text{root outputs}}\max_{P,Q}\mathsf H_{\rm old}^{P,Q}(c,N)
 +CR^{-20}.
\end{equation}
Here $P,Q$ are occupied depth-$c$ squares for the respective classes, and the child amplitudes are the restrictions of the original ones. The estimates are uniform over the amplitudes, so taking the supremum leaves the additive error unchanged. The factor $R^{G_c}$ multiplies only the main term.

The graph estimate uses the extra $\Par0c$ factor, which is the only occurrence specified at depth at most $c$. We may now remove it from the resulting nonnegative independent-angle quantities by Lemma~\ref{amplitude:lem:deletion}.

On a depth-$c$ square, this factor and its angular derivatives are independent of the spatial variable. Thus $U_j=F_jU_j^{\rm rem}$, where $F_j$ depends only on the angle and $|F_j|\le U_{B_{\rm root}}$. The constant depends on finitely many $K_s$ and is uniform in the class and $N$. Thus
\[
 \int|U_j(r,w)|^2\,dw
 \le U_{B_{\rm root}}^2\int|U_j^{\rm rem}(r,w)|^2\,dw.
\]
Applying this inequality at both endpoints gives the fixed factor $U_{B_{\rm root}}^4$, which is absorbed in~\eqref{frequency:eq:after-graph}.

The remaining list at endpoint $j$ is exactly
\begin{equation}\label{frequency:eq:remaining}
 \M_{\rho_j,>c}([c,N])=\M_{\rho_j}(c,N).
\end{equation}
The spatial multiplier and remaining prescribed orders are unchanged. The measures remain the original class restrictions, divided by $M_j(c)$ in the transforms; they are not normalized by the actual child-square masses.

\subsection{Separating the endpoint profiles}\label{frequency:sec:duplication}
For a fixed pair of depth-$c$ squares and the remaining amplitudes, we put
\[
 E_j(r)=\int_{S^1}|U_j^{\rm rem}(r,w)|^2\,dw,
 \qquad d\lambda_N(r)=\Psi(r/R)\frac{dr}{R}.
\]
The functions $E_j$ are nonnegative, and $\lambda_N$ is a common finite positive measure. Each transform still uses its own model denominator $M_j(c)$. The radial comparison~\eqref{angular:eq:radial-CS} gives
\begin{align}\label{frequency:eq:radial-CS}
 &2^{2(N-c)}\int E_1(r)E_2(r)\,d\lambda_N(r)\notag\\
 &\quad\le
 \left(2^{2(N-c)}\int E_1(r)^2\,d\lambda_N(r)\right)^{1/2}
 \left(2^{2(N-c)}\int E_2(r)^2\,d\lambda_N(r)\right)^{1/2}.
\end{align}

Each expression under a square root uses two copies of the same class probability, square, and remaining amplitude. It is a single-profile independent-angle quantity with coincident squares, as allowed by Theorem~\ref{summation:thm:iteration}.

Let $b_j$ be the total prescribed order at endpoint $j$ after factor removal. Conversion gives combined order at most $B_{\rm root}$, and removal only decreases it, so
\[
 b_1+b_2\le B_{\rm root}.
\]
The duplicated one-profile quantity on side $j$ has combined order $2b_j$. To cover every possible output uniformly, apply Theorem~\ref{summation:thm:iteration} with
\begin{equation}\label{frequency:eq:dup-order}
 B_{\rm in}=2B_{\rm root}.
\end{equation}
This allows for the case in which all $B_{\rm root}$ derivatives occur at one endpoint.

Each duplicated side uses its own condition set~\eqref{frequency:eq:remaining}, at level zero, with its original class probability. Its base and current interval are both the nonterminal $[c,N]$. These are the data required by Theorem~\ref{summation:thm:iteration}. Additional derivative estimates use the common bounds $K_s$ without changing the prescribed-order bound~\eqref{frequency:eq:dup-order}.

Let $H_j$ denote the duplicated one-profile quantity. Theorem~\ref{summation:thm:iteration} with~\eqref{frequency:eq:dup-order} gives, at a common threshold,
\begin{equation}\label{frequency:eq:Hj}
 H_j\le C_S(1+N)^{P_{\rm it}}
 R^{2A_{j,c}-V_j([c,N])+C_*(\kappa+\sigma/\kappa)}.
\end{equation}
The estimate is uniform over the squares and amplitudes in the preceding maxima and suprema.

The geometric mean in~\eqref{frequency:eq:radial-CS} is bounded by
\begin{equation}\label{frequency:eq:two-H}
 C(1+N)^{P_{\rm it}}
 R^{A_{1,c}+A_{2,c}-\frac12(V_1([c,N])+V_2([c,N]))
        +C_*(\kappa+\sigma/\kappa)}.
\end{equation}
The geometric mean retains one copy of the polynomial factor and one copy of the common exponent loss.

We use~\eqref{summation:eq:theorem}, which already includes the $R^{-20}$ error from~\eqref{summation:eq:theorem-strong}. Thus the geometric mean applies to the complete bounds, uniformly at the order~\eqref{frequency:eq:dup-order}.

Combining~\eqref{frequency:eq:graph-factor} and~\eqref{frequency:eq:two-H}, the terms involving $A_{1,c}$ and $A_{2,c}$ cancel:
\begin{align}\label{frequency:eq:cancellation}
 &-A_{1,c}-A_{2,c}+\min_jk_j(0,c)\notag\\
 &\quad+\tfrac12\bigl[2A_{1,c}-V_1([c,N])
                     +2A_{2,c}-V_2([c,N])\bigr]\notag\\
 &=\min_jk_j(0,c)-\tfrac12\bigl[V_1([c,N])+V_2([c,N])\bigr].
\end{align}
Lemma~\ref{frequency:lem:entry-minima} bounds the right side by
\[
 (\beta+C_{\rm low}\sigma)-\tfrac12(4\beta-8/N)
 =-\beta+C_{\rm low}\sigma+4/N.
\]
Adding the graph and single-profile losses gives
\begin{equation}\label{frequency:eq:raw-exponent}
 -\beta+(C_{\rm low}+C_G)\sigma+C_*(\kappa+\sigma/\kappa)+4/N.
\end{equation}
Since $R^{4/N}=16$, the term $4/N$ in the exponent contributes only a fixed multiplicative constant. Substitution in~\eqref{frequency:eq:after-graph} proves~\eqref{frequency:eq:component-raw}.

The graph estimate already includes the cost of taking the spatial maximum. The single-profile bound is uniform over that maximum and the amplitude family, so no further factor counting squares or spatial multipliers is introduced.

\subsection{Choosing the parameters and absorbing the polynomial}\label{frequency:sec:parameters}
We choose the small parameters before the derivative orders so that the main exponent is negative.

The successive choices~\eqref{frequency:eq:parameters} give
\begin{equation}\label{frequency:eq:loss-portions}
 C_*\kappa<\beta/16,\qquad
 (C_{\rm low}+C_G)\sigma<\beta/16,\qquad
 C_*\sigma/\kappa<\beta/16.
\end{equation}
They also imply $C_{\rm low}\sigma<\chi/8$ and $0<\sigma<\kappa$. The geometric and entry hypotheses used above therefore hold.

Before polynomial absorption, the component main term is at most
\[
 C(1+N)^{P_{\rm it}}R^{-13\beta/16}.
\]
These parameters also fix $P_{\rm it}=J_{\rm rec}+1$. The derivative orders affect only multiplicative constants and starting scales.

\begin{lemma}\label{frequency:lem:poly}
For any fixed $P>0$ and $\beta>0$, the threshold
\begin{equation}\label{frequency:eq:Npoly}
 N_{\rm poly}(P,\beta)=
 \left\lceil\max\left\{64,\left(\frac{16P}{\beta}\right)^2\right\}\right\rceil
\end{equation}
ensures $(1+N)^P\le R^{\beta/16}$ for every $N\ge N_{\rm poly}(P,\beta)$.
\end{lemma}
\begin{proof}
The elementary bound $\log_2(1+N)\le\sqrt N$ for $N\ge64$ gives
\[
 P\log_2(1+N)\le P\sqrt N\le\beta N/16
\]
under~\eqref{frequency:eq:Npoly}. Exponentiation proves the claim.
\end{proof}
Using $P=P_{\rm it}$ in~\eqref{frequency:eq:component-raw} gives the common component estimate
\begin{equation}\label{frequency:eq:component-final}
 F_{C,D}\le CR^{-3\beta/4}+CR^{-20}.
\end{equation}
As $\beta<1/10$, the remainder is bounded by the first term for \(R\ge1\). In particular~\eqref{frequency:eq:component-final} implies $F_{C,D}\le CR^{-\beta/2}$.

\subsection{Summing with the original class weights}\label{frequency:sec:weights}
The component estimates are uniform over the class probabilities. The measure identity~\eqref{mass:eq:component-identity} gives
\begin{equation}\label{frequency:eq:weighted-identity}
 \boxed{\displaystyle
 \tau_N=\sum_{C,D}a_{1,C}a_{2,D}\,\tau_{C,D,N}.}
\end{equation}
The identity uses $a_{j,C}\rho_{j,C}=\mu_j|_C$ with no normalization by the retained mass. Each product class keeps its own $c(C,D)$ and cutoff families, including their dependence on the opposite class.

We put $w_{C,D}=a_{1,C}a_{2,D}$. Their total is
\begin{equation}\label{frequency:eq:weight-sum}
 s_N:=\sum_{C,D}w_{C,D}
 =\mu_1(E_{1,N})\mu_2(E_{2,N})\le1.
\end{equation}
For each real $r$, Cauchy--Schwarz with nonnegative weights gives
\begin{align}\label{frequency:eq:weighted-CS}
 |\widehat{\tau_N}(r)|^2
 &=\left|\sum_{C,D}w_{C,D}\widehat{\tau_{C,D,N}}(r)\right|^2\notag\\
 &\le s_N\sum_{C,D}w_{C,D}
                |\widehat{\tau_{C,D,N}}(r)|^2\notag\\
 &\le\sum_{C,D}w_{C,D}|\widehat{\tau_{C,D,N}}(r)|^2.
\end{align}
Integrating and using~\eqref{frequency:eq:component-final}, we obtain
\[
 \int_{R/2}^{2R}|\widehat{\tau_N}(r)|^2\,dr
 \le CR^{-3\beta/4}+CR^{-20}.
\]
Both the main estimate and its error are multiplied by weights whose sum is at most one. Thus the bound has no factor counting the classes.

The cutoffs and distance weight are nonnegative, so $\tau_N$ is a positive measure and
\[
 \widehat{\tau_N}(-r)=\overline{\widehat{\tau_N}(r)}.
\]
The two frequency integrals are equal. Doubling the constant gives~\eqref{frequency:eq:frequency-conclusion}, with a common threshold specified below.

We also retain the positive-measure bounds
\begin{equation}\label{frequency:eq:alpha-basic}
 0\le\tau_N\le\alpha,\qquad
 d_1^{-1/2}\le\alpha(\R)\le d_0^{-1/2}.
\end{equation}
Both measures are carried by the compact set $\dd(K_1\times K_2)$. Moreover
\begin{equation}\label{frequency:eq:exact-TV}
 \|\alpha-\tau_N\|_{\TV}
 =\int W(x,y)(1-\Theta_N(x,y))\,d\mu_1(x)d\mu_2(y),
\end{equation}
because the difference is the pushforward of a positive measure. This is~\eqref{mass:eq:exact-TV}, whose right side is estimated in Theorem~\ref{mass:thm:mass} for the same $\Theta_N$.

\subsection{Uniform constants and thresholds}\label{frequency:sec:orders}\label{frequency:sec:distinctions}
We now choose one starting scale for all admissible class pairs and amplitudes.

We first fix the geometry, profile coefficients, and cutoff construction, then choose $\beta,\kappa,\sigma$ by~\eqref{frequency:eq:parameters}. This fixes $J_0$, $J_{\rm rec}$, and $P_{\rm it}$. For error exponent $T_{\rm out}=20$ in the initial conversion and mass-ratio exponent zero,~\eqref{frequency:eq:root-order} fixes $B_{\rm root}$. The initial order for Theorem~\ref{summation:thm:iteration} is $2B_{\rm root}$.

Following~\eqref{summation:eq:T}--\eqref{summation:eq:inverse-orders}, we set
\begin{equation}\label{frequency:eq:recursive-accuracy}
 \Gamma=6+C_G,\qquad
 T_{\rm loc}=(\Gamma+1)(J_{\rm rec}+1)+20.
\end{equation}
The error exponent is independent of $B_{\rm root}$ and the derivative bounds. For the later inverse conversions, we put $S=T_{\rm loc}+2$ and use $\mu=C_{\rm prof}\sigma$, giving
\begin{gather}
 P_{\rm rec}=S+2C_{\rm prof}\sigma+2,\qquad
 J_{\rm inv,rec}=\left\lfloor\frac{P_{\rm rec}+9}{10\sigma}\right\rfloor+1,\notag\\
 B_* =\max\{2B_{\rm root},\ 2(J_{\rm inv,rec}-1)\}.\label{frequency:eq:recursive-order}
\end{gather}
At depth zero, the probabilities have mass ratio one. On a smaller square, mass divided by model mass is at most $R^{C_{\rm prof}\sigma}$, accounting for the different mass-ratio exponents in the two order formulas.

After fixing $B_*$, choose the additional finite derivative orders needed to estimate the remainders. The first graph estimate may use an integer $m_{\rm out}>(20+21)/(26\sigma)$, and later graph estimates may use $m_{\rm rec}>(S+21)/(26\sigma)$. The direct circle orders are those of the estimates already proved. Together they use finitely many common $K_s$.

The bound $B_*$ includes the doubled order required above. Subsequent high-frequency conversions start with undifferentiated cutoffs; graph restriction and factor removal do not increase prescribed order, and the far estimate returns a positive count. Derivatives used to estimate remainders therefore do not accumulate in later amplitudes.

We take $N_s$ to be the maximum of the class threshold in Corollary~\ref{dyad:cor:prepared}, $N_{\rm ent}$, the cutoff and product thresholds, the first inverse and graph thresholds at the chosen orders, the single-profile threshold for initial order $2B_{\rm root}$, and $N_{\rm poly}(P_{\rm it},\beta)$. We also require $N\ge2$ for~\eqref{frequency:eq:root-list-bounds}. The cutoff and product thresholds include $R^{450\sigma}\ge2J_0$ and $4N^3\le R^\sigma$.

Each threshold is uniform over the family covered by its theorem. Their finite maximum therefore works even when the profiles, common starting depth, selected conditions, and spatial multipliers vary with $N$.

The constants depend only on $d_0,d_1$, the arc margins, $\psi,\Psi$, the profile coefficients, $\chi,\kappa,\sigma$, and the finite collection of $K_s$ just fixed. Theorem~\ref{prep:thm:preparation} supplies common numerical geometric bounds, and rotating the fixed cutoff does not change its derivative bounds. After the parameters are chosen from $\eta$, the constants are uniform over the class probabilities, original class masses, number of classes, and amplitudes.

The initial inverse and graph estimates give normalized errors $O(R^{-20})$. Combining them multiplies the graph error only by a fixed constant; $R^{G_c}$ and the factor-removal bound multiply only the main term. The single-profile theorem includes its internal errors, uniformly at the doubled initial order. This gives the remainder in~\eqref{frequency:eq:component-raw}. In~\eqref{frequency:eq:weighted-CS}, the component errors are then summed with total weight at most one.

\begin{proof}[Completion of Theorem~\ref{frequency:thm:frequency}]
The component calculation proves~\eqref{frequency:eq:component-raw}. The parameter bounds and Lemma~\ref{frequency:lem:poly} give~\eqref{frequency:eq:component-final}. Summation with the original class weights in~\eqref{frequency:eq:weighted-identity}--\eqref{frequency:eq:weighted-CS} gives the same bound for $\tau_N$. Since $\tau_N$ is positive, the positive- and negative-frequency integrals are equal. The common starting scale just chosen is independent of the class pair and amplitudes. This proves~\eqref{frequency:eq:frequency-conclusion} with the stated dependence of $C_s,N_s$.
\end{proof}

\paragraph*{Why Fourier decay alone is insufficient.}
Fourier decay of the approximants does not by itself imply convergence to the target measure in total variation. Indeed, let $\alpha$ be any nonzero finite positive measure and set $\widetilde\tau_N=R^{-2}\alpha$. Then $0\le\widetilde\tau_N\le\alpha$ and
\[
 \int_{R/2\le|r|\le2R}|\widehat{\widetilde\tau_N}(r)|^2\,dr
 \le3\alpha(\R)^2R^{-3}.
\]
Nevertheless,
\[
 \|\alpha-\widetilde\tau_N\|_{\TV}
 =(1-R^{-2})\alpha(\R)\longrightarrow\alpha(\R)>0.
\]
Thus positivity, domination, and Fourier decay alone do not imply approximation to $\alpha$ in total variation.

\subsection{The same measures satisfy both estimates}\label{frequency:sec:common}
The mass and Fourier estimates apply to the same sequence $\tau_N$ and fixed positive measure $\alpha$. We collect their common constants for reconstruction.

\Needspace{265pt}
\begin{corollary}\label{frequency:cor:common}
Fix a normalized input margin $0<\eta<1$, put $\chi=\min\{\eta/8,1/16\}$, and choose the parameters in \eqref{frequency:eq:parameters}, for instance by taking half of each minimum. There are positive $a_\eta,\gamma_\eta,C_e,C_s$ and an integer $N_*$, depending only on these fixed numerical data, such that the measures in \eqref{mass:eq:measures} satisfy
\begin{align}\label{frequency:eq:common-estimates}
 \|\alpha-\tau_N\|_{\TV}&\le C_e2^{-a_\eta N},\notag\\
 \int_{2^{N-1}\le|r|\le2^{N+1}}|\widehat{\tau_N}(r)|^2\,dr
 &\le C_s2^{-\gamma_\eta N}\qquad(N\ge N_*),
\end{align}
where
\begin{equation}\label{frequency:eq:common-data}
 a_\eta=\frac{\sigma\chi}{8},\qquad
 \gamma_\eta=\frac{\beta}{2}=\frac{\chi}{40},\qquad
 N_* =\max\{N_e,N_s\}.
\end{equation}
The measure $\alpha$ is independent of $N$ and satisfies
\[
 2\le\alpha(\R)\le\sqrt{2048/7},\qquad 0\le\tau_N\le\alpha,
\]
and all these measures are carried by the compact cross-distance set of the prepared pair. No normalization by the total retained mass is made.
\end{corollary}
\begin{proof}
The parameter choices satisfy those of Corollary~\ref{mass:cor:prepared}. That result supplies the first line, $a_\eta$, and the mass bounds. Theorem~\ref{frequency:thm:frequency} supplies the second line for the same $\Theta_N$ and original class weights. Taking the maximum of the two uniform thresholds makes the estimates simultaneous. The support statement follows from the distance pushforwards in~\eqref{mass:eq:measures}.
\end{proof}

The auxiliary restriction $\Gamma_N$ is not inserted into either measure. The factor counting classes in the mass estimate comes from a failure bound carrying only one endpoint's weight; the Fourier estimate retains both original weights. Neither estimate requires monotonicity of $\tau_N$ or a common class decomposition across scales.

Corollary~\ref{frequency:cor:common} supplies the two tails and mass bounds used to choose a common reconstruction cutoff in Section~\ref{main:normalized-proof}.

\clearpage

\section{The normalized quantitative conclusion}\label{main:normalized-proof}
We apply the uniform reconstruction bound of Corollary~\ref{rec:cor:uniform} to the target measure and its approximants from Corollary~\ref{frequency:cor:common}. For each fixed $0<\eta<1$, we first choose the small parameters and then select a common reconstruction cutoff using the uniform mass, total-variation, and Fourier-decay bounds. This yields a positive lower bound for the Lebesgue measure of the prepared cross-distance set, uniform over all normalized inputs. The similarity from Section~\ref{main:preparation} then transfers this bound to the original distance set, with a factor of eight in length.

\subsection{Common numerical data}\label{normalized:sec:data}
Fix $0<\eta<1$. Let $F\subset\overline B(0,1)$ be compact, and let $\nu$ be a Borel probability carried by $F$ satisfying~\eqref{main:normalized-ball}. Set
\begin{equation}\label{normalized:eq:chi-beta}
 \chi=\min\{\eta/8,1/16\},\qquad \beta=\chi/20.
\end{equation}
We use the structural constants fixed in Sections~\ref{main:cutoffs-mass}--\ref{main:frequency}:
\[
 C_{\rm prof}=10,\qquad C_{\rm low}=12,\qquad
 C_{\rm test}=1000,\qquad C_G=1104,\qquad C_*=5032.
\]
We choose parameters strictly within the ranges permitted by~\eqref{frequency:eq:parameters} by setting
\begin{equation}\label{normalized:eq:kappa}
 \kappa=\tfrac12\min\left\{10^{-4},\frac{\beta}{16(1+C_*)}\right\},
\end{equation}
\begin{equation}\label{normalized:eq:sigma}
 \sigma=\tfrac12\min\left\{
 \frac\kappa{1000},\frac\chi{8(1+C_{\rm low})},
 \frac\beta{16(1+C_{\rm low}+C_G)},
 \frac{\beta\kappa}{16(1+C_*)}\right\}.
\end{equation}
Both minima are positive, and the factors of one half ensure the required strict inequalities. The value of $\beta$ is used to select the entry depth $c(C,D)$ for each ordered class pair through~\eqref{mass:eq:entry}; this depth remains fixed in the construction of the corresponding approximant at scale $N$.

With these choices, Corollary~\ref{frequency:cor:common} gives constants $C_e,C_s>0$ and an integer $N_*$ common to the normalized family, with
\begin{equation}\label{normalized:eq:rates}
 a_\eta=\frac{\sigma\chi}{8},\qquad
 \gamma_\eta=\frac\chi{40},\qquad
 C_e=(6+2M_*(\eta))\sqrt{2048/7}.
\end{equation}
Both rates are positive. We fix one admissible choice of $C_s$ and $N_*$ for all normalized inputs with this $\eta$. As established in Sections~\ref{main:iteration}--\ref{main:frequency}, the required finite derivative orders are chosen before the common starting scales. Consequently, all the numerical choices above depend only on $\eta$.

\subsection{The fixed measure and the common tails}\label{normalized:sec:tails}
We apply Theorem~\ref{prep:thm:preparation} to the fixed input $(F,\nu)$, obtaining compact sets $K_j^0\subset F$, their images $K_j=S(K_j^0)$ under the fixed similarity, and probabilities $\mu_j$. Define
\begin{equation}\label{normalized:eq:compact-set}
 \mathcal K_F=D(K_1,K_2)=\{|x-y|:x\in K_1,\ y\in K_2\}.
\end{equation}
Since $K_1\times K_2$ is compact and the distance map is continuous, $\mathcal K_F$ is compact. It may depend on the input $(F,\nu)$.

We use the measures defined in~\eqref{mass:eq:measures}:
\begin{equation}\label{normalized:eq:measures}
 \alpha=\dd_\#\bigl(W\mu_1\otimes\mu_2\bigr),\qquad
 \tau_N=\dd_\#\bigl(W\Theta_N\mu_1\otimes\mu_2\bigr),
 \qquad W(x,y)=|x-y|^{-1/2}.
\end{equation}
The prepared probabilities $\mu_1,\mu_2$ and the target measure $\alpha$ are independent of $N$, while the class decompositions, the entry depths $c(C,D)$ of ordered class pairs, and $\Theta_N$ may vary with $N$. Both $\alpha$ and $\tau_N$ are carried by $\mathcal K_F$. Corollary~\ref{frequency:cor:common} gives
\begin{equation}\label{normalized:eq:mass-bounds}
 2\le\alpha(\R)\le\sqrt{2048/7},\qquad 0\le\tau_N\le\alpha.
\end{equation}
We may therefore take the reconstruction mass parameters to be
\begin{equation}\label{normalized:eq:mass-data}
 m_\eta=2,\qquad M_\eta^2=2048/7.
\end{equation}
Here $m_\eta$ is a lower bound for $\alpha(\R)$, and $M_\eta^2$ is the square of the chosen upper bound. These bounds already incorporate the distance weight $W$ and the normalization of the prepared probabilities.

For every integer $J\ge N_*$, let $E_J$ and $S_J$ denote the tails defined in~\eqref{rec:eq:E}--\eqref{rec:eq:S} for this target measure and its approximants. Corollary~\ref{frequency:cor:common}, followed by summation of the geometric series, gives
\begin{align}
 E_J&\le\frac{C_e2^{-a_\eta J}}{1-2^{-a_\eta}},
       \label{normalized:eq:variation-tail}\\
 S_J&\le\frac{C_s2^{-\gamma_\eta J}}{1-2^{-\gamma_\eta}}.
       \label{normalized:eq:frequency-tail}
\end{align}
Both series converge because $a_\eta,\gamma_\eta>0$. The Fourier estimate already includes both positive and negative frequencies.

The total-variation and Fourier estimates apply to the same approximants $\tau_N$, defined using the same $\Theta_N$. The auxiliary restriction $\Gamma_N$, used only to estimate the removed mass, does not enter $\Theta_N$. At each scale, $\tau_N$ retains the original class weights without renormalizing the retained union. 

\subsection{One reconstruction cutoff for every input}\label{normalized:sec:cutoff}
We choose a single reconstruction cutoff for the entire normalized family so that the total-variation tail is small relative to the common lower mass bound:
\begin{equation}\label{normalized:eq:Jeta}
 J_\eta=
 \max\left\{N_*,1,
 \left\lceil\frac1{a_\eta}\log_2
 \frac{16C_e}{1-2^{-a_\eta}}\right\rceil\right\}.
\end{equation}
This is a finite integer depending only on $\eta$. The requirement $J_\eta\ge N_*$ ensures that all tail estimates apply, while $J_\eta\ge1$ fixes the low-frequency convention.

By~\eqref{normalized:eq:Jeta},
\[
 2^{a_\eta J_\eta}\ge
 \frac{16C_e}{1-2^{-a_\eta}}.
\]
Together with~\eqref{normalized:eq:variation-tail}, this gives
\begin{equation}\label{normalized:eq:small-tail}
 16E_{J_\eta}\le
 \frac{16C_e2^{-a_\eta J_\eta}}{1-2^{-a_\eta}}
 \le1=\frac{m_\eta}{2}.
\end{equation}
The fixed reconstruction kernel satisfies $\ct=3$ and $\CP<16$, as established in Section~\ref{main:reconstruction}. Consequently,
\begin{equation}\label{normalized:eq:mass-margin}
 \CP E_{J_\eta}\le1<2=m_\eta,\qquad
 m_\eta-\CP E_{J_\eta}\ge1.
\end{equation}
Thus the strict mass inequality required by the reconstruction theorem holds uniformly. This choice satisfies~\eqref{rec:conditional-J} with $m=m_\eta=2$ and depends only on the common lower mass bound, not on the individual value of $\alpha(\R)$.
\subsection{Completion of the normalized theorem}\label{normalized:sec:theorem}
Using the common mass bound and Fourier tail estimate, define
\begin{equation}\label{normalized:eq:denominator}
 \mathcal B_\eta
 =3\cdot2^{J_\eta-1}\frac{2048}{7}
 +\frac{C_s2^{-\gamma_\eta J_\eta}}{1-2^{-\gamma_\eta}}.
\end{equation}
The first term controls the low-frequency part of the reconstruction bound, while the second controls the high-frequency tail. Both are finite, and $\mathcal B_\eta>0$.

\Needspace{205pt}
\begin{theorem}\label{normalized:thm:main}
For each $0<\eta<1$, fix the numerical choices in~\eqref{normalized:eq:chi-beta}--\eqref{normalized:eq:sigma} and the common constants supplied by Corollary~\ref{frequency:cor:common}. Let $J_\eta$ and $\mathcal B_\eta$ be defined by~\eqref{normalized:eq:Jeta} and~\eqref{normalized:eq:denominator}. Then
\begin{equation}\label{normalized:eq:b-eta}
 \boxed{\displaystyle b(\eta)=\frac{16\pi}{\mathcal B_\eta}>0}
\end{equation}
is an admissible constant in Theorem~\ref{main:normalized}. Thus every compact $F\subset\overline B(0,1)$ carrying a Borel probability satisfying~\eqref{main:normalized-ball} obeys $|\Delta F|\ge b(\eta)$, uniformly over the input set and probability.
\end{theorem}
\begin{proof}
Fix an admissible input $(F,\nu)$. The measures~\eqref{normalized:eq:measures} are carried by the compact set $\mathcal K_F$. Equations~\eqref{normalized:eq:mass-bounds}, \eqref{normalized:eq:variation-tail}--\eqref{normalized:eq:frequency-tail}, and~\eqref{normalized:eq:mass-margin} verify the mass bounds, summability conditions, and strict mass inequality required by Theorem~\ref{rec:thm:reconstruction}.

Applying that theorem with $m=2$, $M_0^2=2048/7$, and $\ct=3$ gives
\begin{align}\label{normalized:eq:prepared-length}
 |\mathcal K_F|
 &\ge
 \frac{2\pi(2-\CP E_{J_\eta})^2}
 {3\cdot2^{J_\eta-1}(2048/7)+S_{J_\eta}}
 \ge\frac{2\pi}{\mathcal B_\eta}.
\end{align}
The second inequality follows from~\eqref{normalized:eq:mass-margin}, which gives $2-\CP E_{J_\eta}\ge1$, and from~\eqref{normalized:eq:frequency-tail}, which bounds the denominator by $\mathcal B_\eta$.

The preparation similarity~\eqref{prep:eq:S} is $S(z)=z/8+(1/2,1/2)$. For every $x\in K_1^0$ and $y\in K_2^0$,
\[
 |x-y|=8|S(x)-S(y)|.
\]
Consequently,~\eqref{prep:eq:distance-inclusion} gives
\begin{equation}\label{normalized:eq:return}
 8\mathcal K_F\subseteq\Delta F,\qquad
 |\Delta F|\ge |8\mathcal K_F|=8|\mathcal K_F|.
\end{equation}
Both sets are compact and hence Lebesgue measurable. Combining~\eqref{normalized:eq:return} with~\eqref{normalized:eq:prepared-length} yields $|\Delta F|\ge16\pi/\mathcal B_\eta=b(\eta)$, proving Theorem~\ref{main:normalized}.
\end{proof}

Theorem~\ref{rec:thm:reconstruction} also shows that $\alpha$ is absolutely continuous, with density $g+h$ for some $g\in L^2(\R)$ and $h\in L^1(\R)$. The length bound uses the $L^2$ control of $g$ and the small $L^1$ norm of $h$; neither summand need be nonnegative or individually supported on $\mathcal K_F$.

\bigskip \paragraph{\emph{Uniformity of the constants.}}
For each fixed $\eta$, Theorem~\ref{prep:thm:preparation} provides uniform ball, separation, and radial-projection bounds, although the prepared sets and probabilities depend on $(F,\nu)$. The parameters, finite derivative orders, and common starting scales are chosen independently of the input. At each scale $N$, the classes, profiles, and entry depths $c(C,D)$ may vary, but the total-variation and Fourier estimates apply to the same approximants $\tau_N$, with their original class weights. Consequently, $J_\eta$, $\mathcal B_\eta$, and $b(\eta)$ depend only on $\eta$ and the fixed universal data. The formula establishes a uniform positive constant for each fixed $\eta$; no numerical evaluation, asymptotic behavior as $\eta\downarrow0$, or optimality is claimed.

\bigskip \paragraph{\emph{The two scaling factors.}}
The factor $\sqrt8$ in~\eqref{prep:eq:weighted-scaling} comes from scaling the weight $|x-y|^{-1/2}$ and is already incorporated into the mass bounds~\eqref{normalized:eq:mass-bounds}. By contrast, the factor eight in~\eqref{normalized:eq:return} comes from the scaling of one-dimensional Lebesgue measure when the prepared cross-distance set is returned to the original distance line. These two factors enter at different stages and are not applied twice.

\clearpage
\section{Localization and quantitative consequences}\label{main:localization}
We select a closed ball on which the ratio of mass to the $t$th power of the radius is large, where $t$ is smaller than the exponent in the original ball bound. Restricting the measure to this ball, normalizing by the retained mass, and rescaling gives a probability with ball-growth coefficient two. A lower bound for the selected radius then allows us to deduce Theorem~\ref{main:localized} from Theorem~\ref{main:normalized}.

We use the open and closed balls defined in~\eqref{main:balls} and the measure conventions of Subsection~\ref{main:measure-conventions}. A compact set $E$ carrying the probability need not coincide with its topological support. For a compact set $H\subset\R^2$ and $T(x)=(x-z)/r$, $r>0$, we have
\begin{equation}\label{loc:eq:distance-similarity}
 \Delta(T(H))=r^{-1}\Delta H,\qquad
 |\Delta(T(H))|=r^{-1}|\Delta H|.
\end{equation}
Both distance sets are compact and hence Lebesgue measurable. In the selection below, we choose a ball for which $\mu(\overline B(x,r))/r^t$ is at least half its supremum over $x\in\R^2$ and $0<r\le1$. 

\subsection{The elementary localization theorem}\label{loc:sec:localization}
The strict gap $s-t>0$ allows the upper $s$-power ball bound to give a quantitative lower bound on the radius of a ball with a sufficiently large mass-to-radius ratio at exponent $t$.

\begin{theorem}\label{loc:thm:localization}
Let $0<t<s$, let $E\subset\R^2$ be compact with $\diam E=1$, and let $\mu$ be a Borel probability carried by $E$. Suppose $A\ge1$ and
\begin{equation}\label{loc:eq:growth-unit}
 \mu(B(x,r))\le Ar^s
 \qquad(x\in\R^2,\ 0<r\le1).
\end{equation}
Then there exist $x_0\in\R^2$ and $0<r_0\le1$ such that
\begin{equation}\label{loc:eq:radius}
 r_0\ge(2A)^{-1/(s-t)}.
\end{equation}
Define the compact sets
\begin{equation}\label{loc:eq:F}
 H=E\cap\overline B(x_0,r_0),\qquad
 F=\frac{H-x_0}{r_0}\subset\overline B(0,1),
\end{equation}
and set
\begin{equation}\label{loc:eq:local-probability}
 q_{\rm loc}=\mu(\overline B(x_0,r_0)),\qquad
 T(x)=\frac{x-x_0}{r_0},\qquad
 \nu=T_\#\left(\frac{\mu|_{\overline B(x_0,r_0)}}{q_{\rm loc}}\right).
\end{equation}
Then $q_{\rm loc}=\mu(H)>0$, and $\nu$ is a Borel probability carried by $F$ satisfying
\begin{equation}\label{loc:eq:normalized-t}
 \nu(B(z,u))\le2u^t
 \qquad(z\in\R^2,\ 0<u\le1).
\end{equation}
Moreover,
\begin{equation}\label{loc:eq:distance-inclusion}
 r_0\Delta F=\Delta H\subseteq\Delta E.
\end{equation}
The selected center $x_0$ need not belong to $E$.
\end{theorem}

The probability normalization and diameter assumption give a positive lower bound for the supremum of the mass-to-radius ratios, while the upper ball estimate makes this supremum finite. The strict gap $s-t$ then gives the lower bound on the selected radius. Taking at least half the supremum yields the coefficient two in~\eqref{loc:eq:normalized-t}. The compact restriction provides the set required by the normalized theorem. Subsection~\ref{loc:sec:attainment} gives an optional refinement in which the supremum is attained.

\paragraph*{Proof of Theorem~\ref{loc:thm:localization}.}
We use closed balls in the selection so that their intersections with the compact carrying set are compact.

\begin{lemma}\label{loc:lem:closed}
Under~\eqref{loc:eq:growth-unit},
\begin{equation}\label{loc:eq:closed-growth}
 \mu(\overline B(x,r))\le Ar^s
 \qquad(x\in\R^2,\ 0<r\le1).
\end{equation}
\end{lemma}

\begin{proof}
For $0<r<1$, choose $r_n\downarrow r$ with $r<r_n\le1$. Since $\overline B(x,r)=\bigcap_n B(x,r_n)$ and $\mu$ is finite, continuity from above gives $\mu(\overline B(x,r))=\lim_n\mu(B(x,r_n))\le\lim_n Ar_n^s=Ar^s$. For $r=1$, the probability normalization gives $\mu(\overline B(x,1))\le\mu(\R^2)=1\le A$.
\end{proof}

This argument includes any mass on $\partial B(x,r)$. The open and closed balls need not have equal mass, as the example in Subsection~\ref{loc:subsec:boundary-example} shows.

Define
\begin{equation}\label{loc:eq:Mt}
 M_t=\sup_{x\in\R^2,\ 0<r\le1}
             \frac{\mu(\overline B(x,r))}{r^t}.
\end{equation}
For any $x\in E$, the diameter assumption gives $E\subseteq\overline B(x,1)$. Since $\mu$ is carried by $E$,
\[
 \mu(\overline B(x,1))=1.
\]
On the other hand, Lemma~\ref{loc:lem:closed} gives
\[
 \frac{\mu(\overline B(x,r))}{r^t}
 \le Ar^{s-t}\le A.
\]
Consequently,
\begin{equation}\label{loc:eq:M-range}
 1\le M_t\le A.
\end{equation}

The supremum is taken over all centers in $\R^2$, so it also controls balls centered at $x_0+r_0z$ after rescaling, even when these centers lie outside $E$ or the selected ball.

Since $1\le M_t<\infty$, we may choose $x_0\in\R^2$ and $0<r_0\le1$ such that
\begin{equation}\label{loc:eq:half-maximum}
 q_{\rm loc}:=\mu(\overline B(x_0,r_0))
 \ge\tfrac12M_t r_0^t>0.
\end{equation}
By~\eqref{loc:eq:closed-growth},
\[
 \tfrac12M_t r_0^t\le q_{\rm loc}\le Ar_0^s.
\]
Dividing by $Ar_0^t$ and using $M_t\ge1$, we obtain
\begin{equation}\label{loc:eq:radius-detail}
 r_0^{s-t}\ge\frac{M_t}{2A}\ge\frac1{2A},
\end{equation}
which proves~\eqref{loc:eq:radius}. Moreover,
\begin{equation}\label{loc:eq:mass-lower}
 q_{\rm loc}\ge\tfrac12 r_0^t
 \ge\tfrac12(2A)^{-t/(s-t)}.
\end{equation}

Since $\mu$ is carried by $E$, the restricted measure in~\eqref{loc:eq:local-probability} has mass $q_{\rm loc}$ and is carried by $H$. Thus $\nu$ is a probability measure. For every $z\in\R^2$ and $0<u\le1$,
\begin{align}
 \nu(B(z,u))
 &=q_{\rm loc}^{-1}\mu\bigl(\overline B(x_0,r_0)
                        \cap B(x_0+r_0z,r_0u)\bigr)\notag\\
 &\le q_{\rm loc}^{-1}\mu\bigl(\overline B(x_0+r_0z,r_0u)\bigr)\notag\\
 &\le q_{\rm loc}^{-1}M_t(r_0u)^t
 \le2u^t.\label{loc:eq:rescaled-ball-proof}
\end{align}
Since $0<r_0u\le1$, the definition~\eqref{loc:eq:Mt} applies to the closed ball centered at $x_0+r_0z$. This ball need not be contained in the selected ball.

The set $H$ in~\eqref{loc:eq:F} is compact, being the intersection of the compact set $E$ with a closed ball. Therefore $F=T(H)\subset\overline B(0,1)$ is compact and carries $\nu$.

Finally, for every $x,y\in H$,
\[
 |x-y|=r_0|T(x)-T(y)|.
\]
Taking all pairs in $H$ proves the equality and inclusion in~\eqref{loc:eq:distance-inclusion}, completing the proof of Theorem~\ref{loc:thm:localization}.\qed

The coefficient two comes from taking half the supremum. More generally, for a fixed $0<\theta<1$, a choice satisfying
\[
 q_{\rm loc}\ge\theta M_t r_0^t
\]
gives
\[
 r_0\ge(\theta/A)^{1/(s-t)},\qquad
 \nu(B(z,u))\le\theta^{-1}u^t.
\]
Thus $\theta=1/2$ gives the normalized hypothesis used here. The selected ball may depend on $\theta$.

\subsection{Attainment of the supremum}\label{loc:sec:attainment}
The strict exponent gap and compactness of $E$ also ensure that the supremum $M_t$ is attained. The following refinement is not needed for the main theorem.

\begin{proposition}\label{loc:prop:attainment}
Under the assumptions of Theorem~\ref{loc:thm:localization}, the supremum $M_t$ is attained by a closed ball $\overline B(x_*,r_*)$ with $0<r_*\le1$. Set $q_{\rm max}=\mu(\overline B(x_*,r_*))$ and define $\nu_*$ by restricting $\mu$ to this ball, dividing by $q_{\rm max}$, and pushing forward under $T_*(x)=(x-x_*)/r_*$. Then
\begin{equation}\label{loc:eq:exact-improvement}
 r_*\ge A^{-1/(s-t)},\qquad
 \nu_*(B(z,u))\le u^t\quad(0<u\le1).
\end{equation}
\end{proposition}
\begin{proof}
Choose a sequence of balls $\overline B(x_n,r_n)$ whose mass-to-radius ratios converge to $M_t$. After discarding finitely many terms, we may assume that every ratio is at least $M_t/2$. The same argument as in~\eqref{loc:eq:radius-detail} applies to each of these balls and gives $r_n\ge\delta$, where
\[
 \delta=(2A)^{-1/(s-t)}>0.
\]
Each ball has positive mass and therefore meets $E$, so $\operatorname{dist}(x_n,E)\le r_n\le1$. Since $E$ is compact, its closed distance-one neighborhood is compact. Passing to a subsequence, we obtain $x_n\to x_*$ and $r_n\to r_*\in[\delta,1]$.

For every $\varepsilon>0$, all sufficiently large $n$ satisfy
\[
 \overline B(x_n,r_n)\subseteq\overline B(x_*,r_*+\varepsilon).
\]
Taking the limit superior in $n$ and then letting $\varepsilon\downarrow0$, continuity from above gives
\[
 \limsup_n\mu(\overline B(x_n,r_n))
 \le\mu(\overline B(x_*,r_*)).
\]
This argument does not use the growth estimate at radii greater than one. Since $r_n\to r_*>0$ and the ratios converge to $M_t$, we obtain
\[
 \frac{\mu(\overline B(x_*,r_*))}{r_*^t}\ge M_t.
\]
The reverse inequality follows from the definition of $M_t$, so $q_{\rm max}=M_t r_*^t$. By~\eqref{loc:eq:closed-growth}, $q_{\rm max}\le Ar_*^s$. Consequently, $r_*^{s-t}\ge M_t/A\ge1/A$, proving the radius bound in~\eqref{loc:eq:exact-improvement}.

For every $z\in\R^2$ and $0<u\le1$, the definition of $M_t$ and the identity $q_{\rm max}=M_t r_*^t$ give
\[
 \nu_*(B(z,u))
 \le\frac{\mu(\overline B(x_*+r_*z,r_*u))}{q_{\rm max}}
 \le\frac{M_t(r_*u)^t}{M_t r_*^t}
 =u^t.
\]
This proves the second assertion.
\end{proof}

Using this refinement would replace $2A$ by $A$ in the distance bounds below. The principal statements retain $2A$ and use only the half-supremum argument. The maximizing ball may depend on $\mu$, but the radius bound and normalized ball-growth estimate are uniform.

\subsection{Diameter and concentration dependence}\label{loc:sec:normalized}
We now apply Theorem~\ref{main:normalized} with the function $b(\eta)$ fixed in~\eqref{normalized:eq:b-eta}. The same function applies to every probability obtained by localization.

\phantomsection\label{loc:subsec:unit-consequence}
First suppose that $\diam E=1$ and~\eqref{loc:eq:growth-unit} holds with
\[
 s=1+\varepsilon,\qquad t=1+\eta,
 \qquad 0<\eta<\varepsilon<1.
\]
Theorem~\ref{loc:thm:localization} gives a compact set $F\subset\overline B(0,1)$ carrying a probability satisfying~\eqref{main:normalized-ball}. Hence Theorem~\ref{main:normalized} gives $|\Delta F|\ge b(\eta)$. By~\eqref{loc:eq:distance-inclusion} and~\eqref{loc:eq:radius},
\begin{equation}\label{loc:eq:unit-consequence}
 |\Delta E|\ge r_0|\Delta F|
 \ge(2A)^{-1/(\varepsilon-\eta)}b(\eta).
\end{equation}
Here $s-t=\varepsilon-\eta$.

\begin{theorem}\label{loc:thm:diameter}
Let $0<\eta<\varepsilon<1$, let $E\subset\R^2$ be compact with $D=\diam E>0$, and suppose a Borel probability $\mu$ carried by $E$ satisfies
\begin{equation}\label{loc:eq:original-growth}
 \mu(B(x,r))\le A(r/D)^{1+\varepsilon}
 \qquad(x\in\R^2,\ 0<r\le D),\qquad A\ge1.
\end{equation}
Then
\begin{equation}\label{loc:eq:diameter-conclusion}
 \boxed{|\Delta E|\ge D(2A)^{-1/(\varepsilon-\eta)}b(\eta).}
\end{equation}
For fixed $\eta$ and $\varepsilon$, the bound holds uniformly over all admissible $E,\mu,D,A$.
\end{theorem}
\begin{proof}
Choose any $z_0\in\R^2$ and set
\[
 S_D(x)=\frac{x-z_0}{D},\qquad E'=S_D(E),\qquad \mu'=(S_D)_\#\mu.
\]
Then $\diam E'=1$, and $\mu'$ is a Borel probability carried by $E'$. For every $z\in\R^2$ and $0<u\le1$,
\[
 \mu'(B(z,u))=\mu(B(z_0+Dz,Du))\le Au^{1+\varepsilon}.
\]
Since $0<Du\le D$, this uses~\eqref{loc:eq:original-growth} only within its stated radius range. Applying~\eqref{loc:eq:unit-consequence} to $(E',\mu')$ and using $\Delta E=D\Delta E'$, we obtain~\eqref{loc:eq:diameter-conclusion}.
\end{proof}

This proves Theorem~\ref{main:localized}, with the same function $b$ as in Theorem~\ref{main:normalized}.

The localized restriction can also be expressed in the original coordinates. If $(x_0,r_0)$ is the ball selected for $\mu'$, set
\[
 z_*=z_0+Dx_0,\qquad \ell_0=Dr_0.
\]
Then
\begin{gather}
 H_0=E\cap\overline B(z_*,\ell_0),\qquad
 F=\frac{H_0-z_*}{\ell_0},\label{loc:eq:original-local-set}\\
 q_{\rm loc}=\mu(\overline B(z_*,\ell_0)),\qquad
 \nu=(x\mapsto(x-z_*)/\ell_0)_\#(\mu|_{H_0}/q_{\rm loc}),\label{loc:eq:original-local-law}\\
 \ell_0\ge D(2A)^{-1/(\varepsilon-\eta)},\qquad
 \ell_0\Delta F\subseteq\Delta E.\label{loc:eq:physical-radius}
\end{gather}
Here $q_{\rm loc}=\mu(H_0)>0$ is the actual retained mass, and the similarity has physical scale $\ell_0=Dr_0$. The factor of eight from preparation is already included in $b(\eta)$, so returning to the original distance set multiplies the length bound only by $\ell_0$.

The compact set $F$ lies in the closed unit disk, but its diameter need not equal one and may be as large as two. This is precisely the normalization required by Theorem~\ref{main:normalized}. Since the constants and reconstruction cutoff in that theorem are uniform, the dependence on $A$ and $D$ enters through the radius bound in~\eqref{loc:eq:physical-radius}.

\subsection{Every fixed power above the endpoint}\label{loc:sec:powers}
Choosing a smaller normalized margin $\eta$ decreases the exponent of the concentration coefficient $A$, but also changes the constant $b(\eta)$ supplied by the normalized theorem.

\begin{corollary}\label{loc:cor:power}
Fix $0<\varepsilon<1$ and $p>1/\varepsilon$. Let $E\subset\R^2$ be compact with $D=\diam E>0$, and suppose that a Borel probability $\mu$ carried by $E$ satisfies~\eqref{loc:eq:original-growth} for some $A\ge1$. Then
\begin{equation}\label{loc:eq:power}
 \boxed{|\Delta E|\ge c_{\varepsilon,p}DA^{-p},\qquad
 c_{\varepsilon,p}=2^{-p}b\!\left(\varepsilon-\frac1p\right)>0.}
\end{equation}
\end{corollary}
\begin{proof}
Set $\eta=\varepsilon-1/p$. Since $p>1/\varepsilon$, we have $0<\eta<\varepsilon<1$ and $1/(\varepsilon-\eta)=p$. Substituting into~\eqref{loc:eq:diameter-conclusion} gives~\eqref{loc:eq:power}, since $(2A)^{-p}=2^{-p}A^{-p}$.
\end{proof}

The constant $c_{\varepsilon,p}$ depends only on the fixed pair $(\varepsilon,p)$ and is uniform over the admissible sets, probabilities, diameters, and concentration coefficients. No uniformity as $p$ varies is asserted.

For example, $\varepsilon=1/2$ and $p=3$ give $\eta=1/6$ and $|\Delta E|\ge\tfrac18b(1/6)DA^{-3}$. Choosing $\eta=\varepsilon/2$ gives
\[
 |\Delta E|\ge D(2A)^{-2/\varepsilon}b(\varepsilon/2).
\]
Varying $\eta$ over $(0,\varepsilon)$ gives the full range of exponents in Corollary~\ref{loc:cor:power}.

\phantomsection\label{loc:subsec:endpoint}
The endpoint $p=1/\varepsilon$ corresponds to $\eta=0$, which lies outside the range of Theorem~\ref{main:normalized}. Nor does letting $p\downarrow1/\varepsilon$ in~\eqref{loc:eq:power} establish the endpoint, since no uniform positive lower bound for $b(\eta)$ as $\eta\downarrow0$ has been proved. Proposition~\ref{loc:prop:attainment} would remove the factor $2^{-p}$ but would not resolve this endpoint issue. No optimality of the stated powers is asserted.

\subsection{Frostman's theorem and the dimension consequence}\label{loc:sec:frostman}
Frostman's theorem provides a measure to which we can apply the quantitative distance bound.

We use the compact-set form of Frostman's theorem~\cite[Theorem~2.7, p.~18]{Mattila2015}. The measures in the class $\mathcal M(E)$ have positive finite mass and compact support contained in $E$; they need not be probabilities. We use the closed-ball estimate, which also bounds the open balls appearing in our hypotheses.

Let $E\subset\R^2$ be compact with $\dim_H E>1$, and choose
\[
 1<s<\min\{\dim_H E,2\}.
\]
Then $\mathcal H^s(E)>0$, so Frostman's theorem gives a finite positive measure $\lambda$ carried by $E$ such that
\begin{equation}\label{loc:eq:Frostman-published-form}
 0<m_0:=\lambda(E)<\infty,\qquad
 \lambda(\overline B(x,r))\le r^s
 \quad(x\in\R^2,\ r>0).
\end{equation}
Since $\dim_H E>1$, its diameter $D$ is positive. Define
\begin{equation}\label{loc:eq:Frostman-normalization}
 \mu=\lambda/m_0,\qquad D=\diam E,\qquad
 A=\max\{1,D^s/m_0\}.
\end{equation}
Then $\mu$ is a Borel probability carried by $E$, and for every $x\in\R^2$ and $0<r\le D$,
\begin{equation}\label{loc:eq:Frostman-adaptation}
 \mu(B(x,r))\le\frac{r^s}{m_0}
 =\frac{D^s}{m_0}(r/D)^s
 \le A(r/D)^s.
\end{equation}
Thus $\mu$ satisfies the hypothesis of Theorem~\ref{loc:thm:diameter} with $\varepsilon=s-1$. The normalization by $m_0$ is accounted for in $A$. In fact, for any $x\in E$, we have $E\subseteq\overline B(x,D)$, so~\eqref{loc:eq:Frostman-published-form} gives $m_0\le D^s$ and hence $A=D^s/m_0\ge1$.

\begin{corollary}\label{loc:cor:dimension}
If $E\subset\R^2$ is compact and $\dim_H E>1$, then $|\Delta E|>0$.
\end{corollary}
\begin{proof}
Choose $s$ and $\mu$ as above, and put $\varepsilon=s-1\in(0,1)$. Equation~\eqref{loc:eq:Frostman-adaptation} gives the required ball bound with a finite coefficient $A$. Applying Theorem~\ref{loc:thm:diameter} with $\eta=\varepsilon/2$, we obtain
\[
 |\Delta E|\ge D(2A)^{-2/\varepsilon}b(\varepsilon/2)>0.
\]
\end{proof}

The coefficient $A$ may depend on $E$ and the chosen Frostman measure. Thus the corollary gives positive length for each compact set of Hausdorff dimension greater than one, but no uniform lower bound depending only on dimension and diameter. The example in Subsection~\ref{main:concentration} shows why a uniform estimate must also account for concentration.

\subsection{Examples for the localization argument}\label{loc:sec:examples}
The following examples illustrate why localization requires normalization by the actual mass of the selected ball and why the distinction between open and closed balls matters.

Let $E=[0,1]\times\{0\}$, and let $\mu$ be normalized length measure on this segment. It satisfies~\eqref{loc:eq:growth-unit} with $s=1$ and $A=2$. Take $t=1/2$.

The mass of any closed ball of radius $r$ is at most $\min\{1,2r\}$. At the center $(1/2,0)$ and radius $1/2$, equality holds, giving $M_t=\sqrt2$. For the same center and $r_0=1/8$, we have $q_{\rm loc}=1/4$ and
\[
 \frac{q_{\rm loc}}{M_t r_0^t}=\frac12.
\]
Thus this ball satisfies the half-supremum condition.

Dividing the restricted measure by $M_t r_0^t=1/2$ would give total mass $1/2$, whereas division by the actual mass $q_{\rm loc}=1/4$ gives a probability. The rescaled set $F$ is a segment of length two contained in the closed unit ball. Since $0<t<s=1$, this illustrates the general localization theorem; the distance-set consequence requires $s>1$.

\phantomsection\label{loc:subsec:boundary-example}
For the second example, let $E$ be the circle of radius $1/2$ centered at the origin, and let $\mu$ be its normalized arclength measure. Then $\diam E=1$ and
\[
 \mu(B(0,1/2))=0,\qquad
 \mu(\overline B(0,1/2))=1.
\]
We claim that $\mu(B(x,r))\le4r$ for $0<r\le1$. If $0<r<1/4$ and $B(x,r)$ meets $E$, choose $p\in E\cap B(x,r)$. Every point of $E\cap B(x,r)$ then lies within chordal distance $2r$ of $p$. The normalized arclength of this circular neighborhood is
\[
 \frac2\pi\arcsin(2r)\le2r.
\]
The inequality follows from $\arcsin u\le(\pi/2)u$ for $0\le u\le1$. For $r\ge1/4$, the probability normalization gives $\mu(B(x,r))\le1\le4r$. Thus, for every $0<t<1$, this measure satisfies the hypotheses of the localization theorem with $s=1$ and $A=4$, while the open and closed balls can have different masses. The proof in Subsection~\ref{loc:sec:localization} retains the boundary mass.

Applying only the original $s$-power growth estimate to the rescaled test ball gives
\[
 \nu(B(z,u))\le q_{\rm loc}^{-1}Ar_0^s u^s.
\]
Since $q_{\rm loc}\le Ar_0^s$ gives no upper bound for the coefficient $Ar_0^s/q_{\rm loc}$, the supremum at the smaller exponent $t$ is needed to obtain~\eqref{loc:eq:rescaled-ball-proof}.

Together, Theorem~\ref{loc:thm:diameter} and Corollary~\ref{loc:cor:power} give the quantitative dependence on diameter and concentration, while Corollary~\ref{loc:cor:dimension} yields the positive-length conclusion for compact planar sets of Hausdorff dimension greater than one. This completes the proof of the main results.

\clearpage

\appendix

\section{Notation index}\label{main:notation-index}
\begingroup\small

\par\smallskip
\noindent\textbf{Measures, transforms, and preparation.}
\begin{longtable}{@{}P{31mm}P{94mm}P{19mm}@{}}
\toprule Symbol & Meaning & Page\\\midrule\endfirsthead
\toprule Symbol & Meaning & Page\\\midrule\endhead
$\Delta E$, $|\Delta E|$ & Distance set and its one-dimensional Lebesgue measure. & \pageref{main:distance-set}\\[2pt]
$B(z,r)$, $\overline B(z,r)$ & Open and closed Euclidean balls. & \pageref{main:balls}\\[2pt]
$T_\#\mu$, $\mu|_H$ & Pushforward and unnormalized restriction of a measure. & \pageref{main:measure-conventions}\\[2pt]
$\widehat\mu$, $\check f$ & Fourier transform with exponential $e^{-\ii x\cdot\xi}$ and its inverse convention. & \pageref{main:fourier}\\[2pt]
$\|\mu\|_{\TV}$ & Total-variation norm, with no factor one half. & \pageref{main:variation}\\[2pt]
$\vartheta,q$ & Smooth frequency cutoff and its inverse Fourier kernel. & \pageref{rec:eq:fixed-constants}\\[2pt]
$Q_n,P_n$ & Convolution truncation and difference of consecutive truncations. & \pageref{rec:eq:operators}\\[2pt]
$E_J,S_J$ & Tails of the total-variation errors and the integrals of the squared Fourier modulus on dyadic bands. & \pageref{rec:eq:E}\\[2pt]
$dw=d\theta$ & Unnormalized arclength on the unit circle. & \pageref{main:arclength}\\[2pt]
$\pi_x,P_w$ & Radial projection from $x$ and scalar orthogonal projection onto $w^\perp$. & \pageref{rad:sec:definitions}\\[2pt]
$I_a,U_a$ &
Energy and potential. &
\pageref{rad:eq:energy-def}\\[2pt]
$\mathfrak n_p$ & Extended density norm; infinite if no $L^p(dw)$ density exists. & \pageref{rad:eq:intrinsic}\\[2pt]
$\chi,s,u,t,p$ & Parameters and exponents in the radial projection estimate. & \pageref{rad:eq:parameters}\\[2pt]
$\qs,\qr$ & Lower bounds for original patch mass and final original retained mass. & \pageref{prep:eq:fixed-geometry}\\[2pt]
$K_j^0,K_j,\mu_j$ & Original compact restrictions, their similarity images, and the prepared probabilities. & \pageref{prep:eq:S}\\[2pt]
$\alpha,\tau_N$ & Fixed positive target and scale-dependent approximants in the reconstruction theorem. & \pageref{rec:thm:reconstruction}\\[2pt]

\bottomrule
\end{longtable}
\clearpage
\noindent\textbf{Finite-scale notation and retained measures.}
\begin{longtable}{@{}P{31mm}P{94mm}P{19mm}@{}}
\toprule Symbol & Meaning & Page\\\midrule\endfirsthead
\toprule Symbol & Meaning & Page\\\midrule\endhead
$Q_0,\D_n$ & Half-open root square and its depth-$n$ partition. & \pageref{dyad:sec:theorem}\\[2pt]
$\C_N$ & Family of retained classes at depth $N$. & \pageref{dyad:thm:main}\\[2pt]
$a_C,\rho_C$ & Original weight of class $C$ and its normalized restriction. & \pageref{dyad:eq:class-weight}\\[2pt]
$M_C(n),m_C,A_{C,n}$ & Model mass, mass profile, and excess over the one-dimensional rate. & \pageref{dyad:eq:model}\\[2pt]
$\rho_{C,P}$ & Class restriction normalized by its actual mass in $P$. & \pageref{dyad:eq:conditional}\\[2pt]
$q_{\rm d},J_{\rm d},n_i$ & Dyadic spacing, number of selected steps, and selected depths. & \pageref{dyad:eq:canonical-levels}\\[2pt]
$Q_N,e_N$ & Class-count bound and discarded mass. & \pageref{dyad:thm:main}\\[2pt]
$E_N$ & Union of retained classes. & \pageref{dyad:sec:weights}\\[2pt]
$I(a,b)$ & Increment of the mass profile. & \pageref{profile:eq:I}\\[2pt]
$k(a,b),\delta_A(a,b)$ & Final and initial excesses above the interval minimum. & \pageref{profile:eq:differences}\\[2pt]
$q_0,K_{\rm blk},r_j$ & Block size in depth indices, number of blocks, and their representative depths. & \pageref{lists:eq:blocks}\\[2pt]
$\Par{g}{p},\Proj{p}{t}$ & Indices for the parallel-square and projected-pair conditions in Section~\ref{main:directions}. Both specify depth $p$. & \pageref{lists:eq:ell}\\[2pt]
$\G(J),\M(\mathcal B)$ & Conditions generated by $J$ and their union over subintervals of $\mathcal B$. & \pageref{lists:eq:generated}\\[2pt]
$\M_{>a}(\mathcal B)$ & Conditions in the base with specified depth greater than $a$. & \pageref{lists:eq:active-list}\\[2pt]
$c,\beta$ & Starting depth shared by a class pair, chosen as the last depth where either profile excess is at most $\beta$. & \pageref{entry:eq:entry}\\[2pt]
$\xi,B_{\rm ch}(w,r)$ & Direction subprobability and open chordal balls on the circle. & \pageref{direction:eq:xi}\\[2pt]
$J_0,W_i(n)$ & Number of directional levels and their decreasing strip widths. & \pageref{direction:eq:widths}\\[2pt]
$\mathcal N_{i,g,p,P}$,\newline $\mathcal P_{i,p,t,P}$ & Count of occupied centers near a line and probability of nearly equal projections under the actual conditional law. & \pageref{direction:eq:parallel-stat}\\[2pt]
$G_{i,\mathfrak t}(P)$ & Directions satisfying condition $\mathfrak t$ in its specified square at level $i$. & \pageref{direction:eq:symmetry-levels}\\[2pt]
$C_{\rm test},C_{\rm av},N_{\rm exc}$ & Threshold coefficient, common average constant, and sufficient exceptional-set starting scale. & \pageref{direction:thm:main}\\

$D_\ell,\delta_\ell,\phi_\delta$ & Angular derivative scale, smoothing radius, and the integral-one circle kernel. & \pageref{cutoff:eq:angular-scale}\\[2pt]
$w_{\rho,i,\mathfrak t},K_s$ & Cutoff for class $\rho$, level $i$, and condition $\mathfrak t$, with common derivative constants $K_s$. & \pageref{cutoff:thm:cutoffs}\\[2pt]
$\Theta_N,\tau_N$ & Retained pair cutoff and resulting weighted distance approximant. & \pageref{mass:eq:Theta}, \pageref{mass:eq:measures}\\[2pt]
$F_1,F_2$ & Radial densities for the prepared probabilities. & \pageref{mass:eq:density-identities}\\[2pt]
$\Gamma_N$ & Auxiliary pair restriction used to estimate discarded mass. & \pageref{mass:eq:Gamma}\\[2pt]
$\xi_{1,x,N},\xi_{2,y,N}$ & Unnormalized restricted direction measures. & \pageref{mass:eq:xi}\\[2pt]
$N_e,a_e,C_e$ & Common mass-estimate threshold, rate, and prefactor. & \pageref{mass:thm:mass}\\[2pt]

\bottomrule
\end{longtable}
\clearpage
\noindent\textbf{Amplitudes and angular quantities.}
\begin{longtable}{@{}P{31mm}P{94mm}P{19mm}@{}}
\toprule Symbol & Meaning & Page\\\midrule\endfirsthead
\toprule Symbol & Meaning & Page\\\midrule\endhead
$\mathcal A_j,\mathfrak t_{j,t}$ & Factor index set at endpoint $j$ and geometric condition of factor $t$. Several factors may use one condition. & \pageref{amplitude:sec:data}\\[2pt]
$b_{j,z},f_j$ & Allowed amplitude and its bounded spatial multiplier. & \pageref{amplitude:eq:form}\\[2pt]
$\alpha_{j,t},B$ & Prescribed order of one factor and bound on the combined prescribed order at both endpoints. & \pageref{amplitude:eq:form}, \pageref{amplitude:eq:combined-order}\\[2pt]
$U_B,S_j$ & Uniform bound on amplitude modulus and scale for further differentiation. & \pageref{amplitude:eq:UB}\\[2pt]
$D_{\max},D_*$ & Largest cutoff scale and common derivative scale after absorbing the number of factors. & \pageref{amplitude:eq:Dstar}\\[2pt]
$U_X,V_Y$ & Endpoint transforms divided by their respective model masses. & \pageref{angular:eq:transforms}\\[3pt]
$\Psi,d\lambda_v,\psi$ & Fixed radial cutoff, positive radial measure, and separate angular cutoff. & \pageref{angular:eq:radial-cutoff}\\[2pt]
$\mathsf H_F,\mathsf H$ & Shared-angle and independent-angle expressions with the same frequency normalization. & \pageref{angular:eq:shared}\\[3pt]
$\mathcal B,[a,v]$ & Earlier base retaining the angular conditions, and current interval used for the reduction. & \pageref{angular:eq:base}\\[2pt]
$J\subseteq\mathcal B$ & A supplying interval generates a particular condition in the base. The supplying interval may extend past the current interval. & \pageref{lists:eq:full-list}\\[2pt]
$\mathcal H_{i,B}^{\mathcal B}$ & Supremum over the fixed family of allowed amplitudes, with a finite maximum over occupied child squares. & \pageref{angular:eq:supremum}\\[2pt]
$q_{\rm osc}$ & Frequency depth minus spatial depth, $v-a$. & \pageref{circle:eq:scale-data}\\[3pt]
$L_j,c_{\rm sp}$ & Direct circle expansion operators and Gaussian constant. Here $L_0=c_{\rm sp}\operatorname{Id}$. & \pageref{circle:eq:Lj}, \pageref{circle:eq:gaussian}\\[2pt]
$B_j$ & Finite inverse operators. & \pageref{circle:eq:Bj}\\[2pt]
$J_{\rm c},B_{\rm inv}$ & Number of inverse terms and combined output derivative order. & \pageref{circle:eq:orders}\\[3pt]
$\nu_{f,g}$ & Weighted distance measure in the local conversion. & \pageref{circle:eq:nu}\\[3pt]
$\mathcal K,d_+,d_-$ & Two-displacement radial integral and larger and smaller lengths. & \pageref{circle:eq:radial-integral}\\[3pt]
$h_{\rm g},q_{\rm g},\delta,\tau$ & Spatial and frequency depth gaps, angular radius, and graph spatial tolerance. & \pageref{graph:eq:local-scales}\\[3pt]
$E_j(P,\nu),D_{\rm gr}$ & Squared amplitude integrated over a square and angular cutoff, and bound for row or column degree. & \pageref{graph:eq:eligibility}, \pageref{graph:eq:degree}\\[2pt]
$C_G$ & Structural graph coefficient $C_{\rm test}+64+4C_{\rm prof}$. & \pageref{graph:eq:constant}\\
\bottomrule
\end{longtable}
\clearpage
\noindent\textbf{Positive distance counts and local frequency contributions.}
\begin{longtable}{@{}P{31mm}P{94mm}P{19mm}@{}}
\toprule Symbol & Meaning & Page\\\midrule\endfirsthead
\toprule Symbol & Meaning & Page\\\midrule\endhead
$K,c_-,c_+,\omega_-,\omega_+$ & Fixed square-separation parameter and distance and weight bounds. & \pageref{positive:eq:separated}\\[3pt]
$\theta_i^{X,Y}(a,t)$ & Unnormalized pair measure restricted by the level-$i$ conditions of $\M(a,t)$. & \pageref{positive:eq:positive-count}\\[2pt]
$\mathsf D_i^{X,Y}(a,t),C_0$ & Positive count of nearly equal distances, with $C_0=300$. & \pageref{positive:eq:positive-count}\\[3pt]
$B_\delta(\lambda),n_C$ & Sum of squared interval masses and translated-window factor. & \pageref{scalar:lem:bins}\\[3pt]
$K_b,T_b,K_\Psi$ & Squared-sinc function, triangular frequency coefficient, and smooth-frequency kernel. & \pageref{scalar:eq:triangular}\\[3pt]
$\zeta,S_v,\mathcal L,\mathcal U$ & One local parent measure, its band contribution, and the low and high frequency index sets. & \pageref{positive:eq:zeta}\\[2pt]
$B_{\rm proj},W_k,\lambda_k$ & Projection threshold and mass and measure of a coarse distance group in the suffix proof. & \pageref{positive:sec:suffix}\\[3pt]
$C_{\rm suf},C_{\rm lf},C_{\rm comp}$ & Fixed suffix, low-frequency, and scalar Fourier-comparison constants. & \pageref{positive:thm:three}\\[3pt]
$J_{\rm pos},B,S$ & Number of terms in the local inverse, its
combined output order, and the error exponent $S$. & \pageref{positive:eq:inverse-order}\\[2pt]

$L,L_0,h$ & Largest remaining condition length, the maximum of $L$ and $\lceil10\epsilon\rceil$, and splitting depth $h=v-2L_0$. & \pageref{lists:eq:direct-h}, \pageref{reduce:eq:h}\\[2pt]
$\mathcal N_h,\mathcal F_h$ & The signed near and far parts of the exact four-point expansion. & \pageref{reduce:eq:split-H}\\[3pt]
$\mathcal J(a,h,v),\Omega_b$ & Allowed far-child depths and their comparable-length bands. & \pageref{reduce:eq:J}\\[2pt]
$\mathcal E_b(Z),S_b(Z),D_b(Z)$ & Compatible ordered square pairs, sum of square roots of their counts, and maximum count. & \pageref{reduce:eq:far-sum}\\[2pt]
$G_K$ & Grid bound on the number of possible partners for a square. & \pageref{reduce:prop:far}\\[2pt]
$B_*$ & Common bound for the prescribed orders of the input and child amplitudes. & \pageref{reduce:eq:common-orders}\\[2pt]
$m_{\rm gr},J_{\rm dir}$ & Graph integration order and number of direct circle terms for the far contribution. & \pageref{reduce:sec:uniformity}\\[2pt]
\bottomrule
\end{longtable}
The coefficient $C_{\rm lf}$ bounds a low-frequency contribution and $C_{\rm low}$ enters the lower profile bound. The sum $S_J$ is a global Fourier tail, whereas $S_v$ is one contribution of a local parent. The actual square mass $\rho_C(P)$ normalizes the conditional probability, and the model $M_C(n)$ enters the analytic denominators.

\noindent The exponent $p$ of the radial estimate and the power $p$ in Section~\ref{main:localized-section} are local parameters. The cutoff bounds $K_s$ control derivatives taken to estimate errors as well as derivatives already prescribed in an amplitude.

\clearpage
\noindent\textbf{Repeated reductions and the single-profile bound.}
\begin{longtable}{@{}P{31mm}P{94mm}P{19mm}@{}}
\toprule Symbol & Meaning & Page\\\midrule\endfirsthead
\toprule Symbol & Meaning & Page\\\midrule\endhead
$\epsilon,\mathcal R$ & Frequency-depth gap $\epsilon=\sigma N$ and the fixed set of block representatives. & \pageref{iteration:eq:parameters}\\[2pt]
$J,J'$ & Actual current parent and child intervals in a transition. & \pageref{iteration:sec:transitions}\\[3pt]
$e$ & Additional profile exponent, excluding the change of starting value and structural losses. & \pageref{iteration:eq:suffix-e}\\[3pt]
$V(J)$ & Minimum sum of two excess values at depths separated by at least $10q_0$, defined when $J$ has width at least $10q_0$. & \pageref{iteration:eq:V}\\[2pt]
$J_*,a_f$ & Last interval with a defined minimum and the final starting depth of the branch or prefix. & \pageref{iteration:eq:endpoint}\\[2pt]
$n_S,n_L,n_H,$\newline$n_D,n_N,n_F$ & Counts of suffix, low, high, direct, near, and far transitions. & \pageref{iteration:sec:branch-count}\\[3pt]
$J_{\rm rec}$ & Uniform allowance for the number of transitions and a final stopping node. & \pageref{iteration:eq:Jrec}\\[3pt]
$d_i,h_{i,B}^{\mathcal B}$ & Global spatial maxima for positive counts and angular amplitude suprema for one class. & \pageref{summation:eq:d-global}\\[2pt]
$C_{\rm it},N_{\rm it}$ & Common local prefactor and common starting scale for iteration. & \pageref{summation:sec:local}\\[3pt]
$E_c,P_{\rm it}$ & Initial profile exponent and polynomial exponent $J_{\rm rec}+1$. & \pageref{summation:eq:theorem-data}\\[3pt]
$C_{\rm loss},C_*$ & Structural loss coefficients; the fixed choices give $4792$ and $5032$. & \pageref{summation:eq:theorem}\\[3pt]
$\Gamma,T_{\rm loc}$ & Upper bound for a transition exponent and
the local error exponent $T_{\rm loc}$. & \pageref{summation:eq:Gamma}\\[2pt]
\bottomrule
\end{longtable}
The block construction uses the size $q_0$ and count $K_{\rm blk}$. The dyadic decomposition uses spacing $q_{\rm d}$, and separated square pairs use $K=4096$. A branch follows one sequence of child terms. The full bound combines all contributions as described in Section~\ref{main:iteration}.

\clearpage
\noindent\textbf{Two profiles and the frequency estimate.}
\begin{longtable}{@{}P{31mm}P{94mm}P{19mm}@{}}
\toprule Symbol & Meaning & Page\\\midrule\endfirsthead
\toprule Symbol & Meaning & Page\\\midrule\endhead
$\tau_{C,D,N}$ & Component built from normalized classes $C,D$. It enters $\tau_N$ with weight $a_{1,C}a_{2,D}$. & \pageref{mass:sec:components}\\[2pt]
$V_j([c,N])$ & Two-depth minimum for profile $j$ on the interval beginning at the common starting depth. & \pageref{frequency:eq:V}\\[2pt]
$B_{\rm root}$ & Combined prescribed order allowed after the conversion at depth zero. & \pageref{frequency:eq:root-order}\\[2pt]
$B_{\rm in}$ & Initial order allowance for iteration. After duplicating an endpoint we take $B_{\rm in}=2B_{\rm root}$. & \pageref{summation:thm:iteration}, \pageref{frequency:eq:dup-order}\\[2pt]
$G_c,F_{C,D}$ & Root graph exponent and positive-frequency component integral. & \pageref{frequency:eq:graph-factor}\\[3pt]
$E_j(r),H_j$ & Angular energy at one endpoint and the quantity obtained by duplicating that endpoint. & \pageref{angular:eq:duplicates}\\[2pt]
$N_{\rm poly}$ & Threshold for absorbing the polynomial factor in the Fourier estimate. & \pageref{frequency:eq:Npoly}\\[2pt]
$C_s,N_s$ & Common Fourier-bound prefactor and sufficient starting scale. & \pageref{frequency:eq:frequency-conclusion}\\[2pt]
$N_*,a_\eta,\gamma_\eta$ & Joint starting scale and the total-variation and Fourier decay rates for the same sequence. & \pageref{frequency:cor:common}\\[2pt]
$J_\eta$ & Common reconstruction cutoff. & \pageref{normalized:eq:Jeta}\\[2pt]
$\mathcal B_\eta,b(\eta)$ & Reconstruction denominator and resulting uniform lower bound. & \pageref{normalized:eq:denominator}, \pageref{normalized:eq:b-eta}\\[2pt]
$\mathcal K_F$ & Compact cross-distance set of the prepared pair for the particular input $F$. & \pageref{normalized:eq:compact-set}\\
\bottomrule
\end{longtable}
The common starting depth $c(C,D)$ depends on both profiles. At depth $c$, each endpoint uses its own model mass. The final sum uses the original class weights.

\bigskip 

\noindent\textbf{Localization.}
\begin{longtable}{@{}P{31mm}P{94mm}P{19mm}@{}}
\toprule Symbol & Meaning & Page\\\midrule\endfirsthead
\toprule Symbol & Meaning & Page\\\midrule\endhead
$M_t$ & Supremum of the closed-ball mass divided by $r^t$, over all centers and $0<r\le1$. & \pageref{loc:eq:Mt}\\[3pt]
$x_0,r_0,q_{\rm loc}$ & Selected center, radius at diameter one, and actual mass used to normalize the restriction. & \pageref{loc:eq:local-probability}\\[3pt]
$z_*,\ell_0$ & Selected center and physical radius in original coordinates, with $\ell_0=Dr_0$. & \pageref{loc:eq:original-local-set}\\[3pt]
$c_{\varepsilon,p}$ & Length-bound constant at a fixed $p>1/\varepsilon$. & \pageref{loc:eq:power}\\[2pt]
\bottomrule
\end{longtable}
\endgroup
\clearpage

\section*{Acknowledgments}

G. C. is supported by the National Natural Science Foundation of China
(NSFC), grant no. 12371126. X. F. is  supported by the National Science
and Technology Council (NSTC), Taiwan, grant no. 114-2115-M-A49-003-MY3.

\section*{Use of generative artificial intelligence.}
OpenAI's ChatGPT was used extensively for proof exploration, mathematical drafting and expansion, notation and dependency checks, finite diagnostic programs, and preparation of the \LaTeX{} source. The diagnostic programs are not part of the proofs. All mathematical content was written, checked, and approved
by the authors.

\end{document}